\documentclass[11pt]{article}
\usepackage{smile}

\usepackage{fullpage}
\usepackage{lscape}
\usepackage{bigints}
\usepackage{framed}
\usepackage{mdframed}
\usepackage{enumerate}
\usepackage[inline]{enumitem}
\usepackage[T1]{fontenc}
\usepackage{moresize}
\usepackage{bm}
\usepackage{bbm}
\usepackage{dsfont}
\usepackage{amsmath}
\usepackage{amssymb}
\usepackage{amsthm}
\usepackage{amsfonts}
\usepackage{stmaryrd}
\usepackage{array}
\usepackage{mathrsfs}
\usepackage{mathtools} 
\usepackage{extarrows}
\usepackage{stackrel}
\usepackage{relsize,exscale}
\usepackage{scalerel}
\usepackage[nodisplayskipstretch]{setspace}
\usepackage{color}
\usepackage[usenames,dvipsnames]{xcolor}
\usepackage{cancel}
\usepackage{soul}
\usepackage{undertilde}
\usepackage{xfrac}
\usepackage{siunitx}
\usepackage{graphicx}
\usepackage{float}
\usepackage{rotating}
\usepackage{subcaption}
\usepackage{overpic}
\usepackage[all]{xy}
\usepackage{tikz}
\usetikzlibrary{arrows,matrix,positioning,calc,automata,patterns}
\usepackage{booktabs}
\usepackage{dcolumn}
\usepackage{multirow}
\usepackage{diagbox}
\usepackage{tabularx}
\usepackage{verbatim}
\usepackage{listings}
\usepackage{fancyvrb}
\usepackage{hyperref}
\usepackage[round]{natbib}
\usepackage{sectsty}

\hypersetup{
bookmarks=true,         
unicode=false,          
pdftoolbar=true,        
pdfmenubar=true,        
pdffitwindow=false,     
pdfstartview={FitH},    
pdftitle={My title},    
pdfauthor={Author},     
pdfsubject={Subject},   
pdfcreator={Creator},   
pdfproducer={Producer}, 
pdfkeywords={key1, key2}, 
pdfnewwindow=true,      
colorlinks=true,        
linkcolor=blue,         
citecolor=blue,         
filecolor=blue,         
urlcolor=cyan           
}

\usepackage{stackengine}
\stackMath
\newcommand\tenq[2][1]{%
\def\useanchorwidth{T}%
\ifnum#1>1%
\stackunder[0pt]{\tenq[\numexpr#1-1\relax]{#2}}{\!\scriptscriptstyle\thicksim}%
\else%
\stackunder[1pt]{#2}{\!\scriptstyle\thicksim}%
\fi%
}

\makeatletter
\DeclareRobustCommand\widecheck[1]{{\mathpalette\@widecheck{#1}}}
\def\@widecheck#1#2{%
\setbox\z@\hbox{\m@th$#1#2$}%
\setbox\tw@\hbox{\m@th$#1%
\widehat{%
\vrule\@width\z@\@height\ht\z@
\vrule\@height\z@\@width\wd\z@}$}%
\dp\tw@-\ht\z@
\@tempdima\ht\z@ \advance\@tempdima2\ht\tw@ \divide\@tempdima\thr@@
\setbox\tw@\hbox{%
\raise\@tempdima\hbox{\scalebox{1}[-1]{\lower\@tempdima\box
\tw@}}}%
{\ooalign{\box\tw@ \cr \box\z@}}}
\makeatother

\def\biggiven{\,\big{|}\,}
\def\Biggiven{\,\Big{|}\,}
\def\tr{\mathop{\text{tr}}\kern.2ex}

\def\tY{{\tilde Y}}

\def\tT{{\tilde T}}
\def\P{{\mathrm P}}

\def\E{{\mathrm E}}

\def\cov{{\rm Cov}}

\newcolumntype{L}[1]{>{\raggedright\let\newline\\\arraybackslash\hspace{0pt}}m{#1}}
\newcolumntype{C}[1]{>{  \centering\let\newline\\\arraybackslash\hspace{0pt}}m{#1}}
\newcolumntype{R}[1]{>{ \raggedleft\let\newline\\\arraybackslash\hspace{0pt}}m{#1}}
\newcolumntype{d}[1]{D{.}{.}{#1}}
\newcolumntype{H}{>{\setbox0=\hbox\bgroup}c<{\egroup}@{}}
\newcolumntype{Z}{>{\setbox0=\hbox\bgroup}c<{\egroup}@{\hspace*{-\tabcolsep}}}
\newcolumntype{b}{X}
\newcolumntype{s}{>{\hsize=.5\hsize}X}

\numberwithin{equation}{section}

\newtheorem{theorem}{Theorem}[section]
\newtheorem{lemma}{Lemma}[section]
\newtheorem{proposition}{Proposition}[section]
\newtheorem{assumption}{Assumption}[section]
\newtheorem{corollary}{Corollary}[section]
\newtheorem{lemmaA}{Lemma}[section]
\newtheorem{innercustomgeneric}{\customgenericname}
\providecommand{\customgenericname}{}
\newcommand{\newcustomtheorem}[2]{%
\newenvironment{#1}[1]
{%
\renewcommand\customgenericname{#2}%
\renewcommand\theinnercustomgeneric{##1}%
\innercustomgeneric
}
{\endinnercustomgeneric}
}
\newcustomtheorem{customdefinition}{Definition}
\newcustomtheorem{customdefinitions}{Definitions}
\newcustomtheorem{customtheorem}{Theorem}
\newcustomtheorem{customassumption}{Assumption}
\newcustomtheorem{customlemma}{Lemma}
\newcustomtheorem{customexample}{Example}
\theoremstyle{definition}

\newtheorem{remark}{Remark}[section]

\usepackage{enumitem}
\makeatletter
\newcommand{\mylabel}[2]{#2\def\@currentlabel{#2}\label{#1}}
\makeatother

\graphicspath{{./fig3/}}

\allowdisplaybreaks

\usepackage{algorithm}
\usepackage{algpseudocode}
\algrenewcommand\algorithmicrequire{\textbf{Input:}}
\algrenewcommand\algorithmicensure{\textbf{Output:}}

\def\cal{\mathcal}
\def\hat{\widehat}
\def\tilde{\widetilde}
\def\bar{\overline}

\def\conP{\stackrel{\mathrm{P}} {\to}}                 
\def\conD{\stackrel{\cal D} {\to}}

\def\iid{\mathrm{i.i.d.}}

\def\mathbbR{\mathbb{R}}

\def \tmu{\tilde{\mu}}

\def\conas{\stackrel{\mathrm{a.s.}} \to}                  

\def\d{{\, \mathrm{d}}}

\def\bbeta{{\boldsymbol{\beta}}}

\def\boldsymbolp{{\boldsymbol{p}}}

\def\bfx{\mathbf{x}}
\def\bfX{\mathbf{X}}
\def\bfZ{\mathbf{Z}}

\def\bfW{\mathbf{W}}

\def\calN{\mathcal{N}}

\def\calB{\mathcal{B}}
\def\calG{\mathcal{G}}
\def\calE{\mathcal{E}}

\def\calZ{\mathcal{Z}}

\def\calW{\mathcal{W}}
\def\mathbbR{\mathbb{R}}

\def \bw{\boldsymbol{w}}

\def \tildet{\tilde{t}}

\def \var{\mathrm{Var}}

\def\bX{{\boldsymbol{X}}}
\def\bY{{\boldsymbol{Y}}}
\def\bZ{{\boldsymbol{Z}}}
\def\bW{{\boldsymbol{W}}}
\def\bU{{\boldsymbol{U}}}
\def\bV{{\boldsymbol{V}}}
\def\Ind{\mathbf{1}}
\def\bx{{\boldsymbol{x}}}
\def\bz{{\boldsymbol{z}}}
\def\bw{{\boldsymbol{w}}}
\def\bu{{\boldsymbol{u}}}
\def\bv{{\boldsymbol{v}}}
\def \bfw{{\mathbf{w}}}
\def \tmu{\tilde{\mu}}
\def \tildeg{\tilde{g}}
\def \tildef{\tilde{f}}
\def \tsigma{\tilde{\sigma}}
\def \tcalB{\tilde{\mathcal{B}}}
\def \hatT{\hat{T}}
\def \hatL{\hat{L}}
\def \hatG{\hat{G}}
\def \hatU{\hat{U}}
\def \txi{\tilde{\xi}}
\def \lbr{\llbracket}
\def \rbr{\rrbracket}
\def \supp{\mathrm{supp}}

\def \txi{\tilde{\xi}}
\def \tT{\tilde{T}}

\def \barY{\bar{Y}}
\def \hsigma{\hat{\sigma}}

\begin{document}

\setlength{\abovedisplayskip}{5pt}
\setlength{\belowdisplayskip}{5pt}
\setlength{\abovedisplayshortskip}{5pt}
\setlength{\belowdisplayshortskip}{5pt}
\hypersetup{colorlinks,breaklinks,urlcolor=blue,linkcolor=blue}

\title{\LARGE Limit theorems of Azadkia-Chatterjee's conditional graph correlation}

\author{Muhong Gao\thanks{School of Statistics, University of International Business and Economics, Beijing, China; e-mail: {\tt gaomh@uibe.edu.cn}}, ~~Fang Han\thanks{Department of Statistics, University of Washington, Seattle, WA 98195, USA; e-mail: {\tt fanghan@uw.edu}}, ~ and ~Qizhai Li\thanks{Academy of Mathematics and System Science, Chinese Academy of Sciences, Beijing, China; e-mail: {\tt liqz@amss.ac.cn}}}

\date{\today}

\maketitle

\vspace{-1em}

\begin{abstract} 
Inferring the strength of conditional dependence and testing conditional independence are fundamental problems in statistics. A recent breakthrough by Azadkia and Chatterjee introduced, for the first time, a conditional dependence measure that equals $0$ if and only if the variables under study are conditionally independent, and equals $1$ if and only if they are conditionally perfectly dependent. They further proposed a computationally efficient and strongly consistent estimator, $T_n$, based on an ingenious use of ranks and nearest neighbors. Despite these attractive features, the asymptotic theory of $T_n$ has remained largely undeveloped. This paper closes that gap. We prove that, under general dependence, $T_n$ is asymptotically normal and its limiting variance admits a closed form. We also construct consistent variance estimators that are computationally efficient and implementable in $O(n\log n)$ time. Taken together with existing bias-correction methods, these results provide a complete inferential theory for $T_n$.
\end{abstract}


\section{Introduction} \label{sec:intro}
Consider the random triplet $(Y,\bX,\bZ)$, where $Y\in\mathbb{R}$ is a random scalar, and $\bX\in\mathbb{R}^p$ and $\bZ\in\mathbb{R}^q$ are random vectors of dimensions $p$ and $q$, respectively. Our goal is to infer the strength of conditional dependence between $Y$ and $\bX$ given $\bZ$, as well as to test the null hypothesis
\begin{eqnarray}
H_0:\ \text{$Y$ is conditionally independent of $\bX$ given $\bZ$,}
\label{eq:null}
\end{eqnarray}
on the basis of $n$ independent copies $(Y_i,\bX_i,\bZ_i)$'s of $(Y,\bX,\bZ)$.

While the problem of testing \eqref{eq:null} has been studied extensively in the literature, the problem of quantifying the strength of conditional dependence is arguably equally important. In this direction, the recent work of \cite{azadkia2019simple} constitutes a major breakthrough. Building on ideas from \cite{MR3024030} and \cite{chatterjee2020new} for measuring unconditional dependence, they introduced the following population quantity, where $\P_Y$ denotes the law of $Y$:

\begin{eqnarray}
T = T(Y,\bX \mid \bZ) 
=
\frac{\int \E\Big[\var\Big\{\P(Y \ge y \mid \bX,\bZ)\,\big|\,\bZ\Big\}\Big] \, \d \P_Y(y)}
{\int \E\Big[\var\Big\{\Ind(Y \ge y)\mid \bZ\Big\}\Big] \, \d \P_Y(y)}.
\label{eq:T}
\end{eqnarray}
Azadkia and Chatterjee proved that $T=0$ if and only if $Y$ is conditionally independent of $\bX$ given $\bZ$, whereas $T=1$ if and only if $Y$ is almost surely a measurable function of $\bX$ given $\bZ$. To the best of our knowledge, this is the \emph{first} measure of conditional dependence that captures the full range of dependence strength in this manner.

What makes the contribution of \cite{azadkia2019simple} even more striking is the accompanying statistical estimator. Let $R_i$ denote the rank of $Y_i$ among $\{Y_j\}_{j=1}^n$, let $N(i)$ index the nearest neighbor (NN) of $\bZ_i$, and let $M(i)$ index the nearest neighbor of $(\bX_i,\bZ_i)$, with both nearest neighbors defined under the Euclidean metric. Azadkia and Chatterjee introduced the following rank/graph-based statistic, which we refer to as the ``Azadkia--Chatterjee conditional graph correlation'',
\begin{eqnarray}
T_n = T_n(Y,\bX \mid \bZ)
=
\frac{\sum_{i=1}^n \bigl(\min\{R_i,R_{M(i)}\}-\min\{R_i,R_{N(i)}\}\bigr)}
{\sum_{i=1}^n \bigl(R_i-\min\{R_i,R_{N(i)}\}\bigr)},
\label{eq:T_n}
\end{eqnarray}
as a strongly consistent for $T$. Moreover, $T_n$ possesses several notable advantages:
\begin{enumerate}[label=(\roman*)]
\item it is fully nonparametric and tuning-parameter-free;
\item it completely avoids the need to estimate conditional densities, conditional characteristic functions, or mutual information;
\item it is computable in $O(n\log n)$ time.
\end{enumerate}
These features, together with the conceptual appeal of $T$, make $T_n$ an especially attractive tool for quantifying conditional dependence.

At the same time, $T_n$ has clear and important limitations. As noted in \cite{azadkia2019simple}, it generally converges to $T$ at a subparametric rate, and no limit theory has been available for $T_n$. Consequently, one cannot directly quantify the uncertainty arising from random sampling. Resolving this difficulty is by no means routine, and for years after \cite{azadkia2019simple}, the inferential theory of $T_n$ remained open.

The goal of this paper is to resolve this issue in a definitive manner. Our main contributions are threefold:
\begin{enumerate}[label=(\roman*)]
\item under conditions on the joint distribution $F = F_{\bX,Y,\bZ}$ of $(\bX,Y,\bZ)$, we prove that $T_n$ or its bias-corrected version is asymptotically normal (Theorem~\ref{thm:CLT-main} and Corollary~\ref{cor: CLT-Tn});
\item in addition, the limiting variance of $\sqrt{n}\,T_n$ exists and admits a closed form (the same Theorem~\ref{thm:CLT-main} and Corollary~\ref{cor: CLT-Tn});
\item moreover, a consistent variance estimator exists that is computationally efficient and can be implemented in $O(n\log n)$ time (Theorem \ref{thm: est_var-main} and Proposition \ref{prop:nlogn_new}).
\end{enumerate}
Combined with the bias-correction method developed in \cite{azadkia2026biascorrection}, these results provide a complete inferential theory for $T_n$.

\subsection{Related literature}

The study of conditional dependence is intimately related to that of unconditional dependence. Indeed, when $\bZ$ is degenerate, the Azadkia--Chatterjee conditional graph correlation $T_n(Y,\bX \mid \bZ)$ reduces to an unconditional graph correlation between $Y$ and $\bX$, which we denote by $\xi_n=\xi_n(Y, \bX)$. 
It is therefore natural to discuss the related literature on both conditional and unconditional dependence together.

We first discuss the seminal work of \cite{chatterjee2020new} on quantifying unconditional dependence between $Y$ and $X$ (when $p=1$), together with the subsequent work of \cite{azadkia2019simple} on conditional and unconditional dependence. They have by now generated a large and steadily expanding literature, and below we try to give a brief and highly selective review.

\begin{enumerate}[label=(\roman*)]
\item The general asymptotic normality of $\xi_n$ under arbitrary dependence between $\bX$ and $Y$ was established in \cite{Lin_Han_2025_CLT} via direct moment calculations. Subsequently, \cite{kroll2024asymptotic} and \cite{chhaibi2026martingaleapproachfluctuationsrank} developed complementary theory based on empirical-process and martingale methods, respectively, for the analysis of Chatterjee's rank correlation \citep{chatterjee2020new}.

\item From the perspective of statistical inference, \cite{Lin_Han_2024_boostrap} showed that the classical bootstrap generally fails for $\xi_n$, whereas \cite{Dette_Kroll_2025} and \cite{olivares2025powerful} established the consistency of alternative resampling procedures, namely the $m$-out-of-$n$ bootstrap and the multiplier bootstrap, under different regimes. The closed-form expression for the limiting variance of $\xi_n$ under unconditional independence was derived in \cite{Shi_Drton_Han_2024_Bernoulli} and \cite{han2024azadkia}. Large random matrix theory for matrices built from Chatterjee's rank correlation was developed in \cite{dong2025spectral}.

\item As for statistical efficiency, the (Azadkia--)Chatterjee approach has generally been found to be underpowered for testing marginal independence in regular statistical models \citep{cao2020correlations,shi2020power,Shi_Drton_Han_2024_Bernoulli}, even though it is rate-optimal for estimating the corresponding population quantity \citep{auddy2021exact,Lin_Han_2025_CLT}. See also \cite{azadkia2026kernel} for a re-examination of the kernel-based estimator of \cite{MR3024030} and a discussion of its statistical efficiency. 

\item On the methodological side, the NN graph-based framework introduced in \cite{azadkia2019simple} has inspired a variety of follow-up works. These include, among many others, \cite{deb2020kernel}, \cite{huang2022kernel}, \cite{chatterjee2024kernel}, and \cite{roudaki2026kernel} on combining kernels with graph-based methods; \cite{gamboa2022global}, which extends the idea to sensitivity analysis; \cite{lin2021boosting}, which advocates incorporating multiple NNs into estimation; \cite{hormann2026azadkia}, which extends the framework to functional data; \cite{tran2024rank}, which proposes rank-based metrics for NN graph construction; and \cite{ansari2022direct} and \cite{huang2025multivariate}, which extend the setting to multivariate $\bY$.

\item An equally active line of research concerns the population quantity $T$ itself, as well as related alternatives in the setting of unconditional dependence. Representative contributions include \cite{strothmann2022rearranged}, \cite{bucher2024lack}, \cite{ansari2025exact}, \cite{chierichetti2025metricity}, \cite{ansari2025ordering}, \cite{fuchs2025exact}, and \cite{ansari2026quantifying}, among many others.
\end{enumerate}

Notably, the existing literature has so far focused predominantly on the setting of unconditional dependence, in which the conditioning variable $\bZ$ is absent. To the best of our knowledge, the main exceptions are \cite{Shi_Drton_Han_2024_Bernoulli}, which studied the use of $T_n$ for testing \eqref{eq:null} within the conditional randomization test framework; \cite{huang2022kernel}, which proposed a class of conditional dependence measures by combining graph-based and kernel-based ideas; and \cite{azadkia2025new}, which introduced refined versions of $T$ and $T_n$. Even so, inferential results remain unavailable beyond the simple setting in which $Y$ is further assumed to be independent of $(\bX,\bZ)$.

Concerning the task \eqref{eq:null}, the present paper is also inevitably connected to the vast literature on testing conditional independence, and the (bias-corrected) statistic $T_n$ does yield a consistent test of \eqref{eq:null}. Our work therefore also complements the broad class of nonparametric, consistent conditional independence tests developed in \cite{MR2413488}, \cite{MR3449068}, \cite{10.5555/3020548.3020641}, \cite{cai2022distribution}, and \cite{zhang2026doubly}, among many others. At the same time, it is worth noting that these methods are not designed to consistently capture conditional perfect dependence, and their implementations are typically quadratic in $n$ or more expensive.

\subsection{Technical ingredients}

The present work builds on several earlier contributions, especially \cite{Lin_Han_2025_CLT} and \cite{azadkia2026biascorrection}, which established the asymptotic normality of the unconditional version of $T_n$ and resolved the corresponding bias-correction issue, respectively. It is therefore worth clarifying more explicitly what is technically new in the current paper.

Our first main technical contribution is a central limit theorem (CLT) for the Azadkia--Chatterjee conditional correlation coefficient $T_n$ in general settings. The main difficulty here is to handle the interaction between the following two terms in the nominator of \eqref{eq:T_n}:
\[
\sum_{i=1}^n \min\{R_i,R_{M(i)}\}
\qquad\text{and}\qquad
\sum_{i=1}^n \min\{R_i,R_{N(i)}\},
\]
when $(\bX,Y,\bZ)$ is allowed to be arbitrarily dependent. This issue is in our opinion technically far more challenging than in the unconditional setting considered in \cite{Lin_Han_2025_CLT}, and took substantially additional work of ours. In fact, resolving it necessitates sharpening several results from \cite{Lin_Han_2025_CLT} and \cite{Shi_Drton_Han_2024_Bernoulli}; these improvements are highlighted in Section~\ref{sec:Lin_Han} below.

Our second main contribution is the identification of a closed-form expression for the limiting variance of $T_n$. More precisely, we show that this variance can be represented explicitly as a functional of the joint distribution of $(\bX,Y,\bZ)$; see \eqref{eq:sigma2_tTn}, \eqref{eq:han-Tn-var}, and \eqref{eq:han-Tn-var2} below. Moreover, under $H_0$ in \eqref{eq:null}, a further simplification exists; see \eqref{eq:sigma2_H0} below. These explicit characterizations in turn enable us to construct a consistent and computationally efficient variance estimator for $T_n$ with $O(n\log n)$ complexity. Relative to the variance estimator proposed in \cite{Lin_Han_2025_CLT} and the $m$-out-of-$n$ bootstrap considered in \cite{Dette_Kroll_2025}, this provides a computationally more efficient inferential tool.

\subsection{Paper organization and notation}

\paragraph*{Paper organization.}

The rest of the paper is organized as follows. Section~\ref{sec:Lin_Han} revisits the limit theorems established in \cite{Lin_Han_2025_CLT} for the Azadkia--Chatterjee \emph{unconditional} correlation coefficient, and presents our new findings related to this unconditional dependence measure. Section~\ref{sec:SI} introduces the proposed inferential framework for the Azadkia--Chatterjee conditional correlation coefficient. Section~\ref{sec:theory} develops the corresponding theory. Section~\ref{sec:simu} reports numerical experiments illustrating the finite-sample performance of the proposed procedure. All proofs are deferred to the Appendix.

\paragraph*{Notation.}

For any integer $n \ge 1$, let $\lbr n \rbr = \{1,2,\dots,n\}$. A set consisting of distinct elements $x_1,\dots,x_n$ is written either as $\{x_1,\dots,x_n\}$ or as $\{x_i\}_{i=1}^n$. For a real random vector $\bW$, let $\P_{\bW}$, $F_{\bW}$, and $\supp(\bW)$ denote its induced probability measure, cumulative distribution function, and support, respectively. We write $\Ind(\cdot)$ for the indicator function. For a vector $\bv \in \mathbb{R}^d$, let $\|\bv\|$ denote its Euclidean norm. For any $a,b \in \mathbb{R}$, define $a \vee b = \max\{a,b\}$ and $a \wedge b = \min\{a,b\}$. For a finite set $A$, let $\#A$ or $|A|$ denote its cardinality. The symbols $\lfloor \cdot \rfloor$ and $\lceil \cdot \rceil$ denote the floor and ceiling functions. 
For a random variable $U$ and a random vector $\bV$, let $\mu_U$ denote the law of $U$ and let $\mu_{U \mid \bV}$ denote the conditional law of $U$ given $\bV$. 
Finally, $\conas$, $\conP$, and $\conD$ denote convergence almost surely, in probability, and in distribution, respectively. Unless stated otherwise, the terms ``absolutely continuous'' and ``almost everywhere'' are understood with respect to Lebesgue measure.

\section{Revisiting \cite{Lin_Han_2025_CLT}: CLT of the Azadkia-Chatterjee's unconditional correlation coefficient} \label{sec:Lin_Han}

This section reviews and refines the CLT for Azadkia–Chatterjee’s unconditional correlation coefficient $\xi_n(Y,\bX)$. 
To facilitate comparison with $T_n$ in later sections and to maintain notational consistency throughout the paper, in this section we use $\bZ$ in place of $\bX$ and write $\xi_n(Y,\bZ)$ instead of $\xi_n(Y,\bX)$.

Specifically, in this section let $Y$ be a real-valued random variable and $\bZ$ be a random vector in $\mathbbR^q$, both defined on the same probability space. Let $(Y_1,\bZ_1),\dots,(Y_n,\bZ_n)$ be $n$ independent copies of $(Y,\bZ)$. Under the assumption that $(Y,\bZ)$ is continuously distributed, Azadkia-Chatterjee's unconditional correlation coefficient proposed by \cite{azadkia2019simple} is defined as
\begin{eqnarray}
\xi_n(Y,\bZ) := \frac{6}{n^2-1} \sum_{i=1}^n \min\{R_i, R_{N(i)}\}-\frac{2n+1}{n-1},
\label{eq:xi_n}
\end{eqnarray}
where, as before, $R_i$ denotes the rank of $Y_i$ among $\{Y_j\}_{j=1}^n$, and $N(i)$ denotes the index of the first NN of $\bZ_i$ among $\{\bZ_j\}_{j=1}^n$ under the Euclidean metric. Note that $\xi_n$ extends Chatterjee's original correlation coefficient \citep{chatterjee2020new} from the univariate setting $p=1$ to the multivariate setting $p\ge 1$. As $n\to\infty$, $\xi_n$ converges almost surely to the population quantity
\begin{eqnarray}
\xi(Y,\bZ) := \frac{\int \var\Big[\E \big\{ \Ind(Y\ge y) \mid \bZ \big\}\Big] \d \P_Y(y)}
{\int \var\big\{ \Ind(Y\ge y)\big\} \d \P_Y(y)},
\label{eq:xi}
\end{eqnarray}
which is also known as the Dette--Siburg--Stoimenov dependence measure \citep{MR3024030}.

In the special case where $\bZ$ and $Y$ are independent, a CLT for $\xi_n$ was first established in \citet[Theorem~3.1(ii)]{Shi_Drton_Han_2024_Bernoulli}. The subsequent work of \cite{Lin_Han_2025_CLT} extended this result to the general setting in which $\bZ$ and $Y$ may be arbitrarily dependent. We summarize the main conclusions of \cite{Lin_Han_2025_CLT} in Proposition~\ref{prop:Lin_Han} below.

\begin{proposition}[Summary of results in \cite{Lin_Han_2025_CLT}]
\label{prop:Lin_Han}
Assume that $F_{Y,\bZ}$ is fixed and continuous.
\begin{enumerate}[label=(\roman*)]
\item The limiting variance $\sigma^2_{\xi(Y,\bZ)} := \lim_{n \to \infty} n \var(\xi_n)$ exists. Moreover, $\sigma^2_{\xi(Y,\bZ)} > 0$ if and only if $Y$ is not almost surely a measurable function of $\bZ$.
\item There exists a consistent estimator $\tsigma^2$ of $\sigma^2_{\xi(Y,\bZ)}$. Moreover, $\tsigma^2$ can be computed in $O(n^2)$ time.
\item If $Y$ is not almost surely a measurable function of $\bZ$, then, as $n \to \infty$,
\begin{eqnarray*}
\frac{\xi_n - \E(\xi_n)}{\sqrt{\var(\xi_n)}} \conD N(0,1).
\end{eqnarray*}
\end{enumerate}
\end{proposition}

\vspace{0.2cm}

Despite these results, two notable gaps remain in \cite{Lin_Han_2025_CLT}: (1) no closed-form expression for $\sigma^2_{\xi(Y,\bZ)}$ is available; and (2) the estimator $\tsigma^2$ in \cite{Lin_Han_2025_CLT} requires $O(n^2)$ computational time, which is substantially more demanding than the $O(n\log n)$ complexity typically associated with rank- and graph-based statistics, and thus limits its practical usefulness.

As a byproduct of developing our general inferential theory for $T_n$, we resolve both of these issues in a definitive manner. We present these improvements to \cite{Lin_Han_2025_CLT} first, before turning to the general theory of $T_n$. We hope that this presentation makes it clearer that the present paper is not merely an extension of \cite{Lin_Han_2025_CLT} to the setting of conditional dependence.

\subsection{New probabilistic results on NNGs}

To derive the closed-form expression for $\sigma^2_{\xi(Y,\bZ)}$, we begin by reviewing and establishing several probabilistic asymptotic results for nearest neighbor graphs (NNGs), which form the foundation for our subsequent analysis.

Our first result in this section concerns a sample $\{\bW_i\}_{i=1}^n$ consisting of $n$ independent copies of a random vector $\bW \in \mathbbR^d$. Let $\calG_n$ denote the associated directed nearest-neighbor graph (NNG) with vertex set $\lbr n \rbr$. A directed edge $i \to j$ is drawn between two distinct vertices $i$ and $j$ whenever $\bW_j$ is the nearest neighbor of $\bW_i$. Denote by $\calE(\calG_n)$ the edge set of $\calG_n$.

We begin by recalling Theorem 1 of \cite{MR937563} on the expected number of mutual NN pairs. 

\begin{lemma}[\cite{MR937563}, expected number of mutual NN pairs]
\label{lemma:q_d}
Assume that $\bW$ is Lebesgue absolutely continuous. Then, for any fixed $i$, we have
\begin{eqnarray*}
\E\Big(\#\big\{j \in \lbr n \rbr: i\to j,\ j \to i \in \calE(\calG_n)\big\} \ \Big | \ \bW_i \Big)  \conP \mathfrak{q}_d,
\end{eqnarray*}
where $\mathfrak{q}_d$ is a positive constant depending only on $d$, with explicit expression
\begin{eqnarray}
\mathfrak{q}_d := \Big\{2- I_{3/4}\Big(\frac{d+1}{2}, \frac{1}{2}\Big)\Big\}^{-1}, 
\qquad 
I_{x}(a,b) = \frac{\int_{0}^x t^{a-1} (1-t)^{b-1} \d t}{\int_{0}^1 t^{a-1} (1-t)^{b-1} \d t}.
\label{eq:q_d}
\end{eqnarray}
Consequently,
\begin{eqnarray*}
\E\Big( \frac{1}{n} \#\big\{(i,j) \text{ distinct}: i\to j,\ j \to i \in \calE(\calG_n)\big\}  \Big) \to \mathfrak{q}_d.
\end{eqnarray*}
\end{lemma}

Note that the first statement may equivalently be written as $\P\{N(N(i))=i \mid \bW_i\} \conP \mathfrak{q}_d$, that is, the conditional probability that the NN of $\bW_i$ has $\bW_i$ itself as its NN converges to $\mathfrak{q}_d$, regardless of the specific value of $\bW_i$.

We are now ready to present our first new result, which extends the argument of Lemma~\ref{lemma:q_d} to the conditional expectation of the number of shared nearest-neighbor triplets.

\begin{lemma}[Conditional expected number of shared-NN triplets]
\label{lemma:o_d}
Assume that $\bW$ is Lebesgue absolutely continuous and admits a continuous density on its support. Then, for any fixed $i$, as $n \to \infty$,
\begin{eqnarray*}
\E\Big(\#\big\{j\in \lbr n \rbr: j \to k,\ i \to k \in \calE(\calG_n) \big\} \ \Big | \ \bW_i \Big) \conP \mathfrak{o}_d,
\end{eqnarray*}
where $\mathfrak{o}_d$ is a positive constant depending only on $d$, with explicit expression
\begin{eqnarray}
\mathfrak{o}_d := \int_{\Gamma_{d;2}}
\exp\Big[ - \lambda \big\{ \calB(\bw_1, \|\bw_1\|) \cup \calB(\bw_2, \|\bw_2\|)\big\}\Big]\d(\bw_1, \bw_2),
\label{eq:o_d}
\\
\Gamma_{d;2} :=
\Big\{(\bw_1, \bw_2) \in (\mathbbR^d)^2: \max(\|\bw_1\|, \|\bw_2\|) < \|\bw_1-\bw_2\|\Big\},
\nonumber
\end{eqnarray}
with $\calB(\bw,r)$ denoting the ball of radius $r$ centered at $\bw$, and $\lambda(\cdot)$ denoting Lebesgue measure.
\end{lemma}

Note that, by the bounded convergence theorem, together with the well-known fact that the maximum degree of an NNG is bounded \citep{MR682809}, Lemma~\ref{lemma:o_d} immediately recovers the existing result on the unconditional expectation from \cite{MR914597}:
\begin{eqnarray}
\E\Big(\frac{1}{n}\#\big\{(i,j,k) \text{ distinct}: j \to k,\ i \to k \in \calE(\calG_n) \big\} \Big) \to \mathfrak{o}_d.
\label{eq:E_to_od}
\end{eqnarray}

Table~\ref{table1} reports the values of $\mathfrak{q}_d$ and $\mathfrak{o}_d$ for the first ten dimensions, updating the calculations reported in \citet[Table~1]{han2024azadkia}.

\begin{table}[htbp]
\centering
\caption{The first $10$ values of $\mathfrak{q}_d$ and $\mathfrak{o}_d$.
Specifically, $\mathfrak{q}_d$ is computed by numerical integration according to \eqref{eq:q_d}, whereas $\mathfrak{o}_d$ is estimated by Monte Carlo simulation based on \eqref{eq:E_to_od} with $n = 10^7$.}
\label{table1}
\begin{tabular}{ccccccccccc}
\hline
$d$ & 1 & 2 & 3 & 4 & 5 & 6 & 7 & 8 & 9 & 10 \\
\hline
$\mathfrak{q}_d$ & 0.667 & 0.622 & 0.593 & 0.573 & 0.558 & 0.547 & 0.538 & 0.531 & 0.528 & 0.521 \\
$\mathfrak{o}_d$ & 0.500 & 0.633 & 0.709 & 0.763 & 0.805 & 0.840 & 0.871 & 0.898 & 0.923 & 0.946 \\
\hline
\end{tabular}
\end{table}

\vspace{0.5cm}

Our second result concerns a setting involving two NNGs, generated from the full sample and from a subsample, respectively. Such configurations arise repeatedly in the analysis of the conditional correlation coefficient $T_n$.

To describe this setting, consider a sample $\{\bW_i\}_{i=1}^n$, where each $\bW_i = (\bU_i,\bV_i)$ is independently drawn from the random vector $\bW=(\bU,\bV)$, with $\bU \in \mathbbR^{d_1}$ and $\bV \in \mathbbR^{d_2}$. Let $\calG^{\bW}_n$ denote the NNG associated with the full sample $\{\bW_i\}_{i=1}^n$, and let $\calG^{\bU}_n$ denote the NNG associated with the subsample $\{\bU_i\}_{i=1}^n$. Lemma~\ref{lemma:two_NNGs} below establishes the convergence of the conditional expected number of shared-NN triplets across the two graphs $\calG^{\bW}_n$ and $\calG^{\bU}_n$.

\begin{lemma}[Shared-NN triplets across two NNGs]
\label{lemma:two_NNGs}
Assume that $\bW$ is Lebesgue absolutely continuous and admits a continuous density on its support. Then, for each fixed $i$, as $n \to \infty$,
\begin{eqnarray*}
\E\Big(\#\big\{j\in \lbr n \rbr: j \to k \in \calE(\calG^\bU_n),\ i \to k \in \calE(\calG^\bW_n) \big\} \ \Big | \ \bW_i \Big) \conP 1.
\end{eqnarray*}
\end{lemma}

Of note, the unconditional version
\begin{eqnarray*}
\E\Big(\frac{1}{n}\#\big\{(i,j,k) \text{ distinct}: j \to k \in \calE(\calG^\bU_n),\ i \to k \in \calE(\calG^\bW_n) \big\}\Big) \to 1,
\end{eqnarray*}
was previously established in \citet[Lemma 7.4]{Shi_Drton_Han_2024_Bernoulli}. As in Lemmas~\ref{lemma:q_d} and \ref{lemma:o_d}, here we show that the asymptotic behavior of the corresponding conditional expectation is invariant with respect to the specific value of $\bW_i$.

Lemmas~\ref{lemma:q_d} and \ref{lemma:o_d} will be used to derive the closed-form expression for the limiting variance of $\xi_n$, whereas Lemma~\ref{lemma:two_NNGs} will be used to derive the closed-form expression for the limiting variance of $T_n$ in Section~\ref{sec:theory}.

\subsection{Closed-form expression for the limiting variance of $\xi_n$}

In what follows, let $\tY_i$, $\tY'_i$, and $\tY''_i$ denote copies of $Y_i$ such that, conditional on $\bZ_i$, they are independently and identically distributed ($\iid$) according to the conditional distribution of $Y_i$ given $\bZ_i$.

Theorem~\ref{thm:var_xi} below constitutes the main theoretical result of this section. It provides an explicit expression for the limiting variance of $\xi_n$ when $\bZ$ and $Y$ are possibly dependent. In this way, it complements \cite{Lin_Han_2025_CLT} and further extends the corresponding results of \cite{Shi_Drton_Han_2024_Bernoulli} and \cite{chhaibi2026martingaleapproachfluctuationsrank} to the settings of dependent pairs and multivariate $\bZ$, respectively. 

\begin{theorem}[Asymptotic variance of $\xi_n$ under dependence] \label{thm:var_xi}
Assume that $F_{Y,\bZ}$ is fixed and continuous. Assume further that $\bZ$ is Lebesgue absolutely continuous and admits a continuous density function on its support. We then have
\begin{align}
\hspace{-0.9cm}\sigma^2_{\xi(Y,\bZ)} &:=\lim_{n \to \infty} n \var\big\{\xi_n(Y,\bZ)\big\} \cr
=& \ 36 \, \Big\{	(1+\mathfrak{q}_q) \,T_1 + (2-2\mathfrak{q}_q + \mathfrak{o}_q) \, T_2 - (2-\mathfrak{q}_q + \mathfrak{o}_q) \, T_3 +4 \, T_4 - 2 \, T_5 + T_6 - 4 \, T_7 \Big\}, \label{eq:sigma2_xi}
\end{align}
where
\begin{align}
T_1 &=  \E\big\{ F_Y^2(Y \wedge \tY)\big\}, \quad & T_2 &= \E\big\{F_Y(Y \wedge \tY) \cdot F_Y(Y \wedge\tY')\big\},  \cr
T_3 &= \E\big\{F_Y(Y \wedge \tY) \cdot F_Y(\tY' \wedge\tY'')\big\}, \quad & T_4 &= \E\big\{\Ind(Y_1 \leq Y_2 \wedge \tY_2)\cdot F_Y(Y_1 \wedge \tY_1)\big\},  \cr
T_5 &=  \E\big\{\Ind(Y_1 \leq Y_2 \wedge \tY_2)\cdot F_Y(\tY_1 \wedge \tY_1')\big\}, \quad & T_6 &= \E\big\{F_Y(Y_1 \wedge \tY_1 \wedge Y_2 \wedge \tY_2)\big\}, \cr
T_7 &= \E\big\{F_Y(Y \wedge \tY)\big\}^2, & \label{eq:T1-T7}
\end{align}
with $\mathfrak{q}_q$ and $\mathfrak{o}_q$ defined in \eqref{eq:q_d} and \eqref{eq:o_d}, respectively.
\end{theorem}

Notably, when $Y$ and $\bZ$ are further assumed to be independent, all terms $T_1$ through $T_7$ in \eqref{eq:T1-T7} reduce to distribution-free constants, and \eqref{eq:sigma2_xi} further simplifies to the corresponding expression in \cite{Shi_Drton_Han_2024_Bernoulli}, summarized in Proposition~\ref{prop:var_xi_indep} below.

\begin{proposition}[Asymptotic variance under independence] \label{prop:var_xi_indep}
If $Y$ is independent of $\bZ$, then the terms $T_1$ through $T_7$ in \eqref{eq:T1-T7} all reduce to constants, with values $T_1 = 1/6$, $T_2 = 2/15$, $T_3=1/9$, $T_4 = 1/15$, $T_5 = 1/9$, $T_6 = 1/5$, and $T_7 = 1/9$. Consequently, the limiting variance reduces to
\begin{eqnarray*}
\frac{2}{5} +\frac{2}{5}\mathfrak{q}_q + \frac{4}{5}\mathfrak{o}_q,
\end{eqnarray*}
which agrees with the expression in \citet[Theorem 3.1(ii)]{Shi_Drton_Han_2024_Bernoulli}.
\end{proposition}

\subsection{Variance estimation for $\xi_n$}

Recall that the variance estimator proposed in \cite{Lin_Han_2025_CLT} requires $O(n^2)$ computational time. In contrast, our new Theorem~\ref{thm:var_xi} yields, as a byproduct, a new estimator of $\sigma^2_{\xi(Y,\bZ)}$ that can be computed in $O(n\log n)$ time. We first present the form of this estimator,  along with its theoretical properties, in Theorem~\ref{thm:est_var_xi} below.

\begin{theorem}[Consistent estimator of the limiting variance]\label{thm:est_var_xi}
Assume the conditions of Theorem~\ref{thm:var_xi}. Then the following statistic converges in probability to $\sigma^2_{\xi(Y,\bZ)}$:
\begin{eqnarray}
\hsigma^2_{\xi(Y,\bZ)} &:=&  36\Big\{	(1+\mathfrak{q}_q) \,\hatT_1 + (2-2\mathfrak{q}_q + \mathfrak{o}_q) \, \hatT_2 - (2-\mathfrak{q}_q + \mathfrak{o}_q) \, \hatT_3 \cr
&& \qquad +4 \, \hatT_4 - 2 \, \hatT_5 + \hatT_6 - 4 \, \hatT_7 \Big\},
\label{eq:est_var_cha}
\end{eqnarray}
where
\begin{align}
\hatT_1 &=  \frac{1}{n^3} \sum_{i=1}^n  \big(R_i \wedge R_{N(i)} \big)^2,  &
\hatT_2 &= \frac{1}{n^3} \sum_{i=1}^n \big(R_i \wedge R_{N(i)} \big) \big( R_i \wedge R_{N_2(i)}\big),  \cr
\hatT_3 &= \frac{1}{n^3} \sum_{i=1}^n \big(R_i \wedge R_{N(i)} \big) \big( R_{N_2(i)} \wedge R_{N_3(i)}\big), \qquad &
\hatT_4 &= \frac{1}{n^3}
\sum_{1 \leq i \neq j \leq n}
\Ind\big(R_i \leq R_j \wedge R_{N(j)}\big) \big(R_i \wedge R_{N(i)}\big),  \cr
\hatT_5 &=   \frac{1}{n^3} \sum_{1 \leq i \neq j \leq n} \Ind\big(R_i \leq R_j \wedge R_{N(j)}\big) \big(R_{N(i)} \wedge R_{N_2(i)}\big), \hspace{-4.5em} \cr
\hatT_6 &= \frac{1}{n^3} \sum_{1 \leq i \neq j \leq n}  R_i \wedge R_{N(i)} \wedge R_j \wedge R_{N(j)},  &
\hatT_7 &= \Big(\frac{1}{n^2}\sum_{i=1}^n R_i \wedge R_{N(i)}\Big)^2, \label{eq:hatT1-hatT7}
\end{align}
with $N_2(i)$ and $N_3(i)$ denoting the indices of the second and third NNs of $\bZ_i$, respectively.
\end{theorem}

Compared with the original estimator in \cite{Lin_Han_2025_CLT} (Theorem 1.1), the new estimator $\hsigma^2_{\xi(Y,\bZ)}$ is notably simpler, owing to the explicit closed-form expression of $\hsigma^2_{\xi(Y,\bZ)}$ established in Theorem~\ref{thm:var_xi} through the incorporation of the constants $\mathfrak{q}_q$ and $\mathfrak{o}_q$. Moreover, unlike $\tilde\sigma^2$ in \cite{Lin_Han_2025_CLT}, the following proposition shows that $\hsigma^2_{\xi(Y,\bZ)}$ can be computed in $O(n\log n)$ time.

\begin{proposition} \label{prop:nlogn}
The estimator $\hsigma^2_{\xi(Y,\bZ)}$ in \eqref{eq:est_var_cha} can be computed in $O(n \log n)$ time. In particular, the terms $\hatT_4$, $\hatT_5$, and $\hatT_6$ admit $O(n \log n)$ implementations via Algorithm~\ref{alg1}.
\end{proposition}

\begin{algorithm}[htbp]
\caption{Fast computations of $\hatT_4$, $\hatT_5$, and $\hatT_6$ in $O(n \log n)$ time}
\label{alg1}
\begin{algorithmic}[1]

\Require Sample $\{(\bZ_i,Y_i)\}_{i=1}^n$.

\State Perform $k$-NN search on $\{\bZ_i\}_{i=1}^n$, for $k=1$ and $2$. Obtain $N(i), N_2(i)$, for $i=1,\dots, n$. 
\State Sort $\{Y_i\}_{i=1}^n$ and compute ranks $R_i$. Obtain 
$R_i, R_{N(i)}, R_{N_2(i)}$, for $i=1,\dots, n$.
\State For $i=1,\dots, n$, compute $U_i = R_i \wedge R_{N(i)}$ and $V_i=  R_{N(i)} \wedge R_{N_2(i)}$. The remaining objective is to compute $\hatT_4 
=\sum_{i=1}^n U_i \cdot \sum_{j \in 
\lbr n \rbr, j\neq i}  \Ind(R_i \leq U_j)
$, $\hatT_5= \sum_{i=1}^n V_i \cdot \sum_{j \in 
\lbr n \rbr, j\neq i}  \Ind(R_i \leq U_j)$, and 
$\hatT_6 = \sum_{i\neq j} U_i \wedge U_j = 2\cdot \sum_{i=1}^n U_i \cdot \big\{\sum_{j=1}^n \Ind(U_i \leq U_j) - 1\big\}$.
\State For $i=1,\dots,n$, use binary search to find the rank of $R_i$ among $\{U_j\}_{j=1}^n$, i.e., compute $R^*_i = \#\{j \in \lbr n \rbr: R_i > U_j\}$; also, find the rank of $U_i$ among $\{U_j\}_{j=1}^n$, i.e., compute $R^\#_i = \#\{j \in \lbr n \rbr: U_i > U_j\}$.
\State Compute $\hatT_4 = \sum_{i=1}^n U_i \cdot \big\{n-R^*_i - \Ind(R_i \leq U_i)\big\}$, $\hatT_5 = \sum_{i=1}^n V_i \cdot \big\{n-R^*_i - \Ind(R_i \leq U_i)\big\}$, and $\hatT_6 =2\cdot\sum_{i=1}^n U_i\cdot(n- R^\#_i -1)$.
\Ensure Estimators $\hatT_4$, $\hatT_5$, and $\hatT_6$.

\end{algorithmic}
\end{algorithm}

\section{Statistical inference of $T_n$} \label{sec:SI}

This section introduces inferential procedures for constructing confidence intervals for $T$ in \eqref{eq:T}, as well as for testing $H_0$ in \eqref{eq:null}, based on Azadkia--Chatterjee's conditional correlation coefficient $T_n$ in \eqref{eq:T_n}. Before proceeding, we first introduce some notation and preliminary observations.

Let $(\bX_1,Y_1,\bZ_1),\dots,(\bX_n,Y_n,\bZ_n)$ be $n$ $\iid$ copies of the random triplet $(\bX,Y,\bZ)$, where $Y \in \mathbb{R}$, $\bX \in \mathbb{R}^p$, and $\bZ \in \mathbb{R}^q$, with $p,q \ge 1$. Recall that $T_n$ in \eqref{eq:T_n} takes the form
\begin{eqnarray}
T_n=T_n(Y, \bX \mid \bZ) &=&  \frac{\tau_n(Y, \bX \mid \bZ)}{\kappa_n(Y, \bZ)}, \cr
\text{with}	\hspace{2cm} \tau_n=\tau_n(Y, \bX \mid \bZ) &:=& \frac{1}{n^2}\sum_{i=1}^n \big(\min\{R_i, R_{M(i)}\}- \min\{R_i, R_{N(i)}\}\big), \label{eq:tau_n} \\
\kappa_n=\kappa_n(Y, \bZ) &:=& \frac{1}{n^2}\sum_{i=1}^n\big(R_i-\min\{R_i,R_{N(i)}\}\big), \nonumber
\end{eqnarray}
where $\tau_n$ and $\kappa_n$ denote, respectively, the numerator and denominator of $T_n$ after scaling by $n^{-2}$, and $M(i)$ and $N(i)$ index the NNs of $(\bX_i,\bZ_i)$ and $\bZ_i$, respectively. Whenever $Y$ is not almost surely a function of $\bZ$, $T_n(Y, \bX \mid \bZ)$ converges almost surely to the conditional dependence measure in \eqref{eq:T}, expressed as 
\begin{eqnarray}
T=T(Y, \bX \mid \bZ) &=&  \tau(Y, \bX \mid \bZ) / \kappa(Y, \bZ), \cr
\text{with}	\hspace{2cm}  \tau=\tau(Y, \bX \mid \bZ) &:=& \int \E\big[\var\big\{\P(Y \geq y \mid \bX, \bZ) \ \big | \  \bZ \big\}\big] \d \P_Y(y), \label{eq:tau} \\
\kappa=\kappa(Y, \bZ) &:=& \int \E\big\{ \var \big( \Ind(Y \geq y) \mid \bZ \big)\big\} \d \P_Y(y), \label{eq:kappa}
\end{eqnarray}
where $\tau$ and $\kappa$ denote, respectively, the numerator and denominator of $T$.

\subsection{Confidence intervals} \label{sec:CI}

Constructing confidence intervals for $T$ using $T_n$ hinges on deriving the limiting distribution of $T_n - T$, where
\begin{eqnarray}
T_n - T \ = \  
\frac{\tau_n- T\cdot \kappa_n}{\kappa_n}
\ =: \  \frac{\tT_n}{\kappa_n}.  \label{eq:T_n-T}
\end{eqnarray}
It was shown in \cite{azadkia2019simple} that
\[
\kappa_n \conas \kappa >0
\qquad\text{and}\qquad
\tT_n \conas 0.
\]
Accordingly, the main challenge is to infer the limiting distribution of the numerator term $\tT_n$.

The construction of confidence intervals for $T$ proceeds in the following three steps.

\begin{description}
\item[Step 1] (CLT). Establish a CLT for $\tT_n$:
\begin{eqnarray}
\sqrt{n} \big\{\tT_n - \E(\tT_n) \big\}\conD N(0, \sigma^2), \quad \text{as } n \to \infty, \label{eq:step1-CLT}
\end{eqnarray}
and construct a consistent estimator $\hat{\sigma}^2$ of the limiting variance $\sigma^2$.

\item[Step 2] (Bias correction, if necessary).
Let $L_n =\E(\tT_n)$ be the (asymptotic) bias in \eqref{eq:step1-CLT}. If $\sqrt{n} L_n \to 0$, then it is asymptotically negligible. In that case,
\begin{eqnarray*}
\sqrt{n} \tT_n \conD N(0, \sigma^2),
\end{eqnarray*}
and, by Slutsky's theorem,
\begin{eqnarray*}
\sqrt{n} \big(T_n - T \big)\conD N(0, \sigma^2/\kappa^2), \quad \text{as } n \to \infty.
\end{eqnarray*}

Otherwise, the bias is not negligible. In such cases, let 
\[
L^{(\tau)}_n = \E(\tau_n) - \tau~~~ {\rm and}~~~ L^{(\kappa)}_n = \E(\kappa_n) - \kappa 
\]
denote the biases of $\tau_n$ and $\kappa_n$, respectively.
Construct consistent bias estimators $\hat{L}^{(\tau)}_n$ and $\hat{L}^{(\kappa)}_n$ such that
\begin{eqnarray}
\hat{L}^{(\tau)}_n - L^{(\tau)}_n = o_\P(n^{-1/2}), \qquad  \hat{L}^{(\kappa)}_n - L^{(\kappa)}_n = o_\P(n^{-1/2}).\label{eq:step2-bias}
\end{eqnarray}
Then the bias-corrected conditional correlation coefficient
\begin{eqnarray*}
T_n^{\mathrm{bc}} \ : = \  
\big(\tau_n - 	\hat{L}^{(\tau)}_n\big)/\big(\kappa_n - 	\hat{L}^{(\kappa)}_n\big)
\end{eqnarray*}
satisfies
\begin{eqnarray*}
\sqrt{n} \big(	T_n^{\mathrm{bc}} - T \big)\conD N(0, \sigma^2/\kappa^2), \quad \text{as } n \to \infty.
\end{eqnarray*}

\item[Step 3] (Confidence interval). A $(1-\alpha)$ confidence interval for $T$ is given by
\begin{eqnarray}
&&\mathsf{CI}_{\alpha}(T) = \Big(T_n -\frac{z_{\alpha/2}\cdot \hat{\sigma}}{\kappa_n \cdot\sqrt{n}} \ , \ T_n +\frac{z_{\alpha/2}\cdot \hat{\sigma}}{\kappa_n \cdot\sqrt{n} } \Big), \quad \text{if bias correction is unnecessary,} \cr
\text{or}&&\mathsf{CI}^{\mathrm{bc}}_{\alpha}(T) = \Big(T_n^{\mathrm{bc}} -\frac{z_{\alpha/2}\cdot \hat{\sigma}}{\kappa_n \cdot\sqrt{n}} \ , \ T_n^{\mathrm{bc}} +\frac{z_{\alpha/2}\cdot \hat{\sigma}}{\kappa_n \cdot\sqrt{n} } \Big), \quad \text{if bias correction is necessary,}  \qquad \label{eq:step3-CI}
\end{eqnarray}
where $z_{\alpha/2}$ denotes the $(1-\alpha/2)$-quantile of the standard normal distribution.
\end{description}

\subsection{Conditional independence testing} \label{sec:test}

Further simplifications arise when the goal is to test $H_0$ in \eqref{eq:null}. Specifically, under $H_0$, we have $T=\tau=0$, so that
\[
T_n - T = \tau_n/\kappa_n.
\]
Accordingly, testing $H_0$ is equivalent to testing $\tau=0$, which can be carried out using $\tau_n$ alone.

The construction of a test of $H_0$ based on $\tau_n$ proceeds as follows.

\begin{description}
\item[Step $\mathbf{1'}$] (CLT under $H_0$). Establish a CLT for $\tau_n$ under $H_0$:
\begin{eqnarray}\label{eq:han-CI-test}
\sqrt{n} \big\{\tau_n - \E(\tau_n) \big\}\conD N(0, \sigma_0^2), \quad \text{as } n \to \infty.
\end{eqnarray}
Construct a consistent estimator $\hat{\sigma}_0^2$ of the limiting variance $\sigma_0^2$. \footnote{Note that under $H_0$, $T=0$, so $\tau_n$ coincides with $\tT_n$ in \eqref{eq:T_n-T}; hence \eqref{eq:han-CI-test} is a special case of \eqref{eq:step1-CLT}.}

\item[Step $\mathbf{2'}$] (Bias correction, if necessary). Let $L^{(\tau)}_n = \E(\tau_n) - \tau$ denote the bias of $\tau_n$.
Whenever $\sqrt{n} L_n \nrightarrow 0$, the bias is not negligible.
In such cases, construct a bias estimator $\hat{L}^{(\tau)}_n$ such that
\begin{eqnarray}
\hat{L}^{(\tau)}_n - L^{(\tau)}_n = o_\P(n^{-1/2}), \label{eq:step2'-bias}
\end{eqnarray}
analogously to \eqref{eq:step2-bias} in Step 2 above.

\item[Step $\mathbf{3'}$] (A test of $H_0$).
The resulting level-$\alpha$ test is given by
\begin{align*}
&\mathsf{T}_{\alpha} = \Ind\big(\sqrt{n} \tau_n / \hat{\sigma} > z_\alpha\big),  && \hspace{-1.2cm} \text{if bias correction is unnecessary,} \cr
\text{or } \quad &\mathsf{T}^{\mathrm{bc}}_\alpha = \Ind\big(	 \sqrt{n} (\tau_n - 	\hat{L}^{(\tau)}_n ) / \hat{\sigma} > z_\alpha\big),  && \hspace{-1.2cm} \text{if bias correction is necessary.}
\end{align*}
\end{description}

\section{Theory} \label{sec:theory}

This section provides the theoretical foundation for the inferential procedures described in Section~\ref{sec:SI}. In particular, 
\begin{enumerate}[label=(\roman*)]
\item For Steps 1 and $1'$, we establish a CLT, derive a closed-form expression for the limiting variance (Section~\ref{sec:CLT}), and provide consistent estimators thereof (Section~\ref{sec:est_var});
\item For Steps 2 and $2'$, we show that bias correction could be unnecessary when the combined dimension of $\bX$ and $\bZ$ satisfies $p+q \leq 3$; otherwise, a bias-correction procedure is justified (Section~\ref{sec:bias_correct});
\item Finally, for Steps 3 and $3'$, we establish the validity of the proposed confidence intervals and tests (Section~\ref{sec:infer_theory}).
\end{enumerate}

\subsection{CLT} \label{sec:CLT}

Before presenting the main theorems in this section, we first introduce the following assumptions.

\begin{assumption} \label{assump_4.1}
Assume that $\{(\bX_i,Y_i,\bZ_i): i\in\lbr n\rbr\}$ are $n$ independent copies of $(\bX,Y,\bZ)$.
\end{assumption}

\begin{assumption}\label{assump_4.2}
The joint cumulative distribution function $F_{\bX,Y,\bZ}$ of $(\bX,Y,\bZ)$ is continuous.
\end{assumption}

\begin{assumption} \label{assump_4.3}
$(\bX,\bZ)$ is absolutely continuous and admits a density function $f_{\bX,\bZ}(\bx,\bz)$ that is continuous on its support.
\end{assumption}

\begin{assumption}\label{assump_4.4}
$Y$ is not almost surely equal to a function of $\bZ$.
\end{assumption}

\begin{assumption}\label{assump_4.5}
Define
$G_\bz(t) = \E[\Ind(Y \geq t) \mid \bZ = \bz]$ and $G_{\bx,\bz}(t) = \E[\Ind(Y \geq t) \mid (\bX, \bZ) = (\bx,\bz)]$.
For any fixed $t$, assume that the mapping $\bz \mapsto G_{\bz}(t)$ is continuous almost everywhere on $\supp(\bZ)$, and that the mapping $(\bx,\bz) \mapsto G_{\bx,\bz}(t)$ is continuous almost everywhere on $\supp((\bX,\bZ))$.
\end{assumption}

Recall that $T_n(Y, \bX \mid \bZ) = \tau_n(Y, \bX \mid \bZ) / \kappa_n(Y, \bZ)$ in \eqref{eq:T_n}, with $\tau_n$ and $\kappa_n$ given by
\begin{align}
&	\tau_n(Y, \bX \mid \bZ) = \txi_{1,n} -  \txi_{2,n},     &\kappa_n(Y, \bZ) 	= (n+1)/(2n)- 3^{-1} -  \txi_{2,n}, \cr
\text{where} \hspace{1cm} & \txi_{1,n} := \frac{1}{n^2}\sum_{i=1}^n \min\{R_i, R_{M(i)}\} -\frac{1}{3},
& \txi_{2,n} := \frac{1}{n^2} \sum_{i=1}^n \min\{R_i, R_{N(i)}\} - \frac{1}{3}. \label{eq:txi}
\end{align}
It is worth noting that $\txi_{1,n}$ and $\txi_{2,n}$ can be viewed as unnormalized versions of the Azadkia--Chatterjee unconditional correlation coefficient $\xi_n$ in \eqref{eq:xi_n}, in the sense that
\begin{eqnarray}
\txi_{1,n} &=& 6^{-1}\cdot \xi_n(Y, (\bX, \bZ)) + O_\P(n^{-2}) +O(n^{-1}), \cr \txi_{2,n} &=&  6^{-1} \cdot \xi_n(Y,\bZ) + O_\P(n^{-2}) + O(n^{-1}). \label{eq:txi_2}
\end{eqnarray}
Using this notation, $\tT_n$ in \eqref{eq:T_n-T} admits the decomposition
\begin{eqnarray}
\tT_n \ = \kappa_n\cdot (T_n -T) \ = \ \txi_{1,n} -(1 -T) \cdot \txi_{2,n} - \big\{(n+1)/(2n)-3^{-1}\big\}\cdot T.  \label{eq:tTn}
\end{eqnarray}

We first derive the general CLT. As noted earlier in Section~\ref{sec:CI}, the CLT for $T_n$ relies on that for $\tT_n$. We therefore begin by establishing a CLT for $\tT_n$ in the general case where $Y$ may depend on $\bX$ conditionally on $\bZ$. Throughout Section~\ref{sec:theory}, $\barY$, $\barY'$, $\tY$, and $\tY'$ denote copies of $Y$ such that, conditional on $(\bX,\bZ)$, (i) they are mutually independent, (ii) $\barY,\barY' \sim \mu_{Y \mid \bX,\bZ}$, and (iii) $\tY,\tY' \sim \mu_{Y \mid \bZ}$.

\begin{theorem}[CLT of $\tT_n$] \label{thm:CLT-main}
Assume Assumptions~\ref{assump_4.1}--\ref{assump_4.5}. Then, as $n \to \infty$, $\tT_n \ = \kappa_n\cdot (T_n -T)$ satisfies the CLT
\begin{eqnarray*}
\sqrt{n} \big\{\tT_n - \E(\tT_n) \big\}\conD N(0, \sigma^2),
\end{eqnarray*}
where
\begin{eqnarray}
\sigma^2 = \lim_{n \to \infty} n \var(\tT_n) = \sigma_1^2 + (1-T)^2 \cdot \sigma_2^2 - 2\cdot (1-T) \cdot \sigma_{1,2}.  \label{eq:sigma2_tTn}
\end{eqnarray}
The explicit expressions of $\sigma_1^2$, $\sigma_2^2$, and $\sigma_{1,2}$ are given as:
\begin{eqnarray*}
&&	\sigma_1^2 = \lim_{n \to \infty} n \var(\txi_{1,n}) = \lim_{n \to \infty} n \var\big(6^{-1} \cdot \xi_n(Y,\bZ) \big) = 36^{-1}\cdot \sigma^2_{\xi(Y,\bZ)}, \cr
&&	\sigma_2^2 = \lim_{n \to \infty} n \var(\txi_{2,n}) = \lim_{n \to \infty} n \var\big(6^{-1} \cdot \xi_n(Y,(\bX,\bZ)) \big) = 36^{-1}\cdot \sigma^2_{\xi(Y,(\bX,\bZ))},
\end{eqnarray*}
where $\sigma^2_{\xi(Y,\bZ)}$ and $\sigma^2_{\xi(Y,(\bX,\bZ))}$ are the limiting variances of the Azadkia--Chatterjee unconditional coefficients $\xi(Y,\bZ)$ and $\xi(Y,(\bX,\bZ))$, respectively. 
\footnote{The closed form of $\sigma^2_{\xi(Y,\bZ)}$ is provided in \eqref{eq:sigma2_xi}–\eqref{eq:T1-T7}, and $\sigma^2_{\xi(Y,(\bX,\bZ))}$ is defined in the same manner as $\sigma^2_{\xi(Y,\bZ)}$, with $\bZ$ replaced by $(\bX,\bZ)$.}  Furthermore,
\begin{eqnarray}
\sigma_{1,2} &=& \lim_{n \to \infty} n \cdot \cov\big(\txi_{1,n} \, , \, \txi_{2,n}\big) \cr
&=& 4 \, U_1 -2\, U_2 - U_3 +2 \, U_4 - U_5 + 2\, U_6 - U_7 + U_8 - 4 \, U_9,
\label{eq:sigma12}
\end{eqnarray}
with
\begin{align}
U_1 &=  \E\big\{F_Y(Y \wedge \barY) \cdot F_Y(Y \wedge\tY)\big\}, \quad & U_2 &= \E\big\{F_Y(Y \wedge \tY) \cdot F_Y(\barY \wedge\barY')\big\},  \cr
U_3 &= \E\big\{F_Y(Y \wedge \barY) \cdot F_Y(\tY \wedge\tY')\big\}, \quad & U_4 &= \E\big\{\Ind(Y_1 \leq Y_2 \wedge \tY_2)\cdot F_Y(Y_1 \wedge \barY_1)\big\},  \cr
U_5 &=  \E\big\{\Ind(Y_1 \leq Y_2 \wedge \tY_2)\cdot F_Y(\barY_1 \wedge \barY_1')\big\}, \quad & U_6 &=  \E\big\{\Ind(Y_1 \leq Y_2 \wedge \barY_2)\cdot F_Y(Y_1 \wedge \tY_1)\big\}, \cr
U_7 &=  \E\big\{\Ind(Y_1 \leq Y_2 \wedge \barY_2)\cdot F_Y(\tY_1 \wedge \tY_1')\big\}, \quad & U_8 &= \E\big\{F_Y(Y_1 \wedge \barY_1 \wedge Y_2 \wedge \tY_2)\big\}, \cr
U_9 &= \E\big\{F_Y(Y \wedge \barY)\big\}\cdot \E\big\{F_Y(Y \wedge \tY)\big\}. & \label{eq:U1-U9}
\end{align}
\end{theorem}

By Slutsky's theorem, Theorem~\ref{thm:CLT-main} directly yields the CLT for $T_n$ and $T_n^{\mathrm{bc}}$, stated in Corollary~\ref{cor: CLT-Tn} below.

\begin{corollary}[CLT of $T_n$ and $T_n^{\mathrm{bc}}$] \label{cor: CLT-Tn}
Assume Assumptions~\ref{assump_4.1}--\ref{assump_4.5}. Let $L_n =\E(\tT_n)$ denote the bias of $\tT_n$.
\begin{enumerate}[label=(\roman*)]
\item If $\sqrt{n} L_n \to 0$ as $n \to \infty$, then
\begin{eqnarray}\label{eq:han-Tn-var}
\sqrt{n} \big(T_n - T \big)\conD N(0, \sigma^2/\kappa^2).
\end{eqnarray}
\item Assume that the bias-corrected correlation coefficient $T_n^{\mathrm{bc}}  =
\big(\tau_n - 	\hat{L}^{(\tau)}_n\big)/\big(\kappa_n - 	\hat{L}^{(\kappa)}_n\big)$
satisfies \eqref{eq:step2-bias}. Then, as $n \to \infty$,
\begin{eqnarray}\label{eq:han-Tn-var2}
\sqrt{n} \big(T_n^{\mathrm{bc}}   - T \big)\conD N(0, \sigma^2/\kappa^2).
\end{eqnarray}
\end{enumerate}
\end{corollary}

\begin{remark}
As will be shown in Theorem~\ref{thm:bias_correct} below, the condition $p+q \leq 3$ could imply $\sqrt{n}L_n \to 0$. In this case, $\sqrt{n}(T_n-T)$ is asymptotically normal, so that $T_n$ converges to $T$ at the parametric rate $n^{-1/2}$ without the need for bias correction.
\end{remark}

We next derive the CLT under $H_0$. To this end, only the limiting distribution of $\tau_n$ is needed.

\begin{corollary}[CLT of $\tau_n$ under conditional independence]
\label{cor: CLT-H0}
Assume Assumptions~\ref{assump_4.1}--\ref{assump_4.5}. Assume further that $Y$ is conditionally independent of $\bX$ given $\bZ$.
Then, as $n \to \infty$, $\tau_n = \tT_n$ satisfies the CLT
\begin{eqnarray*}
\sqrt{n} \big\{\tau_n - \E(\tau_n) \big\}\conD N(0, \sigma_0^2),
\end{eqnarray*}
where $\sigma_0^2$ is strictly positive and admits the simplified expression
\begin{eqnarray}
\sigma_0^2 &=& (2+\mathfrak{q}_q+ \mathfrak{q}_{p+q}) \cdot\E\big\{ F_Y^2(Y \wedge \tY)\big\}
\cr
&&+	(\mathfrak{o}_q+ \mathfrak{o}_{p+q} - 2\mathfrak{q}_q- 2\mathfrak{q}_{p+q} -4) \cdot \E\big\{F_Y(Y \wedge \tY) \cdot F_Y(Y \wedge\tY')\big\}
\cr
&&+(2+\mathfrak{q}_q
+ \mathfrak{q}_{p+q}-\mathfrak{o}_q- \mathfrak{o}_{p+q}) \cdot \E\big\{F_Y(Y \wedge \tY) \cdot F_Y(\tY' \wedge\tY'')\big\}, \label{eq:sigma2_H0}
\end{eqnarray}
where, as before, $\mathfrak{q}_d$ and $\mathfrak{o}_d$ are defined in \eqref{eq:q_d} and \eqref{eq:o_d}, respectively.
\end{corollary}

\subsection{Estimation of the limiting variance} \label{sec:est_var}

We begin with the general case. In view of Theorem~\ref{thm:CLT-main}, we can construct a consistent estimator of $\sigma^2$ in a manner analogous to that of Theorem~\ref{thm:est_var_xi}. 

\begin{theorem}[Consistent estimator of limiting variance $\sigma^2$] \label{thm: est_var-main}
Assume Assumptions~\ref{assump_4.1}--\ref{assump_4.5}. Then $\sigma^2$ in Theorem~\ref{thm:CLT-main} admits the consistent estimator
\begin{eqnarray}
\hsigma^2 = \hsigma_1^2 + (1-T_n)^2 \cdot \hsigma_2^2 - 2\cdot (1-T_n) \cdot \hsigma_{1,2}, \label{eq:sigma2_est}
\end{eqnarray}
where the explicit expressions of $\hsigma_1^2$, $\hsigma_2^2$, and $\hsigma_{1,2}$ are given as follows:
\begin{eqnarray*}
\hsigma_1^2 = 36^{-1}\cdot\hsigma^2_{\xi(Y,\bZ)},  \quad \text{and} \quad \hsigma_2^2  = 36^{-1}\cdot\hsigma^2_{\xi(Y,(\bX,\bZ))}.
\end{eqnarray*}
Here $\hsigma^2_{\xi(Y,\bZ)}$ and $\hsigma^2_{\xi(Y,(\bX,\bZ))}$ are consistent estimators of the limiting variances of the Azadkia--Chatterjee unconditional correlation coefficients $\xi(Y,\bZ)$ and $\xi(Y,(\bX,\bZ))$, respectively. 
\footnote{Here the form of $\hsigma^2_{\xi(Y,\bZ)}$ is provided in \eqref{eq:est_var_cha} in Theorem~\ref{thm:est_var_xi}, and $\hsigma^2_{\xi(Y,(\bX,\bZ))}$ is defined in the same manner as $\hsigma^2_{\xi(Y,\bZ)}$, with $\bZ$ replaced by $(\bX,\bZ)$.}
Furthermore, we introduce
\begin{eqnarray*}
\hsigma_{1,2} = 4 \, \hatU_1 -2\, \hatU_2 - \hatU_3 +2 \, \hatU_4 - \hatU_5 + 2\, \hatU_6 - \hatU_7 + \hatU_8 - 4 \, \hatU_9,
\end{eqnarray*}
where
\begin{align*}
\hatU_1 &=  \frac{1}{n^3} \sum_{i=1}^n \big(R_i \wedge R_{M(i)} \big) \big( R_i \wedge R_{N(i)}\big),  &
\hatU_2 &= \frac{1}{n^3} \sum_{i=1}^n \big(R_i \wedge R_{N(i)} \big) \big( R_{M(i)} \wedge R_{M_2(i)}\big),  \cr
\hatU_3 &= \frac{1}{n^3} \sum_{i=1}^n \big(R_i \wedge R_{M(i)} \big) \big( R_{N(i)} \wedge R_{N_2(i)}\big), \qquad &
\hatU_4 &= \frac{1}{n^3}
\sum_{1 \leq i \neq j \leq n}
\Ind\big(R_i \leq R_j \wedge R_{N(j)}\big) \big(R_i \wedge R_{M(i)}\big),  \cr
\hatU_5 &=   \frac{1}{n^3} \sum_{1 \leq i \neq j \leq n} \Ind\big(R_i \leq R_j \wedge R_{N(j)}\big) \big(R_{M(i)} \wedge R_{M_2(i)}\big), \hspace{-4.5em} \cr
\hatU_6 &=   \frac{1}{n^3} \sum_{1 \leq i \neq j \leq n} \Ind\big(R_i \leq R_j \wedge R_{M(j)}\big) \big(R_i \wedge R_{N(i)}\big), \hspace{-4.5em} \cr
\hatU_7 &=   \frac{1}{n^3} \sum_{1 \leq i \neq j \leq n} \Ind\big(R_i \leq R_j \wedge R_{M(j)}\big) \big(R_{N(i)} \wedge R_{N_2(i)}\big), \hspace{-4.5em} \cr
\hatU_8 &= \frac{1}{n^3} \sum_{1 \leq i \neq j \leq n}  R_i \wedge R_{M(i)} \wedge R_j \wedge R_{N(j)},  &
\hatU_9 &= \Big(\frac{1}{n^2}\sum_{i=1}^n R_i \wedge R_{M(i)}\Big) \Big(\frac{1}{n^2}\sum_{i=1}^n R_i \wedge R_{N(i)}\Big), \label{eq:hatT1-hatT7+}
\end{align*}
with $M_2(i)$ and $N_2(i)$ denoting the indices of the second NNs of $(\bX_i,\bZ_i)$ and $\bZ_i$, respectively.
\end{theorem}

According to Theorem~\ref{thm:est_var_xi} and Proposition~\ref{prop:nlogn}, the computational complexities of $\hsigma_1^2$ and $\hsigma_2^2$ are both of order $O(n \log n)$. For $\hsigma_{1,2}$, an analysis similar to that in Proposition \ref{prop:nlogn} shows that its computational complexity is also of order $O(n \log n)$. Indeed, the terms $\hat U_4$--$\hat U_8$ involved in $\hsigma_{1,2}$ can be computed via fast algorithms analogous to Algorithm~\ref{alg1}, so that each term can be evaluated in $O(n \log n)$ time. Consequently, the overall computational complexity of $\hsigma^2$ is of order $O(n \log n)$.

\begin{proposition} \label{prop:nlogn_new}
The estimator $\hsigma^2$ in \eqref{eq:sigma2_est} can be computed in $O(n \log n)$ time.
In particular, the terms $\hat U_4$--$\hat U_8$ can be computed in $O(n \log n)$ time via fast algorithms analogous to Algorithm~\ref{alg1}.
\end{proposition}

Next, for conditional independence testing, it suffices to estimate $\sigma_0^2$ in \eqref{eq:sigma2_H0}. To this end, we consider two alternative estimators: (1) the fast simplified estimator $\hat{\sigma}^2_{0,\mathrm{F}}$, and (2) the $m$-out-of-$n$ bootstrap estimator $\hat{\sigma}^2_{0,\mathrm{B}}$. Compared with $\hsigma^2$ in Theorem~\ref{thm: est_var-main}, both alternative estimators remain consistent under $H_0$, while offering simpler computation and improved estimation accuracy.

We first discuss the direct estimation approach based on Theorem~\ref{thm: est_var-main}, which has time complexity $O(n\log n)$.

\begin{corollary}[Consistency of fast simplified estimator $\hat{\sigma}^2_{0,\mathrm{F}}$]
\label{cor: est_var-H0}
Assume Assumptions~\ref{assump_4.1}--\ref{assump_4.5}. Assume that $Y$ is conditionally independent of $\bX$ given $\bZ$. Then $\sigma_0^2$ in Corollary~\ref{cor: CLT-H0} admits the simplified consistent estimator
\begin{eqnarray}
\hat{\sigma}^2_{0,\mathrm{F}} &=& (2+\mathfrak{q}_q+ \mathfrak{q}_{p+q})\cdot  \frac{1}{n^3} \sum_{i=1}^n  \big(R_i \wedge R_{N(i)} \big)^2 \cr
&& + \	(\mathfrak{o}_q+ \mathfrak{o}_{p+q} - 2\mathfrak{q}_q- 2\mathfrak{q}_{p+q} -4) \cdot \frac{1}{n^3} \sum_{i=1}^n \big(R_i \wedge R_{N(i)} \big) \big( R_i \wedge R_{N_2(i)}\big)  \cr
&&+ \ (2+\mathfrak{q}_q
+ \mathfrak{q}_{p+q}-\mathfrak{o}_q- \mathfrak{o}_{p+q})\cdot \frac{1}{n^3}\sum_{i=1}^n \big(R_i \wedge R_{N(i)} \big) \big( R_{N_2(i)} \wedge R_{N_3(i)}\big),
\label{eq:sigma2_F}
\end{eqnarray}
with $N_2(i)$ and $N_3(i)$ denoting the indices of the second and third NNs of $\bZ_i$, respectively.
\end{corollary}

We next consider the $m$-out-of-$n$ bootstrap procedure proposed in \cite{Dette_Kroll_2025}, which has computational complexity $O(B \, m \log m)$, where $B$ denotes the number of bootstrap replicates. The procedure is as follows. For each bootstrap iteration $b = 1,\dots,B$, draw $m<n$ observations without replacement from $\{(Y_i, \bX_i,\bZ_i)\}_{i=1}^n$, denoted by $\{(Y_{b,j}^*, \bX_{b,j}^*, \bZ_{b,j}^*)\}_{j=1}^m$, and compute the statistic $\tau_m$ in \eqref{eq:tau_n} based on this bootstrap sample, denoted by $\tau_{m,b}^*$. The bootstrap estimator $\hat{\sigma}^2_{0,\mathrm{B}}$ of $\sigma^2$ is then defined by
\begin{eqnarray}
\hat{\sigma}^2_{0,\mathrm{B}} = \frac{m}{B}\sum_{b=1}^B \Big(\tau_{m,b}^* - \frac{1}{B}\sum_{j=1}^B \tau_{m,j}^*\Big)^2.
\label{eq:sigma2_B}
\end{eqnarray}

\begin{proposition}[Consistency of $m$-out-of-$n$ bootstrap estimator $\hat{\sigma}^2_{0,\mathrm{B}}$]
\label{prop:est_var-B}
Assume Assumptions~\ref{assump_4.1}--\ref{assump_4.5}.
Assume that $Y$ is conditionally independent of $\bX$ given $\bZ$.
Assume $B\to \infty$, $m \to \infty$, and $m = o(n)$. Then under the null hypothesis, $\hat{\sigma}^2_{0,\mathrm{B}}$ is consistent, in the sense that $\hat{\sigma}^2_{0,\mathrm{B}} \conP \sigma_0^2$ as $n\to \infty$.
\end{proposition}

\subsection{Bias correction} \label{sec:bias_correct}

Recall that $L^{(\tau)}_n = \E(\tau_n) - \tau$ and $L^{(\kappa)}_n = \E(\kappa_n) - \kappa$ represent the biases of $\tau_n$ and $\kappa_n$, respectively.
The goal of this subsection is to introduce consistent estimators $\hatL^{(\tau)}_n$ and $\hatL^{(\kappa)}_n$ such that $\hatL^{(\tau)}_n - L^{(\tau)}_n= o_\P(n^{-1/2})$ and $\hatL^{(\kappa)}_n - L^{(\kappa)}_n= o_\P(n^{-1/2})$, as required in \eqref{eq:step2-bias} and \eqref{eq:step2'-bias} for our inferential procedures. 

Recall from \eqref{eq:txi} that $\tau_n$ and $\kappa_n$ admit the decompositions
\begin{align*}
&	\tau_n = \txi_{1,n} -  \txi_{2,n},     &\kappa_n	= (n+1)/(2n)- 3^{-1} -  \txi_{2,n}, \cr
\text{where} \hspace{2cm} & \txi_{1,n} := \sum_{i=1}^n \min\{R_i, R_{M(i)}\} -\frac{1}{3},
& \txi_{2,n} := \sum_{i=1}^n \min\{R_i, R_{N(i)}\} - \frac{1}{3}. 
\end{align*}
As noted earlier in \eqref{eq:txi_2}, $\txi_{1,n}$ and $\txi_{2,n}$ can be viewed as unnormalized versions of the Azadkia--Chatterjee unconditional correlation coefficient $\xi_n$ in \eqref{eq:xi_n}.
Therefore, as $n\to \infty$, we have
\begin{eqnarray*}
&&	 \hspace{1cm}\txi_{1,n} \conas \txi_1, \qquad \txi_{2,n} \conas \txi_2, \cr
\text{where}  \hspace{1cm}  &&
\txi_1  := \int \var\big[\E \big\{ \Ind(Y\geq y) \mid \bX, \bZ\big\}\big] \d \P_Y(y), \cr
&& \txi_2 := \int \var\big[\E \big\{ \Ind(Y\geq y) \mid \bZ \big\}\big] \d \P_Y(y),
\end{eqnarray*}
which correspond to the numerator of $\xi$ in \eqref{eq:xi}, with $\bX$ therein replaced by $(\bX,\bZ)$ and $\bZ$, respectively.
It is also straightforward to verify that the population quantities $\tau$ and $\kappa$ in \eqref{eq:tau} and \eqref{eq:kappa} admit analogous decompositions, namely, $\tau = \txi_1 - \txi_2$ and $\kappa = 6^{-1} - \txi_2$.

Denote the biases of $\txi_{1,n}$ and $\txi_{2,n}$ by
\begin{eqnarray*}
L_n^{\bX, \bZ}: = \E(\txi_{1,n}) - \txi_1, \qquad L_n^{\bZ}: = \E(\txi_{2,n}) - \txi_2.
\end{eqnarray*}
From the above decompositions, it is readily verified that
\begin{eqnarray*}
&&L^{(\tau)}_n  \ = \ \E(\tau_n) - \tau  \ = \ \E( \txi_{1,n} -  \txi_{2,n}) - (\txi_1 - \txi_2) \ = \ L_n^{\bX, \bZ} - L_n^{\bZ}, \cr
&&L^{(\kappa)}_n \ = \ \E(\kappa_n) - \kappa \ = \ - \big\{ \E(  \txi_{2,n})   -\txi_2 \big\} + 1/(2n) \ = \ - L_n^{\bZ} + O(n^{-1}).
\end{eqnarray*}
Therefore, it suffices to perform bias correction separately for $L_n^{\bX, \bZ}$ and $L_n^{\bZ}$; that is, to construct estimators $\hatL_n^{\bX, \bZ}$ and $\hatL_n^{\bZ}$ such that $\hatL_n^{\bX, \bZ} =L_n^{\bX, \bZ} + o_\P(n^{-1/2})$ and $\hatL_n^{\bZ} =L_n^{\bZ} + o_\P(n^{-1/2})$. Then
\begin{eqnarray}
\hatL^{(\tau)}_n = \hatL_n^{\bX, \bZ} - \hatL_n^{\bZ}, \qquad \text{and} \quad \hatL^{(\kappa)}_n = - \hatL_n^{\bZ},   \label{eq:bias_est}
\end{eqnarray}
serve as the desired bias estimators satisfying \eqref{eq:step2-bias} and \eqref{eq:step2'-bias}.

Note that the construction of $L_n^{\bX, \bZ}$ is entirely analogous to that of $L_n^{\bZ}$: one simply replaces the data $\{\bZ_i\}_{i=1}^n$ by $\{(\bX_i,\bZ_i)\}_{i=1}^n$.
We therefore focus on the construction of $L_n^{\bZ}$ below to illustrate the bias-correction procedure and its theoretical justification, following the framework of \cite{azadkia2026biascorrection}.

Recall the function $G_\bz(t) = \E\{\Ind(Y\geq t) \mid \bZ=\bz\}$ defined in Assumption~\ref{assump_4.5}.
Section 3 of \cite{azadkia2026biascorrection} establishes the following alternative representation of the bias:
\begin{eqnarray}
L_n^\bZ &=& \int \E \Big\{G_{\bZ_1}(t) G_{\bZ_{N(1)}}(t) - G^2_{\bZ_1}(t)\Big\} \d \P_Y(t) \cr
&=&
\E \Big\{G_{\bZ_1}(Y^*) G_{\bZ_{N(1)}}(Y^*) - G^2_{\bZ_1}(Y^*)\Big\}, \label{eq:bias}
\end{eqnarray}
where $Y^*$ is independent of $\bZ$ and has the same marginal distribution as $Y$.
This identity naturally suggests a two-step procedure for constructing an estimator.
First, construct an appropriate estimator $\hat{G}.(\cdot)$ of the bivariate regression function $G.(\cdot)$.
Second, approximate the expectation in \eqref{eq:bias} by replacing the population mean with the empirical distribution of $Y^*$ and $\bZ$, leading to the estimator
\begin{eqnarray*}
\hat{L}_n^{\bZ} = \frac{1}{n(n-1)} \sum_{1 \leq i \neq j \leq n} \Big\{ \hatG_{\bZ_i}(Y_j) \hatG_{\bZ_{N(i)}}(Y_j) - \hatG^2_{\bZ_i}(Y_j) \Big\}.
\end{eqnarray*}
Note that, for each fixed $t$, $G_\bZ(t) = \E\{\Ind(Y\geq t) \mid \bZ\}$ is the regression mean function of $\Ind(Y\geq t)$ on $\bZ$. This motivates estimating $G.(t)$ by regression techniques.
According to the results in Section 4.2.2 of \cite{azadkia2026biascorrection}, $G.(t)$ can be effectively estimated by ridge least squares \citep{tuo2024asymptotic}, which yields the estimator $\hat{G}.(t)$ and hence $\hat{L}_n^{\bZ}$. The complete procedure for computing $\hat{L}_n^{\bZ}$ is summarized in Algorithm~\ref{alg2}.

\begin{algorithm}[htbp]
\caption{Compute bias estimator $\hat{L}_n^{\bZ}$}
\label{alg2}
\begin{algorithmic}[1]

\Require Sample $\{(Y_i, \bZ_i)\}_{i=1}^n$; $K$ basis functions $\boldsymbolp(\bz) = (p_1(\bz),...,p_K(\bz))^\top$, defined for $\bz \in \mathbbR^q$; regularization penalty parameter $\lambda_n>0$.
\State
For each $j = 1, \dots, n$, let $t= Y_j$. Solve the ridge estimator
\begin{eqnarray*}
\hat{\bbeta}_j = \arg \min_{\bbeta \in \mathbbR^K} \Big\{\frac{1}{n}\sum_{i=1}^n  \big\{ \Ind(Y_i \geq t) - \boldsymbolp(\bZ_i)^\top \bbeta\big\} + \lambda_n \|\bbeta\|^2 \Big\}.
\end{eqnarray*}
Obtain $\hatG_\bz(Y_j) = \hat{\bbeta}_j^\top \boldsymbolp(\bz)$, for any given $\bz \in \mathbbR^q$.
\State Perform NN search on $\{\bZ_i\}_{i=1}^n$. Obtain $N(i)$ for $i=1,\dots,n$.
\State Compute
$
\hat{L}_n^{\bZ} = n^{-1}(n-1)^{-1} \sum_{1 \leq  i \neq j \leq n} \Big\{ \hatG_{\bZ_i}(Y_j) \hatG_{\bZ_{N(i)}}(Y_j) - \hatG^2_{\bZ_i}(Y_j) \Big\}.
$
\Ensure Bias estimator $\hat{L}_n^{\bZ}$.
\end{algorithmic}
\noindent \textit{Note:} The bias estimator $\hat{L}_n^{\bX, \bZ}$ can be obtained by repeating Algorithm~\ref{alg2} with $\bZ$ replaced by $(\bX,\bZ)$.
\end{algorithm}

Theorem~\ref{thm:bias_correct} below establishes the bias rate and the convergence rate of the bias estimators. For ease of exposition, Assumptions~\ref{assump_A.1}--\ref{assump_A.4}, which are needed in this subsection, are deferred to Appendix~\ref{secA:assump}.

\begin{theorem} \label{thm:bias_correct}
We establish the following results:
\begin{enumerate}[label=(\roman*)]
\item (Bias rate) Assume Assumptions~\ref{assump_4.1}, \ref{assump_4.2}, and \ref{assump_A.1}. Then, as $n\to \infty$,
\begin{eqnarray*}
L_n^{\bX, \bZ} = O(n^{-2/(p+q)}),  \qquad  L_n^{\bZ} = O(n^{-2/q} + n^{-1}).
\end{eqnarray*}
If $p+q \leq 3$, then $L^{(\tau)}_n  = L_n^{\bX, \bZ} - L_n^{\bZ}  =  o(n^{-1/2})$ and $L^{(\kappa)}_n  = - L_n^{\bZ} +O(n^{-1}) =o(n^{-1/2})$, that is, both $L^{(\tau)}_n$ and $L^{(\kappa)}_n$ are asymptotically negligible and bias correction is unnecessary.

\item (Bias-correction efficacy) Assume Assumptions~\ref{assump_4.1}, \ref{assump_4.2}, \ref{assump_A.2}, \ref{assump_A.3}, and \ref{assump_A.4}. Then, as $n \to \infty$, the bias estimators $\hat{L}_n^{\bZ}$ and $\hat{L}_n^{\bX, \bZ}$ produced by Algorithm~\ref{alg2} satisfy
\begin{eqnarray*}
\hat{L}_n^{\bZ} - L_n^{\bZ} = o_\P(n^{-1/2}), \qquad \hat{L}_n^{\bX,\bZ} - L_n^{\bX,\bZ} = o_\P(n^{-1/2}).
\end{eqnarray*}
\end{enumerate}
\end{theorem}

\subsection{Inferential validity} \label{sec:infer_theory}

Recall the confidence intervals proposed in \eqref{eq:step3-CI}:
\begin{eqnarray*}
&&\mathsf{CI}_{\alpha}(T) = \Big(T_n -\frac{z_{\alpha/2}\cdot \hat{\sigma}}{\kappa_n \cdot\sqrt{n}} \ , \ T_n +\frac{z_{\alpha/2}\cdot \hat{\sigma}}{\kappa_n \cdot\sqrt{n} } \Big), \qquad 
\mathsf{CI}^{\mathrm{bc}}_{\alpha}(T) = \Big(T_n^{\mathrm{bc}} -\frac{z_{\alpha/2}\cdot \hat{\sigma}}{\kappa_n \cdot\sqrt{n}} \ , \ T_n^{\mathrm{bc}} +\frac{z_{\alpha/2}\cdot \hat{\sigma}}{\kappa_n \cdot\sqrt{n} } \Big).
\end{eqnarray*}
Based on the explicit variance-estimation and bias-correction procedures developed in Section~\ref{sec:theory}, we can now specify the concrete forms of $\hat{\sigma}$ and $T_n^{\mathrm{bc}}$. Specifically, $\hat{\sigma}$ is a consistent estimator of $\sigma$, as given in \eqref{eq:sigma2_est} of Theorem~\ref{thm: est_var-main}, whereas
$
T_n^{\mathrm{bc}} =
\big(\tau_n - 	\hat{L}^{(\tau)}_n\big)/\big(\kappa_n - 	\hat{L}^{(\kappa)}_n\big),
$
with $\hat{L}^{(\tau)}_n$ and $\hat{L}^{(\kappa)}_n$ defined in \eqref{eq:bias_est}.

As a direct consequence of the CLT established in Corollary~\ref{cor: CLT-Tn}, we then obtain the validity of the confidence intervals in Theorem~\ref{thm:CI} below.

\begin{theorem}[Confidence interval validity] \label{thm:CI}
Assume Assumptions~\ref{assump_4.1}--\ref{assump_4.5}. Let $\alpha \in (0,1)$ be a prespecified significance level. Assume that $\sigma^2$ in \eqref{eq:sigma2_tTn} is strictly positive.
\begin{enumerate}[label=(\roman*)]
\item Assume Assumption~\ref{assump_A.1}. If $p+q\leq 3$, then the confidence interval $\mathsf{CI}_{\alpha}(T)$ is valid, in the sense that
\begin{eqnarray*}
\lim_{n \to \infty}\P\big( T \in \mathsf{CI}_{\alpha}(T)\big) = 1-\alpha.
\end{eqnarray*}
\item Assume Assumptions~\ref{assump_A.2}--\ref{assump_A.4}. For general $p,q \geq 1$, the bias-corrected confidence interval $\mathsf{CI}^{\mathrm{bc}}_{\alpha}(T)$ is valid, in the sense that
\begin{eqnarray*}
\lim_{n \to \infty}\P\big( T \in \mathsf{CI}^{\mathrm{bc}}_{\alpha}(T)\big) = 1-\alpha.
\end{eqnarray*}
\end{enumerate}
Moreover, as $n \to \infty$, the lengths of both confidence intervals converge to $0$ at rate $O_\P(n^{-1/2})$.
\end{theorem}

For conditional independence testing, by combining the two limiting variance-estimation methods developed in the latter part of Section~\ref{sec:est_var} with the bias-correction procedure in Section~\ref{sec:bias_correct}, we obtain four distinct level-$\alpha$ tests:
\begin{eqnarray}
&&\mathsf{T}^{\mathrm{F}}_{\alpha} = \Ind \big(	\sqrt{n} \tau_n / \hat{\sigma}_{0,\mathrm{F}}> z_\alpha\big), \qquad 
\mathsf{T}^{\mathrm{F,bc}}_\alpha = \Ind\big(	\sqrt{n} (\tau_n - \hatL_n^{(\tau)}) / \hat{\sigma}_{0,\mathrm{F}} > z_\alpha\big), \cr
&&\mathsf{T}^{\mathrm{B}}_{\alpha} = \Ind\big(\sqrt{n} \tau_n / \hat{\sigma}_{0,\mathrm{B}}> z_\alpha\big), \qquad 
\mathsf{T}^{\mathrm{B,bc}}_\alpha = \Ind\big(	\sqrt{n} (\tau_n - \hatL_n^{(\tau)}) / \hat{\sigma}_{0,\mathrm{B}} > z_\alpha\big), 
\label{eq:four_tests}
\end{eqnarray}
where $\hat{\sigma}_{0,\mathrm{F}}$ and $\hat{\sigma}_{0,\mathrm{B}}$ correspond to the fast estimator in \eqref{eq:sigma2_F} and the bootstrap estimator in \eqref{eq:sigma2_B}, respectively, and $\hatL_n^{(\tau)} = \hatL_n^{\bX, \bZ} - \hatL_n^{\bZ}$ is the bias estimator.
Let $H_1$ denote the alternative hypothesis consisting of all distributions of $(\bX, Y, \bZ)$ under which $Y$ is not conditionally independent of $\bX$ given $\bZ$.
Theorem~\ref{thm:test} establishes the asymptotic size control and consistency of these four tests.

\begin{theorem}[Test validity and consistency] \label{thm:test}
Assume Assumptions~\ref{assump_4.1}--\ref{assump_4.5} and \ref{assump_A.1}--\ref{assump_A.4}. Let $\alpha \in (0,1)$ be a prespecified significance level. Assume $B\to \infty$, $m \to \infty$, and $m = o(n)$.
We have the following results:
\begin{enumerate}[label=(\roman*)]
\item The tests $\mathsf{T}^{\mathrm{F,bc}}_\alpha$ and $\mathsf{T}^{\mathrm{B,bc}}_\alpha$ are valid in the sense that, for any $(\bX, Y, \bZ)$ whose distribution is fixed and satisfies $H_0$ in \eqref{eq:null},
\begin{eqnarray*}
\lim_{n \to \infty} \P(\mathsf{T}^{\mathrm{F,bc}}_\alpha = 1) = \alpha, \qquad  \lim_{n \to \infty} \P(\mathsf{T}^{\mathrm{B,bc}}_\alpha = 1) = \alpha.
\end{eqnarray*}
Moreover, if $p+q \leq 3$, then the same conclusion also holds for the corresponding non-bias-corrected tests $\mathsf{T}^{\mathrm{F}}_{\alpha}$ and $\mathsf{T}^{\mathrm{B}}_{\alpha}$.
\item The four proposed tests are consistent in the sense that, for any $(\bX, Y, \bZ)$ whose distribution is fixed and does not satisfy $H_0$ in \eqref{eq:null},
\begin{eqnarray*}
\lim_{n \to \infty} \P(\mathsf{T} = 1) =1,
\qquad \text{for each }  \mathsf{T} \in \{\mathsf{T}^{\mathrm{F}}_{\alpha}, \mathsf{T}^{\mathrm{B}}_{\alpha}, \mathsf{T}^{\mathrm{F,bc}}_\alpha , \mathsf{T}^{\mathrm{B,bc}}_\alpha\}.
\end{eqnarray*}
\end{enumerate}
\end{theorem}

\begin{remark}
Unfortunately, a calculation of the joint limiting distribution of $\tau_n$ and the log-likelihood ratio, analogous to that in \cite{shi2020power} (see also \cite{shi2020rate}), shows that all the tests considered in Theorem~\ref{thm:test} still have Pitman efficiency zero. Nevertheless, by adopting ideas similar to those in \citet[Section~4]{bickel2022measures}, one may either use $T_n$ solely for measuring conditional dependence, or combine the conditional independence test based on $\tau_n$ with any conditional independence test that possesses positive Pitman efficiency. The resulting combined test can be size-adjusted either via a naive Bonferroni correction or through a more refined analysis of the joint distribution of the two test statistics. In this way, one obtains a combined procedure with nonzero local efficiency.
\end{remark}

\section{Simulations} \label{sec:simu}

We conduct simulation studies to investigate the finite-sample performance of the proposed confidence intervals and conditional independence tests. 
To this end, we consider the following two data-generating models on $(\bX, Y, \bZ)$.

\begin{description}
\item[Model 1] (Uniform distribution). Let $\bZ = (Z_1, \dots, Z_q) \sim \mathrm{Unif}[0,1]^q$ follow a $q$-dimensional uniform distribution. 
Define the $p$-dimensional extension $\tilde\bZ = (\tilde{Z}_1, \dots, \tilde{Z}_p)$ of $\bZ$, with 
\[
\tilde{Z}_j = Z_j \text{ for } j \leq q,\text{ and } \tilde{Z}_j = Z_1 \text{ for } j > q \text{ (if $p>q$)} .
\] 
Let $\epsilon_0, \epsilon_1, \dots, \epsilon_p \sim_\iid \mathrm{Unif}[0,1]$ be independent noise variables. 
For a prespecified $\rho \in [0,1]$, define $Y$ and $\bX = (X_1,\dots, X_p)$ as
\begin{eqnarray}
Y = Z_q + \rho \,\epsilon_1 + \sqrt{1-\rho^2} \, \epsilon_0, \qquad \text{and } \quad  X_j = \tilde{Z}_j + \epsilon_j, \quad \text{for } j = 1\dots, p.  \label{eq:model1}
\end{eqnarray}
It is clear that as $\rho$ increases, the conditional dependence of $Y$ on $\bX$ given $\bZ$ becomes stronger. In particular, $\rho = 0$ corresponds to the null hypothesis \eqref{eq:null}, whereas $\rho = 1$ corresponds to the case where $Y$ is a deterministic function of $\bX$ given $\bZ$.

\item[Model 2] (Gaussian distribution). Model 2 is structurally similar to Model 1, except that all $\mathrm{Unif}[0,1]$ distributions are replaced with the standard normal $N(0,1)$ distributions. 
\end{description}

For the settings of dimensions $p, q$, we consider the following five scenarios 
\[
(p,q) \in \{(1,1), (1,2), (3,1), (3,3), (5,5)\}.
\] 
In each simulation run, we generate $\iid$  $\{(\bX_i, Y_i, \bZ_i)\}_{i=1}^n$ from the selected model, with sample size 
\[
n \in \{1000,5000,10000\}. 
\]
For the the confidence interval $\mathsf{CI}^{\mathrm{bc}}_{\alpha}(T) $ and testing
 methods $\mathsf{T}^{\mathrm{F,bc}}$, $\mathsf{T}^{\mathrm{B,bc}}$ that involve bias correction, we set the penalty parameter $\lambda_n = \lceil n^{-0.85} \rceil$ in Algorithm~\ref{alg2}. 
 For the 
methods $\mathsf{T}^{\mathrm{B}}$ and $\mathsf{T}^{\mathrm{B,bc}}$ that employ the $m$-out-of-$n$ bootstrap, we set the number of bootstrap replications to $B = 200$ and the subsample size to $m = \lceil n^{0.5} \rceil$. In each scenario, the true values of $T$ and $\sigma^2$ (or $\sigma_0^2$) are approximated by Monte Carlo simulation by evaluating $T_n$ and $\hsigma^2$ with $n=10^7$.

Note that {\bf Model 2} represents a challenging setting in which the observed random variables are unbounded, making bias correction---which essentially amounts to a nonparametric regression adjustment---substantially more difficult. We include this setting to contrast it with {\bf Model 1}, which represents the most idealized case, and to assess the robustness of our methods in scenarios that fall outside the scope of the available theoretical guarantees.

\vspace{0.3cm}

All code required to reproduce the results in this paper is publicly available at: \\\url{https://github.com/MuhongGao/Conditional_Independence}.

\subsection{Empirical coverage probabilities}
Tables \ref{Table_2} and \ref{Table_3} report the empirical coverage probabilities (ECPs) of the proposed confidence intervals 
$\mathsf{CI}_{\alpha}(T) $ and $\mathsf{CI}^{\mathrm{bc}}_{\alpha}(T) $ under Models 1 and 2, respectively. In addition, the tables report the relative empirical root mean squared error (rRMSE), defined below, to assess the accuracy of the limiting variance estimator $\hat{\sigma}^2$ in \eqref{eq:sigma2_est}:
\begin{eqnarray}
	\text{rRMSE} = \sqrt{\E \Big\{\Big(\frac{\hat{\sigma}^2 -\sigma^2}{\sigma^2}\Big)^2\Big\}}.   \label{eq:rRMSE}
\end{eqnarray}
From the tables, we observe that when $(p,q)=(1,1)$ and $(1,2)$, so that $p+q\leq 3$, both $\mathsf{CI}$ and $\mathsf{CI}^{\mathrm{bc}}$ achieve ECPs close to the nominal level $1-\alpha=0.9$. By contrast, when $(p,q)=(3,3)$, so that $p+q>3$, the ECPs of $\mathsf{CI}$ deteriorate substantially, whereas those of $\mathsf{CI}^{\mathrm{bc}}$ remain close to $0.9$. This pattern is especially pronounced under {\bf Model 1} when $\rho$ is not too close to $1$. These findings are consistent with Theorem~\ref{thm:bias_correct}, which indicates that bias correction is necessary when $p+q>3$, and they further demonstrate the effectiveness of the proposed bias-correction procedure. The case $(p,q)=(5,5)$, on the other hand, illustrates the curse of dimensionality, as one would expect.

The situation is also of interest under {\bf Model 2}, where the assumptions required for bias correction are violated because $(\bX,\bZ)$ is supported on an unbounded domain. In this case, as shown in Table~\ref{Table_3}, both $\mathsf{CI}$ and $\mathsf{CI}^{\mathrm{bc}}$ continue to perform well when $p+q\leq 3$, suggesting that bias correction may serve as a safe alternative to $\mathsf{CI}$, albeit at a higher computational cost. On the other hand, when $p+q>3$, $\mathsf{CI}^{\mathrm{bc}}$ no longer performs as well as it does in Table~\ref{Table_2}, although it remains substantially superior to $\mathsf{CI}$. This suggests that the bias-correction step is indeed sensitive to tail observations, and points to the potential value of applying a rank transformation, as in \citet[Section~4]{cattaneo2025rosenbaum}, to stabilize the bias correction. Given the already broad scope of the present paper, we do not pursue this direction further.

\begin{table}[tbp]
	\caption{
		\textbf{(Model 1: ECPs).} $\mathsf{CI}$ and $\mathsf{CI}^{\mathrm{bc}}$ denote the ECPs of the proposed $(1-\alpha)$-level confidence intervals $\mathsf{CI}_{\alpha}(T)$ and $\mathsf{CI}^{\mathrm{bc}}_{\alpha}(T)$, respectively, with $\alpha = 0.1$. rRMSE represents the relative empirical root mean squared error of $\hat{\sigma}^2$.
		The results are based on 1000 simulation replicates.
	}
	\label{Table_2}
	\resizebox{\linewidth}{!}{
		\begin{tabular}{cr@{\hspace{6mm}}cccccccccccc}
			\toprule
			\multicolumn{2}{c}{} & \multicolumn{3}{c}{$(p,q)=(1,1)$} & \multicolumn{3}{c}{$(p,q)=(1,2)$} & \multicolumn{3}{c}{$(p,q)=(3,3)$} & \multicolumn{3}{c}{$(p,q)=(5,5)$} \\[1mm]
			\cmidrule(lr){3-5}  \cmidrule(lr){6-8}
			\cmidrule(lr){9-11}  \cmidrule(lr){12-14}
			$\rho$ & \multicolumn{1}{c}{$n$} & $\mathsf{CI}$     & $\mathsf{CI}^{\mathrm{bc}}$  & \footnotesize rRMSE  & $\mathsf{CI}$    & $\mathsf{CI}^{\mathrm{bc}}$ & \footnotesize rRMSE  & $\mathsf{CI}$    & $\mathsf{CI}^{\mathrm{bc}}$ & \footnotesize rRMSE  & $\mathsf{CI}$    & $\mathsf{CI}^{\mathrm{bc}}$ & \footnotesize rRMSE \\[1mm] \midrule
			\multirow{3}[0]{*}{0} & 1000  & 0.88  & 0.88  & 0.40  & 0.85  & 0.86  & 0.39  & 0.81  & 0.85  & 0.40  & 0.72  & 0.85  & 0.39 \\      & 5000  & 0.88  & 0.88  & 0.18  & 0.90  & 0.90  & 0.18  & 0.79  & 0.90  & 0.17  & 0.44  & 0.89  & 0.18 \\      & 10000 & 0.89  & 0.89  & 0.13  & 0.89  & 0.89  & 0.13  & 0.78  & 0.88  & 0.12  & 0.27  & 0.89  & 0.12 \\[2mm]
			\multirow{3}[0]{*}{0.3} & 1000  & 0.89  & 0.88  & 0.40  & 0.88  & 0.88  & 0.38  & 0.78  & 0.86  & 0.40  & 0.70  & 0.84  & 0.39 \\      & 5000  & 0.89  & 0.89  & 0.18  & 0.91  & 0.91  & 0.18  & 0.74  & 0.89  & 0.17  & 0.44  & 0.82  & 0.18 \\      & 10000 & 0.90  & 0.90  & 0.12  & 0.89  & 0.88  & 0.13  & 0.74  & 0.88  & 0.12  & 0.30  & 0.78  & 0.12 \\[2mm]
			\multirow{3}[0]{*}{0.5} & 1000  & 0.88  & 0.88  & 0.39  & 0.88  & 0.88  & 0.37  & 0.70  & 0.87  & 0.40  & 0.48  & 0.86  & 0.41 \\      & 5000  & 0.88  & 0.88  & 0.18  & 0.89  & 0.90  & 0.17  & 0.64  & 0.90  & 0.17  & 0.15  & 0.81  & 0.20 \\      & 10000 & 0.91  & 0.91  & 0.12  & 0.89  & 0.89  & 0.12  & 0.59  & 0.88  & 0.12  & 0.06  & 0.72  & 0.14 \\[2mm]
			\multirow{3}[0]{*}{0.7} & 1000  & 0.89  & 0.89  & 0.40  & 0.87  & 0.88  & 0.38  & 0.48  & 0.89  & 0.52  & 0.08  & 0.91  & 0.52 \\      & 5000  & 0.90  & 0.90  & 0.18  & 0.88  & 0.90  & 0.18  & 0.28  & 0.93  & 0.26  & 0.00  & 0.87  & 0.29 \\      & 10000 & 0.89  & 0.89  & 0.13  & 0.89  & 0.89  & 0.12  & 0.19  & 0.91  & 0.20  & 0.00  & 0.71  & 0.23 \\[2mm]
			\multirow{3}[0]{*}{0.9} & 1000  & 0.85  & 0.85  & 0.54  & 0.84  & 0.92  & 0.89  & 0.01  & 0.76  & 1.93  & 0.00  & 0.72  & 0.97 \\      & 5000  & 0.89  & 0.89  & 0.26  & 0.84  & 0.91  & 0.37  & 0.00  & 0.72  & 1.35  & 0.00  & 0.80  & 0.78 \\      & 10000 & 0.88  & 0.88  & 0.18  & 0.86  & 0.92  & 0.24  & 0.00  & 0.69  & 1.11  & 0.00  & 0.92  & 0.70 \\
			\bottomrule
		\end{tabular}%
	}
\end{table}

Regarding the rRMSE, it decreases uniformly with $n$ across all scenarios, providing empirical evidence for the consistency of $\hat{\sigma}^2$ established in Theorem \ref{thm: est_var-main}. Moreover, the rRMSE increases noticeably with both $p+q$ and $\rho$, indicating that the convergence rate of $\hat{\sigma}^2$ deteriorates in higher-dimensional settings and under stronger dependence. Such behavior is in line with theoretical intuition.

\subsection{Empirical powers of tests of $H_0$} \label{sec:powers}

The empirical powers under {\bf Models 1} and {\bf 2} are reported in Figures~\ref{Figure_1} and \ref{Figure_2}, respectively.
We begin with the results under {\bf Model 1} (Figure~\ref{Figure_1}). When $(p,q) = (1,1), (1,2)$, or $(3,1)$, all four methods perform similarly well, with power curves that nearly overlap. By contrast, when $(p,q) = (3,3)$ or $(5,5)$, the bias-corrected methods $\mathsf{T}^{\mathrm{F,bc}}$ and $\mathsf{T}^{\mathrm{B,bc}}$ clearly outperform their non-bias-corrected counterparts $\mathsf{T}^{\mathrm{F}}$ and $\mathsf{T}^{\mathrm{B}}$, and this advantage becomes more pronounced as $(p,q)$ increases. This pattern is in line with Theorem~\ref{thm:bias_correct}, which shows that bias correction is unnecessary when $p+q < 4$, whereas its benefit becomes increasingly substantial as the dimension grows.
Furthermore, when $p$ and $q$ are large, although the non-bias-corrected methods $\mathsf{T}^{\mathrm{F}}$ and $\mathsf{T}^{\mathrm{B}}$ control size under $H_0$ (that is, when $\rho = 0$) well below the nominal level $\alpha = 0.05$, their power increases much more slowly as $\rho$ grows. This suggests that, under {\bf Model 1}, the bias $L_n = \E(\tau_n) - \tau$ is positive, rendering the tests more conservative and thereby lowering their rejection probabilities.
Finally, comparing the two variance-estimation methods, we observe little difference in either size or empirical power: $\mathsf{T}^{\mathrm{F,bc}}$ and $\mathsf{T}^{\mathrm{B,bc}}$ behave almost identically, and the same is true for $\mathsf{T}^{\mathrm{F}}$ and $\mathsf{T}^{\mathrm{B}}$.

\begin{table}[tbp]
	\caption{
		\textbf{(Model 2: ECPs).} $\mathsf{CI}$ and $\mathsf{CI}^{\mathrm{bc}}$ denote the ECPs of the proposed $(1-\alpha)$-level confidence intervals $\mathsf{CI}_{\alpha}(T)$ and $\mathsf{CI}^{\mathrm{bc}}_{\alpha}(T)$, respectively, with $\alpha = 0.1$. rRMSE represents the relative empirical root mean squared error of $\hat{\sigma}^2$. The results are based on 1000 simulation replicates.
	}
	\label{Table_3}
	\resizebox{\linewidth}{!}{
		\begin{tabular}{cr@{\hspace{6mm}}cccccccccccc}\toprule\multicolumn{2}{c}{} & \multicolumn{3}{c}{$(p,q)=(1,1)$} & \multicolumn{3}{c}{$(p,q)=(1,2)$} & \multicolumn{3}{c}{$(p,q)=(3,3)$} & \multicolumn{3}{c}{$(p,q)=(5,5)$} \\[1mm]
			\cmidrule(lr){3-5}  \cmidrule(lr){6-8}
			\cmidrule(lr){9-11}  \cmidrule(lr){12-14}
			$\rho$ & \multicolumn{1}{c}{$n$} & \multicolumn{1}{c}{$\mathsf{CI}$} & \multicolumn{1}{c}{$\mathsf{CI}^{\mathrm{bc}}$ } & \multicolumn{1}{c}{\footnotesize rRMSE} & \multicolumn{1}{c}{$\mathsf{CI}$} & \multicolumn{1}{c}{$\mathsf{CI}^{\mathrm{bc}}$ } & \multicolumn{1}{c}{\footnotesize rRMSE} & \multicolumn{1}{c}{$\mathsf{CI}$} & \multicolumn{1}{c}{$\mathsf{CI}^{\mathrm{bc}}$ } & \multicolumn{1}{c}{\footnotesize rRMSE} & \multicolumn{1}{c}{$\mathsf{CI}$} & \multicolumn{1}{c}{$\mathsf{CI}^{\mathrm{bc}}$ } & \multicolumn{1}{c}{\footnotesize rRMSE} \\[1mm]
			\midrule\multirow{3}[1]{*}{0} & 1000  & 0.84  & 0.84  & 0.42  & 0.86  & 0.84  & 0.41  & 0.75  & 0.83  & 0.42  & 0.67  & 0.80  & 0.44 \\      & 5000  & 0.90  & 0.89  & 0.19  & 0.89  & 0.88  & 0.19  & 0.64  & 0.86  & 0.18  & 0.28  & 0.87  & 0.19 \\      & 10000 & 0.90  & 0.90  & 0.13  & 0.90  & 0.90  & 0.13  & 0.61  & 0.88  & 0.13  & 0.12  & 0.87  & 0.13 \\[2mm]
			\multirow{3}[0]{*}{0.3} & 1000  & 0.84  & 0.83  & 0.41  & 0.86  & 0.83  & 0.40  & 0.71  & 0.83  & 0.42  & 0.65  & 0.71  & 0.44 \\      & 5000  & 0.91  & 0.90  & 0.18  & 0.89  & 0.88  & 0.18  & 0.59  & 0.85  & 0.18  & 0.36  & 0.67  & 0.19 \\      & 10000 & 0.91  & 0.91  & 0.13  & 0.91  & 0.91  & 0.12  & 0.53  & 0.86  & 0.12  & 0.19  & 0.55  & 0.13 \\[2mm]
			\multirow{3}[0]{*}{0.5} & 1000  & 0.85  & 0.85  & 0.39  & 0.85  & 0.86  & 0.38  & 0.60  & 0.83  & 0.42  & 0.43  & 0.70  & 0.45 \\      & 5000  & 0.91  & 0.91  & 0.17  & 0.89  & 0.89  & 0.17  & 0.40  & 0.83  & 0.18  & 0.09  & 0.57  & 0.20 \\      & 10000 & 0.91  & 0.91  & 0.12  & 0.92  & 0.90  & 0.12  & 0.29  & 0.79  & 0.12  & 0.02  & 0.41  & 0.14 \\[2mm]
			\multirow{3}[0]{*}{0.7} & 1000  & 0.84  & 0.84  & 0.38  & 0.84  & 0.87  & 0.39  & 0.30  & 0.88  & 0.51  & 0.06  & 0.72  & 0.52 \\      & 5000  & 0.90  & 0.90  & 0.17  & 0.88  & 0.88  & 0.17  & 0.07  & 0.84  & 0.27  & 0.00  & 0.51  & 0.26 \\      & 10000 & 0.90  & 0.90  & 0.12  & 0.91  & 0.89  & 0.12  & 0.02  & 0.76  & 0.21  & 0.00  & 0.25  & 0.21 \\[2mm]
			\multirow{3}[1]{*}{0.9} & 1000  & 0.84  & 0.84  & 0.49  & 0.78  & 0.95  & 1.03  & 0.00  & 0.94  & 1.44  & 0.00  & 0.97  & 0.68 \\      & 5000  & 0.90  & 0.90  & 0.23  & 0.78  & 0.86  & 0.46  & 0.00  & 0.97  & 1.10  & 0.00  & 0.95  & 0.45 \\      & 10000 & 0.90  & 0.89  & 0.16  & 0.81  & 0.82  & 0.32  & 0.00  & 0.98  & 0.96  & 0.00  & 0.76  & 0.42 \\\bottomrule\end{tabular}%
	}
\end{table}

We next examine the results under {\bf Model 2} (Figure~\ref{Figure_2}). Overall, the patterns are highly consistent with those observed under {\bf Model 1}, exhibiting similar trends and leading to the same qualitative conclusions. A closer inspection reveals that the discrepancies are, if anything, slightly more pronounced under {\bf Model 2}. In particular, (i) when $(p,q) = (1,2)$ or $(3,1)$, the four methods exhibit less overlap in their power curves; and (ii) when $(p,q) = (5,5)$ and the sample size is relatively small (e.g., $n = 1000$), the bias-corrected methods $\mathsf{T}^{\mathrm{F,bc}}$ and $\mathsf{T}^{\mathrm{B,bc}}$ exhibit size inflation under $H_0$ (that is, when $\rho = 0$), with rejection probabilities around $0.2$, substantially above the nominal level $0.05$. However, this issue is quickly alleviated as the sample size increases.

Overall, these results show that, unlike the confidence-interval counterpart, the proposed testing procedures exhibit consistently stable and favorable performance across different models, even in the presence of unbounded distributions. This is in line with the general intuition that testing is often statistically easier than estimation.

\subsection{Comparison of variance estimators under $H_0$}

We next take a closer look at the performance of the two asymptotic variance estimators under $H_0$: the fast $k$NN-based estimator $\hat{\sigma}^2_{0,\mathrm{F}}$ in \eqref{eq:sigma2_F}$,$ and the $m$-out-of-$n$ bootstrap estimator $\hat{\sigma}^2_{0,\mathrm{B}}$ in \eqref{eq:sigma2_B}.

We focus on the following three aspects:
\begin{enumerate}[label=(\roman*)]
\item the rRMSE for estimating $\sigma_0^2$ (defined analogously to \eqref{eq:rRMSE});
\item the ECP of $\tau$, defined as the empirical probability that the $(1-\alpha)$-level confidence interval
\[
\Big( \tau_n -z_{1-\alpha/2} \hat{\sigma} /\sqrt{n} \ , \  \tau_n +z_{1-\alpha/2} \hat{\sigma} /\sqrt{n}\Big)
\]
covers the true value $\tau = 0$ under $H_0$, where $\hat{\sigma}$ denotes either $\hat{\sigma}_{0,\mathrm{F}}$ or $\hat{\sigma}_{0,\mathrm{B}}$, and where we set $\alpha = 0.05$;
\item the CPU time required to compute each estimator. All numerical experiments were implemented in MATLAB R2023b on a Windows desktop equipped with an Intel Xeon Platinum 8370C 64-core processor.
\end{enumerate}

Figures~\ref{Figure_4} and \ref{Figure_5} report the results for {\bf Models 1} and {\bf 2}, respectively. The overall patterns in the two figures are nearly identical. We therefore focus on Figure~\ref{Figure_4} under {\bf Model 1}.
For the rRMSE, the bootstrap estimator $\hat{\sigma}^2_{0,\mathrm{B}}$ consistently maintains a relatively low level, whereas the fast estimator $\hat{\sigma}^2_{0,\mathrm{F}}$ exhibits larger values when the sample size is small (e.g., $n = 1000$). As $n$ increases, however, the rRMSE of $\hat{\sigma}^2_{0,\mathrm{F}}$ decreases substantially, eventually becoming comparable to, and even slightly smaller than, that of $\hat{\sigma}^2_{0,\mathrm{B}}$ when $n = 10000$. This suggests that $\hat{\sigma}^2_{0,\mathrm{F}}$ converges more slowly than $\hat{\sigma}^2_{0,\mathrm{B}}$, but attains comparable performance once the sample size is sufficiently large.

Despite these differences in rRMSE, they do not translate into noticeable differences in statistical inference. Both in terms of empirical power (see Section~\ref{sec:powers}) and ECP, the two estimators yield nearly identical results, with their corresponding curves largely overlapping.

Finally, in terms of computational efficiency, the $k$NN-based estimator $\hat{\sigma}^2_{0,\mathrm{F}}$ is substantially faster to compute than the bootstrap estimator $\hat{\sigma}^2_{0,\mathrm{B}}$. Even with a relatively small number of bootstrap replications ($B = 200$), the $m$-out-of-$n$ bootstrap incurs a much higher computational cost than the $k$NN-based estimator.

In summary, $\hat{\sigma}^2_{0,\mathrm{B}}$ achieves lower rRMSE when the sample size is small, whereas $\hat{\sigma}^2_{0,\mathrm{F}}$ becomes increasingly competitive, faster, and more stable as the sample size grows.

\begin{figure}[!htbp]
\centering
\raisebox{1 em}{
\includegraphics[width= \linewidth]
{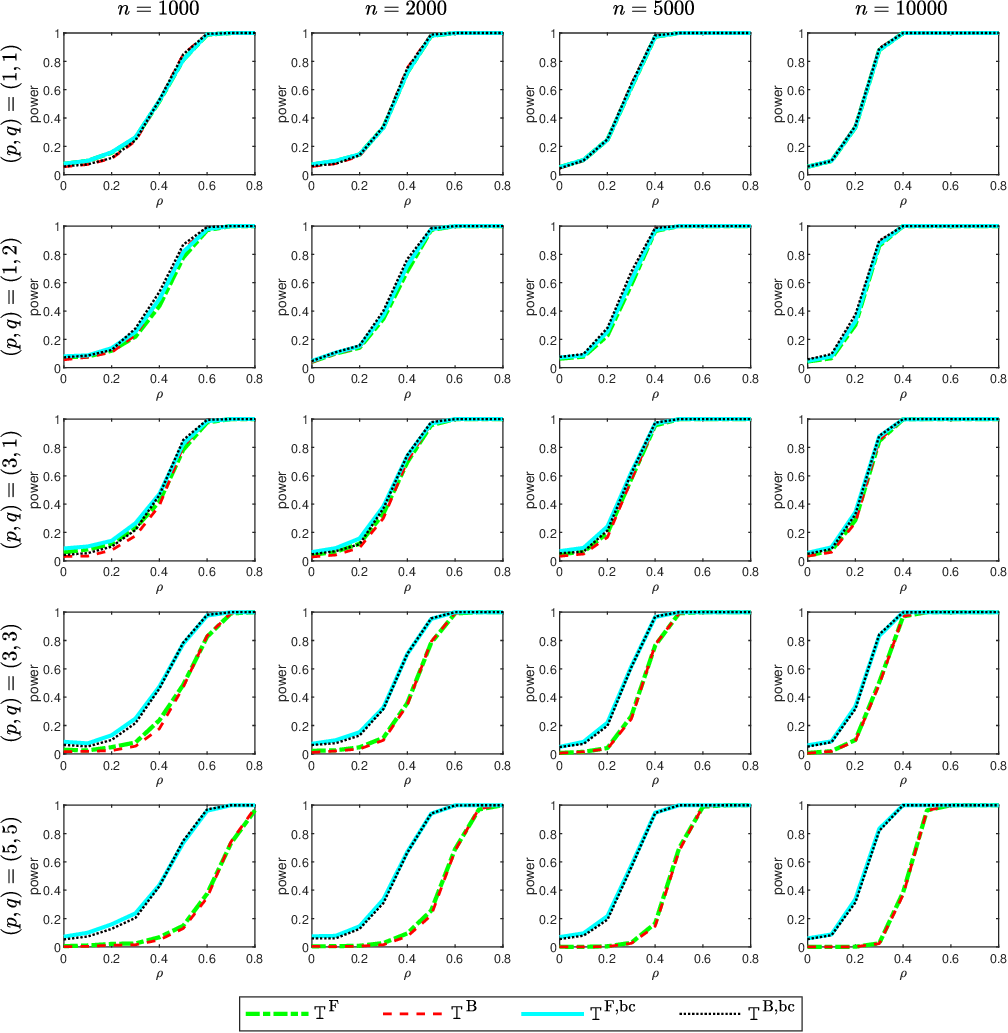}}
\caption{\textbf{(Model 1: power comparison)}
Empirical powers of the four proposed tests in \eqref{eq:four_tests} under Model 1, at nominal significance level $\alpha = 0.05$. The $y$-axis shows rejection frequencies based on $1000$ replicates, while the $x$-axis represents the model parameter $\rho$ in \eqref{eq:model1}, which characterizes the strength of conditional correlation. The case $\rho = 0$ corresponds to the null hypothesis.
}
\label{Figure_1}
\end{figure}

\begin{figure}[!htbp]
\centering
\raisebox{1 em}{
\includegraphics[width= \linewidth]
{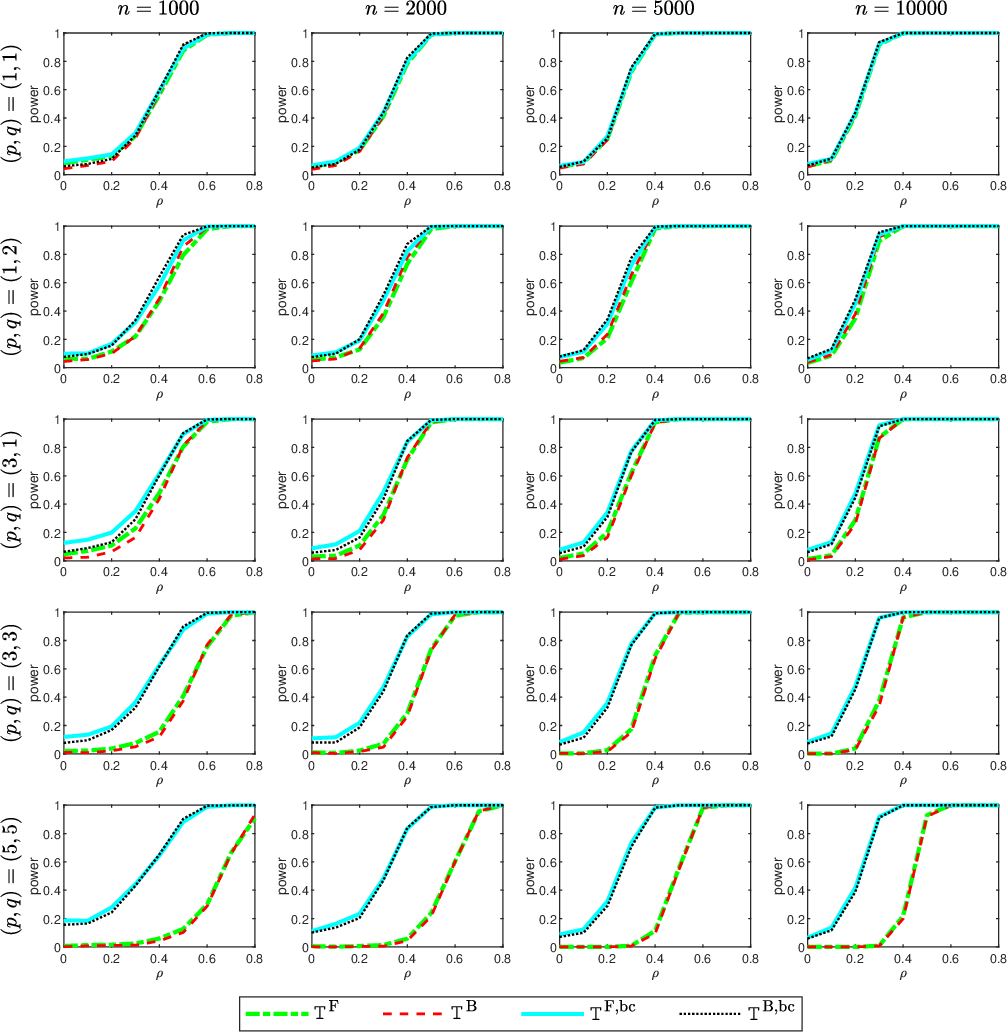}}
\caption{\textbf{(Model 2: power comparison)}
Empirical powers of the four proposed tests in \eqref{eq:four_tests} under Model 3, at  nominal  significance level $\alpha = 0.05$. The $y$-axis shows rejection frequencies based on $1000$ replicates, while the $x$-axis represents the model parameter $\rho$ in \eqref{eq:model1}, which characterizes the strength of conditional correlation. The case $\rho = 0$ corresponds to the null hypothesis.
}
\label{Figure_2}
\end{figure}

\begin{figure}[!htbp]
\centering
\raisebox{1 em}{
\includegraphics[width= \linewidth]
{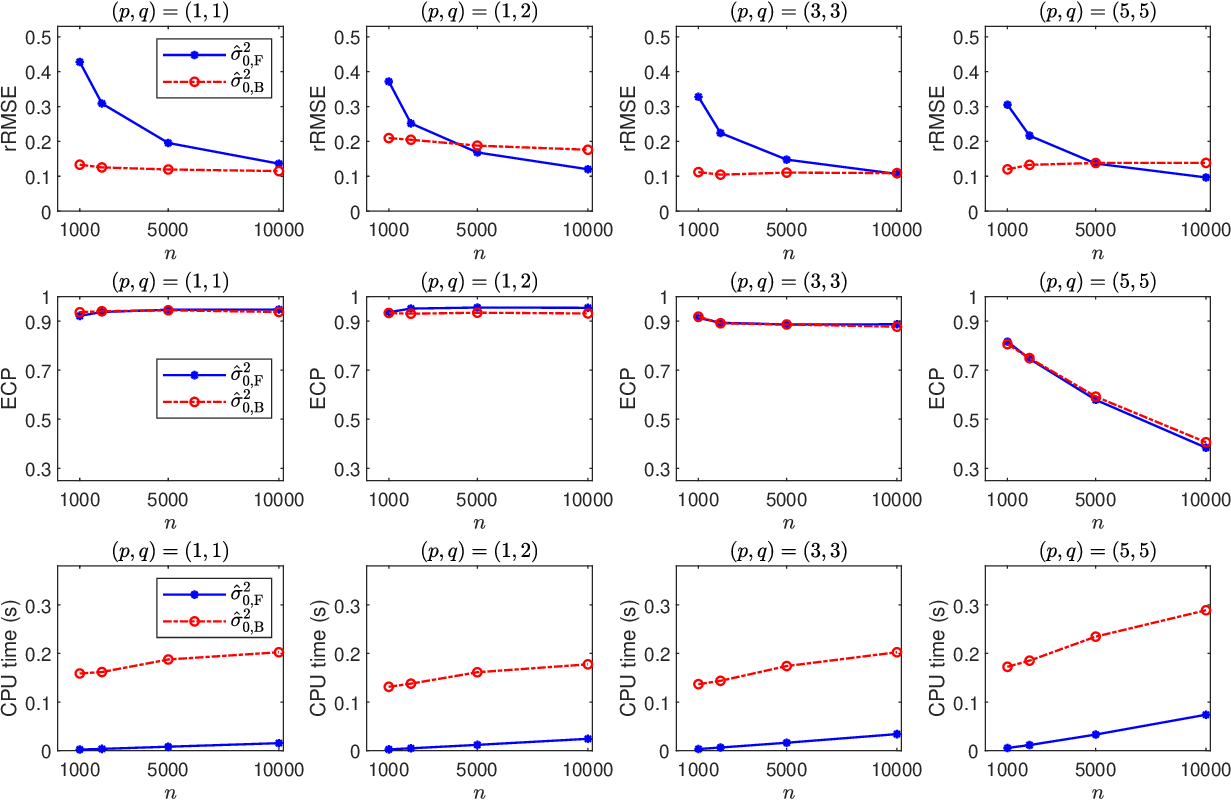}}
\caption{
\textbf{(Model 1: comparison of variance estimators under $H_0$)} 
For the fast $k$NN-based estimator $\hat{\sigma}^2_{0,\mathrm{F}}$ in \eqref{eq:sigma2_F} and the $m$-out-of-$n$ bootstrap estimator $\hat{\sigma}^2_{0,\mathrm{B}}$ in \eqref{eq:sigma2_B}, the first row reports the rRMSE values, the second row reports the ECP values, and the third row presents the CPU time (in seconds). All values are averaged over $1000$ simulation replicates.
}
\label{Figure_4}
\end{figure}

\begin{figure}[!htbp]
\centering
\raisebox{1 em}{
\includegraphics[width= \linewidth]
{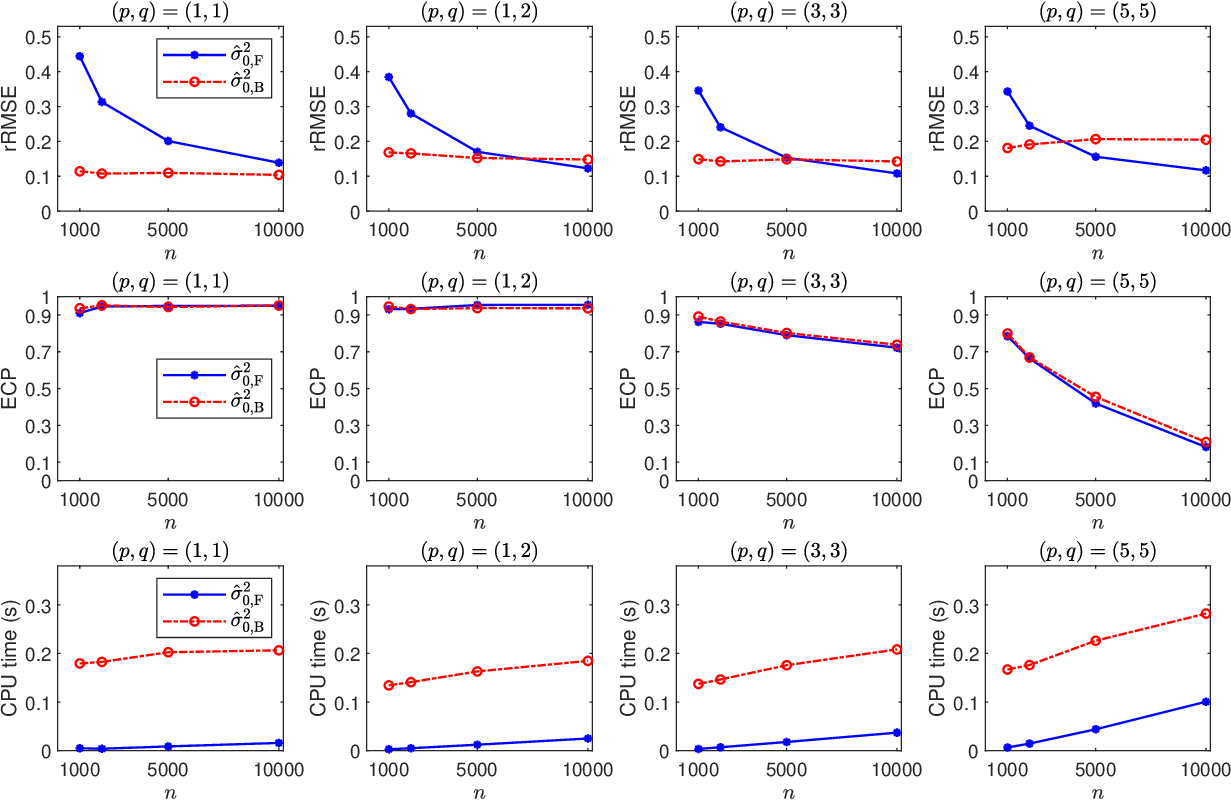}}
\caption{
\textbf{(Model 2: comparison of variance estimators under $H_0$)} 
For the fast $k$NN-based estimator $\hat{\sigma}^2_{0,\mathrm{F}}$ in \eqref{eq:sigma2_F} and the $m$-out-of-$n$ bootstrap estimator $\hat{\sigma}^2_{0,\mathrm{B}}$ in \eqref{eq:sigma2_B}, the first row reports the rRMSE values, the second row reports the ECP values, and the third row presents the CPU time (in seconds). All values are averaged over $1000$ simulation replicates.
}
\label{Figure_5}
\end{figure}

\clearpage

{\centering\huge Supplement to ``Limit theorems of Azadkia-Chatterjee's conditional graph correlation'' \par}

\appendix

\vspace{1cm}

The appendix is organized as follows. Appendix~\ref{secA:assump} presents additional assumptions and notation. Appendix~\ref{secA:proofs-1} provides the proofs of Lemmas~\ref{lemma:o_d} and \ref{lemma:two_NNGs} on NNG theory. Appendix~\ref{secA:proofs-2} contains the proofs of Theorems~\ref{thm:var_xi},  \ref{thm:est_var_xi}, and Proposition \ref{prop:nlogn} in Section~\ref{sec:Lin_Han}. 
Appendix~\ref{secA:proofs-3} presents the proof of Theorem~\ref{thm:CLT-main} in Section~\ref{sec:theory}, while Appendix~\ref{secA:proofs-4} provides the proofs of the remaining theorems in Section~\ref{sec:theory}. Appendix~\ref{secA:aux-lemmas} collects the auxiliary lemmas required for proving Theorem~\ref{thm:CLT-main}.

\section{Supplementary assumptions and notation} \label{secA:assump}

The following assumptions, Assumptions~\ref{assump_A.1}--\ref{assump_A.4}, are imposed on the joint distribution of $(Y, \bZ)$. 
For $\bz \in \supp(\bZ)$ and $t \in \mathbbR$, define
$G_\bz(t) = \E[\Ind(Y \geq t) \mid \bZ = \bz] $.
As an important remark, 
we also assume that the corresponding variants of Assumptions~\ref{assump_A.1}--\ref{assump_A.4} remain valid when $\bZ$ is replaced by $(\bX, \bZ)$.

\begin{assumption}[Regularity conditions on $(Y, \bZ)$ distribution]\label{assump_A.1} We assume that
\begin{enumerate}[label=(\roman*)]
\item 	$\supp(\bZ)$ is compact, and has a Lipschitz boundary (\cite{grisvard2011elliptic}, Definition 1.2.1.1).
\item $\bZ$ admits a density $f_\bZ$ satisfying $\inf_{\bZ \in \supp(\bZ)} f_\bZ \geq C $ for some constant $C>0$. Moreover, $f_\bZ$ is Lipschitz continuous on $\supp(\bZ)$. 
\item For every $t \in \supp(Y)$, there exists a twice continuously differentiable function $\tilde{G}.(t): \mathbbR^q \to [0,1]$ such that $\tilde{G}.(t) = G.(t)$ on $\supp(\bZ)$, and 
\begin{eqnarray*}
\sup_{t \in \supp(Y), \bz' \in \mathbbR^q} \Big(\|\nabla_{\bz} \tilde{G}_{\bz'}(t)\| + \|\nabla_{\bz} \tilde{G}_{\bz'}(t)\|_2\Big) < \infty,
\end{eqnarray*}
where $\|\cdot \|_2$ denotes the matrix spectral norm. 
\end{enumerate}
\end{assumption}

\begin{assumption}[Boundness condition on $\bZ$]\label{assump_A.2}
For any $r>0$, there exists some constant $C_r$ only depending on $r$, such that
\begin{eqnarray*}
\big\{\E(\|\bZ_1-\bZ_{N(1)}\|^r)\big\}^{1/r} \leq C_r n^{-1/q}
\end{eqnarray*}
holds for sufficiently large $n$. 
\end{assumption}

For a positive integer $r$, let $\Lambda_r$ denote the set of multi-indices in $\lbr q \rbr$, with total degree $r$, i.e., $\Lambda_r = \big\{\boldsymbol{\alpha}=(\alpha_1,\dots, \alpha_q): \sum_{i=1}^q \alpha_i = r; \alpha_i \in \mathbb{Z}^{\geq 0}, \text{for } i \in \lbr q \rbr\big\}$, 
where $\mathbb{Z}^{\geq 0}$ is the set of non-negative integers.
For $\boldsymbol{\alpha}=(\alpha_1,\dots, \alpha_q) \in \Lambda_r$ and a $q$-variate function $f$, let $D^{\boldsymbol{\alpha}}f = \partial^{\alpha_1+\dots+\alpha_q}f/\partial z_1^{\alpha_1}\cdots z_q^{\alpha_q}$ be the partial derivative of $f$ on $\Lambda_r $.  
Recall the basis functions $\boldsymbolp(\bz) = (p_1(\bz),...,p_K(\bz))^\top$ in Algorithm \ref{alg2}. 
For integer $r>0$, let $\zeta_{r,K} = \max_{\boldsymbol{\alpha} \in \Lambda_r} \sup_{\bz \in \mathbbR^q}
\|D^{\boldsymbol{\alpha}}\boldsymbolp(\bz)\|$, and let $\zeta_{0,K} = 
\sup_{\bz \in \mathbbR^q}
\|\boldsymbolp(\bz)\|$.
Let $\underline{\lambda}_K = \lambda_{\min}\E\{\boldsymbolp(\bZ)\boldsymbolp(\bZ)^\top\}$.
\begin{assumption}[Smoothness condition for $G_\bz(t)$] \label{assump_A.3}
Assume that
both $\max_{\boldsymbol{\alpha} \in \Lambda_1} \|D^{\boldsymbol{\alpha} } G.(t)\|_{\infty}$ and $\max_{\boldsymbol{\alpha} \in \Lambda_{\lfloor q/2 \rfloor+1}} \|D^{\boldsymbol{\alpha} } G.(t)\|_{\infty}$ are uniformly bounded over $t \in \supp(Y)$. 
\end{assumption}

\begin{assumption}[Conditions on penalty parameter and basis functions]\label{assump_A.4}
Let $\lambda_n$ be the penalty parameter in Algorithm \ref{alg2}. 
Assume $c_1 \cdot n^{-c} \leq \lambda_n \leq  c_2 \cdot n^{-c}$ for some constants $c, c_1, c_2>0$. Assume $K \to \infty$ as $n \to \infty$, and $\underline{\lambda}_K > \lambda_n $ for sufficiently large $n$. Assume $K/n\to 0$, $\zeta_{0,K} = o\big((n/\log n)^{1/4}(\underline{\lambda}_K-\lambda_n)^{1/2}\big)$, and $\underline{\lambda}_K^{-1}\,\zeta_{0,K}^2\,\log K=o(n)$. 
\end{assumption}

Assumptions~\ref{assump_A.1}--\ref{assump_A.4} correspond to Assumptions~3.1, 4.1, 4.2, and 4.5 in \cite{azadkia2026biascorrection}, respectively. Specifically, Assumption~\ref{assump_A.1} imposes regularity conditions on the distribution of $(Y, \bZ)$, which are required to establish the bias rate in Theorem~\ref{thm:bias_correct}(i). 
Assumptions~\ref{assump_A.2} and \ref{assump_A.3} are used to develop the general bias correction theory. Assumption~\ref{assump_A.4} is imposed for the ridge regression estimator in Algorithm~\ref{alg2}, ensuring its effectiveness for bias correction. We refer the reader to \cite{azadkia2026biascorrection} for further details.

\paragraph*{Notation in proofs.}
For a sequence of random vectors $\bX_1,\dots, \bX_n \in \mathbbR^d$,
we use the boldface notation  $\bfX = (\bX_1,\dots, \bX_n) \in (\mathbbR^{d})^n$ to denote the joint vector collecting all $n$ samples. For $\bw \in \mathbbR^d$ and $r >0$, let $\mathcal{B}(\bw,r)$ denote the $d$-dimensional ball with center $\bw$ and radius $r$. For $\bw_1, \bw_2 \in \mathbbR^d$, let $\tilde{\mathcal{B}}(\bw_1,\bw_2) = \mathcal{B}(\bw_1,\|\bw_2-\bw_1\|)$ denote the $d$-dimensional ball centered at $\bw_1$, with $\bw_2$ lying on its surface. 
For a $d$-dimensional manifold $M \subseteq \mathbbR^d$, let $V(M)$ denote the volume of this manifold. For $x \in \mathbbR$ and $A \subseteq \mathbbR$, write $\Ind_A(x) = \Ind(x \in A)$.
For a random vector $\bZ \in \mathbbR^d$,
let $\tmu_{\bz} := \mu_{Y|\bZ=\bz}$ 
denote the conditional law of $Y$ given $\bZ=\bz$. Refer to Lemma \ref{lemA:mu} below for the existence of $\tmu_{\bz}$. Similarly, let $\tmu_{(\bx,\bz)} := \mu_{Y|(\bX,\bZ)=(\bx,\bz)}$ 
denote the conditional law of $Y$ given $(\bX,\bZ)=(\bx,\bz)$.
\begin{lemmaA}[Theorem 2.1.22 and Exercise 4.1.18 in \cite{MR3930614}] \label{lemA:mu}

For each Borel set $A \subseteq \mathbbR$, there is a measurable map 
\begin{eqnarray*}
\bz \longmapsto \tmu_\bz(A), \quad \text{from }\mathrm{supp}(\bZ) \text{ into } [0,1],
\end{eqnarray*}
such that 
\begin{enumerate}
\item for each $A$, $\tmu_\bZ(A)$ is a version of $\P(Y \in A \mid \bZ)$;
\item with probability one, $\tmu_\bZ$ is a probability measure on $\mathbbR$. 
\end{enumerate}
\end{lemmaA}

\section{Proofs of Lemmas \ref{lemma:o_d} and \ref{lemma:two_NNGs}}  \label{secA:proofs-1}

Both the proofs of Lemmas~\ref{lemma:o_d} and \ref{lemma:two_NNGs} require the following  Lemma \ref{lemA:Xn_conP_C}. 
\begin{lemmaA}\label{lemA:Xn_conP_C}
Let $\{X_n\}_{n=1,2\dots}$ be a sequence of nonnegative one-dimensional random variables. Let $C>0$ be some constant. Assume that
$\liminf_{n\to \infty} X_n \geq C$ holds with probability one, and $\limsup_{n\to \infty}\E(X_n)  \leq C$. Then 
\begin{eqnarray*}
X_n \conP C, \quad \text{as } n \to \infty.
\end{eqnarray*}
\end{lemmaA}

\begin{proof}
By condition that $\P(\liminf_{n\to \infty} X_n \geq C)=1$, for any
$\epsilon>0$, we have
\begin{eqnarray}
\P(X_n > C-\epsilon) >1-\epsilon \label{eq:lemA:Xn_conP_C:1}
\end{eqnarray}
holds for sufficiently large $n$. 
Next, we share prove that for any $\epsilon>0$, 
\begin{eqnarray*}
\P(X_n<C+\epsilon) >1-\epsilon
\end{eqnarray*}
holds for sufficiently large $n$. Then the proof is completed by combing the above two results. 

If this is not true, then there exists some $\delta>0$, and a subsequence $X_{i_1},X_{i_2}\dots$ such that 
\begin{eqnarray*}
\P(X_{i_k} >1+\delta )>\delta, \quad \text{for all } k=1,2\dots.
\end{eqnarray*}
Combining this with \eqref{eq:lemA:Xn_conP_C:1}, we obtain that, for any $\epsilon > 0$, the following holds for sufficiently large $k$:
\begin{eqnarray*}
\E(X_{i_k}) 
&\geq& \P(X_{i_k}\leq 1-\epsilon)\cdot 0 + \P(1-\epsilon <X_{i_k}\leq 1+\delta)\cdot (C-\epsilon)
+ \P(X_{i_k}> 1+\delta)\cdot (C+\delta) \cr
&\geq&
(1-\epsilon-\delta) \cdot(C-\epsilon) + \delta \cdot (C+\delta).
\end{eqnarray*}
Note that 
\begin{eqnarray*}
\lim_{\epsilon \to 0} \big\{(1-\epsilon-\delta) \cdot(C-\epsilon) + \delta \cdot (C+\delta)\big\} = C+\delta^2 >C.
\end{eqnarray*}
Thus, by choosing sufficiently small $\epsilon$, we have that $\E(X_{i_k}) >C$ holds for sufficiently large $k$, which is obviously a contradiction with $\limsup_{n\to \infty}\E(X_n)\leq C$. 
This completes the proof. 
\end{proof}

\subsection{Proof of Lemma \ref{lemma:o_d}}
\begin{proof}[Proof of Lemma \ref{lemma:o_d}] 	For $i \in \lbr n \rbr$, 
let $N(i)$ indexes the  NN of $\bW_i$ among $\lbr n \rbr$.
It suffices to show that, as $n \to \infty$,
\begin{eqnarray*}
\E\big[\#\{j\in \lbr n \rbr:j \neq 1, N(j) = N(1)\} \ \big | \ \bW_1\big] \conP \mathfrak{o}_d.
\end{eqnarray*}

Let  $f$ be the density function of $\bW$. Let $\calW^o \subseteq \mathbbR^d$ denote the interior of the support $\supp(\bW)$, and let $\calW^+ = \{\bw \in \calW^o: f(\bw)>0\}$. 	Since $\bW$ is absolutely continuous and $f$ is continuous on $\supp(\bW)$, it follows that $\calW^+$ is an open subset of $\mathbbR^d$ and $\P(\bW \in \calW^+) = 1$.

Fix some $\bw_1 \in \calW^+$. We have
\begin{eqnarray*}
&&\E\big[\#\{j\in \lbr n \rbr:j \neq 1, N(j) = N(1)\} \ \big | \ \bW_1=\bw_1\big] \cr
&=&
\sum_{j=2}^n\sum_{k\in\lbr n \rbr:k\neq 1, k \neq j}
\E\big[\Ind(N(1)=j, N(k) =j) \ \big | \ \bW_1=\bw_1\big] \cr
&=&
(n-1)\cdot(n-2)	\cdot \E\big[\Ind(N(1)=2, N(3) =2) \ \big | \ \bW_1=\bw_1\big].
\end{eqnarray*}
Note that the event $\{N(1)=2,\, N(3)=2\}$ occurs if and only if the following two conditions are satisfied:
\begin{enumerate}
\item $ \hspace{2.5cm}\|\bW_3-\bW_1\| > \max\{\|\bW_2-\bW_1\|, \|\bW_3-\bW_2\|\}.$

This guarantees that $\bW_3$ is farther away from $\bW_1$ than $\bW_2$ is, and also $\bW_1$ is farther away from $\bW_3$ than $\bW_2$ is. Thus, among $\{\bW_1, \bW_2,\bW_3\}$, $\bW_2$ is the NN of $\bW_1$ as well as the NN of $\bW_3$.
\item $	\hspace{2.5cm} \bW_j \notin \tcalB(\bW_1, \bW_2)\cup \tcalB(\bW_3, \bW_2), \quad \text{for }j = 4,\dots, n.$

This ensures that $\bW_j$ is not the NN of either $\bW_1$ or $\bW_3$. 
\end{enumerate}
It follows that
\begin{eqnarray}
&&\E\big[\Ind(N(1)=2, N(3) =2) \ \big | \ \bW_1=\bw_1, \bW_2=\bw_2, \bW_3=\bw_3 \big] \cr
&=&
\Ind\big(\|\bw_3-\bw_1\| > \max\{\|\bw_2-\bw_1\|, \|\bw_3-\bw_2\|\}\big) \cr
&&\times
\P\Big(	\bigcap_{j=4}^n \big\{\bW_j \notin \tcalB(\bw_1, \bw_2)\cup \tcalB(\bw_3, \bw_2)\big\} \ \Big | \ \bW_1=\bw_1, \bW_2=\bw_2, \bW_3=\bw_3\Big) \cr
&=&
\Ind\big(\|\bw_3-\bw_1\| > \max\{\|\bw_2-\bw_1\|, \|\bw_3-\bw_2\|\}\big)
\cdot
\Big(1 - \int_{\bw_4\in \tcalB(\bw_1, \bw_2)\cup \tcalB(\bw_3, \bw_2)} f(\bw_4) \d \bw_4\Big)^{n-3} .
\nonumber
\end{eqnarray}
Therefore,
\begin{eqnarray}
&&\E\big[\#\{j\in \lbr n \rbr:j \neq 1, N(j) = N(1)\} \ \big | \ \bW_1=\bw_1 \big] \cr
&=&
(n-1)\cdot(n-2)	\cdot \E\big[\Ind(N(1)=2, N(3) =2) \ \big | \ \bW_1=\bw_1\big] \cr
&=&
(n-1)\cdot(n-2)	\cr
&& \times\int \int f(\bw_2)\cdot f(\bw_3) \cdot \E\big[\Ind(N(1)=2, N(3) =2) \ \big | \ \bW_1=\bw_1, \bW_2=\bw_2, \bW_3=\bw_3\big]
\d \bw_2 \d \bw_3 \cr
&=&
(n-1)\cdot(n-2)	\cdot \int \int f(\bw_2)\cdot f(\bw_3) \cdot\Ind \Big(\|\bw_3-\bw_1\| > \max\big\{\|\bw_2-\bw_1\|, \|\bw_3-\bw_2\|\big\}\Big) \cr
&& \qquad \qquad \qquad\qquad \qquad
\times \Big(1 - \int_{\bw_4\in \tcalB(\bw_1, \bw_2)\cup \tcalB(\bw_3, \bw_2)} f(\bw_4) \d \bw_4\Big)^{n-3} \d \bw_2 \d \bw_3. \label{eq:lemma:o_d:1}
\end{eqnarray}

Write $f_1: =f(\bw_1)>0$ as the probability density of $\bW$ at $\bw_1$. 
We next apply a change of variables to the integration variable $\bw_2$ and $\bw_3$.
Let 
\begin{eqnarray*}
\bu_2 = n^{1/d}\cdot f_1^{1/d}\cdot (\bw_2-\bw_1), \quad \text{and } \bu_3 = n^{1/d}\cdot f_1^{1/d}\cdot (\bw_2-\bw_3).
\end{eqnarray*}	
Then 
\begin{eqnarray}
\bw_2 = \bw_1 + n^{-1/d}\cdot f_1^{-1/d} \cdot \bu_2, \quad \bw_3 = \bw_1 + n^{-1/d}\cdot f_1^{-1/d} \cdot \bu_2 - n^{-1/d}\cdot f_1^{-1/d} \cdot \bu_3.  \label{eq:lemma:o_d:2}
\end{eqnarray}
Obviously, the mapping from $(\bw_2,\bw_3)$ to $(\bu_2,\bu_3$) is a linear bijection. 
It is easy to verify the following three equations:
\begin{eqnarray}
\text{(1)}&&	\frac{\d \bw_2 \d \bw_3}{\d \bu_2 \d \bu_3 } = \frac{1}{n^2\cdot f_1^2}; \cr
\text{(2)}&& \Ind\big(\|\bw_3-\bw_1\| > \max\{\|\bw_2-\bw_1\|, \|\bw_3-\bw_2\|\}\big) = 
\Ind\big(\|\bu_3-\bu_2\| > \max\{\|\bu_2\|, \|\bu_3\|\}\big); \qquad \quad
\cr
\text{(3)}&&
V( \tcalB(\bw_1, \bw_2)\cup \tcalB(\bw_3, \bw_2))  
= (n\cdot f_1)^{-1}
\cdot V( \mathcal{B}(\bu_2, \|\bu_2\|)\cup \mathcal{B}(\bu_3, \|\bu_3\|)). \label{eq:lemma:o_d:3}
\end{eqnarray}
By replacing $\d \bw_2 \d \bw_3$ with $\d \bu_2 \d \bu_3$ in the above integral \eqref{eq:lemma:o_d:1}, we get
\begin{eqnarray*}
&&\E\big[\#\{j\in \lbr n \rbr:j \neq 1, N(j) = N(1)\} \ \big | \ \bW_1=\bw_1\big] \cr
&=&
\frac{(n-1)\cdot(n-2)	}{n^2\cdot f_1^2}\cdot \int \int f(\bw_2)\cdot f(\bw_3) \cdot\Ind\big(\|\bu_3-\bu_2\| > \max\{\|\bu_2\|, \|\bu_3\|\}\big) \cr
&& \qquad \qquad \qquad\qquad \qquad
\times \Big(1 - \int_{\bw_4\in \tcalB(\bw_1, \bw_2)\cup \tcalB(\bw_3, \bw_2)} f(\bw_4) \d \bw_4\Big)^{n-3} \d \bu_2 \d \bu_3.
\end{eqnarray*}
Note that the original variables $\bw_2$ and $\bw_3$ still appear in the integrand. In this case, they can be regarded as functions of $\bu_2$ and $\bu_3$, as specified in \eqref{eq:lemma:o_d:2}.

For $r >0$, define
\begin{eqnarray}
&&g_n(r):=	\frac{(n-1)\cdot(n-2)	}{n^2\cdot f_1^2}\cdot \int_{\|\bu_3\| <r} \int_{\|\bu_2\|<r} f(\bw_2)\cdot f(\bw_3) \cdot\Ind(\|\bu_3-\bu_2\| > \max\{\|\bu_2\|, \|\bu_3\|\}) \cr
&& \qquad \qquad \qquad\qquad \qquad
\times \Big(1 - \int_{\bw_4\in \tcalB(\bw_1, \bw_2)\cup \tcalB(\bw_3, \bw_2)} f(\bw_4) \d \bw_4\Big)^{n-3} \d \bu_2 \d \bu_3.
\nonumber
\end{eqnarray}
Then 
\begin{eqnarray*}
\E\big[\#\{j\in \lbr n \rbr:j \neq 1, N(j) = N(1)\} \ \big | \ \bW_1=\bw_1\big] = \lim_{r \to \infty} g_n(r)=:g_n(\infty).
\end{eqnarray*}

Now fix $r>0$. As $n \to \infty$, it is clear that
\begin{eqnarray*}
&&	\sup_{(\bw_2, \bw_3): \ \|\bu_2\| <r, \|\bu_3\|<r}\|\bw_3-\bw_1\| \cr
&=& \sup_{(\bu_2, \bu_3): \  \|\bu_2\| <r, \|\bu_3\|<r}\| n^{-1/d}\cdot f_1^{-1/d} \cdot \bu_2 - n^{-1/d}\cdot f_1^{-1/d} \cdot \bu_3\| \to 0.
\end{eqnarray*}
Similarly, we can verify that
\begin{eqnarray*}
&&	\sup_{\bw_2: \ \|\bu_2\| <r}\|\bw_2-\bw_1\| \to 0, \cr
&&\sup_{(\bw_2,\bw_3,\bw_4): \ \|\bu_2\| <r, \|\bu_3\|<r, \bw_4\in \tcalB(\bw_1, \bw_2)\cup \tcalB(\bw_3, \bw_2)}\|\bw_4-\bw_1\| \to 0.
\end{eqnarray*}
Note that the density function $f$ is uniformly continuous in some neighborhood of $\bw=\bw_1$. Thus, 
for any $\epsilon>0$, 
\begin{eqnarray*}
\sup_{(\bw_2,\bw_3,\bw_4): \ \|\bu_2\|<r, \ \|\bu_3\|<r, \ \bw_4\in \tcalB(\bw_1, \bw_2)\cup \tcalB(\bw_3, \bw_2)} \max\Big\{|f(\bw_2)- f_1|, |f(\bw_3)- f_1|, |f(\bw_4)- f_1|\Big\} < \epsilon
\end{eqnarray*}
holds for sufficiently large $n$. 
It follows that
\begin{eqnarray*}
g_n(r) &\geq& 
\frac{(n-1)\cdot(n-2)	}{n^2\cdot f_1^2}\cdot \int_{\|\bu_3\| <r} \int_{\|\bu_2\|<r} (f_1-\epsilon)^2 \cdot\Ind(\|\bu_3-\bu_2\| > \max\{\|\bu_2\|, \|\bu_3\|\}) \cr
&& \qquad \qquad \qquad\qquad \qquad
\times \Big(1 - \int_{\bw_4\in \tcalB(\bw_1, \bw_2)\cup \tcalB(\bw_3, \bw_2)} (f_1+\epsilon) \d \bw_4\Big)^{n-3} \d \bu_2 \d \bu_3 \cr
&=&
\frac{(n-1)\cdot(n-2)	}{n^2\cdot f_1^2}\cdot \int_{\|\bu_3\| <r} \int_{\|\bu_2\|<r} (f_1-\epsilon)^2 \cdot\Ind(\|\bu_3-\bu_2\| > \max\{\|\bu_2\|, \|\bu_3\|\}) \cr
&& \qquad \qquad \qquad
\times \Big(1 - \frac{f_1+\epsilon}{f_1 \cdot n}\cdot V\big( \mathcal{B}(\bu_2, \|\bu_2\|)\cup \mathcal{B}(\bu_3, \|\bu_3\|)\big)\Big)^{n-3} \d \bu_2 \d \bu_3
\end{eqnarray*}
holds for sufficiently large $n$, 
where the last step is due to \eqref{eq:lemma:o_d:3}.

Note that as $n \to \infty$, 
\begin{eqnarray*}
&&\Big(1 - \frac{f_1+\epsilon}{f_1 \cdot n}\cdot V\big( \mathcal{B}(\bu_2, \|\bu_2\|)\cup \mathcal{B}(\bu_3, \|\bu_3\|)\big)\Big)^{n-3} \cr
&\to&
\exp\Big(- \frac{f_1+\epsilon}{f_1} \cdot V\big( \mathcal{B}(\bu_2, \|\bu_2\|)\cup \mathcal{B}(\bu_3, \|\bu_3\|)\big)\Big).
\end{eqnarray*}
Moreover, it is straightforward to verify that this convergence holds uniformly for all $\bu_2$ and $\bu_3$ satisfying $\|\bu_2\| < r$ and $\|\bu_3\| < r$.
Since $\epsilon > 0$ is arbitrary, letting $\epsilon \to 0$ yields
\begin{eqnarray*}
\liminf_{n\to \infty} g_n(r)
&\geq&
\int_{\|\bu_3\| <r} \int_{\|\bu_2\|<r}  \Ind\big(\|\bu_3-\bu_2\| > \max\{\|\bu_2\|, \|\bu_3\|\}\big) \cr
&& \qquad \qquad 
\times 	\exp\Big(-  V\big( \mathcal{B}(\bu_2, \|\bu_2\|)\cup \mathcal{B}(\bu_3, \|\bu_3\|)\big)\Big) \d \bu_2 \d \bu_3.
\end{eqnarray*}
Note that $\E\big[\#\{j\in \lbr n \rbr:j \neq 1, N(j) = N(1)\} \ \big | \ \bW_1=\bw_1\big] = g_n(\infty) \geq g_n(r)$ for all $r>0$. Then
\begin{eqnarray*}
&&\liminf_{n \to \infty}\E\big[\#\{j\in \lbr n \rbr:j \neq 1, N(j) = N(1)\} \ \big | \ \bW_1=\bw_1\big] \cr
&\geq &
\sup_{r>0} \liminf_{n \to \infty} g_n(r) \cr
&\geq&
\sup_{r>0} \int_{\|\bu_3\| <r} \int_{\|\bu_2\|<r}  \Ind\big(\|\bu_3-\bu_2\| > \max\{\|\bu_2\|, \|\bu_3\|\}\big) \cr
&& \qquad \qquad 
\times 	\exp\Big(-  V\big( \mathcal{B}(\bu_2, \|\bu_2\|)\cup \mathcal{B}(\bu_3, \|\bu_3\|)\big)\Big) \d \bu_2 \d \bu_3 \cr
&=&
\int \int \Ind\big(\|\bu_3-\bu_2\| > \max\{\|\bu_2\|, \|\bu_3\|\}\big) \cr
&& \qquad 
\times 	\exp\Big(-  V\big( \mathcal{B}(\bu_2, \|\bu_2\|)\cup \mathcal{B}(\bu_3, \|\bu_3\|)\big)\Big) \d \bu_2 \d \bu_3 \cr
&=& \mathfrak{o}_d,
\end{eqnarray*}
where the last step follows from the definition of $\mathfrak{o}_d$; see \eqref{eq:o_d} and \citet[Theorem 3.1]{Shi_Drton_Han_2024_Bernoulli}. Since $\bw_1 \in \calW^+$ and $\P(\bW \in \calW^+) = 1$, it follows that
\begin{eqnarray*}
\liminf_{n \to \infty}\E\big[\#\{j\in \lbr n \rbr:j \neq 1, N(j) = N(1)\} \ \big | \ \bW_1\big] \geq \mathfrak{o}_d
\end{eqnarray*}
holds with probability one. 

On the other hand, by Lemma 3.7 in \cite{Shi_Drton_Han_2024_Bernoulli}, we have
\begin{eqnarray*}
\lim_{n \to \infty}\E\big[\#\{j\in \lbr n \rbr:j \neq 1, N(j) = N(1)\}\big] = \mathfrak{o}_d.
\end{eqnarray*}
Combining these two results and applying Lemma \ref{lemA:Xn_conP_C}, we complete the proof. 
\end{proof}

\subsection{Proof of Lemma \ref{lemma:two_NNGs}} \label{secB.2}
Before the proof of Lemma~\ref{lemma:two_NNGs}, we first present Lemma~\ref{lemA:cond_pdf}.

\begin{lemmaA} \label{lemA:cond_pdf}
Assume that the $d$-dimensional random vector $\bW$
is absolutely continuous and admits a continuous probability density function
$f(\bw)$ over its support. 
Let $\bW_1, \bW_2,\dots, \bW_n$ be i.i.d random vectors sampled from $\bW$. Let $M(1)$ be the index of the NN of $\bW_1$. Let $\bw_1, \bw_2 \in \supp(\bW)$ be any distinct points satisfying $f(\bw_1)>0$ and $f(\bw_2)>0$. 
Then conditional on the event $\{\bW_1=\bw_1, \bW_2=\bw_2, M(1)=2\}$, $\bW_3,\dots, \bW_n$ are i.i.d with density function
\begin{eqnarray*}
\breve f_{\bw_1,\bw_2}(\bw) := C_{\bw_1, \bw_2}\cdot f(\bw) \cdot \Ind(\bw \notin \tcalB(\bw_1, \bw_2)),  \quad \text{for } \bw \in \supp(\bW),
\end{eqnarray*}
where
\begin{eqnarray*}
C_{\bw_1, \bw_2} = \frac{1}{1-\P(\bW \in \tcalB(\bw_1, \bw_2) )}
= \frac{1}{1-\int_{\bw \in\tcalB(\bw_1, \bw_2) } f(\bw) \d \bw}.
\end{eqnarray*}
\end{lemmaA}
\begin{proof}
Let $A= \{\bW_1=\bw_1, \bW_2=\bw_2, M(1)=2\}$, and let $f(\bw_3, \dots, \bw_n \mid A)$ be the joint density function of $\bW_3, \dots, \bW_n$ conditional on $A$. Since $\bW_2$ is the NN of $\bW_1$ under the event $A$, $f(\bw_3, \dots, \bw_n \mid A)$ must satisfy the following three properties.
\begin{enumerate}
\item (Zero density in $\tcalB(\bw_1, \bw_2)$).
If there exists any $\bw_i \in \tcalB(\bw_1, \bw_2)$ for $i=3, \dots, n$, then
$f(\bw_3, \dots, \bw_n \mid A) = 0$. This is because, conditional on $A$, $\bw_2$ is the NN of $\bw_1$. Hence, $\|\bW_i-\bw_1\|> \|\bw_2-\bw_1\|$, implying that $\bW_i \notin \tcalB(\bw_1, \bw_2)$ for $i=3,\dots,n$. 
\item (Proportionality). 
For any $\bw_3,\dots, \bw_n \in \supp(\bW)$, and $\bw_3', \dots, \bw_n' \in \supp(\bW)$ such that $f(\bw_i)>0$, $f(\bw'_i)>0$, and 
$\bw_i, \bw_i' \notin \tcalB(\bw_1, \bw_2)$ for all $i=3, \dots, n$, we have 
\begin{eqnarray*}
\frac{f(\bw_3, \dots, \bw_n \mid A)}{f(\bw_3', \dots, \bw_n' \mid A)}
=
\frac{f(\bw_3,\dots, \bw_n)}{f(\bw_3',\dots, \bw_n')}
=
\frac{\Pi_{i=3}^n f(\bw_i)}{\Pi_{i=3}^n f(\bw_i')}.
\end{eqnarray*}
\item (Normalization property). 
\begin{eqnarray*}
\int f(\bw_3, \dots, \bw_n \mid A) \d \bw_3 \dots \d \bw_n=1.
\end{eqnarray*}
\end{enumerate}
It is straightforward to verify that $f(\bw_3, \dots, \bw_n \mid A) = \Pi_{i=3}^n \breve{f}_{\bw_1, \bw_2}(\bw_i)$ is the unique joint conditional density function that satisfy all the three properties above. This completes the proof.
\end{proof}

\begin{proof}[Proof of Lemma \ref{lemma:two_NNGs}]
For $i \in \lbr n \rbr$, 
let $M(i)$ and $N(i)$ index the  NNs of $\bW_i = (\bU_i, \bV_i)$ and $\bU_i$ respectively.	
It suffices to show that, as $n \to \infty$, 
\begin{eqnarray*}
\E\big[\#\{j\in \lbr n \rbr:j \neq 1, N(j) = M(1)\} \ \big | \ \bW_1\big] \conP 1.
\end{eqnarray*}

Let $f_\bW$ be the density function of $\bW$. 	Let $\calW^o \subseteq \mathbbR^{d_1+d_2}$ denote the interior of the support of $\bW$, and 
let $\calW^+ = \{\bw \in \calW^o: f_\bW(\bw)>0\}$. 
Since $\bW$ is absolutely continuous and $f_\bW$ is continuous on $\supp(\bW)$, it follows that $\calW^+$ is an open subset of $\mathbbR^{d_1+d_2}$ and $\P(\bW \in \calW^+) = 1$.

Fixing some $\bw_1 = (\bu_1, \bv_1)\in \calW^+$, we have 
\begin{eqnarray}
&&\E\big[\#\{j\in \lbr n \rbr:j \neq 1, N(j) = M(1)\} \ \big | \ \bW_1=\bw_1\big] \cr
&=&
\sum_{k=2}^n \Big\{
\E\big[\#\{j\in \lbr n \rbr:j \neq 1, N(j) = k\} \ \big | \ \bW_1=\bw_1, M(1) =k\big] \cr
&& \qquad \times\P\big(M(1)=k \ \big | \ \bW_1=\bw_1\big) \Big\}\cr
&=&
\E\big[\#\{j\in \lbr n \rbr:j \neq 1, N(j) = 2\} \ \big | \ \bW_1=\bw_1, M(1) =2\big]
\cr
&=&
(n-2) \cdot \P\big[N(3) = 2 \ \big | \ \bW_1=\bw_1, M(1) =2\big] \cr
&=&
(n-2) \cdot  \int f_{M(1)} (\bw_2) \cdot \P\big[ N(3) = 2 \ \big | \ \bW_1=\bw_1, M(1) = 2, \bW_2 = \bw_2\big] \d \bw_2 \cr
&=&
(n-2) \cdot  \int f_{M(1)} (\bw_2) \cdot \P\big[ N(3) = 2 \ \big | \ A\big] \d \bw_2,
\label{eq:lemma:two_NNGs:1}
\end{eqnarray}
where $f_{M(1)}(\cdot)$ denotes the density function of $\bW_{M(1)}$ (conditional on the event $\bW_1 = \bw_1$), and 
$A$ denotes the event $\{\bW_1=\bw_1, M(1) = 2, \bW_2 = \bw_2\}$.
By Lemma \ref{lemA:cond_pdf}, conditional on the event $A$, $\bW_3, \dots, \bW_n$ are i.i.d. with density function 
\begin{eqnarray*}
\breve f_{\bw_1, \bw_2}(\bw) :=  \frac{ f_\bW(\bw) \cdot \Ind\big(\bw \notin \tcalB(\bw_1, \bw_2)\big)}{1-\int_{\bw\in \tcalB(\bw_1, \bw_2)}f_\bW(\bw) \d \bw}. 
\end{eqnarray*} 
Let 
\begin{eqnarray}
\tilde f_{\bw_1, \bw_2}(\bu) := \int 	\breve f_{\bw_1, \bw_2}(\bw) \d \bv \label{eq:lemma:two_NNGs:2}
\end{eqnarray}
be the marginal density function for $\bU$ under $\breve f_{\bw_1, \bw_2}$. Then conditional on the event $A$, $\bU_3 \dots, \bU_n$ are i.i.d. with density function $\tilde f_{\bw_1, \bw_2}(\bu) $.
Hence,
\begin{eqnarray}
&&\P[ N(3) = 2 \mid A] \cr
&=&
\int 
\tilde f_{\bw_1, \bw_2}(\bu_3) \cdot \P\big[ \|\bU_j -\bu_3\| \geq \|\bu_2-\bu_3\| \text{ for all } j =4,\dots,n \ \big | \ A, \bU_3 = \bu_3\big] \d \bu_3 \cr
&=&
\int 
\tilde f_{\bw_1, \bw_2}(\bu_3) \cdot \P\big(\bU_4 \notin \tcalB(\bu_3,\bu_2) \ \big | \ A, \bU_3=\bu_3\big)^{n-3} \d \bu_3 \cr 
&=&
\int 
\tilde f_{\bw_1, \bw_2}(\bu_3) \cdot \Big[1-\int_{\bu_4 \in \tcalB(\bu_3,\bu_2)} \tilde f_{\bw_1, \bw_2}(\bu_4) \d \bu_4 \Big]^{n-3} \d \bu_3 \cr
&=&
\int_{r \in [0, \infty)}\int_{\bu_3 \in \mathcal{S}(\bu_2, r)}
\tilde f_{\bw_1, \bw_2}(\bu_3) \cdot \Big[1-\int_{\bu_4 \in \mathcal{B}(\bu_3,r)} \tilde f_{\bw_1, \bw_2}(\bu_4) \d \bu_4 \Big]^{n-3} \d \bu_3 \d r,
\nonumber
\end{eqnarray}
where $\mathcal{S}(\bu_2, r) = \{ \bu \in \mathbbR^{d_1}: \|\bu-\bu_2\| =r\}$ denotes the sphere with center $\bu_2$ and radius $r$. 

We next apply a change of variables by rescaling the integration variable $r$.
Let $ t = n\cdot r^{d_1}$. 
Then
\begin{eqnarray}
r = n^{-1/d_1}\cdot t^{1/d_1}, \quad \text{and }\d r = n^{-1/d_1}\cdot d_1^{-1} \cdot t^{(1-d_1)/d_1} \d t.   \label{eq:lemma:two_NNGs:3}
\end{eqnarray}
Substituting $\d r$ by $\d t$ in the above integral gives
\begin{eqnarray}
&& \hspace{-0.8cm}\P[ N(3) = 2 \mid A] =
\int_{t \in [0, \infty)} \Big\{\int_{\bu_3 \in \mathcal{S}(\bu_2, r)}
\tilde f_{\bw_1, \bw_2}(\bu_3) \cdot \Big[1-\int_{\bu_4 \in \mathcal{B}(\bu_3,r)} \tilde f_{\bw_1, \bw_2}(\bu_4) \d \bu_4 \Big]^{n-3} \d \bu_3 \Big\}  \cr
&& \hspace{3cm} \times n^{-\frac{1}{d_1}} \cdot d_1^{-1} \cdot t^{\frac{1-d_1}{d_1}}\d t. \nonumber
\end{eqnarray}
Note that the original variable $r$ still appears in the integrand. In this case, it can be regarded as a function of the new integration variable $t$, as specified in \eqref{eq:lemma:two_NNGs:3}.

For $\tildet > 0$, we truncate the upper limit of integration in $t$ from $\infty$ to $\tildet$, and define
\begin{eqnarray}
\P[ N(3) = 2 \mid A](\tildet) &:=& 
\int_{t \in [0, \tildet]} \Big\{\int_{\bu_3 \in \mathcal{S}(\bu_2, r)}
\tilde f_{\bw_1, \bw_2}(\bu_3) \cdot \Big[1-\int_{\bu_4 \in \mathcal{B}(\bu_3,r)} \tilde f_{\bw_1, \bw_2}(\bu_4) \d \bu_4 \Big]^{n-3} \d \bu_3 \Big\} \cr
&& \times  n^{-\frac{1}{d_1}}\cdot d_1^{-1} \cdot t^{\frac{1-d_1}{d_1}}\d t \cr
&=&	\P\big[ N(3) = 2, \|\bU_3 - \bu_2\| \leq  \tilde{r} \mid A\big],
\label{eq:lemma:two_NNGs:4}
\end{eqnarray}
where 
$\tilde{r} = n^{-1/d_1}\cdot \tildet^{1/d_1}$.
Also, define
\begin{eqnarray}
g_n(\tildet) := (n-2) \cdot  \int f_{M(1)} (\bw_2) \cdot \P[ N(3) = 2 \mid A](\tildet) \d \bw_2. \label{eq:lemma:two_NNGs:5}
\end{eqnarray}
Using these definitions, \eqref{eq:lemma:two_NNGs:1} can be rewritten as
\begin{eqnarray*}
&&\E\big[\#\{j\in \lbr n \rbr:j \neq 1, N(j) = M(1)\} \ \big | \ \bW_1=\bw_1\big] \cr
&=&
(n-2) \cdot  \int f_{M(1)} (\bw_2) \cdot \P[ N(3) = 2 \mid A] \d \bw_2 \ =\
\lim_{\tildet \to \infty} g_n(\tildet) : = g_n(\infty).
\end{eqnarray*}
Then it suffices to show that $g_n(\infty) \conP 1$ as $n \to \infty$. 

Let $f(\bu): = f_{\bU}(\bu)$ denote the marginal density function of $\bU$, which is also continuous over its support $\supp(\bU)$. Since $f_\bW(\bw_1)>0$, we also have $f(\bu_1)>0$. 
In the following proofs, most of the effort will be devoted to establishing the following statement: for each fixed $\tildet \in (0,\infty)$, 
\begin{eqnarray}
\lim_{n\to \infty} g_n(\tildet) = 1- \exp\big(- C_1 \cdot f(\bu_1) \cdot \tildet\big), \label{eq:lemma:two_NNGs:6}
\end{eqnarray}
where $C_1 =  \pi^{d_1/2} \cdot \Gamma(d_1/2+1)^{-1}$ is the volume of the $d_1$-dimensional unit ball.

Note that from the definition of $\tilde f_{\bw_1, \bw_2}(\cdot)$ in \eqref{eq:lemma:two_NNGs:2}, it is easy to check that
\begin{eqnarray*}
\lim_{\bw_2 \to \bw_1} \tilde f_{\bw_1, \bw_2}(\bu) = f(\bu), \quad \text{for all } \bu \in \supp(\bU).
\end{eqnarray*}
Moreover, since both functions $f(\bu)$ and $\tilde f_{\bw_1, \bw_2}(\bu)$ are bounded and continuous over $\supp(\bU)$, it can be verified that: for any $\epsilon>0$, there exists some neighborhood $\calN_{\bw_1}$ of $\bw_1$ and $\calN_{\bu_1}$ of $\bu_1$, such that 
\begin{eqnarray*}
\sup_{\bw_2 \in \calN_{\bw_1}, \ \bu \in \calN_{\bu_1}} |\tilde f_{\bw_1, \bw_2}(\bu) - f(\bu)| < \epsilon.
\end{eqnarray*}
Also, $f(\bu)$ is uniformly continuous in some neighborhood of $\bu_1$.
Thus, for any $\epsilon >0$, 
\begin{eqnarray*}
&&\sup_{\bw_2 \in \calN_{\bw_1}, \ \bu \in \calN_{\bu_1}} |\tilde f_{\bw_1, \bw_2}(\bu) - f(\bu_1)| \cr
&\leq& \sup_{\bw_2 \in \calN_{\bw_1}, \ \bu \in \calN_{\bu_1}} |\tilde f_{\bw_1, \bw_2}(\bu) - f(\bu)| + \sup_{\ \bu \in \calN_{\bu_1}} |f(\bu) -f(\bu_1)| \cr
&\leq& \epsilon
\end{eqnarray*}
holds for some neighborhoods $\calN_{\bw_1}$ of $\bw_1$ and $\calN_{\bu_1}$ of $\bu_1$.
This directly  yields the following statement:
for any $\epsilon >0$, there exists $R>0$, such that for any  $\bw_2 \in \mathbbR^{d_1+d_2}$ and $\bu \in
\mathbbR^{d_1}$ satisfying conditions $\|\bw_2-\bw_1\|\leq R$ and $\|\bu-\bu_2\| \leq 2 \cdot n^{-1/d_1}\cdot \tildet^{1/d_1}$,
\begin{eqnarray*}
|\tilde f_{\bw_1, \bw_2}(\bu) - f(\bu_1)| < \epsilon 
\end{eqnarray*}
holds for sufficiently large $n$. 

The above statement implies that, provided $\|\bw_2 - \bw_1\| \leq R$, the quantity $\tilde f_{\bw_1, \bw_2}(\bu)$ appearing in \eqref{eq:lemma:two_NNGs:4} for $\P\big[ N(3) = 2 \ \big | \ A \big](\tildet)$ differs from $f(\bu_1)$ by an arbitrarily small $\epsilon$ for sufficiently large $n$. 
In such case,
\begin{eqnarray*}
&&\P[ N(3) = 2 \mid A](\tildet) \cr
&=&
\int_{t \in [0, \tildet]} \Big\{\int_{\bu_3 \in \mathcal{S}(\bu_2, r)}
\tilde f_{\bw_1, \bw_2}(\bu_3) \cdot \Big[1-\int_{\bu_4 \in \mathcal{B}(\bu_3,r)} \tilde f_{\bw_1, \bw_2}(\bu_4) \d \bu_4 \Big]^{n-3} \d \bu_3 \Big\} \cr
&&\qquad \times  n^{-\frac{1}{d_1}}\cdot d_1^{-1} \cdot t^{\frac{1-d_1}{d_1}}\d t. \cr
&\geq &
\int_{t \in [0, \tildet]} \Big\{\int_{\bu_3 \in \mathcal{S}(\bu_2, r)}
(f(\bu_1)-\epsilon) \cdot \Big[1-\int_{\bu_4 \in \mathcal{B}(\bu_3,r)} (f(\bu_1)+\epsilon) \d \bu_4 \Big]^{n-3} \d \bu_3 \Big\}  \cr
&& \qquad \times  n^{-\frac{1}{d_1}}\cdot d_1^{-1} \cdot t^{\frac{1-d_1}{d_1}}\d t. \cr
&=&
\int_{t \in [0, \tildet]}  S_r \cdot
(f(\bu_1)-\epsilon) \cdot \Big[1- V_r \cdot (f(\bu_1)+\epsilon)  \Big]^{n-3}   \cdot  n^{-\frac{1}{d_1}}\cdot d_1^{-1} \cdot t^{\frac{1-d_1}{d_1}}\d t \cr
&=&
\int_{t \in [0, \tildet]}  C_2 \cdot t^{\frac{d_1-1}{d_1}} \cdot  n^{\frac{1-d_1}{d_1}}\cdot 
(f(\bu_1)-\epsilon) \cdot \Big[1- C_1 \cdot t \cdot n^{-1}\cdot (f(\bu_1)+\epsilon)  \Big]^{n-3}    \cr
&& \qquad \times  n^{-\frac{1}{d_1}}\cdot d_1^{-1} \cdot t^{\frac{1-d_1}{d_1}}\d t, \cr
&=&
\int_{t \in [0, \tildet]}  C_2\cdot n^{-1} \cdot d_1^{-1} \cdot  
(f(\bu_1)-\epsilon) \cdot \Big[1- C_1 \cdot t \cdot n^{-1}\cdot (f(\bu_1)+\epsilon)  \Big]^{n-3} \d t,
\end{eqnarray*}
where $V_r$ and $S_r$ denote the volume and the surface area of the $d_1$-dimensional ball with radius $r$, respectively. Moreover, 
$C_1 =  \pi^{d_1/2} \cdot \Gamma(d_1/2+1)^{-1}$ and $C_2 = 2 \pi^{d_1/2} \cdot \Gamma(d_1/2)^{-1}$
are the volume and the surface area of the $d_1$-dimensional unit ball, respectively.
Note that
\begin{eqnarray*}
\lim_{n \to \infty}\Big[1- C_1 \cdot t \cdot n^{-1}\cdot (f(\bu_1)+\epsilon)  \Big]^{n-3} = \exp\big(-C_1 \cdot t \cdot  (f(\bu_1)+\epsilon)\big).
\end{eqnarray*}
Therefore,
\begin{eqnarray*}
&&\liminf_{n \to \infty} (n-2)\cdot \P[ N(3) = 2 \mid A](\tildet)  \cr
&\geq &	\liminf_{n \to \infty}  (n-2) \cdot \int_{t \in [0, \tildet]}  C_2\cdot n^{-1} \cdot d_1^{-1} \cdot  
(f(\bu_1)-\epsilon) \cdot \Big[1- C_1 \cdot t \cdot n^{-1}\cdot (f(\bu_1)+\epsilon)  \Big]^{n-3} \d t, \cr
&=&
\int_{t \in [0, \tildet]}  C_2 \cdot d_1^{-1} \cdot  
(f(\bu_1)-\epsilon) \cdot \exp\big(-C_1 \cdot t \cdot  (f(\bu_1)+\epsilon)\big) \d t, \cr
&=&
\frac{C_2 \cdot d_1^{-1} \cdot  
(f(\bu_1)-\epsilon)}{C_1 \cdot  (f(\bu_1)+\epsilon)} \cdot \Big[1-  \exp\big(-C_1 \cdot  (f(\bu_1)+\epsilon)\cdot \tildet\big)\Big] \cr
&=&
\frac{f(\bu_1)-\epsilon}{f(\bu_1)+\epsilon}  \cdot \Big[1-  \exp\big(-C_1 \cdot  (f(\bu_1)+\epsilon)\cdot \tildet\big)\Big].
\end{eqnarray*}
Similarly, we also have
\begin{eqnarray*}
&&\limsup_{n \to \infty} (n-2)\cdot \P[ N(3) = 2 \mid A](\tildet) \cr
&\leq &\frac{f(\bu_1)+\epsilon}{f(\bu_1)-\epsilon}  \cdot \Big[1-  \exp\big(-C_1 \cdot  (f(\bu_1)-\epsilon)\cdot \tildet\big)\Big].
\end{eqnarray*}
Combining the above two facts, we arrive at the conclusion:
for any $\epsilon >0$, there exists some $R>0$, such that for any  $\bw_2$ satisfying $\|\bw_2-\bw_1\|\leq R$, 
\begin{eqnarray}
\Big|(n-2)\cdot \P[ N(3) = 2 \mid A](\tildet) -\Big[1-  \exp(-C_1 \cdot  f(\bu_1)\cdot \tildet)\Big]\Big| <\epsilon \label{eq:lemma:two_NNGs:7}
\end{eqnarray}
holds for sufficiently large $n$. 

By Lemma 11.3 of \cite{azadkia2019simple}, we have $\|\bW_{M(1)}- \bW_1\| \conas 0$. Therefore, for any $\epsilon_2 >0$ and  $R_2 \in (0, R)$, 
\begin{eqnarray}
\int_{\bw_2 \in \mathcal{B}(\bw_1, R_2)} f_{M(1)} (\bw_2) \d \bw_2>1-\epsilon_2 \label{eq:lemma:two_NNGs:8}
\end{eqnarray}
holds for sufficiently large $n$. 
Also note that $(n-2)\cdot \P[ N(3) = 2 \mid A](\tildet) \leq (n-2)\cdot \P[ N(3) = 2 \mid A] = \E[ \#\{j\in \lbr n \rbr:j \neq 1, j \neq 2, N(j) = 2\} \mid A]$
does not exceed the number of indices whose NN is $\bU_2$, which is bounded by some constant $C_3>0$ according to \cite{MR682809} (Corollary S1). 
Thus, 
\begin{eqnarray}
g_n(\tildet) &=& (n-2) \cdot  \int f_{M(1)} (\bw_2) \cdot \P[ N(3) = 2 \mid A](\tildet) \d \bw_2 \cr
&=:&
\int_{\bw_2 \in \mathcal{B}(\bw_1, R_2)}  f_{M(1)} (\bw_2) \cdot 	(n-2) \cdot \P[ N(3) = 2 \mid A](\tildet) \d \bw_2 + C(\epsilon_2),
\label{eq:lemma:two_NNGs:9}
\end{eqnarray}
where the term $C(\epsilon_2)$ satisfies $|C(\epsilon_2)| \leq C_3 \cdot \epsilon_2$ for sufficiently large $n$. 
Combining this with \eqref{eq:lemma:two_NNGs:7} and \eqref{eq:lemma:two_NNGs:8}, we have that: for any $\epsilon>0$, 
\begin{eqnarray*}
\Big|g_n(\tildet) -\big[1-  \exp(-C_1 \cdot  f(\bu_1)\cdot \tildet)\big]\Big| < \epsilon 
\end{eqnarray*}
holds for sufficiently large $n$. This proves \eqref{eq:lemma:two_NNGs:6}.

Note that $\E\big[\#\{j\in \lbr n \rbr:j \neq 1, N(j) = M(1)\} \ \big | \ \bW_1=\bw_1\big]= g_n(\infty) \geq g_n(\tildet)$ for all $\tildet>0$. Then
\begin{eqnarray*}
&&\liminf_{n \to \infty}\E\big[\#\{j\in \lbr n \rbr:j \neq 1, N(j) = M(1)\} \ \big | \ \bW_1=\bw_1\big]\cr
&\geq &
\sup_{\tildet>0} \liminf_{n \to \infty} g_n(\tildet) \, = \,	\sup_{\tildet>0} \big\{1- \exp(- C_1 \cdot f(\bu_1) \cdot \tildet)\big\} \, = \, 1.
\end{eqnarray*}
Since $\bw_1 \in \calW^+$ and $\P(\bW \in  \calW^+) = 1$, it follows that
\begin{eqnarray*}
\liminf_{n \to \infty}\E\big[\#\{j\in \lbr n \rbr:j \neq 1, N(j) = M(1)\} \ \big | \ \bW_1\big]\geq 1 
\end{eqnarray*}
holds with probability one.
On the other hand, \citet[Lemma 7.4, Equation (28)]{Shi_Drton_Han_2024_Bernoulli} has proved that 
\begin{eqnarray*}
\lim_{n\to \infty}\E\big[\#\{j\in \lbr n \rbr:j \neq 1, N(j) = M(1)\}\big] =1.
\end{eqnarray*}
Combining these two results and applying Lemma \ref{lemA:Xn_conP_C}, we complete the proof.
\end{proof}

\section{Proofs of Theorems \ref{thm:var_xi}, \ref{thm:est_var_xi}, and Proposition \ref{prop:nlogn}}  \label{secA:proofs-2}

\begin{proof}[Proof of Theorem \ref{thm:var_xi}]
By \cite{Lin_Han_2025_CLT} (proof of Theorem 1.2, p.~19), we have
\begin{eqnarray*}
&& 36^{-1} \cdot
\E(\var(\xi_n \mid \bfZ)) \cr
&=&
\E \Big\{\var\big[F_Y(Y_1 \wedge \tY_1)\mid \bZ_1 \big] \Big\} \cr
&&+
2 \cdot \E \Big\{\cov\big[F_Y(Y_1 \wedge \tY_1),  F_Y(Y_1 \wedge \tY_1')\mid \bZ_1 \big] \cdot \Ind\big(N(N(1)) \neq 1\big)\Big\}
\cr
&&
+ \E \Big\{\var\big[F_Y(Y_1 \wedge \tY_1)\mid \bZ_1 \big] \cdot
\Ind\big(N(N(1)) = 1\big)
\Big\} \cr
&&+
\E \Big\{\cov\big[F_Y(Y_1 \wedge \tY_1),  F_Y(Y_1 \wedge \tY_1')\mid \bZ_1 \big] 
\cdot \#\{j \in \lbr n \rbr: j\neq 1, N(j) = N(1)\}
\Big\} \cr
&&+
\E \Big\{\cov\big[\Ind (Y_3 \leq Y_1 \wedge  \tY_1),  \Ind (Y_3 \leq Y_2 \wedge  \tY_2) \mid \bZ_1, \bZ_2, \bZ_3\big] \Big\} \cr
&&+ \ 
4 \cdot \E \Big\{\cov\big[\Ind (Y_2 \leq Y_1 \wedge  \tY_1),  F_Y( Y_2 \wedge  \tY_2) \mid \bZ_1,\bZ_2\big] \Big\}  \Big] +  o(1).
\end{eqnarray*}
By Lemmas \ref{lemma:q_d} and \ref{lemma:o_d}, we have 
\begin{eqnarray*}
&& \E\big[	\Ind\big(N(N(1)) = 1\big)\mid \bZ_1\big] \conP \mathfrak{q}_q, \cr
&&	\E\big[\#\{j \in \lbr n \rbr: j\neq 1, N(j) = N(1)\} \mid \bZ_1\big] \conP \mathfrak{o}_q.
\end{eqnarray*}
Notice that 
$\#\{j \in \lbr n \rbr: j\neq 1, N(j) = N(1)\}$ is bounded above by the maximum degree of the $1$-NNG, which is itself bounded \citep{MR682809}. By applying the bounded convergence theorem, the second, third, and fourth terms in the above decomposition could be further simplified, which yields
\begin{eqnarray}
36^{-1} \cdot
\E(\var(\xi_n \mid \bfZ))
&=&
(1+\mathfrak{q}_q)\cdot\E \Big\{\var\big[F_Y(Y_1 \wedge \tY_1)\mid \bZ_1 \big] \Big\} \cr
&&+
(2-2\mathfrak{q}_q + \mathfrak{o}_q)\cdot \E \Big\{\cov\big[F_Y(Y_1 \wedge \tY_1),  F_Y(Y_1 \wedge \tY_1')\mid \bZ_1 \big] \Big\} \cr
&&+
\E \Big\{\cov\big[\Ind (Y_3 \leq Y_1 \wedge  \tY_1),  \Ind (Y_3 \leq Y_2 \wedge  \tY_2) \mid \bZ_1, \bZ_2, \bZ_3\big] \Big\} \cr
&&+ \ 
4 \cdot \E \Big\{\cov\big[\Ind (Y_2 \leq Y_1 \wedge  \tY_1),  F_Y( Y_2 \wedge  \tY_2) \mid \bZ_1,\bZ_2\big] \Big\}+  o(1) \cr
&=:&  
(1+\mathfrak{q}_q)\cdot S_1 + 	(2-2\mathfrak{q}_q + \mathfrak{o}_q)\cdot S_2 + S_3 + 4 S_4 + o(1). 
\label{eq:thm:var_xi:1}
\end{eqnarray}
On the other hand, invoking Lemma C.1 (pp.~20, 25) and the proof of Lemma 2.11 (p.~47) in \cite{Lin_Han_2025_CLT}, we obtain
\begin{eqnarray}
36^{-1} \cdot \var(\E(\xi_n \mid \bfZ))  &=& 
\var\Big\{\E (F_Y(Y_1 \wedge \tY_1)\mid \bZ_1 )\Big\}+
\E\Big\{\E\big(\Ind(Y_1 \leq Y_2 \wedge \tY_2)\mid \bZ_1\big)^2\Big\} \cr
&&
-2 \E\Big\{ \Ind(Y_1 \leq Y_2 \wedge \tY_2) \cdot F_Y(\tY_1' \wedge \tY_1'') \Big\} +
\E\big\{ F_Y(Y_1 \wedge \tY_1)\big\} ^2  + o(1) \cr
&=:&
S_5 + S_6 - 2 S_7 +  S_8 + o(1). 
\label{eq:thm:var_xi:2}
\end{eqnarray}
Note that 
\begin{eqnarray}
S_1 &=&\E \Big\{\var\big[F_Y(Y_1 \wedge \tY_1)\mid \bZ_1 \big] \Big\}   \cr
&=&
\E \Big\{\E \big[F_Y(Y_1 \wedge \tY_1)^2\mid \bZ_1 \big] \Big\}  
-
\E \Big\{\E \big[F_Y(Y_1 \wedge \tY_1) \cdot F_Y(Y_1' \wedge \tY_1'') \mid \bZ_1 \big] \Big\}   \cr
&=&  \E\big\{ F_Y^2(Y \wedge \tY)\big\} -
\E\big\{F_Y(Y \wedge \tY) \cdot F_Y(\tY' \wedge\tY'')\big\}= T_1- T_3.
\label{eq:thm:var_xi:3}
\end{eqnarray}
Similarly, we can derive the following identities
\begin{eqnarray}
&&S_2 = T_2 - T_3, \quad S_3 = T_6 - S_6, \quad S_4 = T_4 - T_7,  \cr
&& S_5 = T_3 - T_7, \quad  S_7 = T_5, \quad \text{and } S_8 = T_7. \label{eq:thm:var_xi:4}
\end{eqnarray}
Using $\var(\xi_n) = \E(\var(\xi_n \mid \bfZ))  +  \var(\E(\xi_n \mid \bfZ))$, and combining \eqref{eq:thm:var_xi:1}--\eqref{eq:thm:var_xi:4}, we complete the proof of Theorem \ref{thm:var_xi}. \end{proof}

\begin{proof}[Proof of Theorem \ref{thm:est_var_xi}] 
Using \eqref{eq:thm:var_xi:3}, \eqref{eq:thm:var_xi:4}, and \cite{Lin_Han_2025_CLT} (Lemmas 2.10 and 2.11, pp.~19–20), we have that $\hatT_i \conP T_i$, for $i = 1,2,3,4,5,6,7$. This completes the proof. 
\end{proof}

\begin{proof}[Proof of Proposition \ref{prop:nlogn}]
To compute $\hsigma^2_{\xi(Y,\bZ)}$, we first need to obtain the ranks $R_i$, $R_{N(i)}$, $R_{N_2(i)}$, and $R_{N_3(i)}$ for all $i=1,\ldots,n$
	This requires performing a $k$-nearest neighbor search ($k=1,2,3$) over the entire set $\{\bZ_i\}_{i=1}^n$, followed by sorting the sample $\{Y_i\}_{i=1}^n$. 
	Each of these two steps can be carried out in $O(n \log n)$ time. 
	After this, the terms $\hatT_1$, $\hatT_2$, $\hatT_3$, and $\hatT_7$ only involve summation over single index $i$, and thus 
	can clearly be computed in $O(n \log n)$ time. 
	
	For the remaining terms $\hatT_4$, $\hatT_5$, and $\hatT_6$, although their definitions involve double summations over index pairs $(i,j)$ and thus appear to require $O(n^2)$ operations, Algorithm~\ref{alg1} provides a substantial computational acceleration. Specifically, by introducing the auxiliary ranks $R_i^*$ and $R_i^\#$ in Step~4, the original double summations are reformulated as equivalent single summations over $i$. Since, for each $i$, the quantities $R_i^*$ and $R_i^\#$ can be obtained via binary search in $O(\log n)$ time, the overall computational complexity of Algorithm~\ref{alg1} is bounded by $O(n\log n)$.
This completes the proof.
\end{proof}

\section{Proof of Theorem \ref{thm:CLT-main}} \label{secA:proofs-3}

The proof of Theorem \ref{thm:CLT-main} consists of three steps:
\begin{description}
\item[Step (1).] Derive the H\'ajek representation $	\tT_n^*$ of $\tT_n$, 
\begin{eqnarray}
\tT_n^* &=&  \frac{1}{n}\sum_{i=1}^n \Big\{F_Y(Y_i \wedge Y_{M(i)}) + h_1(Y_i)\Big\} -(1-T)\cdot\frac{1}{n}\sum_{i=1}^n \Big\{ F_Y(Y_i \wedge Y_{N(i)}) +h_2(Y_i)\Big\} \cr
&=:& S_{1,n}  -(1-T) \cdot S_{2,n}, \label{eq:Hajek} \\
\text{where} && h_1(t): = \P(Y \wedge \barY>t), \qquad h_2(t): = \P(Y \wedge \tY>t). \nonumber
\end{eqnarray}
Show that $\lim_{n \to \infty} \var(\tT_n - \tT_n^*) = 0$. Then $\tT_n - \E(\tT_n)$ shares the same limit distribution as $\tT_n^*- \E(\tT_n^*)$. 
\item [Step (2).] Prove that $\sigma^2 = \lim_{n \to \infty} n \var(\tT_n^*) $ exists, and admits the closed-form representation as in \eqref{eq:sigma2_tTn}--\eqref{eq:U1-U9}.

\item[Step (3).] Prove the CLT for $\tT_n^*$:
\begin{eqnarray*}
\sqrt{n}\cdot \big\{\tT_n^* - \E(\tT_n^*)\big\} \conD N(0,\sigma^2), \quad \text{as } n \to \infty.
\end{eqnarray*}
\end{description}

We present the proofs of Steps (1), (2), and (3) separately in Sections~\ref{secA:proof-stepI}, \ref{secA:proof-stepII}, and \ref{secA:proof-stepIII}, respectively.
\subsection{Proof of Step (1)}  \label{secA:proof-stepI}
\begin{proof}
Recall the decomposition of $\tT_n$ in \eqref{eq:tTn}:
\begin{eqnarray*}
&&	\tT_n \ = \kappa_n\cdot (T_n -T) \ = \ \txi_{1,n} -(1 -T) \cdot \txi_{2,n} - \big\{(n+1)/(2n)-3^{-1}\big\}\cdot T, \cr
&& \text{where } \hspace{1.8cm} \txi_{1,n} \ = \ \frac{1}{n^2}\sum_{i=1}^n\min\{R_i, R_{M(i)}\}- \frac{1}{3}, \cr
&& \hspace{3cm} \txi_{2,n} \ = \ \frac{1}{n^2}\sum_{i=1}^n\min\{R_i, R_{N(i)}\}- \frac{1}{3}.
\end{eqnarray*}

As shown in \cite{Lin_Han_2025_CLT} (Theorem 1.3, p.~8), the H\'ajek representations of 	$ \txi_{1,n}$ and $\txi_{2,n}$ take the form
\begin{eqnarray*}
S_{1,n} &:=&
\frac{1}{n}\sum_{i=1}^n F_Y(Y_i \wedge Y_{M(i)}) + \sum_{i=1}^n h_{1}(Y_i),\cr
S_{2,n} &:=& 
\frac{1}{n}\sum_{i=1}^n F_Y(Y_i \wedge Y_{N(i)}) + \sum_{i=1}^n h_{2}(Y_i),
\end{eqnarray*}
where
\begin{eqnarray*}
&&G_{\bX, \bZ}(t) := \E\{\Ind(Y\geq t) \mid \bX, \bZ\}, \quad h_{1}(t): = \E\big[G^2_{\bX, \bZ}(t)\big] = \P(\barY \wedge Y > t),\cr
\text{and} &&	G_\bZ(t) := \E\{\Ind(Y\geq t) \mid \bZ\}, \quad \qquad  h_{2}(t): = \E\big[	G^2_\bZ(t)\big] =  \P(\tY \wedge Y > t).
\end{eqnarray*}
Moreover, we have that 
\begin{eqnarray}
\lim_{n \to \infty} n \var\big(\txi_{1,n}-  S_{1,n}\big) =0, \qquad \lim_{n \to \infty} n \var\big(\txi_{2,n}-  S_{2,n}\big) =0,   \label{eq:lim_var_0}
\end{eqnarray}
which implies that
\begin{eqnarray*}
\lim_{n \to \infty} n \var\Big[\tT_n- \big\{S_{1,n} -  (1-T)\cdot S_{2,n} \big\}\Big] =0.
\end{eqnarray*}
Thus $\tT_n^* = S_{1,n} -  (1-T)\cdot S_{2,n}  $ in \eqref{eq:Hajek} is the H\'ajek representation of $\tT_n$. 
This completes the proof.
\end{proof}

\subsection{Proof of Step (2)} \label{secA:proof-stepII}
\paragraph*{Overview.}
Decompose the limit of variance as follows,
\begin{eqnarray*}
\lim_{n \to \infty} n \var(\tT_n^*)  &=& \lim_{n \to \infty} n \var\big(S_{1,n}  -(1-T) \cdot S_{2,n}\big)   \cr
&=&
\lim_{n \to \infty} n \var(S_{1,n}) + (1-T)^2\cdot \lim_{n \to \infty} n \var(S_{2,n}) -2\cdot(1-T)\cdot 	\lim_{n \to \infty} n \cov(S_{1,n}, S_{2,n}).
\end{eqnarray*}
By \eqref{eq:lim_var_0}, \eqref{eq:txi_2} and Theorem \ref{thm:var_xi}, we have
\begin{eqnarray*}
&&\lim_{n \to \infty} n \var(S_{1,n}) = \lim_{n \to \infty} n \var(\txi_{1,n}) =\lim_{n \to \infty} n \var\big(6^{-1} \cdot \xi_n(Y,\bZ) \big) = 36^{-1}\cdot \sigma^2_{\xi(Y,\bZ)}, \cr
&&\lim_{n \to \infty} n \var(S_{2,n}) = \lim_{n \to \infty} n \var(\txi_{2,n}) =\lim_{n \to \infty} n \var\big(6^{-1} \cdot \xi_n(Y,(\bX,\bZ)) \big) = 36^{-1}\cdot \sigma^2_{\xi(Y,(\bX,\bZ))}.
\end{eqnarray*}
Thus, it remains to show that
\begin{eqnarray*}
\lim_{n \to \infty} n \cov(S_{1,n}, S_{2,n}) = \sigma_{1,2},
\end{eqnarray*}
where $\sigma_{1,2}$ is provided in \eqref{eq:sigma12}--\eqref{eq:U1-U9}. 

Decompose the covariance into four terms:
\begin{eqnarray*}
n \cov(S_{1,n}, S_{2,n}) &=&
n \cdot \cov \Big\{ \frac{1}{n}\sum_{i=1}^n \Big[F_Y(Y_i \wedge Y_{M(i)}) + h_1(Y_i)\Big] , \ \frac{1}{n}\sum_{i=1}^n \Big[ F_Y(Y_i \wedge Y_{N(i)}) +h_2(Y_i)\Big]  \Big\} \cr
&=& \frac{1}{n} \cov\Big\{\sum_{i=1}^n F_Y(Y_i \wedge Y_{N(i)}) , \ \sum_{i=1}^n F_Y(Y_i \wedge Y_{M(i)})\Big\} \cr
&&+
\frac{1}{n} \cov\Big\{\sum_{i=1}^n F_Y(Y_i \wedge Y_{N(i)}) ,  \ \sum_{i=1}^n h_1(Y_i)\Big\}  \cr
&& +
\frac{1}{n} \cov\Big\{\sum_{i=1}^n F_Y(Y_i \wedge Y_{M(i)}) ,  \ \sum_{i=1}^n h_2(Y_i)\Big\} \cr
&& + 
\cov\Big\{h_1(Y), \ h_2(Y)\Big\}\cr
&=:& Q_1 + Q_2 + Q_3 + Q_4.
\end{eqnarray*}
Moreover, decompose $Q_1$ into two terms
\begin{eqnarray*}
Q_1&=&\frac{1}{n}\cov\Big(\sum_{i=1}^n F_Y(Y_i \wedge Y_{N(i)}) , \sum_{i=1}^n F_Y(Y_i \wedge Y_{M(i)}) \Big) \cr
&=&
\frac{1}{n}	
\E \Big\{\cov\Big[\sum_{i=1}^n F_Y(Y_i \wedge Y_{N(i)}), \sum_{j=1}^n F_Y(Y_i \wedge Y_{M(i)}) \mid \bfX, \bfZ\Big] \Big\} \cr
&& +
\frac{1}{n}\cov\Big[ \E\Big(\sum_{i=1}^n F_Y(Y_i \wedge Y_{N(i)})\mid \bfX, \bfZ \Big),
\E\Big(\sum_{i=1}^n F_Y(Y_i \wedge Y_{M(i)}) \mid \bfX, \bfZ \Big)
\Big]  \cr
&=:& Q_{1,1} + Q_{1,2}.  
\end{eqnarray*}
In the following, Sections \ref{secA:proof-stepII:1} and \ref{secA:proof-stepII:2} derive the limits of $Q_{1,1}$ and $Q_{1,2}$, respectively. Section \ref{secA:proof-stepII:3} derives the limits of $Q_2$ and $Q_3$. Finally, Section \ref{secA:proof-stepII:4} combines these results to obtain the limit of $n \cov(S_{1,n}, S_{2,n})$.

\subsubsection{The limit of $Q_{1,1}$} \label{secA:proof-stepII:1}
To analyze the term $Q_{1,1}$, we further decompose it according to the relationships among the indices in the NNG:
\begin{eqnarray}
Q_{1,1} &=& 
\frac{1}{n}\E \Big\{ \sum_{i=1}^n\cov\big[F_Y(Y_i \wedge Y_{N(i)}),  F_Y(Y_i \wedge Y_{M(i)}) \mid \bfX, \bfZ\big] \Big\} \cr
&& +	\frac{1}{n}\E \Big\{\sum_{(i,j)\in \lbr n \rbr\times\lbr n \rbr: \, i,j, N(i), M(j) \,\text{distinct}} \cov\big[F_Y(Y_i \wedge Y_{N(i)}), F_Y(Y_j \wedge Y_{M(j)}) \mid \bfX, \bfZ\big] \Big\} \cr
&&+
\frac{1}{n}	
\E \Big\{
\sum_{\underset{\text{ or } i=M(j), j\neq N(i)}{(i,j)\in \lbr n \rbr\times\lbr n \rbr: j=N(i), i\neq M(j),}	}
\cov\big[F_Y(Y_i \wedge Y_{N(i)}), F_Y(Y_j \wedge Y_{M(j)}) \mid \bfX, \bfZ\big]\Big\} \cr
&&+
\frac{1}{n}	
\E \Big\{\sum_{(i,j)\in \lbr n \rbr\times\lbr n \rbr:i\neq j, N(i) = M(j)}  \cov\big[F_Y(Y_i \wedge Y_{N(i)}), F_Y(Y_j \wedge Y_{M(j)}) \mid \bfX, \bfZ\big] \Big\} \cr
&&+
\frac{1}{n}	
\E \Big\{\sum_{(i,j)\in \lbr n \rbr\times\lbr n \rbr: j = N(i), i = M(j)} \cov\big[F_Y(Y_i \wedge Y_{N(i)}), F_Y(Y_j \wedge Y_{M(j)})  \mid \bfX, \bfZ\big] \Big\} \cr
&=:& 
Q_{1,1,1} + Q_{1,1,2}  +Q_{1,1,3} +Q_{1,1,4}  +Q_{1,1,5} .  \label{eq:Q11_decomp}
\end{eqnarray}
The limit of $Q_{1,1,1}$--$Q_{1,1,5}$ are derived in Lemmas \ref{lemA:lim1}--\ref{lemA:lim5} respectively.

\begin{lemmaA} \label{lemA:lim1}
The limit of $Q_{1,1,1}$ in \eqref{eq:Q11_decomp} is
\begin{eqnarray*}
&&\lim_{n \to \infty}	\E \Big\{\cov\big[F_Y(Y_1 \wedge Y_{N(1)}),  F_Y(Y_1 \wedge Y_{M(1)}) \mid \bfX, \bfZ\big] \Big\}  \cr
&=&
\E \Big\{\cov\big[F_Y(Y_1 \wedge \tY_1),  F_Y(Y_1 \wedge \barY_1) \mid \bX_1, \bZ_1\big] \Big\} .
\end{eqnarray*}
\end{lemmaA}

\begin{proof}
By manipulating the conditional covariance, we have
\begin{eqnarray*}
&&\E \Big\{\cov\big[F_Y(Y_1 \wedge Y_{N(1)}),  F_Y(Y_1 \wedge Y_{M(1)}) \mid \bfX, \bfZ\big] \Big\} \cr
&=&
\E \Big\{\E\big[F_Y(Y_1 \wedge Y_{N(1)})\cdot  F_Y(Y_1 \wedge Y_{M(1)}) \mid \bfX, \bfZ\big] \Big\} \cr
&&- \E\Big\{ \E\big[F_Y(Y_1 \wedge Y_{N(1)}) \mid \bfX, \bfZ\big] \cdot \E\big[F_Y(Y_1 \wedge Y_{M(1)}) \mid \bfX, \bfZ\big] \Big\} \cr
&=:& R_1 - R_2.
\end{eqnarray*}
By \citet[Lemma 7.4, Equation (26)]{Shi_Drton_Han_2024_Bernoulli}, we have that $\P(N(1) = M(1)) =o(1)$. Then for the term $R_1$, we have
\begin{eqnarray*}
R_1
&=& \E \Big\{\E\big[F_Y(Y_1 \wedge Y_{N(1)})\cdot  F_Y(Y_1 \wedge Y_{M(1)}) \mid \bfX, \bfZ\big] \cdot \Ind(N(1) \neq M(1))\Big\} \cr
&& + \E \Big\{\E\big[F_Y(Y_1 \wedge Y_{N(1)})\cdot  F_Y(Y_1 \wedge Y_{M(1)}) \mid \bfX, \bfZ\big] \cdot \Ind(N(1) = M(1))\Big\} \cr
&=&
\E  \Big\{ \int \int \int F_Y(u\wedge v) \cdot F_Y(u \wedge w) 
\d \tmu_{(\bX_1,\bZ_1)}(u) \d \tmu_{(\bX_{N(1)}, \bZ_{N(1)})}(v) \d \tmu_{(\bX_{M(1)}, \bZ_{M(1)})}(w) \cr
&& \quad \times \Ind(N(1) \neq M(1)) \Big\} + o(1).
\end{eqnarray*}
Conditional on the event $	\{ N(1) \neq M(1), (\bX_1, \bZ_1) = (\bx_1,\bz_1), \bZ_{N(1)} = \bz_{N(1)}, (\bX_{M(1)}, \bZ_{M(1)})=(\bx_{M(1)}, \bz_{M(1)}) \}$, the distribution law of $\bX_{N(1)} = \bx$ is $\mu_{\bX=\bx \mid \bZ = \bz_{N(1)}}$. That is, $\bX_{N(1)}$ only depends on $\bz_{N(1)}$ and is irrelevant of $ (\bx_1,\bz_1)$ and $(\bx_{M(1)}, \bz_{M(1)})$. Therefore, by applying the Fubini's theorem, we can eliminate the random term 
$\bX_{N(1)}$ in this expectation as follows,
\begin{eqnarray}
&&\E  \Big\{ \int \int \int F_Y(u\wedge v) \cdot F_Y(u \wedge w) 
\d \tmu_{(\bX_1,\bZ_1)}(u) \d \tmu_{(\bX_{N(1)}, \bZ_{N(1)})}(v) \d \tmu_{(\bX_{M(1)}, \bZ_{M(1)})}(w) \cr
&& \hspace{3cm}\times\Ind(N(1) \neq M(1))\Big\} \cr
&=& \E\Big[\E  \Big\{ \int \int \int F_Y(u\wedge v) \cdot F_Y(u \wedge w) 
\d \tmu_{(\bX_1,\bZ_1)}(u) \d \tmu_{(\bX_{N(1)}, \bZ_{N(1)})}(v) \d \tmu_{(\bX_{M(1)}, \bZ_{M(1)})}(w)  \cr 
&& \hspace{3cm}\Biggiven(\bX_1,\bZ_1), \bZ_{N(1)}, (\bX_{M(1)}, \bZ_{M(1)}) , \Ind(N(1) \neq M(1))\Big\} \cdot \Ind(N(1) \neq M(1))\Big]
\cr
&=&\E \Big[ \Big\{\int \int \int \int F_Y(u\wedge v) \cdot F_Y(u \wedge w) \d \tmu_{(\bX_1,\bZ_1)}(u) \d \tmu_{(\bx, \bZ_{N(1)})}(v)  \cr
&& \hspace{3cm}
\d \tmu_{(\bX_{M(1)}, \bZ_{M(1)})}(w)  \d \mu_{\bX=\bx \mid \bZ = \bZ_{N(1)}}(\bx)\Big\} \cdot \Ind(N(1) \neq M(1))\Big]  \cr
&=&\E \Big[\Big\{ \int \int \int F_Y(u\wedge v) \cdot F_Y(u \wedge w) 
\d \tmu_{(\bX_1,\bZ_1)}(u) \d \tmu_{\bZ_{N(1)}}(v) \d \tmu_{(\bX_{M(1)}, \bZ_{M(1)})}(w)  \Big\} \cr
&& \hspace{3cm} \times \Ind(N(1) \neq M(1))\Big] \cr
&=& \E\Big\{g\big((\bX_1,\bZ_1) , \bZ_{N(1)},  (\bX_{M(1)}, \bZ_{M(1)}) \big) \cdot \Ind(N(1) \neq M(1)) \Big\} \cr
&=&
\E\Big\{g\big((\bX_1,\bZ_1) , \bZ_{N(1)},  (\bX_{M(1)}, \bZ_{M(1)}) \big) \Big\}  +o(1),
\label{eq:lemA:lim1:1}
\end{eqnarray}
where $g$ is defined in \eqref{eq:lemA:aux-2-3:1} in Lemma \ref{lemA:aux-2-3},
and the last equation is because that $g$ is bounded and $\P(N(1) = M(1)) =o(1)$.
By Lemma \ref{lemA:aux-2-3}, we have 
\begin{eqnarray}
g\big((\bX_1, \bZ_1), \bZ_{N(1)}, (\bX_{M(1)}, \bZ_{M(1)})\big) - g\big((\bX_1, \bZ_1), \bZ_1, (\bX_1, \bZ_1)\big)  \conP 0.
\end{eqnarray}
Thus, by applying the bounded convergence theorem,
\begin{eqnarray}
\lim_{n\to \infty} R_1 &=& 	\lim_{n\to \infty} 	\E\Big\{g\big((\bX_1, \bZ_1), \bZ_{N(1)}, (\bX_{M(1)}, \bZ_{M(1)})\big) \Big\} = \E\Big\{g\big((\bX_1, \bZ_1), \bZ_1, (\bX_1, \bZ_1)\big) \Big\} \cr
&=&
\E\Big\{\int \int \int F_Y(u\wedge v) \cdot F_Y(u \wedge w) 
\d \tmu_{(\bX_1,\bZ_1)}(u) \d \tmu_{\bZ_1}(v) \d \tmu_{(\bX_1,\bZ_1)}(w) \Big\} \cr
&=&
\E\Big\{ \E\big[ F_Y(Y_1 \wedge \tY_1)\cdot F_Y(Y_1 \wedge \barY_1)\mid \bX_1, \bZ_1\big]\Big\}. \label{eq:lemA:lim1:2}
\end{eqnarray}

Next, we examine the term $R_2$. 
\begin{eqnarray*}
R_2 &=& 
\E\Big\{ \E\big[F_Y(Y_1 \wedge Y_{N(1)}) \mid \bfX, \bfZ\big] \cdot \E\big[F_Y(Y_1 \wedge Y_{M(1)}) \mid \bfX, \bfZ\big] \Big\}  \cr
&=&
\E\Big\{ \E\big[F_Y(Y_1 \wedge Y_{N(1)}) \mid \bfX, \bfZ\big] \cdot \E\big[F_Y(Y_1 \wedge Y_{M(1)}) \mid \bfX, \bfZ\big]  \cdot \Ind(N(1) \neq M(1)) \Big\} + o(1)
\cr
&=&
\E\Big\{\tildeg\big((\bX_1,\bZ_1), (\bX_{N(1)}, \bZ_{N(1)})\big) \cdot \tildeg\big((\bX_1,\bZ_1), (\bX_{M(1)}, \bZ_{M(1)})\big) \cdot \Ind(N(1) \neq M(1)) \Big\}	+o(1),
\end{eqnarray*}
where the function $\tildeg$ is defined in \eqref{eq:lemA:aux-2-6:1} of  Lemma \ref{lemA:aux-2-6}  and is bounded. 
Similar to the proof of \eqref{eq:lemA:lim1:1}, we can eliminate the term $\bX_{N(1)}$ in the expectation above and obtain
\begin{eqnarray}
&&\E\Big\{\tildeg\big((\bX_1,\bZ_1), (\bX_{N(1)}, \bZ_{N(1)})\big) \cdot \tildeg\big((\bX_1,\bZ_1), (\bX_{M(1)}, \bZ_{M(1)})\big) \cdot \Ind(N(1) \neq M(1)) \Big\}	\cr
&=&
\E\Big[\E\Big\{\tildeg\big((\bX_1,\bZ_1), (\bX_{N(1)}, \bZ_{N(1)})\big) \cdot \tildeg\big((\bX_1,\bZ_1), (\bX_{M(1)}, \bZ_{M(1)})\big) \cr
&& \hspace{2cm}\Biggiven(\bX_1,\bZ_1), \bZ_{N(1)}, (\bX_{M(1)}, \bZ_{M(1)}),  \Ind(N(1) \neq M(1))  \Big\} \cdot \Ind(N(1) \neq M(1)) \Big]	\cr
&=&
\E\Big[ 
\int\tildeg\big((\bX_1,\bZ_1), (\bx, \bZ_{N(1)})\big) \d \mu_{\bX=\bx \mid \bZ = \bZ_{N(1)}}(\bx) \cr
&&\qquad \times \tildeg\big((\bX_1,\bZ_1), (\bX_{M(1)}, \bZ_{M(1)})\big) \cdot \Ind(N(1) \neq M(1)) \Big]\cr
&=&
\E\Big\{\tildeg^\dagger \big((\bX_1,\bZ_1), \bZ_{N(1)}\big) \cdot \tildeg\big((\bX_1,\bZ_1), (\bX_{M(1)}, \bZ_{M(1)})\big) \cdot \Ind(N(1) \neq M(1)) \Big\} \cr
&=&
\E\Big\{\tildeg^\dagger \big((\bX_1,\bZ_1), \bZ_{N(1)}\big) \cdot \tildeg\big((\bX_1,\bZ_1), (\bX_{M(1)}, \bZ_{M(1)})\big) \Big\}  + o(1),
\label{eq:lemA:lim1:3}
\end{eqnarray}
where the function $\tildeg^\dagger$ is defined in \eqref{eq:lemA:aux-2-7:1} of Lemma \ref{lemA:aux-2-7} and is bounded. 
By Lemmas \ref{lemA:aux-2-6} and \ref{lemA:aux-2-7}, we have
\begin{eqnarray*}
&&	\tildeg^\dagger \big((\bX_1,\bZ_1), \bZ_{N(1)}\big) \cdot \tildeg\big((\bX_1,\bZ_1), (\bX_{M(1)}, \bZ_{M(1)})\big) \cr
&& \hspace{3cm} - \tildeg^\dagger \big((\bX_1,\bZ_1), \bZ_1\big) \cdot \tildeg\big((\bX_1,\bZ_1), (\bX_1, \bZ_1)\big) 
\conP 0.
\end{eqnarray*}
Therefore,
\begin{eqnarray}
\lim_{n\to \infty} R_2 &=& 	\lim_{n\to \infty} \E\Big\{\tildeg^\dagger \big((\bX_1,\bZ_1), \bZ_{N(1)}\big) \cdot \tildeg\big((\bX_1,\bZ_1), (\bX_{M(1)}, \bZ_{M(1)})\big)  \Big\}\cr
&=&
\lim_{n\to \infty} \E\Big\{\tildeg^\dagger \big((\bX_1,\bZ_1), \bZ_1\big) \cdot \tildeg\big((\bX_1,\bZ_1), (\bX_1, \bZ_1)\big) \Big\} \cr
&=&
\E\Big\{ \E\big[F_Y(Y_1 \wedge \tY_1)\mid \bX_1, \bZ_1\big]  \cdot \E\big[F_Y(Y_1 \wedge \barY_1)\mid \bX_1, \bZ_1\big]  \Big\}. \label{eq:lemA:lim1:4}
\end{eqnarray}
Combining \eqref{eq:lemA:lim1:2} and \eqref{eq:lemA:lim1:4} completes the proof.
\end{proof}

\begin{lemmaA} \label{lemA:lim2}
The limit of $Q_{1,1,2}$ in \eqref{eq:Q11_decomp} is
\begin{eqnarray*}
\lim_{n\to \infty}\frac{1}{n}\E \Big\{\sum_{(i,j)\in \lbr n \rbr\times\lbr n \rbr: \, i,j, N(i), M(j) \,\text{distinct}} \cov\big[F_Y(Y_i \wedge Y_{N(j)}), F_Y(Y_i \wedge Y_{M(j)}) \mid \bfX, \bfZ\big] \Big\}=0.
\end{eqnarray*}
\end{lemmaA}

\begin{proof} 
\begin{eqnarray*}
&&\frac{1}{n}\E \Big\{\sum_{(i,j)\in \lbr n \rbr\times\lbr n \rbr: \, i,j, N(i), M(j) \,\text{distinct}} \cov\big[F_Y(Y_i \wedge Y_{N(i)}), F_Y(Y_j \wedge Y_{M(j)}) \mid \bfX, \bfZ\big] \Big\}
\cr
&=&
(n-1)\cdot 
\E \Big\{ \cov\big[ F_Y(Y_1 \wedge Y_{N(1)}), F_Y(Y_2\wedge Y_{M(2)}) \mid \bfX, \bfZ\big] \cdot \Ind\big(1,2, N(1), M(2) \,\text{distinct}\big)\Big\}.
\end{eqnarray*}
For $(\bfX, \bfZ)$ such that $1,2, N(1), M(2)$ are distinct, $Y_1, Y_{N(1)}, Y_2, Y_{M(2)}$ conditional on $(\bfX, \bfZ)$ are independently distributed, and thus
\begin{eqnarray*}
\cov\big[ F_Y(Y_1 \wedge Y_{N(1)}), F_Y(Y_2\wedge Y_{M(2)}) \mid \bfX, \bfZ\big]\cdot \Ind\big(1,2, N(1), M(2) \,\text{distinct}\big) =0.
\end{eqnarray*}
Then
\begin{eqnarray*}
\E \Big\{ \cov\big[ F_Y(Y_1 \wedge Y_{N(1)}), F_Y(Y_2\wedge Y_{M(2)}) \mid \bfX, \bfZ\big] \cdot \Ind\big(1,2, N(1), M(2) \,\text{distinct}\big)\Big\}=0.
\end{eqnarray*}
This completes the proof. \end{proof}

\begin{lemmaA} \label{lemA:lim3}
The limit of $Q_{1,1,3}$ in \eqref{eq:Q11_decomp} is
\begin{eqnarray*}
&&\lim_{n\to \infty} \frac{1}{n}	
\E \Big\{
\sum_{\underset{\text{ or } i=M(j), j\neq N(i)}{(i,j)\in \lbr n \rbr\times\lbr n \rbr: j=N(i), i\neq M(j),}	}
\cov\big[F_Y(Y_i \wedge Y_{N(i)}), F_Y(Y_j \wedge Y_{M(j)}) \mid \bfX, \bfZ\big]\Big\} \cr
&=&
2\cdot \E \Big\{\cov\big[F_Y(Y_1 \wedge \tY_1),  F_Y(Y_1 \wedge \barY_1)\mid \bX_1, \bZ_1 \big] \Big\}.
\end{eqnarray*}
\end{lemmaA}

\begin{proof}
The proof is divided into the following two cases.

\noindent \textbf{Case (i): $j=N(i), i\neq M(j)$. }

By \citet[Lemma 7.4, Equation (27)]{Shi_Drton_Han_2024_Bernoulli}, we have that $\P(M(N(1))=1) =o(1)$ and $\P(N(M(1))=1) =o(1)$. Then
\begin{eqnarray*}
&&\frac{1}{n}	
\E \Big\{\sum_{(i,j)\in \lbr n \rbr\times\lbr n \rbr: j=N(i), i\neq M(j)}  \cov\big[F_Y(Y_i \wedge Y_{N(i)}), F_Y(Y_j \wedge Y_{M(j)}) \mid \bfX, \bfZ\big] \Big\} \cr
&=& 	
\E \Big\{\sum_{j\in \lbr n \rbr: j=N(1), 1\neq M(j)}  \cov\big[F_Y(Y_1 \wedge Y_{N(1)}), F_Y(Y_j \wedge Y_{M(j)}) \mid \bfX, \bfZ\big]\Big\} \cr
&=&
\E \Big\{\cov\big[F_Y(Y_1 \wedge Y_{N(1)}), F_Y(Y_{N(1)} \wedge Y_{M(N(1))}) \mid \bfX, \bfZ\big] \cdot \Ind (M(N(1)) \neq 1) \Big\} \cr
&=&
\E \Big\{\cov\big[F_Y(Y_1 \wedge Y_{N(1)}), F_Y(Y_{N(1)} \wedge Y_{M(N(1))}) \mid \bfX, \bfZ\big] \Big\} + o(1) .
\end{eqnarray*}
By the decomposition of conditional covariance, we have
\begin{eqnarray*}
&&	\E \Big\{\cov\big[F_Y(Y_1 \wedge Y_{N(1)}), F_Y(Y_{N(1)} \wedge Y_{M(N(1))}) \mid \bfX, \bfZ\big] \Big\} \cr
&=&
\E \Big\{\E\big[F_Y(Y_1 \wedge Y_{N(1)})\cdot F_Y(Y_{N(1)} \wedge Y_{M(N(1))}) \mid \bfX, \bfZ\big] \Big\} \cr
&&- \E\Big\{ \E\big[F_Y(Y_1 \wedge Y_{N(1)}) \mid \bfX, \bfZ\big] \cdot \E\big[F_Y(Y_{N(1)} \wedge Y_{M(N(1))})\mid \bfX, \bfZ\big] \Big\} \cr
&=:& R_1 - R_2.
\end{eqnarray*}
For the first term $R_1$, 
by Lemma \ref{lemA:aux-2-4}, we have
\begin{eqnarray*}
R_1 &=&
\E  \Big\{ \int \int \int F_Y(u\wedge v) \cdot F_Y(u \wedge w) 
\d \tmu_{(\bX_{N(1)}, \bZ_{N(1)})}(u) \d \tmu_{(\bX_1, \bZ_1)}(v) \d \tmu_{(\bX_{M(N(1))}, \bZ_{M(N(1))})}(w)  \Big\} \cr
&=&
\E\Big\{ g^*\big((\bX_{N(1)}, \bZ_{N(1)}), (\bX_1, \bZ_1), (\bX_{M(N(1))}, \bZ_{M(N(1))})\big)  \Big\} \cr
&=&
\E\Big\{ g^*\big((\bX_{N(1)}, \bZ_{N(1)}), (\bX_1, \bZ_1), (\bX_{N(1)}, \bZ_{N(1)})\big)\Big\} + o(1),
\end{eqnarray*}
where function $g^*$ is defined in \eqref{eq:lemA:aux-2-4:1} of Lemma \ref{lemA:aux-2-4}. 
Note that conditional on $(\bX_{N(1)}, \bZ_{N(1)}, \bZ_1)$, the distribution of $\bX_1$ only depends on $ \bZ_1$ and is irrelevant with respect to $(\bX_{N(1)}, \bZ_{N(1)})$. 
Then using the similar proof as that of \eqref{eq:lemA:lim1:1}, we further have
\begin{eqnarray*}
&&	\E\Big\{ g^*\big((\bX_{N(1)}, \bZ_{N(1)}), (\bX_1, \bZ_1), (\bX_{N(1)}, \bZ_{N(1)})\big)\Big\} \cr
&=&
\E  \Big\{ \int \int \int F_Y(u\wedge v) \cdot F_Y(u \wedge w) 
\d \tmu_{(\bX_{N(1)}, \bZ_{N(1)})}(u) \d \tmu_{(\bX_1, \bZ_1)}(v) \d \tmu_{(\bX_{N(1)}, \bZ_{N(1)})}(w)  \Big\} \cr
&=&
\E  \Big\{ \int \int \int F_Y(u\wedge v) \cdot F_Y(u \wedge w) 
\d \tmu_{(\bX_{N(1)}, \bZ_{N(1)})}(u) \d \tmu_{\bZ_1}(v) \d \tmu_{(\bX_{N(1)}, \bZ_{N(1)})}(w)  \Big\} \cr
&=&
\E\Big\{ g\big((\bX_{N(1)}, \bZ_{N(1)}), \bZ_1, (\bX_{N(1)}, \bZ_{N(1)})\big) \Big\},
\end{eqnarray*}
where $g$ is defined in \eqref{eq:lemA:aux-2-3:1}. By applying Lemma \ref{lemA:aux-2-5},
\begin{eqnarray*}
&&\E\Big\{ g\big((\bX_{N(1)}, \bZ_{N(1)}), \bZ_1, (\bX_{N(1)}, \bZ_{N(1)})\big) \Big\} \cr
&=& 	\E\Big\{ g\big((\bX_{N(1)}, \bZ_{N(1)}), \bZ_{N(1)}, (\bX_{N(1)}, \bZ_{N(1)})\big) \Big\} + o(1).
\end{eqnarray*}
By \eqref{eq:lemA:aux-1-8:4}, we have $(\bX_{N(1)}, \bZ_{N(1)}) \conD (\bX_1, \bZ_1)$. This combined with the fact that both $(\bX_{N(1)}, \bZ_{N(1)})$ and $(\bX_1, \bZ_1)$ are absolutely continuous yields that
\begin{eqnarray}
&& \E\Big\{ g\big((\bX_{N(1)}, \bZ_{N(1)}), \bZ_{N(1)}, (\bX_{N(1)}, \bZ_{N(1)})\big) \Big\}  \cr
&=&
\E\Big\{ g\big((\bX_1, \bZ_1), \bZ_1, (\bX_1, \bZ_1)\big) \Big\} +o(1).  \label{eq:lemA:lim3:1}
\end{eqnarray}
Combining the above pieces gives that
\begin{eqnarray}
\lim_{n \to \infty} R_1
&=&
\E\Big\{ g\big((\bX_1, \bZ_1), \bZ_1, (\bX_1, \bZ_1)\big) \Big\}\cr
&=&
\E\Big\{ \E\big[ F_Y(Y_1 \wedge \tY_1)\cdot F_Y(Y_1 \wedge \barY_1)\mid \bX_1, \bZ_1\big]\Big\}. \label{eq:lemA:lim3:2}
\end{eqnarray}

For the second term $R_2$, we have
\begin{eqnarray*}
R_2&=& 
\E\Big\{  \E\big[F_Y(Y_1 \wedge Y_{N(1)}) \mid \bfX, \bfZ\big] \cdot \E\big[F_Y(Y_{N(1)} \wedge Y_{M(N(1))})\mid \bfX, \bfZ\big] \Big\}  \cr
&=&
\E\Big\{  \E\big[F_Y(Y_1 \wedge Y_{N(1)}) \mid \bfX, \bfZ\big] \cdot \E\big[F_Y(Y_{N(1)} \wedge Y_{M(N(1))})\mid \bfX, \bfZ\big] \cdot  \Ind (M(N(1)) \neq 1)  \Big\} + o(1)
\cr
&=&
\E\Big\{\tildeg\big((\bX_1,\bZ_1), (\bX_{N(1)}, \bZ_{N(1)})\big) \cdot \tildeg\big((\bX_{N(1)}, \bZ_{N(1)}), (\bX_{M(N(1))}, \bZ_{M(N(1))})\big) \cr
&& \hspace{1cm} \times \Ind (M(N(1)) \neq 1)  \Big\}	+o(1),
\end{eqnarray*}
where the function $\tildeg$ is defined in \eqref{eq:lemA:aux-2-6:1} of Lemma \ref{lemA:aux-2-6}. Using the similar proof as that of \eqref{eq:lemA:lim1:3}, we can eliminate the term $\bX_1$ in the expectation above. Then
\begin{eqnarray*}
&&\E\Big\{\tildeg\big((\bX_1,\bZ_1), (\bX_{N(1)}, \bZ_{N(1)})\big) \cdot \tildeg\big((\bX_{N(1)}, \bZ_{N(1)}), (\bX_{M(N(1))}, \bZ_{M(N(1))})\big) \cdot \Ind (M(N(1)) \neq 1)  \Big\}\cr
&=&
\E\Big\{\tildeg^\dagger\big( (\bX_{N(1)}, \bZ_{N(1)}), \bZ_1\big) \cdot \tildeg\big((\bX_{N(1)}, \bZ_{N(1)}), (\bX_{M(N(1))}, \bZ_{M(N(1))})\big) \cdot \Ind (M(N(1)) \neq 1)  \Big\} \cr
&=&
\E\Big\{\tildeg^\dagger\big( (\bX_{N(1)}, \bZ_{N(1)}), \bZ_1\big) \cdot \tildeg\big((\bX_{N(1)}, \bZ_{N(1)}), (\bX_{M(N(1))}, \bZ_{M(N(1))})\big)\Big\}
+o(1),
\end{eqnarray*}
where $\tildeg^\dagger$ is defined in \eqref{eq:lemA:aux-2-7:1} of Lemma \ref{lemA:aux-2-7}. 
By Lemmas \ref{lemA:aux-2-6} and \ref{lemA:aux-2-7}, we have
\begin{eqnarray*}
&&\E\Big\{\tildeg^\dagger\big( (\bX_{N(1)}, \bZ_{N(1)}), \bZ_1\big) \cdot \tildeg\big((\bX_{N(1)}, \bZ_{N(1)}), (\bX_{M(N(1))}, \bZ_{M(N(1))})\big)\Big\} \cr
&=&
\E\Big\{\tildeg^\dagger\big( (\bX_{N(1)}, \bZ_{N(1)}), \bZ_{N(1)}\big) \cdot \tildeg\big((\bX_{N(1)}, \bZ_{N(1)}), (\bX_{N(1)}, \bZ_{N(1)})\big)\Big\}  +o(1).
\end{eqnarray*}
Similar to \eqref{eq:lemA:lim3:1}, it follows that
\begin{eqnarray*}
&&\E\Big\{\tildeg^\dagger\big( (\bX_{N(1)}, \bZ_{N(1)}), \bZ_{N(1)}\big) \cdot \tildeg\big((\bX_{N(1)}, \bZ_{N(1)}), (\bX_{N(1)}, \bZ_{N(1)})\big)\Big\} \cr
&=&
\E\Big\{\tildeg^\dagger\big( (\bX_1, \bZ_1), \bZ_1\big) \cdot \tildeg\big((\bX_1, \bZ_1), (\bX_1, \bZ_1)\big)\Big\} +o(1).
\end{eqnarray*}
Therefore,
\begin{eqnarray}
\lim_{n\to \infty} R_2 &=&
\E\Big\{\tildeg^\dagger\big( (\bX_1, \bZ_1), \bZ_1\big) \cdot \tildeg\big((\bX_1, \bZ_1), (\bX_1, \bZ_1)\big)\Big\} \cr
&=&
\E\Big\{ \E\big[F_Y(Y_1 \wedge \tY_1)\mid \bX_1, \bZ_1\big]  \cdot \E\big[F_Y(Y_1 \wedge \barY_1)\mid \bX_1, \bZ_1\big]  \Big\}. \label{eq:lemA:lim3:3}
\end{eqnarray}
Combining \eqref{eq:lemA:lim3:2} and \eqref{eq:lemA:lim3:3} proves that 
\begin{eqnarray*}
&&\lim_{n\to \infty}\frac{1}{n}	
\E \Big\{\sum_{(i,j)\in \lbr n \rbr\times\lbr n \rbr: j=N(i), i\neq M(j)}  \cov\big[F_Y(Y_i \wedge Y_{N(i)}), F_Y(Y_j \wedge Y_{M(j)}) \mid \bfX, \bfZ\big] \Big\} \cr
&=&
\E \Big\{\cov\big[F_Y(Y_1 \wedge \tY_1),  F_Y(Y_1 \wedge \barY_1)\mid \bX_1, \bZ_1 \big] \Big\}.
\end{eqnarray*}
\textbf{Case (ii): $i=M(j), j\neq N(i)$. }

Similarly, we have
\begin{eqnarray*}
&&\frac{1}{n}	
\E \Big\{\sum_{(i,j)\in \lbr n \rbr\times\lbr n \rbr:  i=M(j), j\neq N(i)}  \cov\big[F_Y(Y_i \wedge Y_{N(i)}), F_Y(Y_j \wedge Y_{M(j)}) \mid \bfX, \bfZ\big] \Big\} \cr
&=&
\E \Big\{\cov\big[F_Y(Y_1 \wedge Y_{M(1)}), F_Y(Y_{M(1)} \wedge Y_{N(M(1))}) \mid \bfX, \bfZ\big] \cdot \Ind (N(M(1)) \neq 1) \Big\} \cr
&=&
\E \Big\{\cov\big[F_Y(Y_1 \wedge Y_{M(1)}), F_Y(Y_{M(1)} \wedge Y_{N(M(1))}) \mid \bfX, \bfZ\big] \Big\}+ o(1). 
\end{eqnarray*}
By the decomposition of conditional covariance, we have
\begin{eqnarray*}
&&	\E \Big\{\cov\big[F_Y(Y_1 \wedge Y_{M(1)}), F_Y(Y_{M(1)} \wedge Y_{N(M(1))}) \mid \bfX, \bfZ\big] \Big\} \cr
&=&
\E \Big\{\E\big[F_Y(Y_1 \wedge Y_{M(1)})\cdot F_Y(Y_{M(1)} \wedge Y_{N(M(1))}) \mid \bfX, \bfZ\big] \Big\} \cr
&&- \E\Big\{ \E\big[F_Y(Y_1 \wedge Y_{M(1)}) \mid \bfX, \bfZ\big] \cdot \E\big[F_Y(Y_{M(1)} \wedge Y_{N(M(1))})\mid \bfX, \bfZ\big] \Big\} \cr
&=:& R_3 - R_4.
\end{eqnarray*}
Conditional on the event $\{ N(M(1)) \neq 1, (\bX_1, \bZ_1) = (\bx_1,\bz_1),  (\bX_{M(1)}, \bZ_{M(1)})=(\bx_{M(1)}, \bz_{M(1)}),$ $ \bZ_{N(M(1))} = \bz_{N(M(1))} \}$, the distribution of $\bX_{N(M(1))}$ only depends on $\bz_{N(M(1))}$ and is irrelevant of $(\bx_1,\bz_1)$ and $(\bx_{M(1)}, \bz_{M(1)})$. Thus, similar to the proof of \eqref{eq:lemA:lim1:1} and \eqref{eq:lemA:lim1:3}, we can eliminate the term $\bX_{N(M(1))}$ in the expectation and obtain
\begin{eqnarray}
R_3
&=&
\E\Big\{ g^*\big((\bX_{M(1)}, \bZ_{M(1)}), (\bX_1, \bZ_1), (\bX_{N(M(1))}, \bZ_{N(M(1))})\big)  \Big\} \cr
&=& 
\E\Big\{ g^*\big((\bX_{M(1)}, \bZ_{M(1)}), (\bX_1, \bZ_1), (\bX_{N(M(1))}, \bZ_{N(M(1))})\big)   \cdot \Ind(N(M(1)) \neq 1) \Big\} +o(1) \cr
&=&
\E\Big\{ g\big((\bX_{M(1)}, \bZ_{M(1)}), \bZ_{N(M(1))}, (\bX_1, \bZ_1) \big)   \cdot \Ind(N(M(1)) \neq 1) \Big\} +o(1) \cr
&=&
\E\Big\{ g\big((\bX_{M(1)}, \bZ_{M(1)}), \bZ_{N(M(1))}, (\bX_1, \bZ_1) \big)    \Big\} +o(1). \label{eq:lemA:lim3:4}
\end{eqnarray}
By Lemma \ref{lemA:aux-2-5},
\begin{eqnarray*}
\E\Big\{ g\big((\bX_{M(1)}, \bZ_{M(1)}), \bZ_{N(M(1))}, (\bX_1, \bZ_1) \big)    \Big\}  
= \E\Big\{ g\big((\bX_1, \bZ_1), \bZ_1, (\bX_1, \bZ_1) \big)    \Big\}  + o(1). 
\end{eqnarray*}
Thus, \begin{eqnarray}
\lim_{n \to \infty} R_3
&=&
\E\Big\{ g\big((\bX_1, \bZ_1), \bZ_1, (\bX_1, \bZ_1)\big) \Big\}\cr
&=&
\E\Big\{ \E\big[ F_Y(Y_1 \wedge \tY_1)\cdot F_Y(Y_1 \wedge \barY_1)\mid \bX_1, \bZ_1\big]\Big\}. \label{eq:lemA:lim3:5}
\end{eqnarray}

For the term $R_4$, note that
\begin{eqnarray*}
R_4 &=& \E\Big\{ \tildeg\big((\bX_{1},\bZ_{1}), (\bX_{M(1)}, \bZ_{M(1)})\big)
\cdot 
\tildeg \big((\bX_{M(1)}, \bZ_{M(1)}),(\bX_{N(M(1))}, \bZ_{N(M(1))}) \big)
\Big\} \cr
&=&
\E\Big\{ \tildeg\big((\bX_{1},\bZ_{1}), (\bX_{M(1)}, \bZ_{M(1)})\big)
\cdot 
\tildeg \big((\bX_{M(1)}, \bZ_{M(1)}),(\bX_{N(M(1))}, \bZ_{N(M(1))}) \big) \cr
&& \qquad
\times \Ind(N(M(1)) \neq 1)
\Big\}  + o(1).
\end{eqnarray*}
Similar to \eqref{eq:lemA:lim3:4}, we can eliminate the term $\bX_{N(M(1))}$ in the expectation and obtain
\begin{eqnarray*}
&& \E\Big\{ \tildeg\big((\bX_{1},\bZ_{1}), (\bX_{M(1)}, \bZ_{M(1)})\big)
\cdot 
\tildeg \big((\bX_{M(1)}, \bZ_{M(1)}),(\bX_{N(M(1))}, \bZ_{N(M(1))}) \big) 
\cdot \Ind(N(M(1)) \neq 1)
\Big\}  \cr
&=&
\E\Big\{ \tildeg \big((\bX_{1},\bZ_{1}), (\bX_{M(1)}, \bZ_{M(1)})\big)
\cdot 
\tildeg^\dagger \big( (\bX_{M(1)}, \bZ_{M(1)}), \bZ_{M(1)} \big) 
\cdot \Ind(N(M(1)) \neq 1)
\Big\}  \cr
&=&
\E\Big\{ \tildeg \big((\bX_{1},\bZ_{1}), (\bX_{M(1)}, \bZ_{M(1)})\big)
\cdot 
\tildeg^\dagger \big((\bX_{M(1)}, \bZ_{M(1)}), \bZ_{M(1)} \big) 
\Big\}  +o(1)
\end{eqnarray*}
By Lemmas \ref{lemA:aux-2-6} and \ref{lemA:aux-2-7}, we have
\begin{eqnarray*}
&& \E\Big\{ \tildeg \big((\bX_{1},\bZ_{1}), (\bX_{M(1)}, \bZ_{M(1)})\big)
\cdot 
\tildeg^\dagger \big((\bX_{M(1)}, \bZ_{M(1)}), \bZ_{M(1)} \big) 
\Big\} \cr
&=&
\E\Big\{ \tildeg \big((\bX_{1},\bZ_{1}), (\bX_1, \bZ_1)\big)
\cdot 
\tildeg^\dagger \big((\bX_1, \bZ_1), \bZ_1 \big) 
\Big\} +o(1).
\end{eqnarray*}
Therefore,
\begin{eqnarray}
\lim_{n\to \infty} R_4 &=&
\E\Big\{ \tildeg\big((\bX_1, \bZ_1), (\bX_1, \bZ_1)\big) \cdot \tildeg^\dagger\big( (\bX_1, \bZ_1), \bZ_1\big) \Big\} \cr
&=&
\E\Big\{  \E\big[F_Y(Y_1 \wedge \barY_1)\mid \bX_1, \bZ_1\big] \cdot 
\E\big[F_Y(Y_1 \wedge \tY_1)\mid \bX_1, \bZ_1\big]   \Big\}. \label{eq:lemA:lim3:6}
\end{eqnarray}
Combining \eqref{eq:lemA:lim3:5} and \eqref{eq:lemA:lim3:6} proves that 
\begin{eqnarray*}
&& \lim_{n \to \infty}\frac{1}{n}	
\E \Big\{\sum_{(i,j)\in \lbr n \rbr\times\lbr n \rbr:  i=M(j), j\neq N(i)}  \cov\big[F_Y(Y_i \wedge Y_{N(i)}), F_Y(Y_j \wedge Y_{M(j)}) \mid \bfX, \bfZ\big] \Big\} \cr
&=&
\E \Big\{\cov\big[F_Y(Y_1 \wedge \tY_1),  F_Y(Y_1 \wedge \barY_1)\mid \bX_1, \bZ_1 \big] \Big\}.
\end{eqnarray*}

Finally, combining the results in the two cases above completes the proof of this lemma. 
\end{proof}

\begin{lemmaA}\label{lemA:lim4}
The limit of $Q_{1,1,4}$ in \eqref{eq:Q11_decomp} is
\begin{eqnarray*}
&&\lim_{n \to \infty}\frac{1}{n}	
\E \Big\{\sum_{(i,j)\in \lbr n \rbr\times\lbr n \rbr:i\neq j, N(i) = M(j)}  \cov\big[F_Y(Y_i \wedge Y_{N(i)}), F_Y(Y_j \wedge Y_{M(j)}) \mid \bfX, \bfZ\big] \Big\} \cr
&=&
\E \Big\{\cov\big[F_Y(Y_1 \wedge \tY_1),  F_Y(Y_1 \wedge \barY_1)\mid \bX_1, \bZ_1 \big] \Big\}.
\end{eqnarray*}
\end{lemmaA}

\begin{proof}
\begin{eqnarray*}
&&\frac{1}{n}	
\E \Big\{\sum_{(i,j)\in \lbr n \rbr\times\lbr n \rbr:i\neq j, N(i) = M(j)}  \cov\big[F_Y(Y_i \wedge Y_{N(i)}), F_Y(Y_j \wedge Y_{M(j)}) \mid \bfX, \bfZ\big] \Big\} \cr
&=&
(n-1)\cdot\E \Big\{ \cov\big[F_Y(Y_1 \wedge Y_{M(1)}), F_Y(Y_2 \wedge Y_{N(2)}) \mid \bfX, \bfZ\big] \cdot \Ind(M(1) = N(2))\Big\}  \cr
&=&
(n-1)\cdot\E \Big\{ \E\big[F_Y(Y_1 \wedge Y_{M(1)})\cdot F_Y(Y_2 \wedge Y_{N(2)}) \mid \bfX, \bfZ\big] \cdot \Ind(M(1) = N(2))\Big\}  \cr
&& 
- (n-1)\cdot \E \Big\{ \E\big[F_Y(Y_1 \wedge Y_{M(1)})\mid \bfX, \bfZ\big] \cdot\E\big[F_Y(Y_2 \wedge Y_{N(2)})\mid \bfX, \bfZ\big] \cdot \Ind(M(1) = N(2))\Big\}  \cr
&=:& R_1 -R_2.
\end{eqnarray*}

For the term $R_1$, we have
\begin{eqnarray*}
R_1 &=& (n-1)\cdot\E  \Big\{ \int \int \int F_Y(u\wedge v) \cdot F_Y(u \wedge w) 
\d \tmu_{(\bX_{M(1)},\bZ_{M(1)})}(u) \d \tmu_{(\bX_1,\bZ_1)}(v) \d \tmu_{(\bX_2,\bZ_2)}(w)  \cr
&& \quad \times \Ind(M(1) = N(2)) \Big\} .
\end{eqnarray*}
Conditional on the event $	\{ M(1) = N(2), (\bX_1, \bZ_1) = (\bx_1,\bz_1), \bZ_2 = \bz_2, (\bX_{M(1)}, \bZ_{M(1)})=(\bx_{M(1)}, \bz_{M(1)}) \}$, the distribution law of $\bX_2 = \bx$ is $\mu_{\bX=\bx \mid \bZ = \bz_2}$. That is, $\bX_2$ only depends on $\bz_2$ and is irrelevant of $ (\bx_1,\bz_1)$ and $(\bx_{M(1)}, \bz_{M(1)})$. Therefore, we can eliminate the random term 
$\bX_2$, in a similar way as in \eqref{eq:lemA:lim1:1},
\begin{eqnarray*}
&&\E  \Big\{ \int \int \int F_Y(u\wedge v) \cdot F_Y(u \wedge w) 
\d \tmu_{(\bX_{M(1)},\bZ_{M(1)})}(u) \d \tmu_{(\bX_1,\bZ_1)}(v) \d \tmu_{(\bX_2,\bZ_2)}(w) \cr
&& \hspace{3cm}\times\Ind(M(1) = N(2))\Big\} \cr
&=& \E\Big[\E  \Big\{ \int \int \int F_Y(u\wedge v) \cdot F_Y(u \wedge w) 
\d \tmu_{(\bX_{M(1)},\bZ_{M(1)})}(u) \d \tmu_{(\bX_1,\bZ_1)}(v) \d \tmu_{(\bX_2,\bZ_2)}(w)  \cr 
&& \hspace{3cm}\Biggiven(\bX_1,\bZ_1), \bZ_2, (\bX_{M(1)}, \bZ_{M(1)}) , \Ind(M(1) = N(2))\Big\} \cdot \Ind(M(1) = N(2))\Big]
\cr
&=&\E \Big[ \Big\{\int \int \int \int F_Y(u\wedge v) \cdot F_Y(u \wedge w) \d \tmu_{(\bX_{M(1)}, \bZ_{M(1)})}(u) \d \tmu_{(\bX_1,\bZ_1)}(v)  \cr
&& \hspace{3cm}
\d \tmu_{(\bx, \bZ_2)}(w)   \d \mu_{\bX=\bx \mid \bZ = \bZ_2}(\bx)\Big\} \cdot \Ind(M(1) = N(2))\Big]  \cr
&=&\E \Big[\Big\{ \int \int \int F_Y(u\wedge v) \cdot F_Y(u \wedge w) 
\d \tmu_{(\bX_{M(1)}, \bZ_{M(1)})}(u)  	\d \tmu_{(\bX_1,\bZ_1)}(v)   \d \tmu_{\bZ_2}(w) \Big\} \cr
&& \hspace{3cm} \times \Ind(M(1) = N(2))\Big] \cr
&=& \E\Big\{g\big((\bX_{M(1)}, \bZ_{M(1)}), \bZ_2,  (\bX_1,\bZ_1) \big) \cdot \Ind(M(1) = N(2))\Big\}.
\end{eqnarray*}
By applying Lemma \ref{lemA:aux-2-8}, we have
\begin{eqnarray*}
\lim_{n\to \infty} R_1 
&=&
\lim_{n\to \infty}  (n-1)\cdot \E\Big\{g\big((\bX_{M(1)}, \bZ_{M(1)}), \bZ_2,  (\bX_1,\bZ_1) \big) \cdot \Ind(M(1) = N(2))\Big\} \cr
&=&
\E\Big\{ g\big((\bX_1, \bZ_1), \bZ_1, (\bX_1, \bZ_1) \Big\} \cr
&=&
\E\Big\{ \E\big[ F_Y(Y_1 \wedge \tY_1)\cdot F_Y(Y_1 \wedge \barY_1)\mid \bX_1, \bZ_1\big]\Big\}.
\end{eqnarray*}	

For the term $R_2$, we have
\begin{eqnarray*}
R_2 =
(n-1)\cdot \E  \Big\{ \tildeg\big((\bX_1,\bZ_1), (\bX_{M(1)},\bZ_{M(1)})\big) \cdot \tildeg\big((\bX_{M(1)},\bZ_{M(1)}), (\bX_2,\bZ_2)\big) \cdot \Ind\big(N(2)=M(1)\big) \Big\}.
\end{eqnarray*}
Similar to \eqref{eq:lemA:lim1:3}, we can eliminate the term $\bX_2$ in the expectation, that is,
\begin{eqnarray*}
&&\E  \Big\{ \tildeg\big((\bX_1,\bZ_1), (\bX_{M(1)},\bZ_{M(1)})\big) \cdot \tildeg\big((\bX_{M(1)},\bZ_{M(1)}), (\bX_2,\bZ_2)\big) \cdot \Ind\big(N(2)=M(1)\big) \Big\} \cr
&=&
\E  \Big\{ \tildeg\big((\bX_1,\bZ_1), (\bX_{M(1)},\bZ_{M(1)})\big) \cdot \tildeg^\dagger \big((\bX_{M(1)},\bZ_{M(1)}), \bZ_2\big) \cdot \Ind\big(N(2)=M(1)\big) \Big\} .
\end{eqnarray*}	
By Lemma \ref{lemA:aux-2-9}, we have
\begin{eqnarray*}
&&\lim_{n\to \infty} R_2 \cr
&=&
\lim_{n\to \infty}  (n-1)\cdot 	\E  \Big\{ \tildeg\big((\bX_1,\bZ_1), (\bX_{M(1)},\bZ_{M(1)})\big) \cdot \tildeg^\dagger \big((\bX_{M(1)},\bZ_{M(1)}), \bZ_2\big) \cdot \Ind\big(N(2)=M(1)\big) \Big\}  \cr
&=&
\E  \Big\{ \tildeg\big((\bX_1,\bZ_1), (\bX_1,\bZ_1)\big) \cdot \tildeg^\dagger \big((\bX_1,\bZ_1), \bZ_1\big)\Big\}  \cr
&=&
\E\Big\{  \E\big[F_Y(Y_1 \wedge \barY_1)\mid \bX_1, \bZ_1\big] \cdot 
\E\big[F_Y(Y_1 \wedge \tY_1)\mid \bX_1, \bZ_1\big]   \Big\}.
\end{eqnarray*}	
Combining the above two results completes the proof. 
\end{proof}

\begin{lemmaA} \label{lemA:lim5}
The limit of $Q_{1,1,5}$ in \eqref{eq:Q11_decomp} is
\begin{eqnarray*}
\lim_{n\to \infty}\frac{1}{n}	
\E \Big\{\sum_{(i,j)\in \lbr n \rbr\times\lbr n \rbr: j = N(i), i = M(j)} \cov\big[F_Y(Y_i \wedge Y_{N(i)}), F_Y(Y_j \wedge Y_{M(j)})  \mid \bfX, \bfZ\big] \Big\} 
=
0.
\end{eqnarray*}
\end{lemmaA}

\begin{proof}
\begin{eqnarray*}
&&\frac{1}{n}	
\E \Big\{\sum_{(i,j)\in \lbr n \rbr\times\lbr n \rbr: j = N(i), i = M(j)} \cov\big[F_Y(Y_i \wedge Y_{N(i)}), F_Y(Y_j \wedge Y_{M(j)})  \mid \bfX, \bfZ\big] \Big\}  \cr
&=&
\E \Big\{\sum_{j\in \lbr n \rbr: j = N(1), 1= M(j)} \cov\big[F_Y(Y_1 \wedge Y_{N(1)}), F_Y(Y_j \wedge Y_{M(j)}) \mid \bfX, \bfZ\big] \Big\} \cr
&=&
\E \Big\{ \cov\big[F_Y(Y_1 \wedge Y_{N(1)}), F_Y(Y_{N(1)} \wedge Y_{M(N(1))}) \mid \bfX, \bfZ\big] \cdot \Ind(M(N(1))=1)\Big\}.
\end{eqnarray*}
By \citet[Lemma 7.4, Equation (27)]{Shi_Drton_Han_2024_Bernoulli}, we have that $\P(M(N(1))=1) =o(1)$. Thus the proof is completed. \end{proof}

\subsubsection{The limit of $Q_{1,2}$} \label{secA:proof-stepII:2}
We have
\begin{eqnarray*}
Q_{1,2} &=& \frac{1}{n}\cov\Big[ \sum_{i=1}^n  \E\Big(F_Y(Y_i \wedge Y_{N(i)}) \mid \bfX, \bfZ \Big),
\sum_{i=1}^n \E\Big( F_Y(Y_i \wedge Y_{M(i)})\mid \bfX, \bfZ \Big)
\Big]  .
\end{eqnarray*}
By Lemma C.1 in \cite{Lin_Han_2025_CLT}, we have
\begin{eqnarray*}
\lim_{n \to \infty} \frac{1}{n} \var\Big[  \sum_{i=1}^n \E\Big( F_Y(Y_i \wedge Y_{M(i)})\mid \bfX, \bfZ \Big) - \sum_{i=1}^n  \phi(\bX_i, \bZ_i)  \Big] = 0,
\end{eqnarray*}
where 
\begin{eqnarray*}
\phi(\bX_i, \bZ_i) : = \E\big\{F_Y(Y_i \wedge \barY_i) \biggiven \bX_i, \bZ_i  \big\}.
\end{eqnarray*}
Therefore, 
\begin{eqnarray*}
Q_{1,2}' &:=& \frac{1}{n}\cov\Big[ \sum_{i=1}^n  \E\Big(F_Y(Y_i \wedge Y_{N(i)}) \mid \bfX, \bfZ \Big),
\sum_{i=1}^n  \phi(\bX_i, \bZ_i) 
\Big]  
\end{eqnarray*}
has the same limit as $Q_{1,2}$ (if the limit exists). Thus it suffices to find $\lim_{n \to \infty} Q_{1,2}'$. 

Invoking the function 
$\tildeg$ defined in \eqref{eq:lemA:aux-2-6:1}, we obtain
\begin{eqnarray*}
\tildeg\big((\bx_1,\bz_1), (\bx_2,\bz_2)\big) = \int \int F_Y(u\wedge v) 
\d \tmu_{(\bx_1,\bz_1)}(u) \d \tmu_{(\bx_2,\bz_2)}(v).
\end{eqnarray*}
Then 
\begin{eqnarray*}
\E\Big(F_Y(Y_i \wedge Y_{N(i)}) \mid \bfX, \bfZ \Big) = \tildeg\big((\bX_i, \bZ_i), (\bX_{N(i)}, \bZ_{N(i)})\big).
\end{eqnarray*}
Decompose $	Q_{1,2}'$ as follows:
\begin{eqnarray}
Q_{1,2}' &=&
\frac{1}{n} \E\Big\{ \cov \Big[\sum_{i=1}^n  \tildeg\big((\bX_i, \bZ_i), (\bX_{N(i)}, \bZ_{N(i)})\big) , 	\sum_{i=1}^n  \phi(\bX_i, \bZ_i)  \Biggiven \bfZ \Big]\Big\} \cr
&&+\frac{1}{n} \cov\Big\{
\E\Big[\sum_{i=1}^n  \tildeg\big((\bX_i, \bZ_i), (\bX_{N(i)}, \bZ_{N(i)})\big)   \Biggiven \bfZ\Big],
\E\Big[\sum_{i=1}^n  \phi(\bX_i, \bZ_i)  \Biggiven \bfZ\Big]
\Big\} \cr
&=:& Q_{1,2,1}' + Q_{1,2,2}'.  \label{eq:Q'_12}
\end{eqnarray}
Moreover, decompose $Q_{1,2,1}'$ into three terms:
\begin{eqnarray}
Q_{1,2,1}' &=& 
\frac{1}{n}\E \Big\{ \sum_{i=1}^n\cov\big[\tildeg\big((\bX_i, \bZ_i), (\bX_{N(i)}, \bZ_{N(i)})\big) ,   \phi(\bX_i, \bZ_i)   \mid \bfZ\big] \Big\} \cr
&& +
\frac{1}{n}\E \Big\{ \sum_{(i,j)\in \lbr n \rbr\times\lbr n \rbr: \, i,j, N(i) \,\text{distinct}} \cov\big[\tildeg\big((\bX_i, \bZ_i), (\bX_{N(i)}, \bZ_{N(i)})\big) ,   \phi(\bX_j, \bZ_j)   \mid \bfZ\big] \Big\} \cr
&& +
\frac{1}{n}\E \Big\{ \sum_{(i,j)\in \lbr n \rbr\times\lbr n \rbr: \, j = N(i)} \cov\big[\tildeg\big((\bX_i, \bZ_i), (\bX_{N(i)}, \bZ_{N(i)})\big) ,   \phi(\bX_j, \bZ_j)   \mid \bfZ\big] \Big\} \cr
&=:&	V_1 + V_2 + V_3. \label{eq:V1-V3}
\end{eqnarray}
In the following, Lemmas \ref{lemA:lim6}--\ref{lemA:lim8} derive the limits of $V_1$--$V_3$ respectively, and Lemma \ref{lemA:lim9} derives the limit of $Q_{1,2,2}'$. 

\begin{lemmaA} \label{lemA:lim6}
The limit of $V_1$ in \eqref{eq:V1-V3} is 
\begin{eqnarray*}
&&\lim_{n\to \infty}
\E \Big\{\cov\big[\tildeg\big((\bX_1, \bZ_1), (\bX_{N(1)}, \bZ_{N(1)})\big) ,   \phi(\bX_1, \bZ_1)   \mid \bfZ\big] \Big\} \cr
&=&
\E\Big[\cov\Big\{ \E\big( F_Y(Y_1 \wedge \tY_1) \biggiven \bX_1, \bZ_1\big) , \E\big( F_Y(Y_1 \wedge \barY_1) \biggiven \bX_1, \bZ_1\big) \Biggiven \bZ_1\Big\}\Big].
\end{eqnarray*}
\end{lemmaA}
\begin{proof}
\begin{eqnarray*}
&&\E \Big\{\cov\big[\tildeg\big((\bX_1, \bZ_1), (\bX_{N(1)}, \bZ_{N(1)})\big) ,   \phi(\bX_1, \bZ_1)   \mid \bfZ\big] \Big\}  \cr
&=& \E\Big\{ \E\big[\tildeg\big((\bX_1, \bZ_1), (\bX_{N(1)}, \bZ_{N(1)})\big) \cdot   \phi(\bX_1, \bZ_1) \biggiven \bfZ \big] \} \cr
&&-
\E\Big\{
\E\big[\tildeg\big((\bX_1, \bZ_1), (\bX_{N(1)}, \bZ_{N(1)})\big) \biggiven \bfZ \big] \cdot 
\E\big[ \phi(\bX_1, \bZ_1) \biggiven \bZ \big]
\Big\} \cr
&=:& R_1 + R_2.
\end{eqnarray*}
For the term $R_1$,
\begin{eqnarray}
R_1 &=& \E\Big\{\tildeg\big((\bX_1, \bZ_1), (\bX_{N(1)}, \bZ_{N(1)})\big) \cdot   \phi(\bX_1, \bZ_1) \Big\} \cr
&=&
\E \Big\{\E\big[\tildeg\big((\bX_1, \bZ_1), (\bX_{N(1)}, \bZ_{N(1)})\big) \cdot   \phi(\bX_1, \bZ_1) \biggiven \bX_1, \bZ_1 ,\bZ_{N(1)}\big]\Big\} \cr
&=&
\E\Big\{\tildeg^\dagger(\bX_1, \bZ_1, \bZ_{N(1)}) \cdot   \phi(\bX_1, \bZ_1) \Big\}, 
\label{R1_1}
\end{eqnarray}
where 
\begin{eqnarray*}
\tildeg^\dagger\big((\bx_1,\bz_1), \bz_2\big)& =& \int \int F_Y(u\wedge v) 
\d \tmu_{(\bx_1,\bz_1)}(u) \d \tmu_{\bz_2}(v) \cr
&=& 
\int \tildeg\big((\bx_1,\bz_1), (\bx,\bz_2)\big) 
\d \mu_{\bX=\bx \mid \bZ = \bz_2}(\bx),
\end{eqnarray*}
is defined in \eqref{eq:lemA:aux-2-7:1}.     It is straightforward to verify that 
\begin{eqnarray*}
\E\big[\tildeg\big((\bX_1, \bZ_1), (\bX_{N(1)}, \bZ_{N(1)})\big) \biggiven \bX_1, \bZ_1 ,\bZ_{N(1)}\big]
&=&
\tildeg^\dagger\big((\bX_1, \bZ_1), \bZ_{N(1)}\big), \cr
\text{and} \qquad \E\big( F_Y(Y_1 \wedge \tY_1) \biggiven \bX_1, \bZ_1\big)  &=&
\tildeg^\dagger\big((\bX_1, \bZ_1), \bZ_1\big).
\end{eqnarray*}
By Lemma \ref{lemA:aux-2-7}, $\tildeg^\dagger\big((\bX_1, \bZ_1), \bZ_{N(1)}\big) - \tildeg^\dagger\big((\bX_1, \bZ_1), \bZ_1\big) \conP 0$. Since both $\tildeg^\dagger$ and $\phi$ are bounded, we then have
\begin{eqnarray*}
\lim_{n \to \infty} R_1 &=& \E\Big\{\tildeg^\dagger\big((\bX_1, \bZ_1), \bZ_1\big) \cdot   \phi(\bX_1, \bZ_1) \Big\} \cr
&=&
\E\Big\{\E\big( F_Y(Y_1 \wedge \tY_1) \biggiven \bX_1, \bZ_1\big) \cdot \E\big( F_Y(Y_1 \wedge \barY_1) \biggiven \bX_1, \bZ_1\big) \Big\}.
\end{eqnarray*}

For the term $R_2$, define
\begin{eqnarray}
\tildeg^\ddagger(\bz_1, \bz_2) &: =& \int \int F_Y(u\wedge v) 
\d \tmu_{\bz_1}(u) \d \tmu_{\bz_2}(v) \cr
&=& 
\int \int \tildeg\big(\bx_1,\bz_1, \bx_2,\bz_2\big) 
\d \mu_{\bX=\bx_1 \mid \bZ = \bz_1}(\bx_1)
\d \mu_{\bX=\bx_2 \mid \bZ = \bz_2}(\bx_2). \label{eq:lemA:lim6:2}
\end{eqnarray}
It is straightforward that
\begin{eqnarray}
\E\big[\tildeg\big((\bX_1, \bZ_1), (\bX_{N(1)}, \bZ_{N(1)})\big) \biggiven \bfZ \big] 
&=& 
\tildeg^\ddagger(\bZ_1, \bZ_{N(1)}), \cr
\text{and} \qquad \E\big( F_Y(Y_1 \wedge \tY_1) \biggiven \bZ_1\big) 
&=&	\tildeg^\ddagger(\bZ_1, \bZ_1). \label{eq:lemA:lim6:3}
\end{eqnarray}
Also, define 
\begin{eqnarray}
\phi^\ddagger(\bz_1) &=:& \E\big\{F_Y(Y_1 \wedge \barY_1) \biggiven\bZ_1=\bz_1 \big\} \cr
&=&
\E\big\{\phi(\bX_1, \bZ_1) \biggiven \bZ_1 = \bz_1 \big\}. \label{eq:lemA:lim6:4}
\end{eqnarray}
Then 
\begin{eqnarray}
R_2 = \E\Big\{	\tildeg^\ddagger(\bZ_1, \bZ_{N(1)}) \cdot 		\phi^\ddagger(\bZ_1) \Big\}. \label{eq:lemA:lim6:5}
\end{eqnarray}
By Lemma \ref{lemA:aux-2-6} with $(\bX, \bZ)$ replaced by $\bZ$, we have
\begin{eqnarray}
\tildeg^\ddagger(\bZ_1, \bZ_{N(1)}) - \tildeg^\ddagger(\bZ_1, \bZ_1) \conP 0.
\label{eq:lemA:lim6:6}
\end{eqnarray}
Therefore,
\begin{eqnarray*}
\lim_{n \to \infty} R_2
&=&
\E\Big\{	\tildeg^\ddagger(\bZ_1, \bZ_1) \cdot 		\phi^\ddagger(\bZ_1) \Big\} \cr
&=&
\E\Big\{\E\big( F_Y(Y_1 \wedge \tY_1) \biggiven \bZ_1\big) \cdot \E\big( F_Y(Y_1 \wedge \barY_1) \biggiven \bZ_1\big) \Big\}.
\end{eqnarray*}
Combining the above two limits completes the proof.
\end{proof}

\begin{lemmaA} \label{lemA:lim7}
The limit of $V_2$ in \eqref{eq:V1-V3} is

\begin{eqnarray*}
\lim_{n \to \infty}	\frac{1}{n}\E \Big\{ \sum_{(i,j)\in \lbr n \rbr\times\lbr n \rbr: \, i,j, N(i) \,\text{distinct}} \cov\big[\tildeg\big((\bX_i, \bZ_i), (\bX_{N(i)}, \bZ_{N(i)})\big) ,   \phi(\bX_j, \bZ_j)   \mid \bfZ\big] \Big\} =0.
\end{eqnarray*}
\end{lemmaA}

\begin{proof}
We have
\begin{eqnarray*}
&&\frac{1}{n}\E \Big\{ \sum_{(i,j)\in \lbr n \rbr\times\lbr n \rbr: \, i,j, N(i) \,\text{distinct}} \cov\big[\tildeg\big((\bX_i, \bZ_i), (\bX_{N(i)}, \bZ_{N(i)})\big) ,   \phi(\bX_j, \bZ_j)   \mid \bfZ\big] \Big\} \cr
&=&
(n-1)\cdot
\E \Big\{  \cov\big[\tildeg\big((\bX_1, \bZ_1), (\bX_{N(1)}, \bZ_{N(1)})\big) ,   \phi(\bX_2, \bZ_2)   \mid \bfZ\big] \cdot \Ind\big(N(1) \neq 2\big) \Big\} .
\end{eqnarray*}
For $\bfZ$ such that $N(1) \neq 2$, the random variables $\tildeg\big((\bX_1, \bZ_1), (\bX_{N(1)}, \bZ_{N(1)})\big)$ and $\phi(\bX_2, \bZ_2) $ conditional on $\bfZ$ are independently distributed, and thus
\begin{eqnarray*}
\cov\big[\tildeg\big((\bX_1, \bZ_1), (\bX_{N(1)}, \bZ_{N(1)})\big) ,   \phi(\bX_2, \bZ_2)   \mid \bfZ\big] \cdot \Ind\big(N(1) \neq 2\big) = 0.
\end{eqnarray*}
Then $V_2 = 0$. This completes the proof. 
\end{proof}

\begin{lemmaA} \label{lemA:lim8}
The limit of $V_3$ in \eqref{eq:V1-V3} is
\begin{eqnarray*}
&&\lim_{n\to \infty}
\E \Big\{\cov\big[\tildeg\big((\bX_1, \bZ_1), (\bX_{N(1)}, \bZ_{N(1)})\big) ,   \phi(\bX_{N(1)}, \bZ_{N(1)})   \mid \bfZ\big] \Big\} \cr
&=&
\E\Big[\cov\Big\{ \E\big( F_Y(Y_1 \wedge \tY_1) \biggiven \bX_1, \bZ_1\big) , \E\big( F_Y(Y_1 \wedge \barY_1) \biggiven \bX_1, \bZ_1\big) \Biggiven \bZ_1\Big\}\Big].
\end{eqnarray*}
\end{lemmaA}

\begin{proof}
\begin{eqnarray*}
&&\E \Big\{\cov\big[\tildeg\big((\bX_1, \bZ_1), (\bX_{N(1)}, \bZ_{N(1)})\big) ,   \phi(\bX_{N(1)}, \bZ_{N(1)})   \mid \bfZ\big] \Big\}  \cr
&=& \E\Big\{ \E\big[\tildeg\big((\bX_1, \bZ_1), (\bX_{N(1)}, \bZ_{N(1)})\big) \cdot   \phi(\bX_{N(1)}, \bZ_{N(1)}) \biggiven \bfZ \big] \} \cr
&&-
\E\Big\{
\E\big[\tildeg\big((\bX_1, \bZ_1), (\bX_{N(1)}, \bZ_{N(1)})\big) \biggiven \bfZ \big] \cdot 
\E\big[ \phi(\bX_{N(1)}, \bZ_{N(1)}) \biggiven \bZ \big]
\Big\} \cr
&=:& R_1 + R_2.
\end{eqnarray*}
For the term $R_1$, similar to \eqref{R1_1}, we have
\begin{eqnarray*}
R_1 &=& \E\Big\{\tildeg\big((\bX_1, \bZ_1), (\bX_{N(1)}, \bZ_{N(1)})\big) \cdot   \phi(\bX_{N(1)}, \bZ_{N(1)}) \Big\} \cr
&=&
\E \Big\{\E\big[\tildeg\big((\bX_1, \bZ_1), (\bX_{N(1)}, \bZ_{N(1)})\big) \cdot   \phi(\bX_{N(1)}, \bZ_{N(1)}) \biggiven \bZ_1 , \bX_{N(1)}, \bZ_{N(1)}\big]\Big\} \cr
&=&
\E\Big\{\tildeg^\dagger\big((\bX_{N(1)}, \bZ_{N(1)}), \bZ_1\big) \cdot   \phi(\bX_{N(1)}, \bZ_{N(1)}) \Big\}.
\end{eqnarray*}
By Lemma \ref{lemA:aux-2-7}, $\tildeg^\dagger\big((\bX_{N(1)}, \bZ_{N(1)}), \bZ_1\big) - \tildeg^\dagger\big((\bX_1, \bZ_1), \bZ_1\big) \conP 0$. Since $\tildeg^\dagger$ and $\phi$ are bounded, we have
\begin{eqnarray*}
&&\E\Big\{\tildeg^\dagger\big((\bX_{N(1)}, \bZ_{N(1)}), \bZ_1\big) \cdot   \phi(\bX_{N(1)}, \bZ_{N(1)}) \Big\} \cr
&=&
\E\Big\{\tildeg^\dagger\big((\bX_{N(1)}, \bZ_{N(1)}), \bZ_{N(1)}\big) \cdot   \phi(\bX_{N(1)}, \bZ_{N(1)}) \Big\} +o(1).
\end{eqnarray*}
By \eqref{eq:lemA:aux-1-8:4}, we have $(\bX_{N(1)}, \bZ_{N(1)}) \conD (\bX_1, \bZ_1)$. This combined with the facts that both $(\bX_{N(1)}, \bZ_{N(1)})$ and $(\bX_1, \bZ_1)$ are absolutely continuous yields that
\begin{eqnarray*}
&&	\E\Big\{\tildeg^\dagger\big((\bX_{N(1)}, \bZ_{N(1)}), \bZ_{N(1)}\big) \cdot   \phi(\bX_{N(1)}, \bZ_{N(1)}) \Big\} \cr
&=&
\E\Big\{\tildeg^\dagger\big((\bX_1, \bZ_1), \bZ_1\big) \cdot   \phi(\bX_1, \bZ_1) \Big\}  +o(1).
\end{eqnarray*}
Therefore,
\begin{eqnarray*}
\lim_{n \to \infty} R_1 &=& \E\Big\{\tildeg^\dagger\big((\bX_1, \bZ_1), \bZ_1\big) \cdot   \phi(\bX_1, \bZ_1) \Big\} \cr
&=&
\E\Big\{\E\big( F_Y(Y_1 \wedge \tY_1) \biggiven \bX_1, \bZ_1\big) \cdot \E\big( F_Y(Y_1 \wedge \barY_1) \biggiven \bX_1, \bZ_1\big) \Big\}.
\end{eqnarray*}

For the term $R_2$, similar to \eqref{eq:lemA:lim6:5}, we have
\begin{eqnarray*}
R_2 = \E\Big\{	\tildeg^\ddagger(\bZ_1, \bZ_{N(1)}) \cdot 		\phi^\ddagger(\bZ_{N(1)}) \Big\}.
\end{eqnarray*}
By \eqref{eq:lemA:lim6:6},
$\tildeg^\ddagger(\bZ_1, \bZ_{N(1)}) - \tildeg^\ddagger(\bZ_1, \bZ_1) \conP 0$.
Since function $\phi^\ddagger$ is measurable and $\|\bZ_{N(1)} -\bZ_1\| \conas 0$, Lemma 11.7 of \cite{azadkia2019simple} implies that
$\phi^\ddagger(\bZ_{N(1)}) - \phi^\ddagger(\bZ_1) \conP 0$. It follows that
\begin{eqnarray*}
\lim_{n \to \infty} R_2
&=&
\E\Big\{	\tildeg^\ddagger(\bZ_1, \bZ_1) \cdot 		\phi^\ddagger(\bZ_1) \Big\} \cr
&=&
\E\Big\{\E\big( F_Y(Y_1 \wedge \tY_1) \biggiven \bZ_1\big) \cdot \E\big( F_Y(Y_1 \wedge \barY_1) \biggiven \bZ_1\big) \Big\}.
\end{eqnarray*}
Combining the above two limits completes the proof.
\end{proof}

\begin{lemmaA} \label{lemA:lim9}
The limit of $Q_{1,2,2}'$ in \eqref{eq:Q'_12} is
\begin{eqnarray*}
&&	\lim_{n \to \infty}\frac{1}{n} \cov\Big\{
\E\Big[\sum_{i=1}^n  \tildeg\big((\bX_i, \bZ_i), (\bX_{N(i)}, \bZ_{N(i)})\big)   \Biggiven \bfZ\Big],
\E\Big[\sum_{i=1}^n  \phi(\bX_i, \bZ_i)  \Biggiven \bfZ\Big]
\Big\}  \cr
&=&
\cov\Big\{\E\big( F_Y(Y_1 \wedge \tY_1 )\mid \bZ_1 \big), \E\big( F_Y(Y_1 \wedge \barY_1 )\mid \bZ_1 \big)\Big\}.
\end{eqnarray*}
\end{lemmaA}

\begin{proof}
Following the notation in \eqref{eq:lemA:lim6:3} and \eqref{eq:lemA:lim6:4}, we have
\begin{eqnarray*}
\E\big\{\tildeg\big((\bX_i, \bZ_i), (\bX_{N(i)}, \bZ_{N(i)})\big)  \biggiven \bfZ \big\} = 
\tildeg^\ddagger(\bZ_i, \bZ_{N(i)}) =
\E\Big( F_Y(Y_i \wedge Y_{N(i)})\mid \bfZ \Big)
\end{eqnarray*}
and 
\begin{eqnarray*}
\E\big\{ \phi(\bX_i, \bZ_i)   \biggiven \bfZ\big\} = \phi^\ddagger(\bZ_i) = \E\big( F_Y(Y_i \wedge \barY_i )\biggiven \bZ_i \big).
\end{eqnarray*}
Therefore,
\begin{eqnarray*}
&&\frac{1}{n} \cov\Big\{
\E\Big[\sum_{i=1}^n  \tildeg\big((\bX_i, \bZ_i), (\bX_{N(i)}, \bZ_{N(i)})\big)   \Biggiven \bfZ\Big],
\E\Big[\sum_{i=1}^n  \phi(\bX_i, \bZ_i)  \Biggiven \bfZ\Big]
\Big\} \cr
&=&
\frac{1}{n}  \cov\Big\{ \sum_{i=1}^n  \tildeg^\ddagger(\bZ_i, \bZ_{N(i)}),
\sum_{i=1}^n   \phi^\ddagger(\bZ_i) 
\Big\} .
\end{eqnarray*}
By the proof of Lemma C.1 in \cite{Lin_Han_2025_CLT}, we have
\begin{eqnarray}
\lim_{n \to \infty} \frac{1}{n} \var\Big[  \sum_{i=1}^n 
\E\Big( F_Y(Y_i \wedge Y_{N(i)})\Biggiven \bfZ \Big) 
- \sum_{i=1}^n  	\E\Big( F_Y(Y_i \wedge \tY_i ) \Biggiven \bZ_i \Big)   \Big] = 0, \label{eq:lemA:lim9:1}
\end{eqnarray}
that is,
\begin{eqnarray*}
\lim_{n \to \infty} \frac{1}{n} \var \Big[  \sum_{i=1}^n  \tildeg^\ddagger(\bZ_i, \bZ_{N(i)}) -  \sum_{i=1}^n  \tildeg^\ddagger(\bZ_i, \bZ_i) \Big] =0 .
\end{eqnarray*}
Therefore,
\begin{eqnarray*}
\frac{1}{n}  \cov\Big\{ \sum_{i=1}^n  \tildeg^\ddagger(\bZ_i, \bZ_{N(i)}),
\sum_{i=1}^n   \phi^\ddagger(\bZ_i) 
\Big\} 
&=&
\frac{1}{n}  \cov\Big\{ \sum_{i=1}^n  \tildeg^\ddagger(\bZ_i, \bZ_i),
\sum_{i=1}^n   \phi^\ddagger(\bZ_i) 
\Big\} +o(1) \cr
&=&
\cov\Big\{ \tildeg^\ddagger(\bZ_1, \bZ_1) , \phi^\ddagger(\bZ_1)  \Big\} +o(1).
\end{eqnarray*}
The proof is completed by noting that
\begin{eqnarray*}
\cov\Big\{ \tildeg^\ddagger(\bZ_1, \bZ_1) , \phi^\ddagger(\bZ_1)  \Big\} 
&=&
\cov\Big\{\E\big( F_Y(Y_1 \wedge \tY_1 )\mid \bZ_1 \big), \E\big( F_Y(Y_1 \wedge \barY_1 )\mid \bZ_1 \big)\Big\}.
\end{eqnarray*}
\end{proof}

\subsubsection{The limits of $Q_2$ and $Q_3$} \label{secA:proof-stepII:3}
Since $Q_2$ and $Q_3$ are structurally similar, we present only the derivation of the limit of $Q_2$. The limit of $Q_3$ can then be obtained analogously by replacing $\bZ$ with $(\bX,\bZ)$ and replacing $h_1$ with $h_2$.
Decompose $Q_2$ into two terms,
\begin{eqnarray}
Q_2&=&\frac{1}{n}\cov\Big(\sum_{i=1}^n F_Y(Y_i \wedge Y_{N(i)}) , \sum_{i=1}^nh_1(Y_i)\Big)\cr
&=&
\frac{1}{n}	
\E \Big\{\cov\Big[\sum_{i=1}^n F_Y(Y_i \wedge Y_{N(i)}), \sum_{i=1}^n\ h_1(Y_i) \Biggiven \bfZ\Big] \Big\} \cr
&& +
\frac{1}{n}\cov\Big[ \E\Big(\sum_{i=1}^n F_Y(Y_i \wedge Y_{N(i)})\mid  \bfZ \Big),
\E\Big(\sum_{i=1}^n h_1(Y_i)\mid \bfZ \Big)
\Big]  \cr
&=:& Q_{2,1} + Q_{2,2}. \label{eq:Q2}
\end{eqnarray}
In the following, Lemmas \ref{lemA:lim10} and \ref{lemA:lim11} present the limits of $Q_{2,1}$ and $Q_{2,2}$ respectively. Whenever no confusion arises, we write function $h_1(\cdot)$ simply as $h(\cdot)$.

\begin{lemmaA} \label{lemA:lim10}
The limit of $Q_{2,1}$ in \eqref{eq:Q2} is
\begin{eqnarray*}
&&\lim_{n\to \infty}	\frac{1}{n}	
\E \Big\{\cov\Big[\sum_{i=1}^n F_Y(Y_i \wedge Y_{N(i)}), \sum_{i=1}^n\ h(Y_i) \Biggiven  \bfZ\Big] \Big\} \cr
&=& 2\cdot \E \Big[ \cov\big\{ F_Y(Y_1 \wedge \tY_1) ,h(Y_1)\biggiven \bZ_1  \big\}\Big].
\end{eqnarray*}
\end{lemmaA}

\begin{proof}	
\begin{eqnarray*}
Q_{2,1} &=& 
\frac{1}{n}\E \Big\{ \sum_{i=1}^n\cov\big[F_Y(Y_i \wedge Y_{N(i)}), h(Y_i)   \mid \bfZ\big]  \Big\} \cr
&& +
\frac{1}{n}\E \Big\{ \sum_{(i,j)\in \lbr n \rbr\times\lbr n \rbr: \, i,j, N(i) \,\text{distinct}} \cov\big[F_Y(Y_i \wedge Y_{N(i)}), h(Y_j)   \mid \bfZ\big] \Big\} \cr
&& +
\frac{1}{n}\E \Big\{ \sum_{(i,j)\in \lbr n \rbr\times\lbr n \rbr: \, j = N(i)} \cov\big[F_Y(Y_i \wedge Y_{N(i)}), h(Y_j)   \mid \bfZ\big] \Big\} \cr
&=:&	Q_{2,1,1} + Q_{2,1,2}+ Q_{2,1,3}.
\end{eqnarray*}
We derive the limits of $Q_{2,1,1}$, $Q_{2,1,2}$, $Q_{2,1,3}$ separately as follows. 

\textbf{Case (i): $	Q_{2,1,1}$.} We have
\begin{eqnarray*}
Q_{2,1,1}
&=&
\E \Big\{ \cov\big[F_Y(Y_1 \wedge Y_{N(1)}), h(Y_1)   \biggiven \bfZ\big]  \Big\} \cr
&=&
\E \Big\{\E\big[F_Y(Y_1 \wedge Y_{N(1)})\cdot h(Y_1) \biggiven \bfZ\big]\}
-\E\Big\{\E\big[F_Y(Y_1 \wedge Y_{N(1)}) \biggiven \bfZ\big] , \E\big[ h(Y_1) \biggiven \bfZ\big]\Big\} \cr
&=:& R_1 + R_2.
\end{eqnarray*}
For the term $R_1$, define function $g^\star: \mathbbR^{2q} \mapsto \mathbbR$,
\begin{eqnarray}
g^\star(\bz_1, \bz_2) := \int F_Y(u \wedge v) \cdot h(u) \d \tmu_{\bz_1}(u) \d \tmu_{\bz_2}(v). \label{eq:lemA:lim10:1}
\end{eqnarray}
Then 
\begin{eqnarray*}
R_1 = \E\big\{ 	g^\star(\bZ_1, \bZ_{N(1)}) \big\}.
\end{eqnarray*}
Since $Y$ is continuous, then function $h(t) = \P(\barY\wedge Y>t )$ is bounded and continuous.
Since both $F_Y(u \wedge v)$ and $h(u)$ are bounded and continuous, it follows from Lemma A.1 of \cite{Gao_Li_2024} that the bivariate function $F_Y(u \wedge v)h(u)$ can be approximated by a simple function of the form $q(u,v) = \sum_{j=1}^m c_j \Ind_{B_j}(u)\Ind_{C_j}(v)$, similarly to the arguments in Lemma \ref{lemA:aux-2-1}. Thus, analogous to the proof of Lemma \ref{lemA:aux-2-3}, by the fact that $\|\bZ_{N(1)} -\bZ_1 \|\conas 0$, we can show that
\begin{eqnarray}
g^\star(\bZ_1, \bZ_{N(1)}) - g^\star(\bZ_1, \bZ_1)  \conP 0.\label{eq:lemA:lim10:2}
\end{eqnarray}
Since $g^\star$ is bounded, then
\begin{eqnarray*}
\lim_{n \to \infty} R_1  \ = \		\lim_{n \to \infty}\E\big\{ 	g^\star(\bZ_1, \bZ_{N(1)}) \big\} \ = \ \E\big\{ 	g^\star(\bZ_1, \bZ_1) \big\} \  =\
\E \Big\{\E\big[F_Y(Y_1 \wedge \tY_1)\cdot h(Y_1) \biggiven \bfZ\big]\}. 
\end{eqnarray*}

For the term $R_2$,
invoke the function $\tildeg^\ddagger$ defined in \eqref{eq:lemA:lim6:2}:
\begin{eqnarray*}
\tildeg^\ddagger(\bz_1, \bz_2) &=& \int \int F_Y(u\wedge v) 
\d \tmu_{\bz_1}(u) \d \tmu_{\bz_2}(v).
\end{eqnarray*}
Also, define function $\tilde{h}(\bz) = \E\big[ h(Y) \biggiven \bZ = \bz\big]$.
Then, 
\begin{eqnarray*}
R_2 = \E\Big\{\tildeg^\ddagger(\bZ_1, \bZ_{N(1)})  \cdot \tilde{h}(\bZ_1)\Big\}.
\end{eqnarray*}
By \eqref{eq:lemA:lim6:6}, we have
\begin{eqnarray*}
\tildeg^\ddagger(\bZ_1, \bZ_{N(1)}) - \tildeg^\ddagger(\bZ_1, \bZ_1) \conP 0. 
\end{eqnarray*}
Therefore, 
\begin{eqnarray*}
\lim_{n \to \infty} R_2  &=& 		\lim_{n \to \infty}\E\Big\{\tildeg^\ddagger(\bZ_1, \bZ_{N(1)})  \cdot \tilde{h}(\bZ_1)\Big\}
\ =\
\E\Big\{\tildeg^\ddagger(\bZ_1, \bZ_1)  \cdot \tilde{h}(\bZ_1)\Big\} \cr
&=&
\E\Big\{\E\big[F_Y(Y_1 \wedge \tY_1) \biggiven \bZ_1\big] , \E\big[ h(Y_1) \biggiven \bZ_1\big]\Big\}.
\end{eqnarray*}
Combining the above two results yields that
\begin{eqnarray*}
\lim_{n \to \infty} Q_{2,1,1} &=&
\E \Big\{\E\big[F_Y(Y_1 \wedge \tY_1)\cdot h(Y_1) \biggiven \bfZ\big]\}-
\E\Big\{\E\big[F_Y(Y_1 \wedge \tY_1) \biggiven \bZ_1\big] , \E\big[ h(Y_1) \biggiven \bZ_1\big]\Big\} \cr
&=&
\E \Big[ \cov\big\{ F_Y(Y_1 \wedge \tY_1) ,h(Y_1)\biggiven \bZ_1  \big\}\Big].
\end{eqnarray*}

\textbf{Case (ii): $	Q_{2,1,2}$.} We have
\begin{eqnarray*}
Q_{2,1,2} &=& 	\frac{1}{n}\E \Big\{ \sum_{(i,j)\in \lbr n \rbr\times\lbr n \rbr: \, i,j, N(i) \,\text{distinct}} \cov\big[F_Y(Y_i \wedge Y_{N(i)}), h(Y_j)   \mid \bfZ\big] \Big\} \cr
&=&
(n-1)\cdot \E \Big\{ \cov\big[F_Y(Y_1 \wedge Y_{N(1)}), h(Y_2)   \mid \bfZ\big] \cdot \Ind\big(N(1)\neq 2\big)\Big\}.
\end{eqnarray*}
For $\bfZ$ such that $N(1) \neq 2$, the random variables $Y_1$, $Y_2$, and $Y_{N(1)}$ conditional on $\bfZ$ are mutually independent. Therefore, 
\begin{eqnarray*}
\cov\big[F_Y(Y_1 \wedge Y_{N(1)}), h(Y_2)   \mid \bfZ\big] \cdot \Ind\big(N(1)\neq 2\big) =0.
\end{eqnarray*}
It follows that $Q_{2,1,2} =0$. 

\textbf{Case (iii): $	Q_{2,1,3}$.} We have
\begin{eqnarray*}
Q_{2,1,3} &=& 	\frac{1}{n}\E \Big\{ \sum_{(i,j)\in \lbr n \rbr\times\lbr n \rbr: \, j = N(i)} \cov\big[F_Y(Y_i \wedge Y_{N(i)}), h(Y_j)   \mid \bfZ\big] \Big\} \cr
&=&
\E \Big\{ \cov\big[F_Y(Y_1 \wedge Y_{N(1)}), h(Y_{N(1)})   \biggiven \bfZ\big]  \Big\} \cr
&=&
\E \Big\{\E\big[F_Y(Y_1 \wedge Y_{N(1)})\cdot h(Y_{N(1)})  \biggiven \bfZ\big]\}
\cr
&&
-\E\Big\{\E\big[F_Y(Y_1 \wedge Y_{N(1)}) \biggiven \bfZ\big] , \E\big[ h(Y_{N(1)})  \biggiven \bfZ\big]\Big\} \cr
&=:& R_1 + R_2.
\end{eqnarray*}
For the term $R_1$, following the notation in \eqref{eq:lemA:lim10:1}, we have
\begin{eqnarray*}
R_1 = \E\big\{ 	g^\star(\bZ_{N(1)}, \bZ_1) \big\}.
\end{eqnarray*}
Similar to the proof of \eqref{eq:lemA:lim10:2}, we can also show that 
\begin{eqnarray*}
g^\star(\bZ_{N(1)}, \bZ_1) - g^\star(\bZ_1, \bZ_1)  \conP 0.
\end{eqnarray*}
Thus,
\begin{eqnarray*}
\lim_{n \to \infty} R_1  \ = \ 		\lim_{n \to \infty}\E\big\{ 	g^\star(\bZ_{N(1)}, \bZ_1)  \big\} \ = \ \E\big\{ 	g^\star(\bZ_1, \bZ_1) \big\} \ = \ 
\E \Big\{\E\big[F_Y(Y_1 \wedge \tY_1)\cdot h(Y_1) \biggiven \bfZ\big]\}. 
\end{eqnarray*}

For the term $R_2$,
following the previous notation, we have
\begin{eqnarray*}
R_2 = \E\Big\{\tildeg^\ddagger(\bZ_1, \bZ_{N(1)})  \cdot \tilde{h}(\bZ_{N(1)})\Big\}.
\end{eqnarray*}
Note that $\tilde{h}$ is a measurable function and 
$\|\bZ_{N(1)} -\bZ_1 \|\conas 0$. By Lemma 11.7 in \cite{azadkia2019simple}, we have $\tilde{h}(\bZ_{N(1)}) - \tilde{h}(\bZ_1) \conP 0$. Also, 	by \eqref{eq:lemA:lim6:6}, we have
$
\tildeg^\ddagger(\bZ_1, \bZ_{N(1)}) - \tildeg^\ddagger(\bZ_1, \bZ_1) \conP 0$.
It follows that
\begin{eqnarray*}
\lim_{n \to \infty} R_2  &=& 		\lim_{n \to \infty}\E\Big\{\tildeg^\ddagger(\bZ_1, \bZ_{N(1)})  \cdot \tilde{h}(\bZ_{N(1)})\Big\}
=
\E\Big\{\tildeg^\ddagger(\bZ_1, \bZ_1)  \cdot \tilde{h}(\bZ_1)\Big\} \cr
&=&
\E\Big\{\E\big[F_Y(Y_1 \wedge \tY_1) \biggiven \bZ_1\big] , \E\big[ h(Y_1) \biggiven \bZ_1\big]\Big\}.
\end{eqnarray*}
Combining the above two results yields that
\begin{eqnarray*}
\lim_{n \to \infty} Q_{2,1,3} &=&
\E \Big\{\E\big[F_Y(Y_1 \wedge \tY_1)\cdot h(Y_1) \biggiven \bfZ\big]\}-
\E\Big\{\E\big[F_Y(Y_1 \wedge \tY_1) \biggiven \bZ_1\big] , \E\big[ h(Y_1) \biggiven \bZ_1\big]\Big\} \cr
&=&
\E \Big[ \cov\big\{ F_Y(Y_1 \wedge \tY_1) ,h(Y_1)\biggiven \bZ_1  \big\}\Big].
\end{eqnarray*}

In summary, combining the three cases completes the proof of this lemma.
\end{proof}

\begin{lemmaA} \label{lemA:lim11}
The limit of $Q_{2,2}$  in \eqref{eq:Q2} is 
\begin{eqnarray*}
&&	\lim_{n\to \infty} 
\frac{1}{n}\cov\Big[ \E\Big(\sum_{i=1}^n F_Y(Y_i \wedge Y_{N(i)})\mid \bfZ \Big),
\E\Big(\sum_{i=1}^n h(Y_i)\mid \bfZ \Big)
\Big]  \cr
&=& \cov \Big\{\E\big[F_Y(Y_1 \wedge \tY_1)\mid \bZ_1 \big] , \E\big[h(Y_1) \mid \bZ_1\big]\Big\}.
\end{eqnarray*}
\end{lemmaA}

\begin{proof}
Recall in \eqref{eq:lemA:lim9:1}, we show that
\begin{eqnarray*}
\lim_{n \to \infty} \frac{1}{n} \var\Big[  \sum_{i=1}^n 
\E\Big( F_Y(Y_i \wedge Y_{N(i)})\Biggiven \bfZ \Big) 
- \sum_{i=1}^n  	\E\Big( F_Y(Y_i \wedge \tY_i ) \Biggiven \bZ_i \Big)   \Big] = 0.
\end{eqnarray*}
Therefore, 
\begin{eqnarray*}
\lim_{n\to \infty}  Q_{2,2} &=&
\lim_{n\to \infty} 
\frac{1}{n}\cov\Big[ \sum_{i=1}^n  \E\Big(F_Y(Y_i \wedge Y_{N(i)})\mid \bfZ \Big),
\sum_{i=1}^n  \E\Big(h(Y_i)\mid \bZ_i \Big)
\Big] \cr
&=&
\lim_{n\to \infty} 
\frac{1}{n}
\cov\Big[ \sum_{i=1}^n  \E\Big( F_Y(Y_i \wedge \tY_i ) \Biggiven \bZ_i \Big),
\sum_{i=1}^n  \E\Big(\sum_{i=1}^n h(Y_i)\mid \bZ_i \Big)
\Big]  \cr
&=&
\cov \Big\{\E\big[F_Y(Y_1 \wedge \tY_1)\mid \bZ_1 \big] , \E\big[h(Y_1) \mid \bZ_1\big]\Big\}.
\end{eqnarray*}
This completes the proof. 
\end{proof}

\subsubsection{Summary of results in Step (2)} \label{secA:proof-stepII:4}

Combining Lemmas \ref{lemA:lim1}--\ref{lemA:lim5} in Section \ref{secA:proof-stepII:1} yields that
\begin{eqnarray}
\lim_{n \to \infty} Q_{1,1} &=& \lim_{n \to \infty}	\frac{1}{n}	
\E \Big\{\cov\Big[\sum_{i=1}^n F_Y(Y_i \wedge Y_{N(i)}), \sum_{j=1}^n F_Y(Y_i \wedge Y_{M(i)}) \mid \bfX, \bfZ\Big] \Big\} \cr
&=&
4\cdot \E \Big\{\cov\big[F_Y(Y_1 \wedge \tY_1),  F_Y(Y_1 \wedge \barY_1)\mid \bX_1, \bZ_1 \big] \Big\}.  \label{eq:limQ11}
\end{eqnarray}
Combining Lemmas \ref{lemA:lim6}--\ref{lemA:lim9} in Section \ref{secA:proof-stepII:2} yields that
\begin{eqnarray}
\lim_{n \to \infty} Q_{1,2} &=& \lim_{n \to \infty}	\frac{1}{n}\cov\Big[ \sum_{i=1}^n  \E\Big(F_Y(Y_i \wedge Y_{N(i)}) \mid \bfX, \bfZ \Big),
\sum_{i=1}^n \E\Big( F_Y(Y_i \wedge Y_{M(i)})\mid \bfX, \bfZ \Big)
\Big]  \cr
&=&
2 \cdot  \E\Big[\cov\Big\{ \E\big( F_Y(Y_1 \wedge \tY_1) \biggiven \bX_1, \bZ_1\big) , \E\big( F_Y(Y_1 \wedge \barY_1) \biggiven \bX_1, \bZ_1\big) \Biggiven \bZ_1\Big\}\Big] \cr
&& 
+ \ 
\cov\Big\{\E\big( F_Y(Y_1 \wedge \tY_1 )\mid \bZ_1 \big), \E\big( F_Y(Y_1 \wedge \barY_1 )\mid \bZ_1 \big)\Big\}.  \label{eq:limQ12}
\end{eqnarray}
Combining Lemmas \ref{lemA:lim10}--\ref{lemA:lim11} in Section \ref{secA:proof-stepII:3} yields that
\begin{eqnarray*}
\lim_{n \to \infty} Q_{2} &=& \lim_{n \to \infty}	\frac{1}{n}\cov\Big(\sum_{i=1}^n F_Y(Y_i \wedge Y_{N(i)}) , \sum_{i=1}^nh_1(Y_i)\Big)\cr
&=&
2\cdot \E \Big[ \cov\big\{ F_Y(Y_1 \wedge \tY_1) ,h_1(Y_1)\biggiven \bZ_1  \big\}\Big]  \cr
&& + \cov \Big\{\E\big[F_Y(Y_1 \wedge \tY_1)\mid \bZ_1 \big] , \E\big[h_1(Y_1) \mid \bZ_1\big]\Big\}.
\end{eqnarray*}
Similarly, replacing $\bZ$ by $(\bX, \bZ)$ and $h_1(\cdot)$ by $h_2(\cdot)$ gives that
\begin{eqnarray*}
\lim_{n \to \infty} Q_{3} &=& \lim_{n \to \infty}	\frac{1}{n}\cov\Big(\sum_{i=1}^n F_Y(Y_i \wedge Y_{M(i)}) , \sum_{i=1}^nh_1(Y_i)\Big)\cr
&=&
2\cdot \E \Big[ \cov\big\{ F_Y(Y_1 \wedge \barY_1) ,h_2(Y_1)\biggiven \bX_1, \bZ_1  \big\}\Big]  \cr
&& + \cov \Big\{\E\big[F_Y(Y_1 \wedge \barY_1)\mid \bX_1, \bZ_1 \big] , \E\big[h_2(Y_1) \mid \bX_1,\bZ_1\big]\Big\}.
\end{eqnarray*}
Note that $Q_4 =  \cov\big\{h_1(Y), \ h_2(Y)\big\}$. 
Therefore, combining the above pieces together gives that
\begin{eqnarray*}
\lim_{n \to \infty}  n \cov(S_{1,n}, S_{2,n})
&=&
\lim_{n \to \infty}  (Q_1 + Q_2+Q_3+Q_4) \cr
&=&4\cdot \E \Big\{\cov\big[F_Y(Y_1 \wedge \tY_1),  F_Y(Y_1 \wedge \barY_1)\mid \bX_1, \bZ_1 \big] \Big\} \cr
&&+2 \cdot  \E\Big[\cov\Big\{ \E\big( F_Y(Y_1 \wedge \tY_1) \biggiven \bX_1, \bZ_1\big) , \E\big( F_Y(Y_1 \wedge \barY_1) \biggiven \bX_1, \bZ_1\big) \Biggiven \bZ_1\Big\}\Big] \cr
&& 
+
\cov\Big\{\E\big( F_Y(Y_1 \wedge \tY_1 )\mid \bZ_1 \big), \E\big( F_Y(Y_1 \wedge \barY_1 )\mid \bZ_1 \big)\Big\} \cr
&&
+2\cdot \E \Big[ \cov\big\{ F_Y(Y_1 \wedge \tY_1) ,h_1(Y_1)\biggiven \bZ_1  \big\}\Big]  \cr
&& + \cov \Big\{\E\big[F_Y(Y_1 \wedge \tY_1)\mid \bZ_1 \big] , \E\big[h_1(Y_1) \mid \bZ_1\big]\Big\} \cr
&&
+2\cdot \E \Big[ \cov\big\{ F_Y(Y_1 \wedge \barY_1) ,h_2(Y_1)\biggiven \bX_1, \bZ_1  \big\}\Big]  \cr
&& + \cov \Big\{\E\big[F_Y(Y_1 \wedge \barY_1)\mid \bX_1, \bZ_1 \big] , \E\big[h_2(Y_1) \mid \bX_1,\bZ_1\big]\Big\}\cr
&& + \cov\big\{h_1(Y), \ h_2(Y)\big\} \cr
&=:& 4 S_1 + 2 S_2 +S_3+ 2 S_4 + S_5 + 2 S_6 + S_7 + S_8.
\end{eqnarray*}
It can be checked that the terms $S_1$--$S_8$ above and $U_1$--$U_9$ in \eqref{eq:U1-U9} admit the following identities,
\begin{eqnarray*}
&&S_1 = U_1 -U_2,   \quad S_2 = U_2 - U_3, \quad S_3 = U_3-U_9, \quad S_4 = U_4-U_5, \cr
&& S_5 = U_5-U_9, \quad S_6 = U_6-U_7, \quad S_7 = U_7-U_9, \quad S_8 = U_8-U_9.
\end{eqnarray*}
Therefore, 
\begin{eqnarray*}
\lim_{n \to \infty}  n \cov(S_{1,n}, S_{2,n}) = 4 \, U_1 -2\, U_2 - U_3 +2 \, U_4 - U_5 + 2\, U_6 - U_7 + U_8 - 4 \, U_9.
\end{eqnarray*}
This completes the proof of Step (2).

\subsection{Proof of Step (3)} \label{secA:proof-stepIII}

The proof of Step (3) is primarily based on the normal approximation technique developed in \cite{MR2435859}, adapted to our $\tT_n^*$.  Before presenting the proof, we first introduce some necessary notation and definitions. 

Consider an index set $A = \{a_1, \dots, a_r\} \subseteq \lbr d \rbr$ with $|A| =r\geq 1$.
For a vector $\bw = (w_1, \dots, w_d) \in \mathbbR^d$ with $d >1$,
define $\bw^A = (w_{a_1}, \dots, w_{a_r})$ as the restriction of $\bw$ to the index set $A$.
Let $\bw_1, \dots, \bw_n \in \mathbbR^d$, with $n >1$. For each $i \in \lbr n \rbr$, define
\begin{eqnarray*}
N^A(i) := \arg\min_{j \in \lbr n \rbr} \|\bw^A_j-\bw^A_i\|.
\end{eqnarray*}
We call $N^A(i)$ the ``$A$-1-NN" of $i$. 

Let $f: (\mathbbR^d)^n \to \mathbbR$ be a function of the form
\begin{eqnarray}
f(\bw_1, \dots, \bw_n) = \sum_{\ell=1}^n f_{\ell}(\bw_1, \dots, \bw_n),
\label{eq:StepIII:1}
\end{eqnarray}
where, for each $\ell$, $f_{\ell}(\bw_1, \dots, \bw_n)$ is a function of only $\bw_\ell$ and its $A$-1-NN $\bw_{N^A(\ell)}$.

For $\bfw =(\bw_1, \dots, \bw_n) \in  (\mathbbR^d)^n$, and distinct $i,j \in \lbr n \rbr$, define
\begin{eqnarray*}
D^A_{\bfw}(i,j) := \#\{\ell: \|\bw^A_i - \bw^A_\ell\| < \|\bw^A_i - \bw^A_j\| \},
\end{eqnarray*}
which is the number of indices that is closer to $i$ than $j$ is to $i$. 
Given any $\bfw  \in (\mathbbR^d)^n$, let $G(\bfw)$ be the 
graphical rule 
(i.e., undirected graph) with vertex set  $\lbr n \rbr$, constructed as follows:
\begin{eqnarray}
&&	\quad \text{puts an edge between $i$ and $j$, if and only if} \cr
&& \text{there exists some $\ell \in \lbr n \rbr$, such that $D^A_{\bfw}(\ell,i) \leq 2$ and $D^A_{\bfw}(\ell,j) \leq 2$.} \label{eq:StepIII:2}
\end{eqnarray}

Having introduced these definitions, we proceed to state Lemma~\ref{lemA:sym_interac}.
\begin{lemmaA} \label{lemA:sym_interac}
$G$ in \eqref{eq:StepIII:2} is a symmetric interaction on $f$ in \eqref{eq:StepIII:1}, based on the following definitions.

1. Definition of ``symmetric rule" (page 1588 of \cite{MR2435859}): for any permutation $\pi$ on $\lbr n \rbr$ and any $(\bw_1,\dots, \bw_n) \in (\mathbbR^d)^n$, the set of edges in $G(\bw_{\pi(1)},\dots, \bw_{\pi(n)})$ is exactly
\begin{eqnarray*}
\Big\{(\pi(i), \pi(j)): (i,j)\in G(\bw_1, \dots \bw_n)\Big\}.
\end{eqnarray*}

2. Definition of ``interaction rule" (page 1589 of \cite{MR2435859}): for any choices of distinct $\bfw, \bfw' \in (\mathbbR^d)^n $ and $i,j \in\lbr n \rbr$, 
\begin{eqnarray*}
&& \text{$(i,j)$ is not an edge in the graphs $G(\bfw), G(\bfw^i), G(\bfw^j), G(\bfw^{ij})$,} \cr
&&\quad \text{implies that } \quad f(\bfw) - f(\bfw^i) =  f(\bfw^j) - f(\bfw^{ij}).
\end{eqnarray*}
Here, for each $i \in \lbr n \rbr$, $\bfw^i$ denotes the vector obtained by replacing $\bw_i$ with $\bw'_i$ in $\bfw$. For two distinct $i,j \in \lbr n \rbr$, $\bfw^{ij}$ denotes the vector obtained by replacing $\bw_i$ with $\bw'_i$ and $\bw_j$ with $\bw'_j$.
\end{lemmaA}

\begin{proof}
In the proof of Theorem 3.4 in \cite{MR2435859} (pages 1597-1598), the case of $A = \lbr n \rbr$ is proved. For general case of $A$, the proof remains unaffected and can still proceed.
\end{proof}

We next extend the previous definitions to the case involving two graphs.

Let $A,B \subseteq \lbr n \rbr$ be two distinct nonempty index sets. 
Let $f^{(1)}: (\mathbbR^d)^n \to \mathbbR$ be a function of the form
\begin{eqnarray*}
f^{(1)}(\bw_1, \dots, \bw_n) = \sum_{\ell=1}^n f^{(1)}_{\ell}(\bw_1, \dots, \bw_n),
\end{eqnarray*}
where, for each $\ell$, $f^{(1)}_{\ell}(\bw_1, \dots, \bw_n)$ is a function of only $\bw_\ell$ and its $A$-1-NN $\bw_{N^A(\ell)}$.

Also, let $f^{(2)}: (\mathbbR^d)^n \to \mathbbR$ be a function of the form
\begin{eqnarray*}
f^{(2)}(\bw_1, \dots, \bw_n) = \sum_{\ell=1}^n f^{(2)}_{\ell}(\bw_1, \dots, \bw_n),
\end{eqnarray*}
where, for each $\ell$, $f^{(2)}_{\ell}(\bw_1, \dots, \bw_n)$ is a function of only $\bw_\ell$ and its $B$-1-NN $\bw_{N^B(\ell)}$.

Let $G^{(1)}(\bfw)$ be the graphical rule on $\lbr n \rbr$ as defined in \eqref{eq:StepIII:2} based on $D^A_{\bfw}$. Also, let $G^{(2)}(\bfw)$ be the graphical rule on $\lbr n \rbr$ as defined in \eqref{eq:StepIII:2} based on $D^B_{\bfw}$. Let 
$G = G^{(1)} \cup G^{(2)}$ be the union of the graphs, in the sense that
\begin{eqnarray}
\text{$i,j$ is connected in $G$ if and only if $i,j$ is connected in either $G^{(1)}$ or $G^{(2)}$.}  \label{eq:StepIII:3}
\end{eqnarray}

The following Lemma \ref{lemA:sym_interac_2} is presented.
\begin{lemmaA} \label{lemA:sym_interac_2}
The graphical rule $G = G^{(1)} \cup G^{(2)}$ is a symmetric interaction on $f = f^{(1)} - f^{(2)}$. 
\end{lemmaA}

\begin{proof}
By Lemma \ref{lemA:sym_interac}, $G^{(1)}$ and $G^{(2)}$ are both symmetric. It follows that
$G$ is also symmetric. Note that $(i,j)$ is not an edge in $G$ if and only if $(i,j)$ is not an edge in both $G^{(1)}$ and $G^{(2)}$. 
Thus, for any choices of $\bfw, \bfw'$ and $i,j$, if $(i,j)$ is not an edge in any of the graphs 
$G(\bfw)$, $G(\bfw^i)$, $G(\bfw^j)$, and $G(\bfw^{ij})$, then $(i,j)$ is also not an edge in any of the graphs 
$G^{(1)}(\bfw)$, $G^{(1)}(\bfw^i)$, $G^{(1)}(\bfw^j)$, $G^{(1)}(\bfw^{ij})$, 
as well as 
$G^{(2)}(\bfw)$, $G^{(2)}(\bfw^i)$, $G^{(2)}(\bfw^j)$, and $G^{(2)}(\bfw^{ij})$.
By Lemma \ref{lemA:sym_interac},  $G^{(1)}$ and $G^{(2)}$ are interaction rules for $f^{(1)}$ and $f^{(2)}$ respectively, which implies that $f^{(1)}(\bfw) - f^{(1)}(\bfw^i) =  f^{(1)}(\bfw^j) - f^{(1)}(\bfw^{ij})$ and $f^{(2)}(\bfw) - f^{(2)}(\bfw^i) =  f^{(2)}(\bfw^j) - f^{(2)}(\bfw^{ij})$. Finally, by $f = f^{(1)} - f^{(2)}$, it follows that $f(\bfw) - f(\bfw^i) =  f(\bfw^j) - f(\bfw^{ij})$. Therefore, $G$ is an interaction rule with respect to $f$.  The proof is completed. 
\end{proof}

Now we proceed to prove Step (3).
\begin{proof}[Proof of Step (3)]
In our context, collect all variables in the vector $\bw = (\bz,\bfx,y) \in \mathbbR^{p+q+1}$. 
Define functions
\begin{eqnarray*}
f^{(1)}(\bw_1, \dots, \bw_n) &=& \frac{1}{\sqrt{n}}\sum_{\ell=1}^n f^{(1)}_{\ell}(\bw_1, \dots, \bw_n) =  \frac{1}{\sqrt{n}}
\sum_{\ell=1}^n \big\{F_Y(y_\ell) \wedge F_Y(y_{M(\ell)}) +h_1(y_\ell)\big\}, \cr
f^{(2)}(\bw_1, \dots, \bw_n) &=&  \frac{1}{\sqrt{n}}\sum_{\ell=1}^n f^{(2)}_{\ell}(\bw_1, \dots, \bw_n) = \frac{1-T}{\sqrt{n}}
\sum_{\ell=1}^n \big\{F_Y(y_\ell) \wedge F_Y(y_{N(\ell)}) +h_2(y_\ell) \big\}, \cr
\text{and} \quad  \quad	\tildef &=& f^{(1)} -f^{(2)}.
\end{eqnarray*}
Let index sets $A = \{1,2\dots, p+q\}$ and $B=\{1,2,\dots, q\}$. It is obvious that $f^{(1)}_{\ell}(\bw_1, \dots, \bw_n)$ is a function of only $\bw_\ell$ and its $A$-1-NN $\bw_{M(\ell)}$. Also, $f^{(2)}_{\ell}(\bw_1, \dots, \bw_n)$ is a function of only $\bw_\ell$ and its $B$-1-NN $\bw_{N(\ell)}$. Thus, by Lemma \ref{lemA:sym_interac_2}, the graphical rule $G = G^{(1)} \cup G^{(2)}$ (as defined in \eqref{eq:StepIII:3}) is a symmetric interaction on $\tildef = f^{(1)} - f^{(2)}$. 

Now that we have verified the existence of a symmetric interaction $G$ on $\tildef$, the results in Theorem 2.5 and 3.4 in \cite{MR2435859} can be applied to establish the CLT in Step (3). 
In the proof below, $C$ represents a generic constant which may differ between steps but is not explicitly defined.

Let $\bW = (\bZ, \bX, Y) \in \mathbbR^{p+q+1}$ be the random vector that collect $\bZ, \bX, Y$ as in our context. Let $\bW_1,\dots, \bW_n$ be i.i.d. random vectors sampled from $\bW$. Let $\bfW = (\bW_1,\dots, \bW_n)$.
Let $\bfW' = (\bW_1',\dots, \bW_n')$ be an $\iid$ copy of $\bfW$. For a function $f$ defined on $(\mathbbR^{p+q+1})^n$ and $j \in \lbr n \rbr$, denote
\begin{eqnarray*}
\Delta_j f(\bfW) &:=& f(\bfW) - f(\bfW^j)\cr
&=& f(\bfW) - f(\bW_1, \dots, \bW_{j-1}, \bW'_j, \bW_{j+1},\dots, \bW_n).
\end{eqnarray*}
It is obvious that $\Delta_j \tildef(\bfW) = 	\Delta_j f^{(1)}(\bfW) - 	\Delta_j f^{(2)}(\bfW)$. Let $M_1 = \max_{j\in\lbr n \rbr} |	\Delta_j f^{(1)}(\bfW)|$, $M_2 = \max_{j\in\lbr n \rbr} |	\Delta_j f^{(2)}(\bfW)|$, and $M = \max_{j\in\lbr n \rbr} |	\Delta_j \tildef(\bfW)|$. 
For either case of $f= f^{(1)}$ or $f^{(2)}$, define
\begin{eqnarray*}
M_f: = \max_{\ell \in \lbr n \rbr } |f_\ell(\bfW)| \vee \max_{j,\ell \in \lbr n \rbr} |f_\ell(\bfW^j)|
\end{eqnarray*}
Since both $f^{(1)}_\ell$ and $f^{(2)}_\ell$ are bounded from above by $1$ for all $\ell \in \lbr n \rbr$, it is clear that $M_{f^{(1)}} \leq 1$ and $M_{f^{(2)}} \leq 1$. By Theorem 3.4 of \cite{MR2435859}, we have
\begin{eqnarray*}
&&M_1 \leq 4 n^{-1/2} \alpha(p+q) M_{f^{(1)}} \leq 4  n^{-1/2} \alpha(p+q), \cr
&&M_2 \leq 4 n^{-1/2} \alpha(q) M_{f^{(2)}} \leq 4  n^{-1/2} \alpha(q),
\end{eqnarray*}
where $\alpha(d)$ is the minimum number of $60^\circ$ cones at the origin required to cover $\mathbbR^d$. Hence,
\begin{eqnarray}
M &=& \max_{j\in\lbr n \rbr} |	\Delta_j \tildef(\bfW)| \ = \ \max_{j\in\lbr n \rbr} |	\Delta_j f^{(1)}(\bfW) - \Delta_j f^{(2)}(\bfW)| \cr
&\leq&  \max_{j\in\lbr n \rbr} |	\Delta_j f^{(1)}(\bfW) | +\max_{j\in\lbr n \rbr} |	\Delta_j f^{(2)}(\bfW) |   \ = \ M_1 + M_2 \cr
&\leq&  4  n^{-1/2} \big(\alpha(p+q) + \alpha(q)\big) \leq C n^{-1/2}, \label{eq:StepIII:4}
\end{eqnarray}
for some constant $C \in (0, \infty)$. Also, for any dimension $d$, we have $\gamma^{(1)}_d: = \max_{\ell \in \lbr n \rbr} \E(f^{(1)}_\ell(\bfW)) \leq 1$, and the same holds for $\gamma^{(2)}_d$. By the proof of Theorem 3.4 of \cite{MR2435859}, for any $d \geq 8$, we have
\begin{eqnarray*}
\E \big(|\Delta_j f^{(1)}(\bfW)|^3\big) &\leq& C \alpha(d)^3 n^{-3/2} (n \gamma^{(1)}_d)^{3/d},  \cr
\E\big(|\Delta_j f^{(2)}(\bfW)|^3\big) &\leq& C \alpha(d)^3 n^{-3/2} (n \gamma^{(1)}_d)^{3/d}.
\end{eqnarray*}
By choosing some sufficiently large $d$, we further have
\begin{eqnarray*}
\max\big\{\E\big(|\Delta_j f^{(1)}(\bfW)|^3\big) , \E\big(|\Delta_j f^{(2)}(\bfW)|^3 \big)\big\} \leq C n^{-4/3}
\end{eqnarray*}
for some constant $C \in (0,\infty)$.
This yields that
\begin{eqnarray}
\E\big(|\Delta_j \tildef(\bfW)|^3\big) &=& 	\E\big(|\Delta_j f^{(1)}(\bfW) - \Delta_j f^{(2)}(\bfW)|^3\big) \cr
&\leq& 8 \cdot \max\big\{\E\big(|\Delta_j f^{(1)}(\bfW)|^3\big) , \E\big(|\Delta_j f^{(2)}(\bfW)|^3\big) \big\} \leq C n^{-4/3} \label{eq:StepIII:5}
\end{eqnarray}
for some constant $C \in (0,\infty)$.

By the proof of Theorem 3.4 of \cite{MR2435859}, there exists symmetric extensions $G^{(1)'}$ and $G^{(2)'}$ of $G^{(1)}$ and $G^{(2)}$ respectively on $\lbr n+4 \rbr$, with the maximum degree of $G^{(1)'}(\bfW)$ and $G^{(2)'}(\bfW)$ being bounded by $12\alpha(p+q)$ and $12\alpha(q)$ respectively. Then $G': = G^{(1)'} \cup G^{(2)'}$ is obviously a symmetric extension of $G = G^{(1)} \cup G^{(2)}$, with maximum vertex degree $d_{\max}$ 
satisfying
\begin{eqnarray}
d_{\max} \leq 12\alpha(p+q) + 12\alpha(q). \label{eq:StepIII:6}
\end{eqnarray}

Let $	\delta_{\tildef(\bfW)}$  denote the Wasserstein distance between the law of $$\big[\tildef(\bfW)- \E(\tildef(\bfW))\big]/\sqrt{\var(\tildef(\bfW))}$$ and the standard Gaussian law (see Definition 2.1 of \cite{MR2435859}).  By combining the bounds in \eqref{eq:StepIII:4}, \eqref{eq:StepIII:5}, and \eqref{eq:StepIII:6}, and by applying
Theorem 2.5 of \cite{MR2435859}, we get 
\begin{eqnarray*}
\delta_{\tildef(\bfW)} &\leq& \frac{C n^{1/2}}{\sigma_n^2} \E(M^8)^{1/4} d_{\max} + \frac{1}{2\sigma_n^3} \sum_{j=1}^n \E\big(|\Delta_j \tildef(\bfW)|^3\big) \cr
&\leq & C n^{-1/2} \sigma_n^{-2} + C n^{-1/3}\sigma_n^{-3},
\end{eqnarray*}
where $\sigma_n^2 = \var(\tildef(\bfW))$ and $C$ is some positive constant.

Note that $\sqrt{n} \tT_n^* = \tildef(\bfW)$.
By Step (2), we know $\sigma^2 = \lim_{n \to \infty} \var(\sqrt{n} \tT_n^*)$ exists. If 
$\sigma^2 = 0$, then $\sqrt{n} \big(\tT_n^* -\E(\tT_n^*)\big) = o_\P(1)$, implying that
$\sqrt{n} \big(\tT_n^* -\E(\tT_n^*)\big)  \conD 0$, where $0$ coincides with $N(0, \sigma^2)$.

If $\sigma^2 > 0$,
then $\sigma_n^2 = \sigma^2 + o(1)$ implies that $\sigma_n^2 >C$ for some constant $C>0$ when $n$ is sufficiently large. It follows that
\begin{eqnarray*}
\delta_{\tildef(\bfW)} = O(n^{-1/2}) + O(n^{-1/3}) = o(1),
\end{eqnarray*}
and
\begin{eqnarray*}
\frac{\tT_n^* - \E(\tT_n^*)}{\sqrt{\var(\tT_n^*)}} = \frac{\tildef(\bfW)- \E(\tildef(\bfW))}{\sqrt{\var(\tildef(\bfW))}}  \conD N(0,1).
\end{eqnarray*}
The proof is completed. 
\end{proof}

\section{Proofs of the rest theorems in Section \ref{sec:theory}} \label{secA:proofs-4}

\subsection{Proof of Corollary \ref{cor: CLT-Tn}}
\begin{proof}
If $\sqrt{n} L_n \to 0$, then Theorem \ref{thm:CLT-main} implies that
\begin{eqnarray*}
\sqrt{n} \big\{\tT_n - 0 \big\}\conD N(0, \sigma^2).
\end{eqnarray*}
Since $\kappa_n \conas \kappa >0$, applying Slutsky's theorem yields that
\begin{eqnarray*}
\sqrt{n} \big\{\tT_n/\kappa_n - 0 \big\}\conD N(0, \sigma^2/\kappa^2).
\end{eqnarray*}
Thus, Statement (i) is proved by noticing that $\tT_n/\kappa_n  = T_n -T$. 

Next, we prove Statement (ii). 
\begin{eqnarray*}
T_n^{\mathrm{bc}}   - T &=& \big(\tau_n - 	\hat{L}^{(\tau)}_n\big)/\big(\kappa_n - 	\hat{L}^{(\kappa)}_n\big) - T  \ = \
\frac{\tau_n - T \cdot \kappa_n - (	\hat{L}^{(\tau)}_n - T\cdot \hat{L}^{(\kappa)}_n ) }{\kappa_n - 	\hat{L}^{(\kappa)}_n}. 
\end{eqnarray*}
By \eqref{eq:step2-bias}, $	\hat{L}^{(\tau)}_n - L^{(\tau)}_n = o_\P(n^{-1/2})$ and $\hat{L}^{(\kappa)}_n - L^{(\kappa)}_n = o_\P(n^{-1/2})$. It follows that
\begin{eqnarray*}
&&\frac{\tau_n - T \cdot \kappa_n - (	\hat{L}^{(\tau)}_n - T\cdot \hat{L}^{(\kappa)}_n ) }{\kappa_n - 	\hat{L}^{(\kappa)}_n}  \cr
&=& 
\frac{\tau_n - T \cdot \kappa_n - (	L^{(\tau)}_n - T\cdot L^{(\kappa)}_n ) + o_\P(n^{-1/2}) }{\kappa_n - 	\hat{L}^{(\kappa)}_n}  \cr
&=&
\frac{\tau_n - T \cdot \kappa_n - \big(	\E(\tau_n) - T\cdot \E(\kappa_n) \big) + (\tau -T\cdot  \kappa) + o_\P(n^{-1/2}) }{\kappa_n - 	\hat{L}^{(\kappa)}_n} \cr
&=&
\frac{\tT_n - \E(\tT_n) + o_\P(n^{-1/2})}{\kappa_n - 	\hat{L}^{(\kappa)}_n}.
\end{eqnarray*}
Since $L^{(\kappa)}_n = o(1)$ and $\hat{L}^{(\kappa)}_n - L^{(\kappa)}_n = o_\P(n^{-1/2})$, we have $	\hat{L}^{(\kappa)}_n = o_\P(1)$, and $\kappa_n - 	\hat{L}^{(\kappa)}_n \conP \kappa>0$. Again, applying Slutsky's theorem gives that
\begin{eqnarray*}
\frac{\sqrt{n}\cdot \big\{\tT_n - \E(\tT_n) \big\}}{\kappa_n - 	\hat{L}^{(\kappa)}_n} \conD N(0, \sigma^2/\kappa^2).
\end{eqnarray*}
Therefore, $\sqrt{n} \cdot (T_n^{\mathrm{bc}}   - T)\conD N(0, \sigma^2/\kappa^2)$.
This completes the proof of Statement (ii). 
\end{proof}

\subsection{Proof of Corollary \ref{cor: CLT-H0}}
\begin{proof}
By applying Theorem \ref{thm:CLT-main}, the CLT follows immediately. It therefore remains to show that, under $H_0$, the variance $\sigma^2$ is simplified to $\sigma_0^2$ in \eqref{eq:sigma2_H0} and is strictly positive.
Indeed, this can be verified directly by combining the 23 terms appearing in Theorem \ref{thm:CLT-main}. However, we instead present an alternative proof, which is more intuitive and better reveals the underlying mechanism behind this simplification.

Recall the  the H\'ajek representation $\tT_n^*$ of $\tT_n$ in \eqref{eq:Hajek}, 
\begin{eqnarray*}
\tT_n^* &=&  \frac{1}{n}\sum_{i=1}^n \Big\{F_Y(Y_i \wedge Y_{M(i)}) + h_1(Y_i)\Big\} -(1-T)\cdot\frac{1}{n}\sum_{i=1}^n \Big\{ F_Y(Y_i \wedge Y_{N(i)}) +h_2(Y_i)\Big\}, 
\end{eqnarray*}
where $h_1(t) = \P(Y \wedge \barY>t)$ and $h_2(t) = \P(Y \wedge \tY>t)$. 
When $H_0$ holds, $\mu_{Y \mid \bX, \bZ}$ and $\mu_{Y \mid \bZ}$ are identical conditional distributions. Therefore, $\barY$ and $\tY$ are the same type copies of $Y$, implying that 
$h_1(t) = h_2(t)$ for all $t \in \mathbbR$. Moreover, under $H_0$, we have $T=0$. Thus $\tT_n^*$ reduces to
\begin{eqnarray*}
\tT_n^* &=&  \frac{1}{n}\sum_{i=1}^n F_Y(Y_i \wedge Y_{M(i)})  -\frac{1}{n}\sum_{i=1}^n F_Y(Y_i \wedge Y_{N(i)}) .
\end{eqnarray*}
In the following, we will prove the three equations
\begin{eqnarray}
\lim_{n \to \infty}\frac{1}{n}\var\Big[\sum_{i=1}^n F_Y(Y_i \wedge Y_{M(i)}) \Big]   
&=& (1+\mathfrak{q}_{p+q})\cdot\E \Big\{\var\big[F_Y(Y_1 \wedge \tY_1)\mid \bZ_1 \big] \Big\} \cr
&&+
(2-2\mathfrak{q}_{p+q} + \mathfrak{o}_{p+q})\cdot \E \Big\{\cov\big[F_Y(Y_1 \wedge \tY_1),  F_Y(Y_1 \wedge \tY_1')\mid \bZ_1 \big] \Big\}  \cr
&& + 
\var \Big\{\E\big[F_Y(Y_1 \wedge \tY_1)\mid \bZ_1 \big] \Big\}, \label{eq:pf-cor-H0:1} \\
\lim_{n \to \infty}\frac{1}{n}\var\Big[\sum_{i=1}^n F_Y(Y_i \wedge Y_{N(i)}) \Big]   
&=& (1+\mathfrak{q}_q)\cdot\E \Big\{\var\big[F_Y(Y_1 \wedge \tY_1)\mid \bZ_1 \big] \Big\} \cr
&&+
(2-2\mathfrak{q}_q + \mathfrak{o}_q)\cdot \E \Big\{\cov\big[F_Y(Y_1 \wedge \tY_1),  F_Y(Y_1 \wedge \tY_1')\mid \bZ_1 \big] \Big\}  \cr
&& + 
\var \Big\{\E\big[F_Y(Y_1 \wedge \tY_1)\mid \bZ_1 \big] \Big\},
\label{eq:pf-cor-H0:2}
\end{eqnarray}
and
\begin{eqnarray}
&&\lim_{n\to \infty}\frac{1}{n}\cov\Big(\sum_{i=1}^n F_Y(Y_i \wedge Y_{N(i)}) , \sum_{i=1}^n F_Y(Y_i \wedge Y_{M(i)}) \Big) \cr
&=&
4 \cdot \E \Big\{\cov\big[F_Y(Y_1 \wedge \tY_1),  F_Y(Y_1 \wedge \tY_1') \mid \bZ_1\big] \Big\} + \var\big\{\E[F_Y(Y_1 \wedge \tY_1) \mid \bZ_1]\big\}. \label{eq:pf-cor-H0:3}
\end{eqnarray}

We first prove \eqref{eq:pf-cor-H0:2}.
By \cite{Lin_Han_2025_CLT} (Lemmas 2.1, 2.2, and the proof of Theorem 1.3, pages 12 and 13), we have
\begin{eqnarray*}
&& \frac{1}{n} \E\Big[  \var\Big(\sum_{i=1}^n F_Y(Y_i \wedge Y_{N(i)}) \  \Big| \ \bfZ\Big)\Big]\cr
&=&
\E \Big\{\var\big[F_Y(Y_1 \wedge \tY_1)\mid \bZ_1 \big] \Big\} \cr
&&+
2 \E \Big\{\cov\big[F_Y(Y_1 \wedge \tY_1),  F_Y(Y_1 \wedge \tY_1')\mid \bZ_1 \big] \cdot \Ind\big(N(N(1)) \neq 1\big)\Big\}
\cr
&&
+ \E \Big\{\var\big[F_Y(Y_1 \wedge \tY_1)\mid \bZ_1 \big] \cdot
\Ind\big(N(N(1)) = 1\big)
\Big\} \cr
&&+
\E \Big\{\cov\big[F_Y(Y_1 \wedge \tY_1),  F_Y(Y_1 \wedge \tY_1')\mid \bZ_1 \big] 
\cdot \#\{j \in \lbr n \rbr: j\neq 1, N(j) = N(1)\}
\Big\}  +o(1).
\end{eqnarray*}
By Lemmas \ref{lemma:q_d} and \ref{lemma:o_d}, we have 
\begin{eqnarray*}
&& \E\big[	\Ind\big(N(N(1)) = 1\big)\mid \bX_1\big] \conP \mathfrak{q}_q, \cr
&&	\E\big[\#\{j \in \lbr n \rbr: j\neq 1, N(j) = N(1)\} \mid \bX_1\big] \conP \mathfrak{o}_q.
\end{eqnarray*}
Note that both $F_Y$ and $\#\{j \in \lbr n \rbr: j\neq 1, N(j) = N(1)\}$ are bounded \citep{MR682809}. Applying the bounded convergence theorem yields that
\begin{eqnarray}
&& \frac{1}{n} \E\Big[  \var\Big(\sum_{i=1}^n F_Y(Y_i \wedge Y_{N(i)}) \  \Big| \  \bfZ\Big)\Big]\cr
&=&
(1+\mathfrak{q}_q)\cdot\E \Big\{\var\big[F_Y(Y_1 \wedge \tY_1)\mid \bZ_1 \big] \Big\} \cr
&&+
(2-2\mathfrak{q}_q + \mathfrak{o}_q)\cdot \E \Big\{\cov\big[F_Y(Y_1 \wedge \tY_1),  F_Y(Y_1 \wedge \tY_1')\mid \bZ_1 \big] \Big\} +o(1).
\label{eq:pf-cor-H0:4}
\end{eqnarray}
Define $h(x) := \E[F_Y(Y \wedge \tY)\mid X=x]$. By applying \cite{Lin_Han_2025_CLT} (Lemma C.1, p.~25), we have
\begin{eqnarray*}
\frac{1}{n} \var\Big[ \E \Big(\sum_{i=1}^n F_Y(Y_i \wedge Y_{N(i)}) \  \Big| \  \bfZ\Big) - \sum_{i=1}^n h(\bZ_i)\Big] \to 0,
\end{eqnarray*}
which yields that
\begin{eqnarray}
\frac{1}{n} \var\Big[  \E \Big(\sum_{i=1}^n F_Y(Y_i \wedge Y_{N(i)}) \  \Big| \  \bfZ\Big)\Big]
&=& \var\big\{h(\bZ_1)\big\} 	+o(1) \cr
&=&
\var \Big\{\E\big[F_Y(Y_1 \wedge \tY_1)\mid \bZ_1 \big] \Big\} 
+o(1). 	\label{eq:pf-cor-H0:5}
\end{eqnarray}
Combining \eqref{eq:pf-cor-H0:4} and \eqref{eq:pf-cor-H0:5} proves \eqref{eq:pf-cor-H0:2}. 
Then \eqref{eq:pf-cor-H0:1} can be derived similarly as \eqref{eq:pf-cor-H0:2}by switching $\bZ$ to $(\bX, \bZ)$.

For \eqref{eq:pf-cor-H0:3}, combining \eqref{eq:limQ11} and \eqref{eq:limQ12} yields that 
\begin{eqnarray*}
&&\lim_{n\to \infty}\frac{1}{n}\cov\Big(\sum_{i=1}^n F_Y(Y_i \wedge Y_{N(i)}) , \sum_{i=1}^n F_Y(Y_i \wedge Y_{M(i)}) \Big) \cr
&=&
4\cdot \E \Big\{\cov\big[F_Y(Y_1 \wedge \tY_1),  F_Y(Y_1 \wedge \barY_1)\mid \bX_1, \bZ_1 \big] \Big\} \cr
&&+ 
2 \cdot  \E\Big[\cov\Big\{ \E\big( F_Y(Y_1 \wedge \tY_1) \biggiven \bX_1, \bZ_1\big) , \E\big( F_Y(Y_1 \wedge \barY_1) \biggiven \bX_1, \bZ_1\big) \Biggiven \bZ_1\Big\}\Big] \cr
&& 
+
\cov\Big\{\E\big( F_Y(Y_1 \wedge \tY_1 )\mid \bZ_1 \big), \E\big( F_Y(Y_1 \wedge \barY_1 )\mid \bZ_1 \big)\Big\} .
\end{eqnarray*}
Under $H_0$, $\barY$ and $\tY$ are the same type of copies, and 
$ \E\Big[\cov\Big\{ \E\big( F_Y(Y_1 \wedge \tY_1) \biggiven \bX_1, \bZ_1\big) , \E\big( F_Y(Y_1 \wedge \barY_1) \biggiven \bX_1, \bZ_1\big) \Biggiven \bZ_1\Big\}\Big] =0$. 
Thus, the above equation reduces to the form in \eqref{eq:pf-cor-H0:3}. This proves \eqref{eq:pf-cor-H0:3}.

Combining \eqref{eq:pf-cor-H0:1}, \eqref{eq:pf-cor-H0:2}, and \eqref{eq:pf-cor-H0:3}, we obtain
\begin{eqnarray}
\sigma_0^2&=& \lim_{n \to \infty } n \var(\tT_n^*) \cr
&=& 
(2+\mathfrak{q}_q+ \mathfrak{q}_{p+q})\cdot\E \Big\{\var\big[F_Y(Y \wedge \tY)\mid \bZ \big] \Big\} \cr
&&+
(\mathfrak{o}_q+ \mathfrak{o}_{p+q} - 2\mathfrak{q}_q- 2\mathfrak{q}_{p+q} -4)\cdot \E \Big\{\cov\big[F_Y(Y \wedge \tY),  F_Y(Y \wedge \tY')\mid \bZ \big] \Big\}. \label{eq:pf-cor-H0:6}
\end{eqnarray}
It is straightforward to verify that this is identical to the limiting variance in \eqref{eq:sigma2_H0}, upon noting that
\begin{eqnarray*}
\E \Big\{\var\big[F_Y(Y \wedge \tY)\mid \bZ \big] \Big\} &=&\E\big\{ F_Y^2(Y \wedge \tY)\big\} - \E\big\{F_Y(Y \wedge \tY) \cdot F_Y(\tY' \wedge\tY'')\big\} , \cr
\E \Big\{\cov\big[F_Y(Y \wedge \tY),  F_Y(Y \wedge \tY')\mid \bZ \big] \Big\}
&=&
\E\big\{F_Y(Y \wedge \tY) \cdot F_Y(Y \wedge\tY')\big\} \cr
&& - \ \E\big\{F_Y(Y \wedge \tY) \cdot F_Y(\tY' \wedge\tY'')\big\}. 
\end{eqnarray*}

Finally, an application of Lemma \ref{lemA:var_cov_ineq} (stated below this proof) yields
\begin{eqnarray*}
\var\big[F_Y(Y \wedge \tY)\mid \bZ \big] >  2 \cdot\cov\big[F_Y(Y \wedge \tY),  F_Y(Y \wedge \tY')\mid \bZ \big]  >0,
\end{eqnarray*}
which further implies that
\begin{eqnarray*}
\sigma_0^2 >(\mathfrak{o}_q+ \mathfrak{o}_{p+q}) \cdot \E \Big\{\cov\big[F_Y(Y \wedge \tY),  F_Y(Y \wedge \tY')\mid \bZ \big] \Big\} >0.
\end{eqnarray*}
This completes the proof. 
\end{proof}

\begin{lemmaA} \label{lemA:var_cov_ineq}
Let $U,V,W$ be i.i.d. non-degenerate random variables. 
Then
\begin{eqnarray*}
\var(U \wedge V) > 2 \cdot\cov(U \wedge V, U \wedge W) >0.
\end{eqnarray*}
\end{lemmaA}
\begin{proof}
We first prove the following result:
\begin{eqnarray}
\E\big[\var(U \wedge V \mid U)\big] > \var\big[\E(U \wedge V \mid U)\big].  
\label{eq:lemA:var_cov_ineq:1}
\end{eqnarray}
Note that
\begin{eqnarray*}
&&\E[\var(U \wedge V \mid U)] - \var[\E(U \wedge V \mid U)] \cr
&=&
\E((U \wedge V)^2) - \E[\E(U \wedge V \mid U)^2] - \big\{\E[\E(U \wedge V \mid U)^2] - \E(U \wedge V)^2\big\} \cr
&=&
\E((U \wedge V)^2)  + \E(U \wedge V)^2 - 2\cdot \E[\E(U \wedge V \mid U)^2].
\end{eqnarray*}
Let $\mu(\cdot)$ be the distribution law of $U$ and $V$. Let $G(t) = \P(U \geq t)$. For $t \in \mathbbR$,
\begin{eqnarray*}
\E((U \wedge V)^2)  
&=&
\int \int \Big[\int \Ind(u \geq t) \Ind(v \geq t) \d t\Big] \cdot \Big[\int \Ind(u \geq s) \Ind(v \geq s) \d s\Big] \d\mu(u) \d \mu(v) \cr
&=&
\int \int \Big[ \int \Ind(u\geq t) \Ind(u \geq s) \d \mu(u)\Big] \cdot \Big[\int \Ind(v\geq t) \Ind(v \geq s) \d \mu(v) \Big] \d t  \d s \cr
&=&
\int \int G^2(t \vee s) \d t  \d s.
\end{eqnarray*}
Note that 
\begin{eqnarray*}
\E(U \wedge V \mid U =u) 
=
\int \Ind(u \geq t) \Ind(v \geq t)  \d t  \d\mu(v) =
\int G(t) \Ind(u \geq t)  \d t.
\end{eqnarray*}
Then 
\begin{eqnarray*}
\E[\E(U \wedge V \mid U)^2] 
&=&
\int \Big[ \int G(t) \Ind(u\geq t) \d t\Big] \cdot \Big[ \int G(s) \Ind(u\geq s) \d s\Big] \d \mu(u) \cr
&=&
\int \int G(t) G(s) \Big[\int \Ind(u\geq t) \Ind(u \geq s) \d \mu(u) \Big] \d t \d s \cr
&=&
\int \int G(t) G(s) G(t \vee s) \d t \d s.
\end{eqnarray*}
Also,
\begin{eqnarray*}
\E(U \wedge V) 
=
\int \int \int \Ind(u \geq t) \Ind(v \geq t) \d t \d \mu(u) \d \mu(v)
=
\int G^2(t) \d t,
\end{eqnarray*}
which yields 
\begin{eqnarray*}
\E(U \wedge V)^2
= \int \int G^2(t) \cdot G^2(s) \d t \d s.
\end{eqnarray*}
Putting these pieces together, we have
\begin{eqnarray*}
&&\E((U \wedge V)^2)  + \E(U \wedge V)^2 - 2\cdot \E[\E(U \wedge V \mid U)^2] \cr
&=& 
\int \int \Big[G^2(t \vee s)  + G^2(t) \cdot G^2(s) - 2G(t) G(s) G(t \vee s) \Big] \d t \d s \cr
&=&
\int \int \big[G(t\vee s) - G(t) G(s) \big ]^2 \d t \d s >0. 
\end{eqnarray*}
This completes the proof of \eqref{eq:lemA:var_cov_ineq:1}.

Next, we prove Lemma \ref{lemA:var_cov_ineq}.
Note that
\begin{eqnarray*}
\cov(U \wedge V, U \wedge W) 
&=&
\E \big[\cov(U \wedge V, U \wedge W \mid U) \big]
+ \cov\big[ \E(U \wedge V \mid U), \E(U \wedge W \mid U)  \big]
\cr
&=&
0 + \var\big[\E(U \wedge V \mid U)\big], \label{eq:7.2}
\end{eqnarray*}
which proves that $\cov(U \wedge V, U \wedge W) >0$. Also,
\begin{eqnarray*}
\var(U \wedge V)   = \E \big[ \var \big( U \wedge V \mid U \big)\big] + \var\big[ \E(U \wedge V\mid U) \big].
\end{eqnarray*}
By \eqref{eq:lemA:var_cov_ineq:1}, we have
\begin{eqnarray*}
\var(U \wedge V)  > 2 \cdot\var\big[ \E(U \wedge V\mid U) \big] = 2 \cdot\cov(U \wedge V, U \wedge W).
\end{eqnarray*}
This completes the proof. 
\end{proof}

\subsection{Proof of Theorem \ref{thm: est_var-main} }
\begin{proof}
By Theorem \ref{thm:est_var_xi}, $\hsigma_1^2 $ and $\hsigma_2^2$ are consistent estimators for $\sigma_1^2 $ and $\sigma_2^2$ respectively. It remains to show $\hsigma_{1,2}$ is also consistent in estimating $\sigma_{1,2}$. 

Noting that $\hsigma_{1,2}$ and $\sigma_{1,2}$ share the similar structure as $\hsigma^2_{\xi(Y,\bZ)}$ and $\sigma^2_{\xi(Y,\bZ)}$ in Theorem \ref{thm:est_var_xi}, we can analogously establish that $\hsigma_{1,2} \conP \sigma_{1,2}$ by following the proof of Theorem \ref{thm:est_var_xi}.
\end{proof}

\subsection{Proof of Corollary \ref{cor: est_var-H0}}
\begin{proof}
Corollary \ref{cor: est_var-H0} directly follows from combining Theorem \ref{thm:est_var_xi} and Corollary \ref{cor: CLT-H0}. 
\end{proof}

\subsection{Proof of Proposition \ref{prop:est_var-B}}
\begin{proof}
By Corollary \ref{cor: CLT-H0}, $\tau_n$ satisfies a CLT, 
\begin{eqnarray*}
\sqrt{n} \big\{\tau_n - \E(\tau_n) \big\}\conD N(0, \sigma_0^2),
\end{eqnarray*}
with $\lim_{n \to \infty} n\,\var(\tau_n) = \sigma_0^2$ existing and strictly positive. Then the proof of Proposition \ref{prop:est_var-B} directly follows from that of Theorem 1 (ii) of \cite{Dette_Kroll_2025}.
\end{proof}

\subsection{Proof of Theorem \ref{thm:bias_correct}}
\begin{proof}
By Theorem 3.1 of \cite{azadkia2026biascorrection}, we have
\begin{eqnarray*}
L_n^{\bZ} = O(n^{-2/q} + n^{-1}). 
\end{eqnarray*}
By replacing $\bZ$ with $(\bX, \bZ)$ in the above result yields that 
\begin{eqnarray*}
L_n^{\bX, \bZ} = O(n^{-2/(p+q)} + n^{-1}). 
\end{eqnarray*}
This proves Result (i). 

By Theorem 4.2 of \cite{azadkia2026biascorrection}, the ridge regression estimator $\hatG_\bz(\cdot)$ provided in Algorithm \ref{alg2} satisfies the regularity conditions on $\hatG$ in Assumption 4.3 of \cite{azadkia2026biascorrection}. Then applying Theorem 4.1 of \cite{azadkia2026biascorrection} yields that $\sqrt{n}\E \big(	|\hat{L}_n^{\bZ} - L_n^{\bZ}| \big) \to 0$, as $n \to \infty$. This proves  	$\hat{L}_n^{\bZ} - L_n^{\bZ} = o_\P(n^{-1/2})$. 
By applying a similar proof procedure with $\bZ$ replaced by $(\bX, \bZ)$, we can analogously show that 
$\hat{L}_n^{\bX,\bZ} - L_n^{\bX,\bZ} = o_{\P}(n^{-1/2})$. This proves Result (ii).
\end{proof}

\subsection{Proof of Theorem \ref{thm:CI}}

\begin{proof}
By Theorem \ref{thm:bias_correct}, $p+q\leq 3$ implies that $\sqrt{n} L_n \to 0$. Then Corollary \ref{cor: CLT-Tn} (i) gives that
\begin{eqnarray*}
\sqrt{n} \big(T_n - T \big)\conD N(0, \sigma^2/\kappa^2).
\end{eqnarray*}
By Theorem \ref{thm: est_var-main}, $\hsigma \conP \sigma$. Also, $\kappa_n \conas \kappa >0$. When $\sigma^2>0$, applying Slutsky's theorem yields that
\begin{eqnarray*}
\sqrt{n} \big( (T_n - T) \cdot \kappa_n  /  \hsigma \big)\conD N(0, 1).
\end{eqnarray*}
Therefore,
\begin{eqnarray*}
\P\Big(  -\frac{\hsigma\cdot z_{\alpha/2}}{\sqrt{n}\cdot \kappa_n }
< T_n - T  < \frac{\hsigma\cdot z_{\alpha/2}}{\sqrt{n}\cdot \kappa_n }\Big ) \to 1-\alpha.
\end{eqnarray*}
This proves result (i) of Theorem \ref{thm:CI}.

Result (ii) can be proved analogously by using Corollary \ref{cor: CLT-Tn} (ii) instead of (i). This completes the proof.
\end{proof}

\subsection{Proof of Theorem \ref{thm:test}}

\begin{proof}
We first prove Result (i). By \cite{azadkia2019simple}, $\tau(Y, \bX \mid \bZ) $ defined in \eqref{eq:tau} equals $0$ if $H_0$ holds, and is strictly positive if $H_0$ is violated. 
Therefore, the first statement of result (i) directly follows from combining the CLT of $\tau_n$ in Corollary \ref{cor: CLT-H0}, the consistency of variance estimation established in Corollary \ref{cor: est_var-H0}, and the bias correction result in Theorem \ref{thm:bias_correct} (ii). 

For the second statement of Result (i), it follows directly from the fact that the bias $L_n^{(\tau)} = o(n^{-1/2})$ is negligible when $p+q \leq 3$, as established in Theorem~\ref{thm:bias_correct} (i).

We next prove Result (ii). 
Note that under $H_1$, the fast simplified estimator $\hat{\sigma}_{0,\mathrm{F}}^2$ in Corollary \ref{cor: est_var-H0} and the original estimator $\hsigma^2$ in Theorem \ref{thm: est_var-main} have different limits:  $\hsigma^2 \conP \sigma^2$ in Theorem \ref{thm:CLT-main}, while  
$ \hat{\sigma}_{0,\mathrm{F}}^2 \conP \sigma_0^2$ in  Corollary \ref{cor: CLT-H0}.
By \eqref{eq:pf-cor-H0:6}, we further have
\begin{eqnarray*}
\sigma_0^2 &=&
(2+\mathfrak{q}_q+ \mathfrak{q}_{p+q})\cdot\E \Big\{\var\big[F_Y(Y \wedge \tY)\mid \bZ \big] \Big\} \cr
&&+
(\mathfrak{o}_q+ \mathfrak{o}_{p+q} - 2\mathfrak{q}_q- 2\mathfrak{q}_{p+q} -4)\cdot \E \Big\{\cov\big[F_Y(Y \wedge \tY),  F_Y(Y \wedge \tY')\mid \bZ \big] \Big\}. 
\end{eqnarray*}
Whenever $Y$ is not almost surely equal to a function of $\bZ$, applying Lemma \ref{lemA:var_cov_ineq} yields that
\begin{eqnarray*}
\sigma_0^2  >(\mathfrak{o}_q+ \mathfrak{o}_{p+q}) \cdot \E \Big\{\cov\big[F_Y(Y \wedge \tY),  F_Y(Y \wedge \tY')\mid \bZ \big] \Big\} >0.
\end{eqnarray*}
This proves that $\sigma_0^2 $ remains strictly positive under $H_1$. 
Therefore,
\begin{eqnarray}
\lim_{n \to \infty} \P(\mathsf{T}^{\mathrm{F}}_{\alpha} = 1)	&=&\lim_{n \to \infty} \P(	\sqrt{n} \tau_n  /\hat{\sigma}_{0,\mathrm{F}} > z_\alpha  ) \cr &=&  \lim_{n \to \infty} \P \big(	 (\tau +o_\P(1)) /(\sigma_0+ o_\P(1))> z_\alpha /\sqrt{n} \big)  \cr
&=& \lim_{n \to \infty} \P(	 \tau  /\sigma_0> z_\alpha /\sqrt{n} ) \ = \ 1.
\label{eq:thm:test:1}
\end{eqnarray}

Now we examine the $m$-out-of-$n$ bootstrap estimator $\hat{\sigma}_{0,\mathrm{B}}$ defined in \eqref{eq:sigma2_B}. Note that $\tau_n$ by its definition in \eqref{eq:tau_n} is bounded for any $n$. Specifically, we have $ -1 \leq \tau_n \leq 1$. Hence, the bootstrap samples $\{\tau_{m,b}^* \}_{b=1}^B$ in  \eqref{eq:sigma2_B} are also bounded, implying that its sample variance is bounded from above by some constant $C>0$.  Therefore, we have $0 < \hat{\sigma}_{0,\mathrm{B}}^2 < m \cdot C$. It follows that
\begin{eqnarray}
\lim_{n \to \infty} \P(\mathsf{T}^{\mathrm{B}}_{\alpha} = 1)	&=&\lim_{n \to \infty} \P(	\sqrt{n} \tau_n  /\hat{\sigma}_{0,\mathrm{B}} > z_\alpha  ) \cr &\geq&  \lim_{n \to \infty} \P \big(	 (\tau +o_\P(1)) /\sqrt{m\cdot C}> z_\alpha /\sqrt{n} \big)  \cr
&=& \lim_{n \to \infty} \P\big(	 \tau  /\sqrt{C}> z_\alpha \cdot \sqrt{m/n} \big) \ = \ 1, \label{eq:thm:test:2}
\end{eqnarray}
where the last equation is due to $m = o(n)$.

By Theorem \ref{thm:bias_correct}, we have $L_n^{(\tau)} = O(n^{-2/(p+q)})$ and $\hatL_n^{(\tau)} - L_n^{(\tau)}  =  o_\P(n^{-1/2})$, yielding that $\hatL_n^{(\tau)}  = o_\P(1)$, which is dominated by $\tau_n$ and $\tau$.  Therefore, we can prove the consistency for $\mathsf{T}^{\mathrm{F,bc}}_\alpha$ and $\mathsf{T}^{\mathrm{B,bc}}_\alpha$ in the same way, as in \eqref{eq:thm:test:1} and \eqref{eq:thm:test:2}. 	This completes the proof of Result (ii). 
\end{proof}

\section{Auxiliary lemmas } \label{secA:aux-lemmas}
This section presents auxiliary lemmas required for proving Theorem~\ref{thm:CLT-main} (Step (2): limiting variance; see Section \ref{secA:proof-stepII}). For clarity, we divide these lemmas into two groups, Group~1 and Group~2. 
Specifically,  Lemmas \ref{lemA:aux-1-1}--\ref{lemA:aux-1-9} in Group~1 mainly concern probabilistic properties of the NNG, serving as preparatory results for Group~2, whereas Lemmas \ref{lemA:aux-2-1}--\ref{lemA:aux-2-9} in Group~2 are directly used in the proof of Step (2) of Theorem~\ref{thm:CLT-main}.

The notations used in this section are consistent with those in Sections \ref{sec:SI} and \ref{sec:theory}. Throughout this section, we always assume Assumptions \ref{assump_4.1} -- \ref{assump_4.5}.

\subsection{Auxiliary lemmas: Group 1}
\begin{lemmaA} \label{lemA:aux-1-1}
For any $\epsilon>0$, the following inequality holds with probability one:
\begin{eqnarray*}
&& \liminf_{n \to \infty} n\P\big[N(3) = M(1), \|\bZ_3-\bZ_1\| <\epsilon \ \big | \ (\bX_1, \bZ_1) \big]\geq 1.
\end{eqnarray*}
\end{lemmaA}

\begin{proof}
The proof of this lemma utilizes similar techniques to those used in Lemma~\ref{lemma:two_NNGs} (see Section~\ref{secB.2}).
Let $(\bx_1, \bz_1)$ be some fixed point in the interior of the support of $(\bX, \bZ)$ with positive density. Similar to \eqref{eq:lemma:two_NNGs:1}, we have
\begin{eqnarray} 
&&n\P\big[N(3) = M(1), \|\bZ_3-\bZ_1\| <\epsilon \mid (\bX_1, \bZ_1) =(\bx_1, \bz_1)\big]\cr
&=&
n \sum_{k=2}^n \Big\{\P\big[N(3) = M(1), \|\bZ_3-\bZ_1\| <\epsilon \ \big | \ (\bX_1, \bZ_1) =(\bx_1, \bz_1), M(1) = k\big]\cr&& \qquad \times \P\big(M(1)=k \ \big | \ (\bX_1, \bZ_1)=(\bx_1, \bz_1)\big) \Big\} \cr
&=&
n P\big[N(3) =2, \|\bZ_3-\bZ_1\| <\epsilon \ \big | \ (\bX_1, \bZ_1) =(\bx_1, \bz_1), M(1) = 2\big]
\cr
&=&
n \int f_{M(1)} (\bx_2, \bz_2) \cdot \P\big[ N(3) = 2 , \|\bZ_3-\bZ_1\| <\epsilon \ \big | \ A\big] \d (\bx_2, \bz_2), \label{eq:lemA:aux-1-1:1}
\end{eqnarray}
where $f_{M(1)}(\cdot)$ denotes the density function of $(\bX_{M(1)}, \bZ_{M(1)})$ (conditional on the event $(\bX_1, \bZ_1) = (\bx_1, \bz_1)$), 
and $A$ denotes the event $\{(\bX_1, \bZ_1)=(\bx_1, \bz_1), M(1) = 2, (\bX_2, \bZ_2) = (\bx_2, \bz_2)\}$. 

Now fix some $\tildet >0$, and let $\tilde{r} = n^{-1/q}\cdot \tildet^{1/q}$. Since $\lim_{n\to \infty}\tilde{r}  =0$, there exists some $R>0$, such that
\begin{eqnarray}
\sup_{	(\bx_2, \bz_2) \in \mathcal{B}((\bx_1, \bz_1), R), \  \bz_3 \in \mathcal{B}(\bz_2, \tilde{r})} \|\bz_3-\bz_1\| < \epsilon  \label{eq:lemA:aux-1-1:2}
\end{eqnarray}
holds for sufficiently large $n$. 
Fix some small $\epsilon_0 \in (0, \epsilon)$.
By Lemma 11.3 of \cite{azadkia2019simple}, $\|(\bX_{M(1)}, \bZ_{M(1)}) - (\bX_1, \bZ_1)\| \to 0$ almost surely. Thus,
\begin{eqnarray}
\int_{(\bx_2, \bz_2) \in \mathcal{B}((\bx_1, \bz_1), R)} f_{M(1)} (\bx_2, \bz_2) >1-\epsilon_0 \label{eq:lemA:aux-1-1:3}
\end{eqnarray}
holds for sufficiently large $n$.  Combining \eqref{eq:lemA:aux-1-1:2} and \eqref{eq:lemA:aux-1-1:3} gives that
\begin{eqnarray}
&&n \int f_{M(1)} (\bx_2, \bz_2) \cdot \P\big[ N(3) = 2 , \|\bZ_3-\bZ_1\| <\epsilon \ \big | \ A\big] \d (\bx_2, \bz_2) \cr
&\geq&
n \int_{(\bx_2, \bz_2) \in \mathcal{B}((\bx_1, \bz_1), R)} f_{M(1)} (\bx_2, \bz_2) \cdot \P\big[ N(3) = 2 , \|\bZ_3-\bz_2\| \leq \tilde{r} \ \big | \ A\big] \d (\bx_2, \bz_2) \cr
&=&
n \int_{(\bx_2, \bz_2) \in \mathcal{B}((\bx_1, \bz_1), R)} f_{M(1)} (\bx_2, \bz_2) \cdot \P[ N(3) = 2 \mid A](\tildet) \d (\bx_2, \bz_2) \cr
&=& (1+o(1))\cdot g_n(\tildet) -O(\epsilon_0), \label{eq:lemA:aux-1-1:4}
\end{eqnarray}
where 
\begin{eqnarray}
&&\P[ N(3) = 2 \mid A](\tildet) := \P\big[ N(3) = 2, \|\bZ_3 - \bz_2\| \leq  \tilde{r} \mid A\big] \label{def_tt} \\
\text{and }&&g_n(\tildet) := (n-2) \cdot  \int f_{M(1)} (\bx_2, \bz_2) \cdot \P[ N(3) = 2 \mid A](\tildet) \d (\bx_2, \bz_2) \nonumber
\end{eqnarray}
are defined in the same way as in \eqref{eq:lemma:two_NNGs:4} and \eqref{eq:lemma:two_NNGs:5}.
The last equation of \eqref{eq:lemA:aux-1-1:4} is by the similar argument as that of \eqref{eq:lemma:two_NNGs:9}. 
By \eqref{eq:lemma:two_NNGs:6}, we have
\begin{eqnarray*}
\lim_{n\to \infty} g_n(\tildet) = 1- \exp\big(- C_2 \cdot f(\bz_1) \cdot \tildet \big),
\end{eqnarray*}
where $f$ denotes the density function of $\bZ$. 
This combined with \eqref{eq:lemA:aux-1-1:1}, \eqref{eq:lemA:aux-1-1:4} and the fact that $\epsilon_0$ can be arbitrarily small, we obtain
\begin{eqnarray}
&&\liminf_{n\to \infty}n\P\big[N(3) = M(1), \|\bZ_3-\bZ_1\| <\epsilon \ \big | \ (\bX_1, \bZ_1) =(\bx_1, \bz_1)\big] \cr
&\geq& \sup_{\tildet >0} \liminf_{n\to \infty} g_n(\tildet)  \geq  \sup_{\tildet >0} \big\{1-  \exp(-C_2 \cdot  f(\bz_1)\cdot \tildet)\big\} = 1
\label{eq:lemA:aux-1-1:5}.
\end{eqnarray}
This completes the proof. 
\end{proof}

\begin{lemmaA} \label{lemA:aux-1-2}
Let $f: \supp(\bZ) \to \mathbbR$ be a measurable function that is continuous almost everywhere.
For any $\epsilon>0$, the following inequality holds with probability one:
\begin{eqnarray*}
&& \liminf_{n \to \infty} n\P\big[N(3) = M(1), |f(\bZ_3)-f(\bZ_1)| <\epsilon \ \big | \ (\bX_1, \bZ_1)\big]\geq 1.
\end{eqnarray*}
\end{lemmaA}

\begin{proof}
Let $\calW^o \subseteq \mathbbR^{p+q}$ be the interior of the support of $(\bX, \bZ)$, and let $\calW^+ = \{(\bx,\bz) \in \calW^o : f_{\bX, \bZ}(\bx, \bz) >0\}$. 
Let $\calZ^o \subseteq \mathbbR^q$ be the interior of the support of $\bZ$, and let $\calZ^* = \{\bz \in \calZ^o: f \text{ is continuous at } \bz\}$.
By Assumption \ref{assump_4.3} and the fact that $f$ is continuous almost everywhere, it is clear that 
$\P((\bX,\bZ) \in \calW^+) =1$ and $\P(\bZ \in \calZ^*) =1$. Moreover, 
for any $\bz_1 \in \calZ^*$, 
there exists some neighborhood $\mathcal{B}(\bz_1, R_1) \subseteq \calZ^o$ such that 
\begin{eqnarray}
\sup_{\bz_3 \in \mathcal{B}(\bz_1, R_1)} |f(\bz_3)-f(\bz_1)| < \epsilon. \label{eq:lemA:aux-1-2:1}
\end{eqnarray}

We shall prove for any $(\bx_1, \bz_1) \in \calW^+$ with $\bz_1 \in \calZ^*$,
\begin{eqnarray}
&& \liminf_{n \to \infty} n\P\big[N(3) = M(1), \|\bZ_3-\bZ_1\| <\epsilon\} \ \big | \ (\bX_1, \bZ_1) = (\bx_1, \bz_1)\big]\geq 1. \label{eq:lemA:aux-1-2:2}
\end{eqnarray}
The proof proceeds along the same lines as that of Lemma~\ref{lemA:aux-1-1}. The only difference is that \eqref{eq:lemA:aux-1-1:2} needs to be replaced by the following statement:
there exists some $R>0$, such that
\begin{eqnarray}
\sup_{	(\bx_2, \bz_2) \in \mathcal{B}((\bx_1, \bz_1), R), \  \bz_3 \in \mathcal{B}(\bz_2, \tilde{r})} |f(\bz_3)-f(\bz_1)| < \epsilon  \label{eq:lemA:aux-1-2:3}
\end{eqnarray}
holds for sufficiently large $n$. Note that 
\begin{eqnarray*}
\sup_{	(\bx_2, \bz_2) \in \mathcal{B}((\bx_1, \bz_1), R), \  \bz_3 \in \mathcal{B}(\bz_2, \tilde{r})} \|\bz_3- \bz_1\| \leq R + \tilde{r}. \label{4.17}
\end{eqnarray*}
Recall that $\tilde{r} = n^{-1/q}\cdot \tildet^{1/q}\to 0$. Thus, there exists some small $R$ such that $\sup\|\bz_3- \bz_1\|\leq R+\tilde{r}  <  R_1$ holds for sufficiently large $n$. Then by \eqref{eq:lemA:aux-1-2:1}, we have $\sup |f(\bz_3)-f(\bz_1)| < \epsilon$. This proves \eqref{eq:lemA:aux-1-2:3}. 

It remains to follow the same proof procedure as in Lemma~\ref{lemA:aux-1-1}to establish \eqref{eq:lemA:aux-1-2:2}. Finally, noting that $\P((\bX,\bZ) \in \calW^+) = 1$ and $\P(\bZ \in \calZ^*) = 1$, we conclude that the statement of the lemma holds.
\end{proof}

\begin{lemmaA} \label{lemA:aux-1-3}
Let $f: \supp((\bX,\bZ)) \to \mathbbR$ be a measurable function that is continuous almost everywhere. For any $\epsilon>0$, the following inequality holds with probability one:
\begin{eqnarray*}
&& \liminf_{n \to \infty} n\P\big[N(3) = M(1), |f(\bX_{M(1)},\bZ_{M(1)})-f(\bX_1,\bZ_1)| <\epsilon \ \big | \ (\bX_1, \bZ_1)\big]\geq 1.
\end{eqnarray*}
\end{lemmaA}
\begin{proof}
Let $\calW^o \subseteq \mathbbR^{p+q}$ be the interior of the support of $(\bX, \bZ)$, and let $\calW^+ = \{(\bx,\bz) \in \calW^o : f_{\bX, \bZ}(\bx, \bz) >0\}$. 
Let $\calW^* = \{(\bx, \bz) \in \calW^+: f \text{ is continuous at } (\bx, \bz)\}$.
By Assumption \ref{assump_4.3} and the fact that $f$ is continuous almost everywhere, it is clear that 
$\P((\bX,\bZ) \in \calW^*) =1$. Moreover, 
for any $(\bx_1,\bz_1) \in \calW^*$, 
there exists some neighborhood $\mathcal{B}((\bx_1,\bz_1), R_1) \subseteq \calW^*$ such that 	
\begin{eqnarray*}
\sup_{(\bx_2,\bz_2) \in \mathcal{B}((\bx_1,\bz_1), R_1)}  |f(\bx_2,\bz_2)-f(\bx_1, \bz_1)| 
< \epsilon.  
\end{eqnarray*}
Thus, by applying a similar derivation as in \eqref{eq:lemA:aux-1-1:1}, we obtain
\begin{eqnarray*} 
&&n\P\big[N(3) = M(1), |f(\bX_{M(1)},\bZ_{M(1)})-f(\bX_1,\bZ_1)|  <\epsilon \ \big | \ (\bX_1, \bZ_1) =(\bx_1, \bz_1)\big]\cr
&=&
n \sum_{k=2}^n \Big\{\P\big[N(3) = M(1), |f(\bX_{M(1)},\bZ_{M(1)})-f(\bX_1,\bZ_1)|<\epsilon \ \big | \ (\bX_1, \bZ_1) =(\bx_1, \bz_1), M(1) = k\big]\cr&& \qquad \times \P\big(M(1)=k \ \big | \ (\bX_1, \bZ_1)=(\bx_1, \bz_1)\big) \Big\} \cr
&=&
n P\big[N(3) =2, |f(\bX_2,\bZ_2)-f(\bX_1,\bZ_1)| <\epsilon \ \big | \ (\bX_1, \bZ_1) =(\bx_1, \bz_1), M(1) = 2\big]
\cr
&\geq &
n \int_{(\bx_2,\bz_2) \in \mathcal{B}((\bx_1,\bz_1), R_1)} f_{M(1)} (\bx_2, \bz_2) \cdot \P[ N(3) = 2 \mid A] \d (\bx_2, \bz_2),
\end{eqnarray*}
where $A$ denotes the event $\{(\bX_1, \bZ_1)=(\bx_1, \bz_1), M(1) = 2, (\bX_2, \bZ_2) = (\bx_2, \bz_2)\}$. Since $\|(\bX_{M(1)}, \bZ_{M(1)}) - (\bX_1, \bZ_1)\| \to 0$ almost surely, for any $\epsilon_0 >0$, there exists some $R \in (0, R_1)$, such that
\begin{eqnarray*}
\int_{(\bx_2, \bz_2) \in \mathcal{B}((\bx_1, \bz_1), R)} f_{M(1)} (\bx_2, \bz_2) >1-\epsilon_0 
\end{eqnarray*}
holds for sufficiently large $n$.

Therefore, for any $\tildet>0$,
\begin{eqnarray*}
&&n \int_{(\bx_2,\bz_2) \in \mathcal{B}((\bx_1,\bz_1), R_1)} f_{M(1)} (\bx_2, \bz_2) \cdot \P[ N(3) = 2 \mid A] \d (\bx_2, \bz_2) \cr
&\geq&
n \int_{(\bx_2, \bz_2) \in \mathcal{B}((\bx_1, \bz_1), R)} f_{M(1)} (\bx_2, \bz_2) \cdot \P[ N(3) = 2 \mid A] \d (\bx_2, \bz_2) \cr
&\geq &
n \int_{(\bx_2, \bz_2) \in \mathcal{B}((\bx_1, \bz_1), R)} f_{M(1)} (\bx_2, \bz_2) \cdot \P[ N(3) = 2 \mid A](\tildet) \d (\bx_2, \bz_2) \cr
&=& (1+o(1))\cdot g_n(\tildet) -O(\epsilon_0),
\end{eqnarray*}
where 
$\P[ N(3) = 2 \mid A](\tildet)$ is defined in \eqref{def_tt}, and
the last equation is due to \eqref{eq:lemA:aux-1-1:4}.  By the similar arguments as those in \eqref{eq:lemA:aux-1-1:4} to \eqref{eq:lemA:aux-1-1:5}, we have
\begin{eqnarray*}
\liminf_{n \to \infty} n\P\big[N(3) = M(1), |f(\bX_{M(1)},\bZ_{M(1)})-f(\bX_1,\bZ_1)| <\epsilon \ \big | \ (\bX_1, \bZ_1) =(\bx_1, \bz_1)\big]\geq 1.
\end{eqnarray*}
Finally, by $\P((\bX,\bZ) \in \calW^*) = 1$, the proof is completed. 
\end{proof}

\begin{lemmaA} \label{lemA:aux-1-4}
Let $f: \supp(\bZ) \to \mathbbR$ be a measurable function that is continuous almost everywhere.
For any $\epsilon>0$, 
\begin{eqnarray*}
\lim_{n\to \infty} n \P\big(|f(\bZ_2)-f(\bZ_1)| <\epsilon, N(2)=M(1)\big) = 1.
\end{eqnarray*}

\end{lemmaA}
\begin{proof}

Let $S_n = n\cdot \P\big(|f(\bZ_2)-f(\bZ_1)|<\epsilon , N(2)=M(1) \ \big | \  \bX_1, \bZ_1\big)$. Then
\begin{eqnarray*}
\E(S_n) &=& n \P\big(|f(\bZ_2)-f(\bZ_1)|<\epsilon , N(2)=M(1) \big) 
\cr
&\leq& n \P\big( N(2)=M(1) \big)  \ = \ 1 + o(1),
\end{eqnarray*}
where the last equation is due to \citet[Lemma 7.4, Equation (28)]{Shi_Drton_Han_2024_Bernoulli}. On the other hand, Lemma \ref{lemA:aux-1-2} proves that $\liminf_{n \to \infty}S_n \geq 1$ holds with probability one. This combined with Lemma \ref{lemA:Xn_conP_C} yields that $S_n \conP 1$. 

Note that 
\begin{eqnarray*}
S_n &\leq & n\cdot \P\big(N(2)=M(1) \ \big | \  \bX_1, \bZ_1\big) \cr
&=& \frac{n}{n-1} \E\Big[ \sum_{k=2}^n \Ind\big( N(k) = M(1)\big) \ \Big | \  \bX_1, \bZ_1 \Big].
\end{eqnarray*}
By Corollary S1 of \cite{MR682809}, $\sum_{k=2}^n \Ind( N(k) = M(1))$ is bounded from above by some constant. Thus $S_n$ is also bounded from above. This combined with $S_n \conP 1$ gives that $\E(S_n) \to 1$. \end{proof}

\begin{lemmaA} \label{lemA:aux-1-5}
Let $f: \supp((\bX,\bZ)) \to \mathbbR$ be a measurable function that is continuous almost everywhere.
For any $\epsilon>0$, 
\begin{eqnarray*}
\lim_{n\to \infty} n \P\big(|f(\bX_{M(1)},\bZ_{M(1)})-f(\bX_1,\bZ_1)| <\epsilon, N(2)=M(1)\big) = 1.
\end{eqnarray*}

\end{lemmaA}

\begin{proof}
The proof is similar to that of Lemma \ref{lemA:aux-1-4}. The only difference is that Lemma \ref{lemA:aux-1-3} is used in place of Lemma \ref{lemA:aux-1-2}.
\end{proof}

\begin{lemmaA} \label{lemA:aux-1-6}
Let $f: \supp(\bZ) \to \mathbbR$ be a measurable function that is continuous almost everywhere.
For any $\epsilon>0$, 
\begin{eqnarray*}
\lim_{n \to \infty}\P\big(|f(\bZ_2)-f(\bZ_1)|<\epsilon \ \big | \ N(2)=M(1)\big) =1.
\end{eqnarray*}
\end{lemmaA}

\begin{proof} Note that
\begin{eqnarray*}
\P\big(|f(\bZ_2)-f(\bZ_1)|<\epsilon \ \big | \ N(2)=M(1)\big) =
\frac{n \P\big( |f(\bZ_2)-f(\bZ_1)| <\epsilon, N(2)=M(1)\big) }{n \P\big(N(2) = M(1)\big)}.
\end{eqnarray*} 
By Lemma \ref{lemA:aux-1-4}, $	\lim_{n\to \infty} n \P\big(|f(\bZ_2)-f(\bZ_1)| <\epsilon, N(2)=M(1)\big) = 1$. By \citet[Lemma 7.4, Equation (28)]{Shi_Drton_Han_2024_Bernoulli}, $\lim_{n\to \infty} n \P\big(  N(2)=M(1)\big) = 1$. Combining these two results completes the proof. \end{proof}

\begin{lemmaA} \label{lemA:aux-1-7}
Let $f: \supp((\bX,\bZ)) \to \mathbbR$ be a measurable function that is continuous almost everywhere.
For any $\epsilon>0$, 
\begin{eqnarray*}
\lim_{n \to \infty}\P\big(|f(\bX_{M(1)},\bZ_{M(1)})-f(\bX_1,\bZ_1)|<\epsilon \ \big | \ N(2)=M(1)\big) =1.
\end{eqnarray*}
\end{lemmaA}

\begin{proof} This lemma can be proved by following the same procedure as in the proof of Lemma \ref{lemA:aux-1-6}. 
The only difference is that Lemma \ref{lemA:aux-1-5} is used in place of Lemma \ref{lemA:aux-1-4}.
\end{proof}

\begin{lemmaA} \label{lemA:aux-1-8}
As $n \to \infty$, we have
\begin{eqnarray}
&&\|\bZ_{M(N(1))} - \bZ_{1}\| \conP 0, \label{eq:lemA:aux-1-8:1}\\
\text{and} &&\|\bZ_{N(M(1))} - \bZ_{1}\| \conP 0. \label{eq:lemA:aux-1-8:2}
\end{eqnarray}
\end{lemmaA}

\begin{proof} 
We first prove \eqref{eq:lemA:aux-1-8:1}. 
Let $\calW \subseteq \mathbbR^{p+q}$ be the support of $(\bX, \bZ)$. For any $\epsilon >0$ and $\delta>0$, define the set
\begin{eqnarray*}
\calW_{\epsilon, \delta} = \Big\{(\bx, \bz) \in \calW: \P\big((\bX, \bZ) \in \calB((\bx, \bz),\epsilon)\big) >\delta \Big\}.
\end{eqnarray*}
It is obvious that $\calW_{\epsilon,\delta}$ is nondecreasing as $\delta \to 0$, with $\cup_{\ell=1}^\infty \calW_{\epsilon,1/\ell}=  \calW$. Thus, for any $\epsilon>0$, there exists some $\delta$, such that 
\begin{eqnarray}
\P\big((\bX, \bZ) \in \calW_{\epsilon, \delta}\big) >1-\epsilon. \label{eq:lemA:aux-1-8:3}
\end{eqnarray}
By Lemma 11.3 in \cite{azadkia2019simple}, $\|\bZ_{N(1)}-\bZ_1\| \to 0$ almost surely as $n \to \infty$, and thus $\bZ_{N(1)} \conD \bZ_1$. Note that both $(\bX, \bZ)$ and $(\bX_{N(1)}, \bZ_{N(1)})$ are absolutely continuous, and
the conditional distribution $\bX_{N(1)} \mid \bZ_{N(1)}=\bz$ is identical to $\bX_1 \mid \bZ_1=\bz$ for all $\bz$ in the support of $\bZ$. Then we have 
\begin{eqnarray}
(\bX_{N(1)}, \bZ_{N(1)}) \conD (\bX_1, \bZ_1), \label{eq:lemA:aux-1-8:4}
\end{eqnarray}
which yields that 
\begin{eqnarray}
\lim_{n\to \infty}\P\big((\bX_{N(1)}, \bZ_{N(1)}) \in \calW_{\epsilon, \delta}\big) = \P\big((\bX_1, \bZ_1) \in \calW_{\epsilon, \delta}\big). \label{eq:lemA:aux-1-8:5}
\end{eqnarray}
Combining \eqref{eq:lemA:aux-1-8:5} and \eqref{eq:lemA:aux-1-8:3}, for any $\epsilon>0$, there exists some small $\delta>0$ such that 
\begin{eqnarray*}
\P\big((\bX_{N(1)}, \bZ_{N(1)}) \in \calW_{\epsilon, \delta}\big) >1-2\epsilon 
\end{eqnarray*}
holds for sufficiently large $n$. 
Note that for any $(\bx, \bz)\in \calW_{\epsilon, \delta}$, 
\begin{eqnarray*}
&&\P\big(\|(\bX_{M(N(1))},\bZ_{M(N(1))}) - (\bX_{N(1)}, \bZ_{N(1)})\|>\epsilon \ \big | \ (\bX_{N(1)}, \bZ_{N(1)})=(\bx, \bz)\big) \cr
&=& \big\{1- \P\big((\bX, \bZ) \in \calB((\bx, \bz),\epsilon)\big)\big\}^{n-1} \cr
&\leq & (1-\delta)^{n-1}.
\end{eqnarray*}
Therefore, for sufficiently large $n$,
\begin{eqnarray*}
&& \P\big(\|(\bX_{M(N(1))},\bZ_{M(N(1))}) - (\bX_{N(1)}, \bZ_{N(1)})\|>\epsilon\big) \cr
&\leq& \P\big(\|(\bX_{M(N(1))},\bZ_{M(N(1))}) - (\bX_{N(1)}, \bZ_{N(1)})\|>\epsilon, \ (\bX_{N(1)}, \bZ_{N(1)}) \in \calW_{\epsilon, \delta}\big) \cr
&&+ \P\big((\bX_{N(1)}, \bZ_{N(1)}) \notin \calW_{\epsilon, \delta}\big)\cr
&\leq&  (1-\delta)^{n-1} + 2\epsilon,
\end{eqnarray*}
which can be arbitrarily small. This proves that 
\begin{eqnarray}
\|(\bX_{M(N(1))},\bZ_{M(N(1))}) - (\bX_{N(1)}, \bZ_{N(1)})\| \conP 0, \label{eq:lemA:aux-1-8:6}
\end{eqnarray}
which further yields $\|\bZ_{M(N(1))} - \bZ_{N(1)}\| \conP 0$. This combined with the fact that $\|\bZ_{N(1)}-\bZ_1\| \conP 0$ completes the proof of \eqref{eq:lemA:aux-1-8:1}. 

Next, we prove \eqref{eq:lemA:aux-1-8:2}. The procedure is similar to proving \eqref{eq:lemA:aux-1-8:1}. Let $\calZ \subseteq \mathbbR^{q}$ be the support of $\bZ$. For any $\epsilon >0$ and $\delta>0$, define the set 
\begin{eqnarray*}
\calZ_{\epsilon, \delta} = \Big\{\bz \in \calZ: \P\big(\bZ \in \calB(\bz,\epsilon)\big) >\delta \Big\}. 
\end{eqnarray*}
For any $\epsilon>0$, there exists some $\delta>0$, such that 
\begin{eqnarray*}
\P(\bZ \in \calZ_{\epsilon, \delta}) >1-\epsilon. 
\end{eqnarray*}
By Lemma 11.3 in \cite{azadkia2019simple}, $\|(\bX_{M(1)}, \bZ_{M(1)})-(\bX_1, \bZ_1)\| \to 0$ almost surely, which yields $\bZ_{M(1)} \conD \bZ_1$. Thus
\begin{eqnarray*}
\lim_{n\to \infty}\P(\bZ_{M(1)} \in \calZ_{\epsilon, \delta}) = \P(\bZ_1\in \calZ_{\epsilon, \delta}). 
\end{eqnarray*}
Then for any $\epsilon>0$, there exists  small $\delta>0$ such that 
\begin{eqnarray*}
\P(\bZ_{M(1)} \in \calZ_{\epsilon, \delta}) >1-2\epsilon 
\end{eqnarray*}
holds for sufficiently large $n$. Note that for any $\bz\in \calZ_{\epsilon, \delta}$, 
\begin{eqnarray*}
\P\big(\|\bZ_{N(M(1))}- \bZ_{M(1)}\|>\epsilon \ \big | \ \bZ_{M(1)}=\bz\big) 
\ = \ \big\{1- \P\big(\bZ \in \calB(\bz,\epsilon)\big)\big\}^{n-1} \ \leq \  (1-\delta)^{n-1}.
\end{eqnarray*}
Thus for sufficiently large $n$, 
\begin{eqnarray*}
&& \P(\|\bZ_{N(M(1))}- \bZ_{M(1)}\|>\epsilon) \cr
&\leq& \P\big(\|\bZ_{N(M(1))}- \bZ_{M(1)}\|>\epsilon, \bZ_{M(1)} \in \calZ_{\epsilon, \delta}\big) + \P(\bZ_{M(1)} \notin \calZ_{\epsilon, \delta})\cr
&\leq&  (1-\delta)^{n-1} + 2\epsilon,
\end{eqnarray*}
which can be arbitrarily small. This proves that $\|\bZ_{N(M(1))}- \bZ_{M(1)}\|\conP 0$. This combined with the fact that $\|\bZ_{M(1)}-\bZ_1\| \conP 0$ completes the proof of \eqref{eq:lemA:aux-1-8:2}. \end{proof}

\begin{lemmaA} \label{lemA:aux-1-9}
Let $\bU \in \mathbbR^d$ be a $d$-dimensional random vector. Assume that $\bU$ is absolutely continuous, and admits a continuous density function over its support. 
Let $\{\bU_n\}_{n=1}^\infty$ be a sequence of $d$-dimensional random vector, such that $\bU_n \conD \bU$ as $n \to \infty$. Let $\{\bV_n\}_{n=1}^\infty$ be a sequence of $d$-dimensional random vector, satisfying that $\|\bV_n - \bU_n\| \conP 0$ as $n \to \infty$. Then for any measurable function $f: \mathbbR^d \mapsto \mathbbR$, we have
\begin{eqnarray*}
f(\bV_n) - f(\bU_n) \conP 0, \quad \text{as } n \to \infty.
\end{eqnarray*}
\end{lemmaA}
\begin{proof}
Fix some $\epsilon >0$. For any random vector $\bU$, there exists a connected compact set $A \subseteq \mathbbR^d$, such that $\P(\bU \in A) > 1- \epsilon$. Since $\bU_n \conD \bU$ and $\|\bV_n - \bU_n\| \conP 0$, it is clear that both $\P(\bU_n \in A)$ and $\P(\bV_n \in A)$ converge to $\P(\bU \in A)$. Thus, for sufficiently large $n$, we have
\begin{eqnarray*}
\P(\bU_n \in A) > 1- 2 \epsilon,  \quad \text{and} \quad  \P(\bV_n \in A) > 1- 2 \epsilon.
\end{eqnarray*} 
Note that $f$ is measurable. By applying the Lusin's theorem, there exists a continuous function $g$, such that $\lambda(B) < \epsilon$, where
\begin{eqnarray*}
B = \{ \bx \in A: f(\bx) \neq g(\bx)\},
\end{eqnarray*}
and $\lambda(\cdot)$ denotes the Lebesgue measure in $\mathbbR^d$. 

Since 	$\bU$ is absolutely continuous and admits a  continuous density function, we have
$\P(\bU \in B) < C \cdot \epsilon$ for some constant $C >0$ not depending on $\epsilon$. Also, note that both $\P(\bU_n \in B)$ and $\P(\bV_n \in B)$ converge to $\P(\bU \in B)$, it follows that
\begin{eqnarray*}
\P(\bU_n \in B) < C \cdot \epsilon,  \quad \text{and}  \quad \P(\bV_n \in B) < C \cdot \epsilon,
\end{eqnarray*}
hold for sufficiently large $n$. 

By applying the triangle inequality, 
\begin{eqnarray*}
|f(\bV_n) - f(\bU_n)| < |f(\bU_n) - g(\bU_n)| + |g(\bV_n) - g(\bU_n)| + |f(\bV_n) - g(\bV_n)| . 
\end{eqnarray*}
For the first term $ |f(\bU_n) - g(\bU_n)| $, we have
\begin{eqnarray*}
\P( |f(\bU_n) - g(\bU_n)| =0) \ \geq \  \P(\bU_n \in A) - \P(\bU_n \in B) \ > \ 1-(C+2)\cdot  \epsilon,
\end{eqnarray*}
holds for sufficiently large $n$. The same inequality also holds for the third term $|f(\bV_n) - g(\bV_n)|$.

Next we examine the second term $|g(\bV_n) - g(\bU_n)| $. Since $A$ is a connected compact set and $g$ is continuous, then $g$ is uniformly continuous in $A$.  Thus there exists some $\delta >0$, such that $|g(\bV_n) - g(\bU_n)| < \epsilon$ whenever $\bV_n, \bU_n \in A$, and $\|\bV_n - \bU_n\| < \delta$. By $\|\bV_n - \bU_n\| \conP 0$, we have $\P(\|\bV_n - \bU_n\| > \delta) \to 0$. Therefore, for sufficiently large $n$, 
\begin{eqnarray*}
\P(|g(\bV_n) - g(\bU_n)|> \epsilon)  \ \leq  \ \P(\bV_n \notin A) + \P(\bU_n \notin A) + \P(\|\bV_n - \bU_n\| > \delta) \ < \ 5 \cdot \epsilon,
\end{eqnarray*}
Combining the above inequalities, it follows that
\begin{eqnarray*}
&&\P(|f(\bV_n) - f(\bU_n)|  > \epsilon) \cr
& \leq& \P(|g(\bV_n) - g(\bU_n)|> \epsilon) + \P( |f(\bU_n) - g(\bU_n)| \neq 0)  + \P( |f(\bV_n) - g(\bV_n)| \neq 0)  \cr
&<& 5 \cdot \epsilon + 2 \cdot (C+2) \cdot \epsilon \ < \ C_2 \cdot \epsilon
\end{eqnarray*}
holds for sufficiently large $n$, where $C_2 >0$ is some constant. 
Thus, by noting that $\epsilon$ can be arbitrarily small, we prove that $ |f(\bV_n) - f(\bU_n)| \conP 0$.
\end{proof}

\subsection{Auxiliary lemmas: Group 2}

\begin{lemmaA} \label{lemA:aux-2-1}
For any $\epsilon>0$, there exists a simple function of form
\begin{eqnarray*}
q(u,v) = \sum_{j=1}^m c_j \Ind_{B_j}(u)\Ind_{C_j}(v)
\end{eqnarray*}
with $m<\infty$, $c_j>0$, and $B_j, C_j$ being intervals of $\mathbbR$, such that
\begin{eqnarray*}
\sup_{(u,v)\in \mathbbR^2} | q(u,v)- F_Y(u \wedge v) | <\epsilon.
\end{eqnarray*}
\end{lemmaA}

\begin{proof} Note that $F_Y(u \wedge v) = \min\{F_Y(u), F_Y(v)\}$ is a continuous bivariate function of $(F_Y(u), F_Y(v))$. This result follows directly by applying Lemma~A.1 of \cite{Gao_Li_2024} with $h(\cdot,\cdot) = \min\{\cdot,\cdot\}$. \end{proof}

\begin{lemmaA} \label{lemA:aux-2-2}
For any $\epsilon>0$, there exists a simple function of form
\begin{eqnarray*}
q(u,v,w) = \sum_{j=1}^m c_j \Ind_{B_j}(u)\Ind_{C_j}(v)\Ind_{D_j}(w)
\end{eqnarray*}
with $m<\infty$, $c_j>0$ and $B_j, C_j, D_j$ being intervals of $\mathbbR$, such that
\begin{eqnarray}
\sup_{(u,v)\in \mathbbR^2} | q(u,v,w)- F_Y(u \wedge v) \cdot F_Y(u \wedge w)  | <\epsilon. \label{eq:lemA:aux-2-2:1}
\end{eqnarray}
\end{lemmaA}

\begin{proof} Fix some $\epsilon >0$. By Lemma \ref{lemA:aux-2-1}, there exist simple functions $$q_1(u,v) = \sum_{j=1}^{m_1} c_{1,j} \Ind_{B_{1,j}}(u)\Ind_{C_{1,j}}(v)$$ and $$q_2(u,w) = \sum_{j=1}^{m_2}  c_{2,j} \Ind_{B_{2,j}}(u)\Ind_{C_{2,j}}(w)$$ with 
\[
\{B_{1,j}, C_{1,j}\}_{j=1}^{m_1}~~~ {\rm and}~~~  \{B_{2,j}, C_{2.j}\}_{j=1}^{m_2}
\]
being intervals in $\mathbbR$, 
such that 
\[
\sup_{(u,v)\in \mathbbR^2} | q_1(u,v)- F_Y(u \wedge v) | <\epsilon~~~ {\rm and}~~~ \sup_{(u,w)\in \mathbbR^2} | q_2(u,w)- F_Y(u \wedge w) | <\epsilon.
\]
Let
\begin{eqnarray}
q(u,v,w)&:=& q_1(u,v) \cdot q_2(u,w) \cr
&=& \Big[\sum_{j=1}^{m_1} c_{1,j} \Ind_{B_{1,j}}(u)\Ind_{C_{1,j}}(v)\Big]\cdot \Big[ \sum_{j=1}^{m_2}  c_{2,j} \Ind_{B_{2,j}}(u)\Ind_{C_{2,j}}(w)\Big] \cr
&=&
\sum_{i=1}^{m_1}\sum_{j=1}^{m_2} c_{1,i}c_{2,j} \Ind_{B_{1,i} \cap B_{2,j}}(u)\Ind_{C_{1,i}}(v) \Ind_{C_{2,j}}(w), \label{eq:lemA:aux-2-2:2}
\end{eqnarray}
which is a simple function of $(u,v, w)$. Write $\Delta_1 =  q_1(u,v)- F_Y(u \wedge v) $ and $\Delta_2 =  q_2(u,w)- F_Y(u \wedge w) $. Then
\begin{eqnarray*}
|q(u,v,w) - F_Y(u \wedge v) \cdot F_Y(u \wedge w)| 
&=&
|q_1(u,v)\cdot q_2(u,w) - F_Y(u \wedge v) \cdot F_Y(u \wedge w)|\cr
&=& |\Delta_1 \cdot F_Y(u \wedge w) + \Delta_2 \cdot  F_Y(u \wedge v) + \Delta_1 \cdot \Delta_2| \cr
&\leq &  2\epsilon + \epsilon^2.
\end{eqnarray*}
Since $2\epsilon + \epsilon^2$ above could be arbitrarily small,
we prove that the simple  function $q(u,v,w)$ in  \eqref{eq:lemA:aux-2-2:2} satisfies \eqref{eq:lemA:aux-2-2:1}. \end{proof}

\begin{lemmaA} \label{lemA:aux-2-3}
Define the function $g:\mathbbR^{2p+3q} \mapsto [0,\infty)$,
\begin{eqnarray}
g\big( (\bx_1, \bz_1), \bz_2, (\bx_3,\bz_3)\big): = \int \int \int F_Y(u\wedge v) \cdot F_Y(u \wedge w) 
\d \tmu_{(\bx_1, \bz_1)}(u) \d \tmu_{\bz_2}(v) \d \tmu_{(\bx_3,\bz_3)}(w), \qquad  \label{eq:lemA:aux-2-3:1}
\end{eqnarray}
where $\tmu_\bz(\cdot)$ is the conditional law of $Y$ given $\bZ=\bz$, and $\tmu_{(\bx, \bz)}(\cdot)$ is the conditional law of $Y$ given $(\bX,\bZ)=(\bx ,\bz)$ as defined in  Lemma \ref{lemA:mu}.
Then as $n\to \infty$, we have the following convergence results:
\begin{eqnarray}
&& g\big((\bX_1, \bZ_1), \bZ_{N(1)}, (\bX_{M(1)}, \bZ_{M(1)})\big) - g\big((\bX_1, \bZ_1), \bZ_1, (\bX_1, \bZ_1)\big)  \conP 0, \label{eq:lemA:aux-2-3:2}
\end{eqnarray}
\end{lemmaA}

\begin{proof}
Fix some $\epsilon >0$ and $\delta >0$. From Lemma \ref{lemA:aux-2-2}, there exists a simple function $	q(u,v,w) = \sum_{j=1}^m c_j \Ind_{B_j}(u)\Ind_{C_j}(v)\Ind_{D_j}(w)$ such that $	\sup_{(u,v,w)\in \mathbbR^3} | q(u,v,w)- F_Y(u \wedge v) \cdot F_Y(u \wedge w)  | <\epsilon$. Define the function 
\begin{eqnarray}
r\big( (\bx_1, \bz_1), \bz_2, (\bx_3,\bz_3)\big) &:=& \int \int \int q(u,v,w) 	\d \tmu_{(\bx_1, \bz_1)}(u) \d \tmu_{\bz_2}(v) \d \tmu_{(\bx_3,\bz_3)}(w)\cr
&=&
\sum_{j=1}^m c_j \int \int \int \Ind_{B_j}(u)\Ind_{C_j}(v)\Ind_{D_j}(w) 	\d \tmu_{(\bx_1, \bz_1)}(u) \d \tmu_{\bz_2}(v) \d \tmu_{(\bx_3,\bz_3)}(w)\cr
&=&
\sum_{j=1}^m c_j \cdot \tmu_{(\bx_1, \bz_1)}(B_j) \tmu_{\bz_2}(C_j)  \tmu_{(\bx_3,\bz_3)}(D_j) . \label{eq:lemA:aux-2-3:3}
\end{eqnarray}
Then 
\begin{eqnarray*}
&&\sup_{((\bx_1, \bz_1), \bz_2, (\bx_3,\bz_3)) \in \mathbbR^{2p+3q} } \Big| g\big((\bx_1, \bz_1), \bz_2, (\bx_3,\bz_3)\big) - r\big((\bx_1, \bz_1), \bz_2, (\bx_3,\bz_3)\big) \Big|  \cr
&\leq&
\sup_{((\bx_1, \bz_1), \bz_2, (\bx_3,\bz_3)) \in \mathbbR^{2p+3q}  } 
\int \int \int \big| F_Y(u\wedge v) \cdot F_Y(u \wedge w) - q(u,v,w) \big|\d \tmu_{(\bx_1, \bz_1)}(u) \d \tmu_{\bz_2}(v) \d \tmu_{(\bx_3,\bz_3)}(w)\cr
&<& \epsilon.
\end{eqnarray*}

Note that 
\begin{eqnarray}
&& r\big((\bX_1, \bZ_1), \bZ_{N(1)}, (\bX_{M(1)}, \bZ_{M(1)})\big) - r\big((\bX_1, \bZ_1), \bZ_1, (\bX_1, \bZ_1)\big)  \cr
&=& \sum_{j=1}^m 
c_j \cdot \tmu_{\bZ_1}(B_j)\cdot\big[ \tmu_{\bZ_{N(1)}}(C_j)  \tmu_{(\bX_{M(1)},\bZ_{M(1)})}(D_j)   - \tmu_{\bZ_1}(C_j) \tmu_{(\bX_1,\bZ_1)}(D_j) \big]. \label{eq:lemA:aux-2-3:4}
\end{eqnarray}
By Lemma 11.3 in \cite{azadkia2019simple},  $\|\bZ_{N(1)}-\bZ_1\|\to 0$ and $\|(\bX_{M(1)}, \bZ_{M(1)})-(\bX_1, \bZ_1)\| \to 0$ almost surely as $n \to \infty$. 
Note that  for any Borel set $A$, both $\bz \mapsto \tmu_\bz(A)$ and $(\bx,\bz) \mapsto \tmu_{(\bx,\bz)}(A)$ 
are measurable mappings. 
Then by Lemma 11.7 in \cite{azadkia2019simple}, we have that 
$\tmu_{\bZ_{N(1)}}(C_j) - \tmu_{\bZ_1}(C_j) \conP 0 $ and $\tmu_{(\bX_{M(1)}, \bZ_{M(1)})}(D_j) - \tmu_{(\bX_1,\bZ_1)}(D_j) \conP 0 $ for all $j$, which further implies that $\tmu_{\bZ_{N(1)}}(C_j)  \tmu_{(\bX_{M(1)},\bZ_{M(1)})}(D_j)   - \tmu_{\bZ_1}(C_j) \tmu_{(\bX_1,\bZ_1)}(D_j) \conP 0$. This combined with \eqref{eq:lemA:aux-2-3:4} gives that $r\big((\bX_1, \bZ_1), \bZ_{N(1)}, (\bX_{M(1)}, \bZ_{M(1)})\big) - r\big((\bX_1, \bZ_1), \bZ_1, (\bX_1, \bZ_1)\big)   \conP 0$. Therefore, for sufficiently large $n$, we have
\begin{eqnarray*}
\P\Big( \Big| r\big((\bX_1, \bZ_1), \bZ_{N(1)}, (\bX_{M(1)}, \bZ_{M(1)})\big) - r\big((\bX_1, \bZ_1), \bZ_1, (\bX_1, \bZ_1)\big)  \Big| >\epsilon \Big) < \delta.
\end{eqnarray*}
Note that
\begin{eqnarray*}
&&\Big|g\big((\bX_1, \bZ_1), \bZ_{N(1)}, (\bX_{M(1)}, \bZ_{M(1)})\big) - g\big((\bX_1, \bZ_1), \bZ_1, (\bX_1, \bZ_1)\big) \Big| \cr
&\leq&
\Big| g\big((\bX_1, \bZ_1), \bZ_{N(1)}, (\bX_{M(1)}, \bZ_{M(1)})\big) - r\big((\bX_1, \bZ_1), \bZ_{N(1)}, (\bX_{M(1)}, \bZ_{M(1)})\big) \Big| \cr
&& +
\Big|r\big((\bX_1, \bZ_1), \bZ_{N(1)}, (\bX_{M(1)}, \bZ_{M(1)})\big) - r\big((\bX_1, \bZ_1), \bZ_1, (\bX_1, \bZ_1)\big) \Big| \cr
&& +
\Big| g\big((\bX_1, \bZ_1), \bZ_1, (\bX_1, \bZ_1)\big)  -
g\big((\bX_1, \bZ_1), \bZ_1, (\bX_1, \bZ_1)\big) \Big| \cr
&\leq &  \Big|r\big((\bX_1, \bZ_1), \bZ_{N(1)}, (\bX_{M(1)}, \bZ_{M(1)})\big) - r\big((\bX_1, \bZ_1), \bZ_1, (\bX_1, \bZ_1)\big) \Big| + 2 \epsilon.
\end{eqnarray*}
Combining the above results, for sufficiently large $n$, 
\begin{eqnarray*}
\P\Big(|\Big|g\big((\bX_1, \bZ_1), \bZ_{N(1)}, (\bX_{M(1)}, \bZ_{M(1)})\big) - g\big((\bX_1, \bZ_1), \bZ_1, (\bX_1, \bZ_1)\big) \Big| > 3\epsilon \Big) < \delta.
\end{eqnarray*}
The proof is completed given that $\epsilon$ and $\delta$ are arbitrary.
\end{proof}

\begin{lemmaA} \label{lemA:aux-2-4}
Define the function $g^*:\mathbbR^{3p+3q} \mapsto [0,\infty)$,
\begin{eqnarray}
&& g^*\big( (\bx_1, \bz_1), (\bx_2,\bz_2), (\bx_3,\bz_3)\big) \cr
&: =& \int \int \int F_Y(u\wedge v) \cdot F_Y(u \wedge w) 
\d \tmu_{(\bx_1, \bz_1)}(u) \d \tmu_{(\bx_2, \bz_2)}(v) \d \tmu_{(\bx_3,\bz_3)}(w), \qquad  \label{eq:lemA:aux-2-4:1}
\end{eqnarray}
Then as $n\to \infty$, we have the following convergence results:
\begin{eqnarray*}
&& g^*\big((\bX_{N(1)}, \bZ_{N(1)}), (\bX_1, \bZ_1), (\bX_{M(N(1))}, \bZ_{M(N(1))})\big) \cr
&& \hspace{3cm}- g^*\big((\bX_{N(1)}, \bZ_{N(1)}), (\bX_1, \bZ_1), (\bX_{N(1)}, \bZ_{N(1)})\big)  \conP 0. 
\end{eqnarray*}
\end{lemmaA}

\begin{proof}
The proof is similar to that of Lemma \ref{lemA:aux-2-3}. We provide some sketches here.

Fix some $\epsilon >0$ and $\delta >0$. From Lemma \ref{lemA:aux-2-2}, there exists a simple function $	q(u,v,w) = \sum_{j=1}^m c_j \Ind_{B_j}(u)\Ind_{C_j}(v)\Ind_{D_j}(w)$ such that $	\sup_{(u,v,w)\in \mathbbR^3} | q(u,v,w)- F_Y(u \wedge v) \cdot F_Y(u \wedge w)  | <\epsilon$. Define the function 
\begin{eqnarray*}
&&r^*\big( (\bx_1, \bz_1),(\bx_2, \bz_2), (\bx_3,\bz_3)\big) \cr
&:=& \int \int \int q(u,v,w) 	\d \tmu_{(\bx_1, \bz_1)}(u) \d \tmu_{(\bx_2, \bz_2)}(v) \d \tmu_{(\bx_3,\bz_3)}(w)\cr
&=&
\sum_{j=1}^m c_j \cdot \tmu_{(\bx_1, \bz_1)}(B_j)  \tmu_{(\bx_2, \bz_2)}(C_j)  \tmu_{(\bx_3,\bz_3)}(D_j) .
\end{eqnarray*}
Then 
\begin{eqnarray*}
&&\sup_{((\bx_1, \bz_1), (\bx_2, \bz_2), (\bx_3,\bz_3)) \in \mathbbR^{3p+3q} } \Big| g^*\big((\bx_1, \bz_1), (\bx_2, \bz_2), (\bx_3,\bz_3)\big) - r^*\big((\bx_1, \bz_1), (\bx_2, \bz_2), (\bx_3,\bz_3)\big) \Big| < \epsilon.
\end{eqnarray*}
Note that 
\begin{eqnarray*}
&& r^*\big((\bX_{N(1)}, \bZ_{N(1)}), (\bX_1, \bZ_1), (\bX_{M(N(1))}, \bZ_{M(N(1))})\big) \cr
&&-
r^*\big((\bX_{N(1)}, \bZ_{N(1)}), (\bX_1, \bZ_1), (\bX_{N(1)}, \bZ_{N(1)})\big) \cr
&=& \sum_{j=1}^m 
c_j \cdot \tmu_{(\bX_{N(1)}, \bZ_{N(1)})}(B_j)\cdot
\tmu_{(\bX_1, \bZ_1)}(C_j) \cdot 
\big[  \tmu_{(\bX_{M(N(1))}, \bZ_{M(N(1))})}(D_j)  - \tmu_{(\bX_{N(1)},\bZ_{N(1)})}(D_j)  \big]. 
\end{eqnarray*}
By \eqref{eq:lemA:aux-1-8:6}, $\|(\bX_{M(N(1))},\bZ_{M(N(1))}) - (\bX_{N(1)}, \bZ_{N(1)})\| \conP 0$. By \eqref{eq:lemA:aux-1-8:4}, $(\bX_{N(1)}, \bZ_{N(1)}) \conD (\bX_1, \bZ_1)$. Note that 
for any Borel set $A$, $(\bx,\bz) \mapsto \tmu_{(\bx,\bz)}(A)$ 
is measurable mapping. 
These three facts combined with 
Lemma \ref{lemA:aux-1-9} give that $  \tmu_{(\bX_{M(N(1))}, \bZ_{M(N(1))})}(D_j)  - \tmu_{(\bX_{N(1)},\bZ_{N(1)})}(D_j) \conP 0$ for all $j$. Therefore, for sufficiently large $n$, we have
\begin{eqnarray*}
&&\P\Big( \Big| r^*\big((\bX_{N(1)}, \bZ_{N(1)}), (\bX_1, \bZ_1), (\bX_{M(N(1))}, \bZ_{M(N(1))})\big) \cr
&&\hspace{3cm}
-r^*\big((\bX_{N(1)}, \bZ_{N(1)}), (\bX_1, \bZ_1), (\bX_{N(1)}, \bZ_{N(1)})\big)   \Big| >\epsilon \Big) < \delta.
\end{eqnarray*}
By the similar arguments in the proof of Lemma \ref{lemA:aux-2-3}, for sufficiently large $n$, we have
\begin{eqnarray*}
&&\P\Big(\Big|g^*\big((\bX_{N(1)}, \bZ_{N(1)}), (\bX_1, \bZ_1), (\bX_{M(N(1))}, \bZ_{M(N(1))})\big) \cr
&& 
\hspace{3cm} - g^*\big((\bX_{N(1)}, \bZ_{N(1)}), (\bX_1, \bZ_1), (\bX_{N(1)}, \bZ_{N(1)})\big) \Big| > 3\epsilon \Big) < \delta.
\end{eqnarray*}
The proof is completed given that $\epsilon$ and $\delta$ are arbitrary.
\end{proof}

\begin{lemmaA} \label{lemA:aux-2-5}
Recall the function $g$ defined in \eqref{eq:lemA:aux-2-3:1}. Then as $n\to \infty$,  we have
\begin{eqnarray*}
&& g\big((\bX_{N(1)}, \bZ_{N(1)}), \bZ_{N(1)}, (\bX_{N(1)}, \bZ_{N(1)})\big) - g\big((\bX_{N(1)}, \bZ_{N(1)}), \bZ_1, (\bX_{N(1)}, \bZ_{N(1)})\big) \conP 0, \cr
&& g\big((\bX_{M(1)}, \bZ_{M(1)}), \bZ_{N(M(1))}, (\bX_1, \bZ_1) \big) - 
g\big((\bX_1, \bZ_1), \bZ_1, (\bX_1, \bZ_1)\big) \conP 0.
\end{eqnarray*}
\end{lemmaA}

\begin{proof}
This lemma can be proved by a similar approach as that of Lemmas \ref{lemA:aux-2-3} and \ref{lemA:aux-2-4}. The proof utilizes the following convergence results: (i). $\|(\bX_{M(1)}, \bZ_{M(1)})-(\bX_1, \bZ_1)\| \conas 0$, and $\|\bZ_{N(1)} -\bZ_1\|\conas 0 $ (By Lemma 11.3 in \cite{azadkia2019simple}); (ii). $\|\bZ_{N(M(1))} - \bZ_{1}\| \conP 0$ (by Lemma \ref{lemA:aux-1-8}).
\end{proof}

\begin{lemmaA} \label{lemA:aux-2-6}
Define the function $\tildeg: \mathbbR^{2p+2q} \mapsto [0,\infty)$,
\begin{eqnarray}
\tildeg\big((\bx_1,\bz_1), (\bx_2,\bz_2)\big): = \int \int F_Y(u\wedge v) 
\d \tmu_{(\bx_1,\bz_1)}(u) \d \tmu_{(\bx_2,\bz_2)}(v).  \label{eq:lemA:aux-2-6:1}
\end{eqnarray}
Then as $n \to \infty$, we have the following convergence results: 
\begin{eqnarray*}
&&\tildeg\big((\bX_1,\bZ_1), (\bX_{M(1)}, \bZ_{M(1)})\big) - \tildeg\big((\bX_1,\bZ_1), (\bX_1, \bZ_1)\big) \conP 0, \cr
&&\tildeg\big((\bX_{M(1)},\bZ_{M(1)}), (\bX_{M(1)}, \bZ_{M(1)})\big) - \tildeg\big((\bX_1,\bZ_1), (\bX_{M(1)}, \bZ_{M(1)})\big) \conP 0, \cr
&&\tildeg\big((\bX_{N(1)}, \bZ_{N(1)}), (\bX_{M(N(1))}, \bZ_{M(N(1))})\big) -
\tildeg\big((\bX_{N(1)}, \bZ_{N(1)}), (\bX_{N(1)}, \bZ_{N(1)})\big)
\conP 0.
\end{eqnarray*}
\end{lemmaA}
\begin{proof}
This lemma can be proved by a similar approach as that of Lemmas \ref{lemA:aux-2-3} and \ref{lemA:aux-2-4}. The only difference is that we need to apply Lemma~\ref{lemA:aux-2-1} instead of  Lemma~\ref{lemA:aux-2-2} and find the bivariate simple function $q(u,v)$ such that $\sup_{(u,v)\in \mathbbR^2} | q(u,v)- F_Y(u \wedge v) | <\epsilon$. The proof utilizes the following convergence results: (i). $\|(\bX_{M(1)}, \bZ_{M(1)})-(\bX_1, \bZ_1)\| \conas 0$ (By Lemma 11.3 in \cite{azadkia2019simple}); (ii). $\|(\bX_{M(N(1))},\bZ_{M(N(1))}) - (\bX_{N(1)}, \bZ_{N(1)})\| \conP 0$ (by \eqref{eq:lemA:aux-1-8:6}).
\end{proof}

\begin{lemmaA} \label{lemA:aux-2-7}
Define the function: 
$\tildeg^\dagger: \mathbbR^{p+2q} \mapsto [0,\infty)$,
\begin{eqnarray}
\tildeg^\dagger\big((\bx_1,\bz_1), \bz_2\big)&: =& \int \int F_Y(u\wedge v) 
\d \tmu_{(\bx_1,\bz_1)}(u) \d \tmu_{\bz_2}(v) \cr
&=& 
\int \tildeg\big((\bx_1,\bz_1), (\bx,\bz_2)\big) 
\d \mu_{\bX=\bx \mid \bZ = \bz_2}(\bx).
\label{eq:lemA:aux-2-7:1}
\end{eqnarray}
Then as $n \to \infty$, we have the following convergence results: 
\begin{eqnarray*}
&&\tildeg^\dagger \big((\bX_1,\bZ_1), \bZ_{N(1)})\big) - \tildeg^\dagger\big((\bX_1,\bZ_1), \bZ_1\big) \conP 0, \cr
&&\tildeg^\dagger\big( (\bX_{N(1)}, \bZ_{N(1)}), \bZ_{N(1)}\big) - \tildeg^\dagger\big( (\bX_{N(1)}, \bZ_{N(1)}), \bZ_1\big) \conP 0, \cr
&& \tildeg^\dagger \big((\bX_{M(1)}, \bZ_{M(1)}), \bZ_{M(1)} \big)  
- \tildeg^\dagger \big((\bX_1,\bZ_1), \bZ_1\big) \conP 0.
\end{eqnarray*}
\end{lemmaA}
\begin{proof}
This lemma can be proved by a similar approach as that of Lemma \ref{lemA:aux-2-6}. 
\end{proof}

\begin{lemmaA} \label{lemA:aux-2-8}
Recall the function $g$ defined in \eqref{eq:lemA:aux-2-3:1}. We have
\begin{eqnarray*}
&&\lim_{n\to \infty} n\cdot \E\Big\{g\big((\bX_{M(1)}, \bZ_{M(1)}), \bZ_2,  (\bX_1,\bZ_1) \big) \cdot \Ind(M(1) = N(2))\Big\} \cr
&=&
\E\Big\{ g\big((\bX_1, \bZ_1), \bZ_1, (\bX_1, \bZ_1) \Big\}.
\end{eqnarray*}
\end{lemmaA}

\begin{proof} 
Fix some $\epsilon>0$ and $\delta>0$. 
Let 
\begin{eqnarray*}
r\big( (\bx_1, \bz_1), \bz_2, (\bx_3,\bz_3)\big) &=& 
\sum_{j=1}^m c_j \cdot \tmu_{(\bx_1, \bz_1)}(B_j) \tmu_{\bz_2}(C_j)  \tmu_{(\bx_3,\bz_3)}(D_j) 
\end{eqnarray*}
be the simple function as defined in \eqref{eq:lemA:aux-2-3:3}, such that 
\begin{eqnarray*}
\sup_{((\bx_1, \bz_1), \bz_2, (\bx_3,\bz_3)) \in \mathbbR^{2p+3q} } \Big| g\big((\bx_1, \bz_1), \bz_2, (\bx_3,\bz_3)\big) - r\big((\bx_1, \bz_1), \bz_2, (\bx_3,\bz_3)\big) \Big|  \ < \ \epsilon.
\end{eqnarray*}
Since the function $g$ is bounded, the function $r$ is also bounded. 

Define
\begin{eqnarray*}
S_1 &:=& n \cdot g\big((\bX_{M(1)}, \bZ_{M(1)}), \bZ_2,  (\bX_1,\bZ_1) \big)  \cdot \Ind\big(M(1)=N(2)\big), \cr
S_2 &:=& n \cdot  r\big((\bX_{M(1)}, \bZ_{M(1)}), \bZ_2,  (\bX_1,\bZ_1) \big) 
\cdot \Ind\big(M(1)=N(2)\big), \cr
S_3 &:=& n \cdot   r\big((\bX_1, \bZ_1), \bZ_1,  (\bX_1,\bZ_1) \big)  \cdot \Ind\big(M(1)=N(2)\big), \cr
S_4 &:=& n \cdot g\big((\bX_1, \bZ_1), \bZ_1,  (\bX_1,\bZ_1) \big)\cdot \Ind\big(M(1)=N(2)\big), \cr
S_5 &:=& g\big((\bX_1, \bZ_1), \bZ_1,  (\bX_1,\bZ_1) \big).
\end{eqnarray*}
Then it suffices to prove $\lim_{n \to \infty} \E(S_1) = \E(S_5)$. To this end,
we will bound the differences in expectations between each pair of consecutive terms.

\noindent \textbf{(i)} Case $|S_1-S_2|$. 
\begin{eqnarray}
\E(|S_1-S_2|)
&\leq& \epsilon \cdot  n \cdot \E\big\{\Ind\big(M(1)=N(2)\big)\big\} \ = \ \epsilon \cdot \frac{n}{n-1}\E\Big\{\sum_{k=2}^n \Ind\big(N(k)=M(1)\big)\Big\}.\quad  \label{eq:lemA:aux-2-8:1}
\end{eqnarray}
By Corollary S1 of \cite{MR682809}, $\sum_{k=2}^n \Ind( N(k) = M(1))$ is bounded from above by some constant. This proves that $\E(|S_1-S_2|) = O(\epsilon)$.

\noindent \textbf{(ii)} Case $|S_2-S_3|$.  Note that 
\begin{eqnarray*}
&& r\big((\bX_{M(1)}, \bZ_{M(1)}), \bZ_2,  (\bX_1,\bZ_1) \big) -    r\big((\bX_1, \bZ_1), \bZ_1,  (\bX_1,\bZ_1) \big)   \cr
&=&
\sum_{j=1}^m 
c_j \cdot \tmu_{(\bX_1,\bZ_1)}(D_j)\cdot \big[\tmu_{(\bX_{M(1)},\bZ_{M(1)})}(B_j)  \tmu_{\bZ_2}(C_j)
- \tmu_{(\bX_1,\bZ_1)}(B_j)  \tmu_{\bZ_1}(C_j) \big].
\end{eqnarray*} 
By Assumption~\ref{assump_4.5}, for any fixed $t \in \mathbb{R}$, the mapping $(\bx,\bz) \mapsto G_{(\bx,\bz)}(t) = \E[\Ind(Y \ge t)\mid  (\bX,\bZ)=(\bx,\bz)]$ is continuous almost everywhere on $\supp((\bX,\bZ))$.
It follows that for any interval $A \subseteq \mathbbR$, the mapping $(\bx,\bz) \mapsto \tmu_{(\bx,\bz)}(A) = \E( \Ind(Y \in A) \mid (\bX,\bZ)=(\bx,\bz))$ is continuous almost everywhere on $\supp((\bX,\bZ))$. 
Similarly,  
the mapping $\bz \mapsto \tmu_{\bz}(A) = \E( \Ind(Y \in A) \mid \bZ=\bz)$ is also continuous almost everywhere on $\supp(\bZ)$.
By Lemmas \ref{lemA:aux-1-6} and \ref{lemA:aux-1-7}, 
for any $\epsilon_0 >0$ and $\delta>0$, the following holds for sufficiently large $n$:
\begin{eqnarray*}
&&\P\Big( \big|\tmu_{(\bX_{M(1)},\bZ_{M(1)})}(B_j)  - \tmu_{(\bX_1,\bZ_1)}(B_j) \big| > \epsilon_0 \ \Big | \ M(1)=N(2)\Big) < \delta, \cr
\text{and} &&\P\Big(| \tmu_{\bZ_2}(C_j) - \tmu_{\bZ_1}(C_j)| > \epsilon_0 \ \Big | \ M(1)=N(2)\Big) < \delta,
\qquad \text{for } j=1,\dots, m.
\end{eqnarray*}
Therefore, for sufficiently large $n$, we have
\begin{eqnarray}
\P\Big( \big|r\big((\bX_{M(1)}, \bZ_{M(1)}), \bZ_2,  (\bX_1,\bZ_1) \big) -    r\big((\bX_1, \bZ_1), \bZ_1,  (\bX_1,\bZ_1) \big)  \big| > \epsilon  \ \Big | \ M(1)=N(2)\Big)  < \delta.  \cr \label{eq:lemA:aux-2-8:2}
\end{eqnarray}
Since the function $r$ is bounded from above, we further have
\begin{eqnarray*}
\E\Big( \big|r\big((\bX_{M(1)}, \bZ_{M(1)}), \bZ_2,  (\bX_1,\bZ_1) \big) -    r\big((\bX_1, \bZ_1), \bZ_1,  (\bX_1,\bZ_1) \big)  \big| \ \Big | \ M(1)=N(2)\Big)  = O(\epsilon+\delta).
\end{eqnarray*}
By \citet[Lemma 7.4, Equation (28)]{Shi_Drton_Han_2024_Bernoulli}, $\lim_{n\to \infty} n \P\big(  M(1)=N(2)\big) = 1$. Combining these gives that
\begin{eqnarray}
&&\E(|S_2-S_3|) \cr
&= & n\cdot  \E\Big( \big|r\big((\bX_{M(1)}, \bZ_{M(1)}), \bZ_2,  (\bX_1,\bZ_1) \big) -    r\big((\bX_1, \bZ_1), \bZ_1,  (\bX_1,\bZ_1) \big)  \big|  \cdot \Ind\big(M(1)=N(2)\big)\Big) \cr
&=&
\E\Big(\big|r\big((\bX_{M(1)}, \bZ_{M(1)}), \bZ_2,  (\bX_1,\bZ_1) \big) -    r\big((\bX_1, \bZ_1), \bZ_1,  (\bX_1,\bZ_1) \big)  \big|   \ \Big | \ M(1)=N(2)\Big) \cr
&&   \times  n \cdot  \P\big(  M(1)=N(2)\big)  \cr
&=&
O(\epsilon+\delta). \label{eq:lemA:aux-2-8:3}
\end{eqnarray}

\noindent \textbf{(iii)} Case $|S_3-S_4|$. 

Using the similar approach for proving Case (i) of $|S_1-S_2|$, we have
\begin{eqnarray*}
\E(|S_3-S_4|)
\leq \epsilon \cdot \frac{n}{n-1}\E\Big\{\sum_{k=2}^n \Ind\big(N(k)=M(1)\big)\Big\} = O(\epsilon).
\end{eqnarray*}

\noindent \textbf{(iv)} Case $S_4$. 
\begin{eqnarray*}
\E(S_4) &=& \E\big[\E(S_4 \ \big | \ \bX_1, \bZ_1)\big] \cr
&=& \E\Big[  g\big((\bX_1, \bZ_1), \bZ_1,  (\bX_1,\bZ_1) \big) \cdot n \E \big( \Ind(M(1)=N(2)) \ \big | \ \bX_1, \bZ_1\big)\Big].
\end{eqnarray*}
By Lemma \ref{lemma:two_NNGs}, 
\begin{eqnarray}
&& n \E \big( \Ind(M(1)=N(2)) \ \big | \ \bX_1, \bZ_1\big) \cr
&=&\frac{n}{n-1}\E\big[\#\{j\in \lbr n \rbr:j \neq 1, N(j) = M(1)\} \ \big | \ \bX_1, \bZ_1 \big] \ = \ 1 +o_\P(1).  \label{eq:lemA:aux-2-8:4}
\end{eqnarray}
By \cite{MR682809} (Corollary S1), $\#\{j\in \lbr n \rbr:j \neq 1, N(j) = M(1)\}$ is bounded from above. 
Also, $g$ is bounded from above. Therefore, by bounded convergence theorem, 
\begin{eqnarray*}
\E(S_4) = \E\Big\{ g\big((\bX_1, \bZ_1), \bZ_1, (\bX_1, \bZ_1) \Big\}+o(1) = \E(S_5) + o(1). 
\end{eqnarray*} 

Finally, putting the above cases together yields that
\begin{eqnarray*}
|\E(S_1) - \E(S_5)| &\leq& 	|\E(S_1) - \E(S_4)| + 	|\E(S_4) - \E(S_5)| \cr
&\leq &
\E(|S_1-S_2|) + \E(|S_2-S_3|) + \E(|S_3-S_4|) + |\E(S_4) - \E(S_5)| \cr
&=&
O(\epsilon+\delta).
\end{eqnarray*}
Since $\epsilon$ and $\delta$ are arbitrary, we get $\lim_{n\to \infty} \E(S_1) = \E(S_5)$. This completes the proof. \end{proof}

\begin{lemmaA} \label{lemA:aux-2-9}
Recall the function $\tildeg$ and $\tildeg^\dagger$ defined in \eqref{eq:lemA:aux-2-6:1} and \eqref{eq:lemA:aux-2-7:1}. We have
\begin{eqnarray*}
&&\lim_{n\to \infty} n\cdot 	\E  \Big\{ \tildeg\big((\bX_1,\bZ_1), (\bX_{M(1)},\bZ_{M(1)})\big) \cdot \tildeg^\dagger \big((\bX_{M(1)},\bZ_{M(1)}), \bZ_2\big) \cdot \Ind\big(M(1)=N(2)\big) \Big\}  \cr
&=&
\E  \Big\{ \tildeg\big((\bX_1,\bZ_1), (\bX_1,\bZ_1)\big) \cdot \tildeg^\dagger \big((\bX_1,\bZ_1), \bZ_1\big)\Big\}  
\end{eqnarray*}	
\end{lemmaA}

\begin{proof} 
The proof follows the similar procedure as that of Lemma \ref{lemA:aux-2-8}.
Fix some $\epsilon >0$ and $\delta >0$. From Lemma \ref{lemA:aux-2-1}, there exists a simple function $	q(u,v) = \sum_{j=1}^m c_j \Ind_{B_j}(u)\Ind_{C_j}(v)$ such that $	\sup_{(u,v)\in \mathbbR^2} | q(u,v)- F_Y(u \wedge v)  | <\epsilon$. Define the functions 
\begin{eqnarray*}
r\big( (\bx_1,\bz_1), (\bx_2, \bz_2)\big) &:=& \int \int q(u,v) 	\d \tmu_{(\bx_1,\bz_1)}(u) \d \tmu_{(\bx_2, \bz_2)}(v) \cr
&=&
\sum_{j=1}^m c_j \cdot \tmu_{(\bx_1,\bz_1)}(B_j) \tmu_{(\bx_2, \bz_2)}(C_j)  ,\cr
r^\dagger\big( (\bx_1,\bz_1), \bz_2\big) &:=& \int \int q(u,v) 	\d \tmu_{(\bx_1,\bz_1)}(u) \d \tmu_{\bz_2}(v) \cr
&=&
\sum_{j=1}^m c_j \cdot \tmu_{(\bx_1,\bz_1)}(B_j) \tmu_{\bz_2}(C_j).
\end{eqnarray*}
Then 
\begin{eqnarray}
&& \sup_{((\bx_1,\bz_1),(\bx_2, \bz_2)) \in \mathbbR^{2p+2q}} \Big|\tildeg\big( (\bx_1,\bz_1), (\bx_2, \bz_2)\big)- r\big( (\bx_1,\bz_1), (\bx_2, \bz_2)\big) \Big| < \epsilon, \cr
&& \sup_{((\bx_1,\bz_1),\bz_2) \in \mathbbR^{p+2q}} \Big|\tildeg^\dagger\big( (\bx_1,\bz_1), \bz_2\big)- r^\dagger\big( (\bx_1,\bz_1),  \bz_2\big) \Big| < \epsilon.\label{eq:lemA:aux-2-9:1}
\end{eqnarray} 
Since functions $\tildeg$ and $\tildeg^\dagger$ are bounded, then functions $r$ and $r^\dagger$ are also bounded. 

Define
\begin{eqnarray*}
R_1 &:=& n \cdot  \tildeg\big((\bX_1,\bZ_1), (\bX_{M(1)},\bZ_{M(1)})\big) \cdot \tildeg^\dagger \big((\bX_{M(1)},\bZ_{M(1)}), \bZ_2\big)  \cdot \Ind\big(M(1)=N(2)\big), \cr
R_2 &:=&  n \cdot r\big((\bX_1,\bZ_1), (\bX_{M(1)},\bZ_{M(1)})\big) \cdot r^\dagger \big((\bX_{M(1)},\bZ_{M(1)}), \bZ_2\big) \cdot \Ind\big(M(1)=N(2)\big), \cr
R_3 &:=& n \cdot r\big((\bX_1,\bZ_1), (\bX_1,\bZ_1)\big) \cdot r^\dagger \big((\bX_1,\bZ_1), \bZ_1\big)  \cdot \Ind\big(M(1)=N(2)\big), \cr
R_4 &:=& n \cdot  \tildeg\big((\bX_1,\bZ_1), (\bX_1,\bZ_1)\big) \cdot \tildeg^\dagger \big((\bX_1,\bZ_1), \bZ_1\big)  \cdot \Ind\big(M(1)=N(2)\big), \cr
R_5 &:=& \tildeg\big((\bX_1,\bZ_1), (\bX_1,\bZ_1)\big) \cdot \tildeg^\dagger \big((\bX_1,\bZ_1), \bZ_1\big) .
\end{eqnarray*}
We will bound the difference in expectations between each pair of consecutive terms.

\noindent \textbf{(i)} Case $|R_1-R_2|$. 

By \eqref{eq:lemA:aux-2-9:1} and the fact that $\tildeg$ and $\tildeg^\dagger$ are bounded, it is straightforward that
\begin{eqnarray*}
&&\sup_{\big((\bx_1,\bz_1), (\bx_2, \bz_2), (\bx_3,\bz_3), \bz_4\big) \in \mathbbR^{3p+4q}} \Big|\tildeg\big( (\bx_1,\bz_1), (\bx_2, \bz_2)\big) \cdot \tildeg^\dagger\big( (\bx_3,\bz_3), \bz_4\big) \cr
&& \hspace{6cm}	-
r\big( (\bx_1,\bz_1), (\bx_2, \bz_2)\big) \cdot r^\dagger\big( (\bx_3,\bz_3), \bz_4\big)
\Big| < O(\epsilon).
\end{eqnarray*}
Then 
\begin{eqnarray*}
\E(|R_1-R_2|)
\leq O(\epsilon) \cdot  n \E\big[\Ind\big(M(1)=N(2)\big)\big].
\end{eqnarray*}
Using the similar argument as that of \eqref{eq:lemA:aux-2-8:1}, we get $\E(|R_1-R_2|) = O(\epsilon)$. 

\noindent \textbf{(ii)} Case $|R_2-R_3|$. 

Using the similar argument as that of \eqref{eq:lemA:aux-2-8:2},
for any $\epsilon_0>0$, we have
\begin{eqnarray*}
&&	\P\Big( \big|r\big((\bX_1,\bZ_1), (\bX_{M(1)},\bZ_{M(1)})\big) -  r\big((\bX_1,\bZ_1), (\bX_1,\bZ_1)\big) \big| > \epsilon_0  \ \Big | \ M(1)=N(2)\Big)  < \delta, \cr
&&	\P\Big( \big|r^\dagger \big((\bX_{M(1)},\bZ_{M(1)}), \bZ_2\big)  - r^\dagger \big((\bX_1,\bZ_1), \bZ_1\big) \big| > \epsilon_0  \ \Big | \ M(1)=N(2)\Big)  < \delta,
\end{eqnarray*}
holds for sufficiently large $n$.
Since the functions $r$ and $r^\dagger$ are bounded, we further get
\begin{eqnarray*}
&&\P\Big( \big|r\big((\bX_1,\bZ_1), (\bX_{M(1)},\bZ_{M(1)})\big) \cdot r^\dagger \big((\bX_{M(1)},\bZ_{M(1)}), \bZ_2\big)  \cr
&&  \qquad -r\big((\bX_1,\bZ_1), (\bX_1,\bZ_1)\big) \cdot r^\dagger \big((\bX_1,\bZ_1), \bZ_1\big)  \big| > \epsilon  \ \Big | \ M(1)=N(2)\Big)  < \delta
\end{eqnarray*}
holds for sufficiently large $n$. 
Then 
\begin{eqnarray*}
&&\E\Big( \big|r\big((\bX_1,\bZ_1), (\bX_{M(1)},\bZ_{M(1)})\big) \cdot r^\dagger \big((\bX_{M(1)},\bZ_{M(1)}), \bZ_2\big)  \cr
&&  \qquad -r\big((\bX_1,\bZ_1), (\bX_1,\bZ_1)\big) \cdot r^\dagger \big((\bX_1,\bZ_1), \bZ_1\big)  \big| > \epsilon  \ \Big | \ M(1)=N(2)\Big) = O(\epsilon+\delta).
\end{eqnarray*}
By applying the similar argument as in \eqref{eq:lemA:aux-2-8:3}, we have
$\E(|R_2-R_3|) = O(\epsilon+\delta)$.

\noindent \textbf{(iii)} Case $|R_3-R_4|$. 

This is similar to Case (i) for bounding $|R_1-R_2|$. We have $\E(|R_3-R_4|)=O(\epsilon)$. 

\noindent \textbf{(iv)} Case $R_4$. 
\begin{eqnarray*}
\E(R_4) &=& \E\big[\E(R_4 \mid \bX_1, \bZ_1)\big] \cr
&=& \E\Big[ \tildeg\big((\bX_1,\bZ_1), (\bX_1,\bZ_1)\big) \cdot \tildeg^\dagger \big((\bX_1,\bZ_1), \bZ_1\big)  \cdot n \E\big( \Ind(N(2) = M(1)) \mid  \bX_1, \bZ_1\big)\Big].
\end{eqnarray*}
By \eqref{eq:lemA:aux-2-8:4} and the fact that $\tildeg$ and $\tildeg^\dagger$ are bounded, we get
\begin{eqnarray*}
\E(R_4)= \E\Big[\tildeg\big((\bX_1,\bZ_1), (\bX_1,\bZ_1)\big) \cdot \tildeg^\dagger \big((\bX_1,\bZ_1), \bZ_1\big) \Big] + o(1) = \E(R_5) + o(1).
\end{eqnarray*}

Finally, putting the pieces together yields that $|\E(R_1)-\E(R_5)| = O(\epsilon+\delta)$. Since $\epsilon$ and $\delta$ are arbitrary, we get $\lim_{n\to \infty} \E(R_1) = \E(R_5)$. This completes the proof. \end{proof}

{
\bibliographystyle{apalike}
\bibliography{AMS}
}

\end{document}